\documentclass[a4paper,english]{amsart}
\usepackage[T1]{fontenc}
\usepackage[utf8]{inputenc}
\usepackage{color}
\usepackage{babel}
\usepackage{mathtools}
\usepackage{bm}
\usepackage{amstext}
\usepackage{amsthm}
\usepackage{amssymb}
\usepackage{esint}
\PassOptionsToPackage{normalem}{ulem}
\usepackage{ulem}
\usepackage[unicode=true,pdfusetitle,
 bookmarks=true,bookmarksnumbered=false,bookmarksopen=false,
 breaklinks=false,pdfborder={0 0 1},backref=false,colorlinks=false]
 {hyperref}
\usepackage{enumitem}

\makeatletter

\pdfpageheight\paperheight
\pdfpagewidth\paperwidth

\numberwithin{equation}{section}
\numberwithin{figure}{section}
\theoremstyle{plain}
\newtheorem{thm}{\protect\theoremname}[section]
\theoremstyle{plain}
\newtheorem{cor}[thm]{\protect\corollaryname}
\theoremstyle{remark}
\newtheorem{rem}[thm]{\protect\remarkname}
\theoremstyle{plain}
\newtheorem{prop}[thm]{\protect\propositionname}
\theoremstyle{plain}
\newtheorem{lem}[thm]{\protect\lemmaname}

\providecommand{\corollaryname}{Corollary}
\providecommand{\lemmaname}{Lemma}
\providecommand{\propositionname}{Proposition}
\providecommand{\remarkname}{Remark}
\providecommand{\theoremname}{Theorem}

\usepackage{tensor}

\definecolor{green}{rgb}{0,0.8,0} 

\newcommand{\nrm}{\@ifstar{\nrmb}{\nrmi}}
\newcommand{\nrmi}[1]{\Vert{#1}\Vert}
\newcommand{\nrmb}[1]{\left\Vert{#1}\right\Vert}
\newcommand{\abs}{\@ifstar{\absb}{\absi}}
\newcommand{\absi}[1]{\vert{#1}\vert}
\newcommand{\absb}[1]{\left\vert{#1}\right\vert}
\newcommand{\brk}{\@ifstar{\brkb}{\brki}}
\newcommand{\brki}[1]{\langle{#1}\rangle}
\newcommand{\brkb}[1]{\left\langle{#1}\right\rangle}
\newcommand{\set}{\@ifstar{\setb}{\seti}}
\newcommand{\seti}[1]{\{#1\}}
\newcommand{\setb}[1]{\left\{ #1\right\}}

\newcommand{\td}[1]{\widetilde{#1}}
\newcommand{\br}[1]{\overline{#1}}
\newcommand{\ul}[1]{\underline{#1}}
\newcommand{\wh}[1]{\widehat{#1}}

\newcommand{\VERT}[1]{{\left\vert\kern-0.25ex\left\vert\kern-0.25ex\left\vert #1 
    \right\vert\kern-0.25ex\right\vert\kern-0.25ex\right\vert}}

\DeclareMathOperator{\supp}{supp}

\let\Re\relax
\DeclareMathOperator{\Re}{Re}
\let\Im\relax
\DeclareMathOperator{\Im}{Im}

\newcommand{\aeq}{\sim}
\newcommand{\aleq}{\lesssim}
\newcommand{\ageq}{\gtrsim}

\newcommand{\lap}{\Delta}

\newcommand{\ud}{d}
\newcommand{\rd}{\partial}
\newcommand{\nb}{\nabla}

\newcommand{\peq}{\relphantom{=}}			

\newcommand{\alp}{\alpha}

\newcommand{\gmm}{\gamma}
\newcommand{\Gmm}{\Gamma}
\newcommand{\dlt}{\delta}
\newcommand{\Dlt}{\Delta}
\newcommand{\eps}{\epsilon}

\newcommand{\lmb}{\lambda}
\newcommand{\Lmb}{\Lambda}

\newcommand{\tht}{\theta}

\newcommand{\zt}{\zeta}

\newcommand{\bff}{{\bf f}}

\newcommand{\bfu}{{\bf u}}
\newcommand{\bfv}{{\bf v}}

\newcommand{\bfz}{{\bf z}}

\newcommand{\bfD}{{\bf D}}

\newcommand{\bfV}{{\bf V}}

\newcommand{\bbC}{\mathbb C}

\newcommand{\bbN}{\mathbb N}

\newcommand{\bbQ}{\mathbb Q}
\newcommand{\bbR}{\mathbb R}
\newcommand{\bbS}{\mathbb S}

\newcommand{\bbZ}{\mathbb Z}

\newcommand{\calA}{\mathcal A}

\newcommand{\calC}{\mathcal C}
\newcommand{\calD}{\mathcal D}
\newcommand{\calE}{\mathcal E}
\newcommand{\calF}{\mathcal F}

\newcommand{\calH}{\mathcal H}

\newcommand{\calK}{\mathcal K}
\newcommand{\calL}{\mathcal L}
\newcommand{\calM}{\mathcal M}
\newcommand{\calN}{\mathcal N}
\newcommand{\calO}{\mathcal O}
\newcommand{\calP}{\mathcal P}

\newcommand{\calS}{\mathcal S}
\newcommand{\calT}{\mathcal T}
\newcommand{\calU}{\mathcal U}

\newcommand{\calZ}{\mathcal Z}

\newcommand{\frkm}{\mathfrak m}

\newcommand{\To}{\longrightarrow}

\newcommand{\weakto}{\rightharpoonup}
\newcommand{\embed}{\hookrightarrow}

\newcommand{\lin}{\mathrm{lin}}
\newcommand{\remain}{\mathrm{rem}}

\newcommand{\rN}{\mathring{\calN}}				
\newcommand{\rV}{\mathring{V}}				

\makeatother

\providecommand{\corollaryname}{Corollary}
\providecommand{\lemmaname}{Lemma}
\providecommand{\propositionname}{Proposition}
\providecommand{\remarkname}{Remark}
\providecommand{\theoremname}{Theorem}

\begin{document}
\global\long\def\bbC{\mathbb{C}}%
\global\long\def\bbN{\mathbb{N}}%
\global\long\def\bbQ{\mathbb{Q}}%
\global\long\def\bbR{\mathbb{R}}%
\global\long\def\bbS{\mathbb{S}}%
\global\long\def\bbZ{\mathbb{Z}}%
\global\long\def\bfD{{\bf D}}%
\global\long\def\bfv{{\bf v}}%
\global\long\def\bfV{{\bf V}}%
\global\long\def\calA{\mathcal{A}}%
\global\long\def\calC{\mathcal{C}}%
\global\long\def\calD{\mathcal{D}}%
\global\long\def\calE{\mathcal{E}}%
\global\long\def\calF{\mathcal{F}}%
\global\long\def\calH{\mathcal{H}}%
\global\long\def\calK{\mathcal{K}}%
\global\long\def\calL{\mathcal{L}}%
\global\long\def\calM{\mathcal{M}}%
\global\long\def\calN{\mathcal{N}}%
\global\long\def\calU{\mathcal{U}}%
\global\long\def\calO{\mathcal{O}}%
\global\long\def\calP{\mathcal{P}}%
\global\long\def\calS{\mathcal{S}}%
\global\long\def\calT{\mathcal{T}}%
\global\long\def\calU{\mathcal{U}}%
\global\long\def\calZ{\mathcal{Z}}%
\global\long\def\frkm{\mathfrak{m}}%
\global\long\def\alp{\alpha}%
\global\long\def\dlt{\delta}%
\global\long\def\Dlt{\Delta}%
\global\long\def\eps{\epsilon}%
\global\long\def\gmm{\gamma}%
\global\long\def\zt{\zeta}%
\global\long\def\Gmm{\Gamma}%
\global\long\def\tht{\theta}%
\global\long\def\lmb{\lambda}%
\global\long\def\Lmb{\Lambda}%
\global\long\def\rd{\partial}%
\global\long\def\lap{\Delta}%
\global\long\def\aeq{\sim}%
\global\long\def\aleq{\lesssim}%
\global\long\def\ageq{\gtrsim}%
\global\long\def\lan{\langle}%
\global\long\def\ran{\rangle}%
\global\long\def\peq{\mathrel{\phantom{=}}}%
\global\long\def\To{\longrightarrow}%
\global\long\def\weakto{\rightharpoonup}%
\global\long\def\embed{\hookrightarrow}%
\global\long\def\chf{\mathbf{1}}%
\global\long\def\td#1{\widetilde{#1}}%
\global\long\def\br#1{\overline{#1}}%
\global\long\def\ul#1{\underline{#1}}%
\global\long\def\wh#1{\widehat{#1}}%
\global\long\def\tint#1#2{{\textstyle \int_{#1}^{#2}}}%
\global\long\def\tsum#1#2{{\textstyle \sum_{#1}^{#2}}}%

\global\long\def\CR{\mathbf{D}_{+}}%
\global\long\def\dec{\mathrm{dec}}%
\global\long\def\avg{\mathrm{(avg)}}%
\global\long\def\lin{\mathrm{lin}}%
\global\long\def\rN{\mathring{\calN}}%
\global\long\def\rV{\mathring{V}}%
\global\long\def\remain{\mathrm{rem}}%

\title[Blow-up construction of strongly interacting regime]{Blow-up dynamics for radial self-dual Chern--Simons--Schrödinger
equation with prescribed asymptotic profile }
\author{Kihyun Kim}
\email{kihyun.kim@snu.ac.kr}
\address{Department of Mathematical Sciences and Research Institute of Mathematics,
Seoul National University, 1 Gwanak-ro, Gwanak-gu, Seoul 08826, Korea}
\author{Soonsik Kwon}
\email{soonsikk@kaist.edu}
\address{Department of Mathematical Sciences, Korea Advanced Institute of Science
and Technology, 291 Daehak-ro, Yuseong-gu, Daejeon 34141, Korea}
\author{Sung-Jin Oh}
\email{sjoh@math.berkeley.edu}
\address{Department of Mathematics, UC Berkeley, Evans Hall 970, Berkeley,
CA 94720-3840, USA and Korea Institute for Advanced Study, 80 Hoegi-ro,
Dongdaemun-gu, Seoul 02455, Korea}
\begin{abstract}
We construct finite energy blow-up solutions for the radial self-dual Chern--Simons--Schr\"odinger equation with a continuum of blow-up rates. Our result stands in stark contrast to the rigidity of blow-up of $H^{3}$ solutions proved by the first author for equivariant index $m \geq 1$, where the soliton-radiation interaction is too weak to admit the present blow-up scenarios. It is optimal (up to an endpoint) in terms of the range of blow-up rates and the regularity of the asymptotic profiles, in view of the authors' previous proof of $H^{1}$ soliton resolution for the self-dual Chern--Simons--Schr\"odinger equation in any equivariance class. 

Our approach is a backward construction combined with modulation analysis, starting from prescribed asymptotic profiles and deriving the corresponding blow-up rates from their strong interaction with the soliton. In particular, our work may be seen as an adaptation of the method of Jendrej--Lawrie--Rodriguez (developed for energy critical equivariant wave maps) to the Schr\"odinger case. 
However, the Schrödinger nature of the equation (in particular, the lack of finite speed of propagation) and the optimal range (up to the $H^{1}$-endpoint) of our blow-up construction give rise to new challenges. Notably, the construction of (approximate) radiation from the prescribed asymptotic profile is one of our key novelties and might be of independent interest.
\end{abstract}

\maketitle
\tableofcontents{}

\section{Introduction}

The subject of this paper is the Chern--Simons--Schrödinger equation,
which is a Lagrangian field theory associated with the action 
\begin{equation}
\calS[\phi,A]\coloneqq\int_{\bbR^{1+2}}\Big(\frac{1}{2}\Im(\br{\phi}\bfD_{t}\phi)+\frac{1}{2}|\bfD_{x}\phi|^{2}-\frac{g}{4}|\phi|^{4}\Big)+\int_{\bbR^{1+2}}\frac{1}{2}A\wedge F.\label{eq:action}
\end{equation}
Here, $g\in\bbR$, $\phi:\bbR^{1+2}\to\bbC$ is a scalar field, $\bfD_{\alpha}\coloneqq\rd_{\alpha}+iA_{\alpha}$
for $\alpha\in\{t,1,2\}$ are the covariant derivatives associated
with the real-valued $1$-form $A\coloneqq A_{t}dt+A_{1}dx_{1}+A_{2}dx_{2}$,
and $F\coloneqq dA$ is the curvature. The Chern--Simons--Schrödinger
equation is introduced by the physicists Jackiw--Pi \cite{JackiwPi1990PRL}
as a nonrelativistic planar quantum electromagnetic model. Since the
action is simply the sum of the Chern--Simons action and the action
for the (gauge-covariant) cubic nonlinear Schrödinger equation, it
can be viewed as a gauged nonlinear Schrödinger equation. With the
special choice of the coupling constant $g=1$, the system becomes
\emph{self-dual.} Remarkably, Jackiw--Pi \cite{JackiwPi1990PRL}
observed that there exist explicit vortex solitons in the self-dual
case. We refer to \cite{JackiwPi1990PRL,JackiwPi1990PRD,JackiwPi1991PRD,JackiwPi1992Progr.Theoret.,Dunne1995Springer}
for more physical background and the self-duality.

Preceded by \cite{KimKwon2019arXiv,KimKwon2020arXiv,KimKwonOh2020arXiv},
we continue our study on the blow-up dynamics of the self-dual Chern--Simons--Schrödinger
equation under equivariance. Our main result (Theorem~\ref{thm:MainThm})
is on the construction of finite-time blow-up solutions with exotic
blow-up rates in the radial case with the optimal range of blow-up rates. The exotic blow-up regime arises from the strong interaction between the soliton and the radiation. We use the approach of Jendrej--Lawrie--Rodriguez \cite{JendrejLawrieRodriguez2019arXiv}, developed in the context of energy critical one-equivariant wave maps. However, the Schr\"odinger case requires a new idea in constructing the radiation solution from the asymptotic profile, due to a different propagation of singularity for the Schr\"odinger equations.

\subsection{Self-dual Chern--Simons--Schrödinger equation}

We consider the \emph{self-dual Chern--Simons--Schrödinger} equation
\begin{equation}
\left\{ \begin{aligned}\bfD_{t}\phi & =i\bfD_{j}\bfD_{j}\phi+i|\phi|^{2}\phi,\\
F_{t1} & =-\Im(\br{\phi}\bfD_{2}\phi),\\
F_{t2} & =\Im(\br{\phi}\bfD_{1}\phi),\\
F_{12} & =-\tfrac{1}{2}|\phi|^{2},
\end{aligned}
\right.\label{eq:CSS-cov}
\end{equation}
which is the Euler--Lagrange equation associated with the action \eqref{eq:action}
with $g=1$. Repeated index $j$ means that we sum over $j\in\{1,2\}$.
The particular choice of the coefficient $g=1$ in front of $|\phi|^{2}\phi$
is referred to as the \emph{self-dual} case.

We remark that \eqref{eq:CSS-cov} has \emph{gauge invariance}; any
solution $(\phi,A)$ to \eqref{eq:CSS-cov} and a function $\chi:\bbR^{1+2}\to\bbR$
give rise to a gauge-equivalent solution $(e^{i\chi}\phi,A-d\chi)$
to \eqref{eq:CSS-cov}.

\subsubsection*{Symmetries and conservation laws}

\eqref{eq:CSS-cov} enjoys various gauge-covariant\footnote{Each symmetry described here consists of a pre-composition of $\phi$
with a coordinate transform $(t',x')\mapsto(t,x)$ and a further transformation
of the resulting $\phi(t',x')$. \emph{Gauge covariance} refers to
the feature that the $1$-form $A$ is simply pulled back by $(t',x')\mapsto(t,x)$.} symmetries and associated conservation laws. In this aspect, \eqref{eq:CSS-cov}
shares many similarities with the cubic NLS 
\begin{equation}
i\rd_{t}\psi+\Delta\psi+|\psi|^{2}\psi=0\quad\text{on }\bbR^{1+2}.\tag{NLS}\label{eq:NLS}
\end{equation}

Among the most basic symmetries are the \emph{time translation symmetry}
\[
(t,x)=(t'+t_{0},x'),\quad\phi'=\phi,\qquad(t_{0}\in\bbR)
\]
and the \emph{phase rotation symmetry} 
\[
(t,x)=(t',x'),\quad\phi'=e^{i\gmm}\phi.\qquad(\gmm\in\bbR)
\]
Associated to these symmetries are the conservation laws for the \emph{energy}
and the \emph{charge}:\footnote{We use the conventional notation $M$ for the charge.}
\begin{align}
E[\phi,A] & \coloneqq\int_{\bbR^{2}}\frac{1}{2}|\bfD_{x}\phi|^{2}-\frac{1}{4}|\phi|^{4}\,dx\label{eq:energy-cov}\\
M[\phi] & \coloneqq\int_{\bbR^{2}}|\phi|^{2}\,dx.\nonumber 
\end{align}
Of particular importance in this work are the \emph{scaling symmetry},
\[
(t,x)=(\lmb^{-2}t',\lmb^{-1}x'),\quad\phi'=\lmb^{-1}\phi,\qquad(\lmb>0)
\]
which preserves the $L^{2}$-norm (or $M[\phi]$), and the \emph{pseudoconformal
symmetry}, 
\begin{equation}
(t,x)=(-\tfrac{1}{t'},\tfrac{x'}{t'}),\quad[\calC\phi](t',x')=\phi'(t',x')=\tfrac{1}{t'}e^{i\frac{|x'|^{2}}{4t'}}\phi.\label{eq:def-pseudoconf}
\end{equation}
Thus \eqref{eq:CSS-cov} and \eqref{eq:NLS} are called \emph{mass-critical}.
Associated to these symmetries are the \emph{virial identities} 
\begin{equation}
\left\{ \begin{aligned}\rd_{t}\left(\int_{\bbR^{2}}|x|^{2}|\phi|^{2}dx\right) & =4\int_{\bbR^{2}}x^{j}\Im(\br{\phi}\bfD_{j}\phi)dx,\\
\rd_{t}\left(\int_{\bbR^{2}}x^{j}\Im(\br{\phi}\bfD_{j}\phi)dx\right) & =4E[\phi,A].
\end{aligned}
\right.\label{eq:virial-cov}
\end{equation}

\subsubsection*{Self-duality}

Introducing the \emph{covariant Cauchy--Riemann operator} 
\[
\CR\coloneqq\bfD_{1}+i\bfD_{2},
\]
a remarkable consequence is the self-dual expression of the energy
\begin{equation}
E[\phi,A]=\frac{1}{2}\int_{\bbR^{2}}|\CR\phi|^{2}\,dx,\label{eq:energy-sd}
\end{equation}
provided $F_{12}=-\frac{1}{2}|\phi|^{2}$. Therefore, the energy minimizers
(i.e., zero energy solutions) obey the \emph{Bogomol'nyi equation}
\begin{equation}
\left\{ \begin{aligned}\CR\phi & =0,\\
F_{12} & =-\frac{1}{2}|\phi|^{2}.
\end{aligned}
\right.\label{eq:bog}
\end{equation}
The last property is the manifestation of \emph{self-duality}. Any
zero-energy solution (or equivalently, a solution to \eqref{eq:bog})
is a static (i.e., $\rd_{t}\phi=0$) solution to \eqref{eq:CSS-cov}
with $A_{t}=-\tfrac{1}{2}|\phi|^{2}$. Conversely, any static solution
with $\phi\in H^{1}$ and mild conditions on $A_{t},A_{j}$ (e.g.,
boundedness) necessarily has zero energy and $A_{t}=-\tfrac{1}{2}|\phi|^{2}$
\cite{HuhSeok2013JMP}.

It was observed by Jackiw--Pi \cite{JackiwPi1990PRD} that, at points
where $\phi$ is nonzero, \eqref{eq:bog} implies that $|\phi|^{2}$
solves the Liouville equation $\Delta(\log|\phi|^{2})=-|\phi|^{2}$.
Using this, they found explicit $m$-equivariant static solutions
(see \eqref{eq:Q-formula} below) for all $m\geq0$.

\subsubsection*{Cauchy problem formulation and Coulomb gauge}

Due to the gauge invariance, we need to fix a gauge in order to consider
the Cauchy problem of \eqref{eq:CSS-cov}. In this paper, we impose
the \emph{Coulomb gauge condition}: 
\begin{equation}
\rd_{1}A_{1}+\rd_{2}A_{2}=0,\label{eq:coulomb}
\end{equation}
along with a suitable decay condition for $A(t,x)$ as $|x|\to\infty$
(at every $t$). We mention that the connection 1-form is determined
by $\phi$ (i.e., $A=A[\phi]$), and hence \eqref{eq:CSS-cov} in
Coulomb gauge can be viewed as an evolution equation solely for $\phi$.
As observed in \cite{JackiwPi1990PRD}, it admits the following Hamiltonian
formulation: 
\begin{equation}
\rd_{t}\phi=-i\nabla E[\phi],\label{eq:CSS-ham}
\end{equation}
where $\nabla$ (acting on a functional) is the Fréchet derivative
with respect to the real inner product $\int_{\bbR^{2}}\Re(\br{\psi}\phi)dx$,
and $E[\phi]=E[\phi,A[\phi]]$ is the energy.

\subsection{Equivariant self-dual Chern--Simons--Schrödinger equation}

In this work, we will work within equivariance symmetry (and Coulomb
gauge). A complex-valued function $\psi$ on $\bbR^{2}$ is said to
be \emph{$m$-equivariant}, $m\in\bbZ$, if 
\[
\psi(r,\tht)=e^{im\tht}v(r),
\]
where $(r,\tht)$ are the polar coordinates. Note that $m=0$ corresponds
to radial symmetry.

Within $m$-equivariance and Coulomb gauge, $A_{t}$,
$A_{r}$, $A_{\tht}$ become radial and the Coulomb gauge condition reduces
to 
\[
A_{r}=0.
\]
The radial profile $u$ of $\phi$, defined by 
\[
\phi(t,r,\tht)=e^{im\tht}u(t,r),
\]
obeys

\begin{equation}
i\rd_{t}u+\rd_{rr}u+\frac{1}{r}\rd_{r}u-\Big(\frac{m+A_{\tht}[u]}{r}\Big)^{2}u-A_{t}[u]u+|u|^{2}u=0,\tag{CSS}\label{eq:CSS-m-equiv}
\end{equation}
where $A_{t}[u]$ and $A_{\tht}[u]$ are given by 
\begin{equation}
A_{t}[u]=-\int_{r}^{\infty}(m+A_{\tht}[u])|u|^{2}\frac{dr'}{r'},\qquad A_{\tht}[u]=-\frac{1}{2}\int_{0}^{r}|u|^{2}r'dr'.\label{eq:def-A}
\end{equation}
Equations \eqref{eq:CSS-m-equiv} and \eqref{eq:def-A} furnish an
evolutionary equation for the radial profile $u$ of an $m$-equivariant
solution $\phi$ to \eqref{eq:CSS-cov} in Coulomb gauge.

The Cauchy--Riemann operator $\bfD_{+}$, with $A$ determined by
$u$ through \eqref{eq:def-A}, can be written in polar coordinates
as 
\[
\bfD_{+}=e^{i\tht}(\rd_{r}+\frac{i}{r}\rd_{\tht}-\frac{1}{r}A_{\tht}[u]).
\]
Thus, it maps $m$-equivariant functions to $(m+1)$-equivariant functions,
and we denote by $\bfD_{u}$ the radial operator acting on the radial
profile of $m$-equivariant functions defined by the relation 
\[
\bfD_{+}(e^{im\tht}w(r))=e^{i(m+1)\tht}[\bfD_{u}w](r).
\]
We have the formula 
\begin{equation}
\bfD_{u}w=\rd_{r}w-\frac{1}{r}(m+A_{\tht}[u])w.\label{eq:CR-radial}
\end{equation}
We call the (nonlinear) operator $u\mapsto\bfD_{u}u$, the \emph{Bogomol'nyi
operator} under $m$-equivariance.

Moreover, we can rewrite the energy functional for the radial profile
$u$ as (cf. \eqref{eq:energy-cov} and \eqref{eq:energy-sd})

\begin{align}
E[u] & =\int\frac{1}{2}|\rd_{r}u|^{2}+\frac{1}{2}\Big(\frac{m+A_{\tht}[u]}{r}\Big)^{2}|u|^{2}-\frac{1}{4}|u|^{2},\label{eq:energy-Coulomb-form}\\
 & =\int\frac{1}{2}|\bfD_{u}u|^{2},\label{eq:energy-self-dual-form}
\end{align}
where we denoted $\int f=2\pi\int_{0}^{\infty}f\,rdr$. The virial
identities (cf. \eqref{eq:virial-cov}) read 
\begin{align}
\rd_{t}\int r^{2}|u|^{2} & =4\int\Im(\br u\cdot r\rd_{r}u),\label{eq:virial-1}\\
\rd_{t}\int\Im(\br u\cdot r\rd_{r}u) & =4E[u].\label{eq:virial-2}
\end{align}

The Hamiltonian structure \eqref{eq:CSS-ham} reads 
\begin{equation}
\rd_{t}u=-i\nabla E[u].\label{eq:CSS-m-equiv-ham}
\end{equation}
As observed in \cite{KimKwon2019arXiv}, by the Hamiltonian structure
\eqref{eq:CSS-m-equiv-ham} and the energy expression \eqref{eq:energy-self-dual-form},
the evolution equation for $u$ can also be written in the self-dual
form: 
\begin{equation}
\rd_{t}u=-iL_{u}^{\ast}\bfD_{u}u,\label{eq:CSS-self-dual-form}
\end{equation}
where $L_{u}$ is obtained by linearizing the Bogomol'nyi operator
around $u$, and $L_{u}^{\ast}$ is its $L^{2}$-adjoint. Note that
$L_{u}$ is a nonlocal operator, which takes the form 
\[
L_{u}w=\rd_{r}w-\frac{1}{r}(m+A_{\tht}[u])w+\frac{u}{r}\int_{0}^{r}\Re(\br uw)r'dr'.
\]
See Section~\ref{subsec:Linearization-of-CSS} for more details.

Finally, for each $m\geq0$, there is a unique solution (Jackiw--Pi vortex) to the Bogomol'nyi equation 
\begin{equation}
\bfD_{Q}Q=0.\label{eq:Bogomol'nyi-eq}
\end{equation}
This gives an \emph{explicit} $m$-equivariant
static solution to \eqref{eq:CSS-self-dual-form} unique up to the symmetries of the equation: 
\begin{equation}
Qe^{im\tht}=\sqrt{8}(m+1)\frac{r^{m}}{1+r^{2m+2}}e^{im\tht},\qquad m\geq0.\label{eq:Q-formula}
\end{equation}
For the simplicity of notation, we suppressed the $m$-dependences
in $\bfD_{u},L_{u},Q$.

We remark that the dynamics of \eqref{eq:CSS-m-equiv} for $m\geq0$
and $m<0$ are completely different, since in the latter case there
are no nontrivial Jackiw--Pi vortices \cite{KimKwonOh2022arXiv1}.
There is a time reversal symmetry\footnote{The time reversal symmetry is \emph{not} simply given by conjugating
the scalar field $\phi$; it reads $\phi(t,x_{1},x_{2})\mapsto\br{\phi}(-t,x_{1},-x_{2})$,
$A_{\alpha}(t,x_{1},x_{2})\mapsto A_{\alpha}(-t,x_{1},-x_{2})$ for
$\alpha\in\{0,2\}$, and $A_{1}(t,x_{1},x_{2})\mapsto-A_{1}(-t,x_{1},-x_{2})$.} for \eqref{eq:CSS-cov}, but it \emph{does not} flip the equivariance
index. This is a sharp contrast to the \eqref{eq:NLS} case, where
there is a simple space-reflection symmetry $(t,x_{1},x_{2})=(t',x_{1}',-x_{2}')$
that can flip the equivariance index.

\subsection{\label{subsec:Known-results}Known results}

We briefly discuss the known results and motivations on our work.
Indeed, many results cover the non-self-dual case, where $i|\phi|^{2}\phi$
of \eqref{eq:CSS-cov} is replaced by $ig|\phi|^{2}\phi$ with arbitrary
$g\in\bbR$. In the following, the results hold true for general $g$,
unless otherwise mentioned.

The well-posedness of \eqref{eq:CSS-cov} was first studied in Coulomb
gauge; after the earlier works \cite{BergeDeBouardSaut1995Nonlinearity,Huh2013Abstr.Appl.Anal},
Lim \cite{Lim2018JDE} proved $H^{1}$-local well-posedness, though
the scaling-critical regularity is $L^{2}$. Under the heat gauge,
small data $H^{0+}$ local well-posedness is proved by Liu--Smith--Tataru
\cite{LiuSmithTataru2014IMRN}. Under equivariance within Coulomb
gauge, the equation becomes semilinear and the $L^{2}$-critical well-posedness
can be achieved by a standard application of the Strichartz estimates;
see \cite[Section 2]{LiuSmith2016}.

There are also works on the long-term dynamics. Bergé--de Bouard--Saut
\cite{BergeDeBouardSaut1995Nonlinearity} used Glassey's convexity
argument \cite{Glassey1977JMP} to derive a sufficient condition for
finite-time blow up. However, this method essentially applies for
negative energy solutions, which exist only in the focusing non-self-dual
case ($g>1$). The same authors \cite{BergeDeBouardSaut1995PRL} carried
out a formal computation to derive the log-log blow-up for negative
energy solutions, similarly to the mass-critical NLS case. Recently,
Oh--Pusateri \cite{OhPusateri2015} showed global existence and scattering
for small data in weighted Sobolev spaces.

From now on, we fix $g=1$. Much more is known under equivariance
within Coulomb gauge. The static solution $Q$ plays a central role
in describing the global dynamics. Liu--Smith \cite{LiuSmith2016}
proved the \emph{subthreshold theorem}: for each equivariance class
$m\geq0$, there hold the global well-posedness and scattering below
the charge of the ground state, $M[Q]$. Recently, Li--Liu \cite{LiLiu2020arXiv}
considered the threshold dynamics ($m\geq0$ again) and proved that
any $H^{1}$-solution to \eqref{eq:CSS-m-equiv} with charge $M[Q]$
must be, up to obvious symmetries, either (i) the static solution
$Q$, (ii) the explicit finite-time blow-up solution\footnote{When $m=0$, one must exclude this scenario due to $S(t)\notin H^{1}$.}
\[
S(t,r)=\frac{1}{|t|}Q\Big(\frac{r}{|t|}\Big)e^{-i\frac{r^{2}}{4|t|}},\qquad t<0
\]
which is obtained by taking the pseudoconformal transform to $Q$,
or (iii) a global scattering solution. We remark that the results
\cite{LiuSmith2016,LiLiu2020arXiv} also treat the non-self-dual case. Moreover, in the non-self-dual case, the threshold rigidity result was very recently extended to $L^{2}$-solutions by Dodson \cite{Dodson2024arXiv}.

Beyond the threshold, it is generally believed that \emph{soliton
resolution} holds, i.e., any maximal solutions to \eqref{eq:CSS-m-equiv}
must decompose into the sum of modulated solitons and a radiation.
Recently, the authors \cite{KimKwonOh2022arXiv1} established
\emph{soliton resolution for solutions in a weighted Sobolev class}
($H_{m}^{1,1}$) with a striking fact that \emph{the nonscattering
part must be a single soliton}. In particular, there is no multisoliton
configuration. When $m<0$, there is no soliton and every $H_{m}^{1,1}$-solution
scatters.

The soliton resolution result \cite{KimKwonOh2022arXiv1} in fact
applies to any finite energy finite-time blow-up solutions, which
we recall as follows. Here, we denote the modulated soliton by 
\[
Q_{\lmb,\gmm}(r)\coloneqq\frac{e^{i\gmm}}{\lmb}Q\Big(\frac{r}{\lmb}\Big),\qquad\lmb\in(0,\infty),\ \gmm\in\bbR/2\pi\bbZ,
\]
and denote by $H_{m}^{1}$ and $H_{m}^{1,1}$ the (weighted) Sobolev
spaces $H^{1,1}$ and $H^{1}$ restricted to $m$-equivariant functions. 
\begin{thm}[Soliton resolution for $H_{m}^{1}$ finite-time blow-up solutions
\cite{KimKwonOh2022arXiv1}]
\label{thm:asymptotic-description}Let $m\geq0$; let $u$ be a $H_{m}^{1}$-solution
to \eqref{eq:CSS-m-equiv} which blows up forwards in time at $T\in(0,\infty)$.
Then, $u(t)$ admits the decomposition 
\begin{equation}
u(t,\cdot)-Q_{\lmb(t),\gmm(t)}\to z^{\ast}\text{ in }L^{2}\text{ as }t\to T^{-},\label{eq:thm1-decomp}
\end{equation}
for some $\lmb(t)\in(0,\infty)$, $\gmm(t)\in\bbR/2\pi\bbZ$, and
$z^{\ast}\in L^{2}$ with the following properties: 
\begin{itemize}
\item (Further regularity of $z^{\ast}$) We have $\rd_{r}z^{\ast},\frac{1}{r}z^{\ast}\in L^{2}.$
Moreover, if $u$ is a $H_{m}^{1,1}$ finite-time blow-up solution,
then we also have $rz^{\ast}\in L^{2}$. 
\item (Bound on the blow-up speed) As $t\to T$, we have 
\begin{equation}
\lmb(t)\aleq_{M[u]}\sqrt{E[u]}(T-t).\label{eq:thm1-lmb-upper-bound}
\end{equation}
\item (Improved bound on the blow-up speed when $m=0$) When $m=0$, we
further have 
\begin{equation}
\lmb(t)\aleq_{M[u]}\frac{\sqrt{E[u]}(T-t)}{|\log(T-t)|^{\frac{1}{2}}}.\label{eq:thm1-lmb-upper-bound-radial}
\end{equation}
\end{itemize}
\end{thm}

In particular, any finite-time blow-up solution must be decomposed
in the form \eqref{eq:thm1-decomp}. A natural question is the refined
dynamics of \eqref{eq:thm1-decomp}. For example, we may ask for which
$\lmb(t)$ (and $\gmm(t)$) the blow-up solutions exist. We may also
ask in what extent these blow-up dynamics are stable.

When $m\geq1$, the first and second authors gave a quantitative description
of the dynamics in the vicinity of $S(t)$, and in particular studied
the \emph{pseudoconformal blow-up solutions}, i.e., 
\[
u(t,r)-\frac{e^{i\gmm^{\ast}}}{\ell(T-t)}Q\Big(\frac{r}{\ell(T-t)}\Big)\to z^{\ast}\quad\text{in }L^{2}
\]
for some $\gmm^{\ast}\in\bbR$ and $\ell\in(0,\infty)$ as $t\to T$.
On one hand, using backward construction, the authors in \cite{KimKwon2019arXiv}
constructed pseudoconformal blow-up solutions with a prescribed asymptotic
profile. Moreover, they exhibited an instability mechanism (the \emph{rotational
instability}) of these solutions. This is an analogue of the construction
of Bourgain--Wang solutions and their instability in the NLS context
\cite{BourgainWang1997,MerleRaphaelSzeftel2013AJM}. On the other
hand, in \cite{KimKwon2020arXiv}, using forward construction, they
studied conditional stability of pseudoconformal blow-up solutions
in the context of the Cauchy problem. Indeed, they constructed a codimension
one set of (smooth and finite energy) initial data leading to pseudoconformal
blow-up. Moreover, when $m\geq3$, they showed that the constructed
codimension one data set is Lipschitz. In view of the rotational instability
in \cite{KimKwon2019arXiv}, the codimension one condition seems to
be optimal.

When $m=0$, the explicit pseudoconformal blow-up solution $S(t)$
no longer has finite energy, due to the slow spatial decay of $Q\sim\langle r\rangle^{-2}$.
In \cite{KimKwonOh2020arXiv}, using forward construction, the authors
constructed a codimension one set of smooth finite energy initial
data leading to finite-time blow-up with the blow-up rate $\lmb(t)\sim|\log(T-t)|^{-2}(T-t)$.
Note that the pseudoconformal blow-up is excluded for finite energy
solutions due to \eqref{eq:thm1-lmb-upper-bound-radial}.

Finally, we refer the reader to Dodson \cite{Dodson2024DCDS, Dodson2023arXiv} for very recent progress on understanding above-threshold solutions merely in $L^{2}$.

\subsection{Main result}

Our main result is on the construction of finite energy finite-time
blow-up solutions with much wider range of blow-up rates in the radial
case $m=0$. More precisely, for all $p>1$, we construct such solutions
with the blow-up rate $\lmb(t)\sim|\log(T-t)|^{-1}(T-t)^{p}$. Note
that the range $p>1$ is almost optimal in view of \eqref{eq:thm1-lmb-upper-bound-radial}.
However, it turns out to be \emph{optimal} in view of the regularity
of the asymptotic profile $z^{\ast}$; see Remark~\ref{rem:optimality}
below. The blow-up solutions considered in this paper are derived
from the \emph{strong soliton-radiation interaction} in the radial
case. This is a consequence of the fact that the soliton $Q$ exhibits
the weakest spatial decay when $m=0$, and hence the result is restricted
to the radial case. 
\begin{thm}[Exotic finite-time blow-up solutions]
\label{thm:MainThm}Let $m=0$. Let $q,\nu\in\bbC$ with $\Re(\nu)>0$
and $q\neq0$. Define the asymptotic profile 
\begin{equation}
z^{\ast}(r)\coloneqq qr^{\nu}\chi(r),\label{eq:def-asymp-profile}
\end{equation}
where $\chi(r)$ is a smooth radial cutoff function with $\chi(r)=1$
for $r\leq1$ and $\chi(r)=0$ for $r\geq2$. Then, there exist $t_{0}^{\ast}<0$
and a $H_{0}^{1,1}$-solution $u(t)$ on the time interval $[t_{0}^{\ast},0)$
that satisfies 
\begin{equation}
u(t)-Q_{\lmb_{q,\nu}(t),\gmm_{q,\nu}(t)}\to z^{\ast}\text{ in }L^{2}\text{ as }t\to0^{-},\label{eq:thm-decomp}
\end{equation}
where the parameters $\lmb_{q,\nu}(t)\in(0,\infty)$ and $\gmm_{q,\nu}(t)\in\bbR/2\pi\bbZ$
are defined by 
\begin{equation}
\lmb_{q,\nu}(t)e^{i\gmm_{q,\nu}(t)}=-\frac{\sqrt{2}}{4}\frac{\Gmm(\tfrac{\nu}{2})}{\Re(\nu)+1}\cdot q\frac{(4it)^{\frac{\nu}{2}+1}}{|\log|t||}.\label{eq:lmb-gmm-formula}
\end{equation}
In particular, we have 
\[
\lmb_{q,\nu}(t)\sim_{q,\nu}\frac{|t|^{\frac{\Re(\nu)}{2}+1}}{|\log|t||}.
\]
\end{thm}

Our result is largely inspired from the work \cite{JendrejLawrieRodriguez2019arXiv}
on corotational wave maps by Jendrej, Lawrie, and Rodriguez. As in
\cite{JendrejLawrieRodriguez2019arXiv}, we construct the radiation
from the asymptotic profile and then capture the strong interaction
between the radiation and the blow-up profile. As opposed to the wave
maps, the construction of the radiation in the present case is more
nontrivial, and we also deal with the optimal range $\nu$ (which is
$\Re(\nu)>0$). See Section~\ref{subsec:Strategy-of-the-proof} for
more details.

Applying the pseudoconformal transform to the solution constructed
in Theorem~\ref{thm:MainThm}, we obtain infinite-time blow-up solutions (cf.~the global solutions case in \cite[Theorem~1.1]{KimKwonOh2022arXiv1}).
\begin{cor}[Infinite-time blow-up]
\label{cor:InfiniteTimeBlowUp}Let $m=0$. Let $q$, $\nu$, and
$z^{\ast}$ be as in Theorem~\ref{thm:MainThm}. Let $u^{\ast}\in H_{0}^{1,1} \cap r L^{2}$
be defined by
\begin{equation}
u^{\ast}(\rho)=-\frac{i}{2}\int_{0}^{\infty}J_{-2}(\frac{1}{2}\rho r)z^{\ast}(r)rdr,\label{eq:def-u-ast}
\end{equation}
where $J_{-2}$ is the Bessel function of the first kind of order $-2$, so that $u^{\ast}$ satisfies
$[e^{it\Delta^{(-2)}}u^{\ast}](r)=\calC[e^{it\Delta^{(-2)}}z^{\ast}](t,r)$,
where\footnote{Formally, $\Delta^{(-2)}$ coincides with the Laplacian applied to a $(-2)$-equivariant function $u(r) e^{-2 i \tht}$ on $\bbR^{2}$ written as an operator acting on $u(r)$, hence our notation.} $\Delta^{(-2)} = \rd_{r}^{2} + \frac{1}{r} \rd_{r} - \frac{4}{r^{2}}$. Then, there exist $t_{0}>0$
and a $H_{0}^{1,1}$-solution $u$ on the time interval $[t_{0},\infty)$
such that 
\[
u(t)-Q_{\wh{\lmb}_{q,\nu}(t),\wh{\gmm}_{q,\nu}(t)}-e^{it\Delta^{(-2)}}u^{\ast}\to0\text{ in }L^{2}
\]
as $t\to\infty$, where the modulation parameters $\wh{\lmb}_{q,\nu}(t)\in(0,\infty)$
and $\wh{\gmm}_{q,\nu}(t)\in\bbR/2\pi\bbZ$ are defined by 
\[
\wh{\lmb}_{q,\nu}(t)e^{i\wh{\gmm}_{q,\nu}(t)}\coloneqq t\cdot\lmb_{q,\nu}(-\tfrac{1}{t})e^{i\wh{\gmm}_{q,\nu}(-\frac{1}{t})}=-\frac{\sqrt{2}}{4}\frac{\Gmm(\tfrac{\nu}{2})}{\Re(\nu)+1}\cdot q\frac{t\cdot(-4i/t)^{\frac{\nu}{2}+1}}{\log t}.
\]
In particular, we have 
\[
\wh{\lmb}_{q,\nu}(t)\sim_{q,\nu}\frac{1}{t^{\frac{\Re(\nu)}{2}}\log t}.
\]
\end{cor}

\begin{rem}[Optimality]
\label{rem:optimality} As long as we consider an asymptotic profile
of the form \eqref{eq:def-asymp-profile}, our range of $\nu$ in
Theorem~\ref{thm:MainThm} is \emph{optimal}. Indeed, when $\Re(\nu)\leq0$,
we have $\frac{1}{r}z^{\ast}\not\in L^{2}$ and hence $z^{\ast}$
is not admissible by Theorem~\ref{thm:asymptotic-description}.

We have constructed the blow-up rates $\lmb(t)\sim|t|^{p}/|\log|t||$
for any $p>1$, whereas the condition $p\geq1$ is necessary according
to \eqref{eq:thm1-lmb-upper-bound-radial}. In this sense, our constructions
are almost optimal. However, we do not know whether the (end-point)
blow-up rate $|t|/|\log|t||$ is possible or not. Moreover, we do not claim any optimality regarding the powers of $\log |t|$ in the constructible blow-up rates with the current method, nor in the upper bound \eqref{eq:thm1-lmb-upper-bound-radial}.
\end{rem}

\begin{rem}[Strongly interacting regime]
Our method crucially utilizes the strong soliton-radiation interaction
in the radial case ($m=0$), which directly affects the blow-up rate
\eqref{eq:lmb-gmm-formula}. A similar blow-up construction exhibiting
a direct connection between the asymptotic profile and the blow-up
rate is provided in \cite{JendrejLawrieRodriguez2019arXiv} for the
$1$-equivariant wave maps.

Our analysis in the radial case shows a drastic contrast to the case of higher equivariance
indices $m\geq1$, where the soliton has better spatial decay and
the soliton-radiation interaction is weaker. There is pseudoconformal blow-up $\lmb(t)\sim|t|$,
which is purely driven by pseudoconformal phases of solitons, both in the
$m\geq1$ case \cite{KimKwon2019arXiv,KimKwon2020arXiv} and the \eqref{eq:NLS}
case \cite{BourgainWang1997,MerleRaphaelSzeftel2013AJM,KimKwon2019arXiv}.
However, the recent work \cite{Kim2022arXiv2} of the first author shows that pseudoconformal blow-up is the only possible finite-time blow-up scenario for $H^3_m$, $m\geq1$-equivariant solutions; the continuum of blow-up rates (for regular solutions) seems to arise only in the radial case.
\end{rem}

\begin{rem}[On the continuum of blow-up rates and regularity class]
The continuum of blow-up rates of the blow-up solutions $u(t)$ in Theorem~\ref{thm:MainThm} reveals that there is a complex zoo of blow-up dynamics for solutions merely in $H^{1, 1}_{0}$. Nevertheless, in view of the (backward) propagation of the singularity $r^{\nu} \chi(r)$ at $r = 0$ for the Schr\"odinger equation, we expect the solutions $u(t)$ to fail to belong to $H^{s}_{0}$ with $s$ too large (see also Remark~\ref{rem:z-lin-td}). It is an interesting open problem whether for $H^{\infty}_{0}$ solutions -- which do not include the blow-up solutions in Theorem~\ref{thm:MainThm} -- the possible blow-up rates become rigid and discrete.
(cf.~the discussion under ``comments on the result'' in \cite{RaphaelSchweyer2014AnalPDE})
\end{rem}

\begin{rem}[Comparison with the Krieger--Schlag--Tataru approach]
Another powerful approach for obtaining a continuum of blow-up rates, which predates \cite{JendrejLawrieRodriguez2019arXiv}, is the approach of Krieger--Schlag--Tataru \cite{KriegerSchlagTataru2008Invent,KriegerSchlagTataru2009AdvMath, KriegerSchlagTataru2009Duke}. It is also a backward blow-up construction, but instead of prescribing the asymptotic profile $z^{\ast}$ and deriving the blow-up rate via modulation analysis, it begins with a \emph{directly prescribed} blow-up rate, constructs an approximate blow-up solution (of which $z^{\ast}$ is a part), and finally obtains a genuine blow-up solution by a perturbative argument. This approach was initiated in the setting of wave-type equations, but subsequently, it has also been extended to the setting of Schr\"odinger-type equations \cite{Perelman2014CMP, OrtolevaPerelman2013, Schmid2023}. It would be interesting to also adapt the Krieger--Schlag--Tataru approach to \eqref{eq:CSS-m-equiv}.
\end{rem}

\begin{rem}[On uniqueness]
In the work \cite{JendrejLawrieRodriguez2019arXiv} by Jendrej, Lawrie,
and Rodriguez, the authors also showed \emph{the uniqueness of blow-up
rates} in the context of $1$-equivariant wave maps. Here, the uniqueness
means that if $u_{1}-Q_{\lmb_{1}}\to z^{\ast}$ and $u_{2}-Q_{\lmb_{2}}\to z^{\ast}$,
then $\frac{\lmb_{1}}{\lmb_{2}}\to1$. Such a question in our case
$m=0$ remains an interesting open problem.

On the other hand, for Bourgain--Wang solutions to \eqref{eq:NLS} constructed
in \cite{BourgainWang1997,MerleRaphaelSzeftel2013AJM} or pseudoconformal
blow-up solutions to $m\geq1$-equivariant \eqref{eq:CSS-m-equiv}
constructed in \cite{KimKwon2019arXiv}, one can easily generate blow-up
rates $\lmb(t)=c|t|$ for any $c>0$ with the same prescribed asymptotic
profile. Therefore, it would be interesting to figure out \emph{to
what extent a prescribed asymptotic profile determines the blow-up rate}. Note that in the case of the zero
asymptotic profile (with blow-up time at $t=0$), the classification
results \cite{Merle1993Duke} and \cite{LiLiu2020arXiv} of the threshold
dynamics say that there is only a two-dimensional family of finite-time
blow-up solutions (within radial or equivariance symmetry) achieving
the zero asymptotic profile. 
\end{rem}

\subsection{\label{subsec:Strategy-of-the-proof}Strategy of the proof}

Our proof of Theorem~\ref{thm:MainThm} proceeds via a backward blow-up construction and modulation analysis. Specifically, we adapt the approach developed in \cite{JendrejLawrieRodriguez2019arXiv} in the context of energy critical one-equivariant wave maps to the present case of the radial self-dual Chern--Simons--Schr\"odinger equation. Some major differences -- which make \eqref{eq:CSS-m-equiv} challenging -- are: 
\begin{enumerate}[label=(\roman*)]
\item the construction of the radiation from the asymptotic profile $z^{\ast}(r) = q r^{\nu} \chi(r)$ is more intricate in the Schr\"odinger case (in part due to the infinite speed of propagation, and also due to our optimal range of $\nu$ for $H^{1}_{m}$-solutions),
\item the scaling and phase rotation parameters $(\lmb(t), \gmm(t))$ for the blow-up profile $Q$ are strongly coupled (as opposed to the wave maps case, where $\gmm(t) = 0$ can be enforced by symmetry), and 
\item the nonlinearity of \eqref{eq:CSS-m-equiv} is nonlocal.
\end{enumerate}
With these differences in mind, we now turn to the outline of the proof.

\subsubsection{Construction of the associated radiation}
We first discuss the construction of what we refer to as the \emph{radiation} $z(t, r)$, which is the part of the solution arising from the asymptotic profile $z^{\ast}(r) = q r^{\nu} \chi(r)$ (by solving backwards in time). The effective equation for $z$ reads
\begin{equation} \label{eq:z-eqn-intro}
i\partial_{t}z+\rd_{rr}z+\frac{1}{r}\rd_{r}z - \frac{4}{r^{2}} z + (\hbox{nonlocal nonlinearity in $z$}) = 0,
\end{equation}
where we note the presence of the additional potential $- \frac{4}{r^{2}}$ arising from the nonlocal effect of the concentrated soliton to $z$; see Section~\ref{subsec:z-eqn} for the complete discussion.

The radiation $z(t,r)$ needs to be constructed for $t<0$. In view of time reversal symmetry, and for the sake of simplicity, we will discuss its construction for $t>0$.

In \cite{JendrejLawrieRodriguez2019arXiv}, the radiation is constructed by solving the corresponding equation (which is simply the original wave maps equation) and is described inside the backward light cone using the fundamental solution for the wave equation. In comparison, our construction requires some new ideas due to (i)~the different (in particular, infinite speed) propagation of singularity for the Schr\"odinger equation, and (ii)~allowing $\nu$ that is arbitrarily close to $0$. It proceeds in two steps:

\smallskip
\textbf{Step 1.} Approximate solution $\wh{z}_{\lin}$ to the linear equation. 

Taking cue from the homogeneity of the singular part $q r^{\nu}$, we first look for a \emph{self-similar} solution of the form $t^{\frac{\nu}{2}} Z(t^{-\frac{1}{2}} r)$ to the linear equation $(i \rd_{t} + \rd_{rr} + \frac{1}{r} \rd_{r} - 4 r^{-2}) v = 0$ on $\set{t > 0}$. As is well-known, this leads to a second-order ODE (of the confluent hypergeometric type) for $Z$. Also requiring that the self-similar solution be regular at $r = 0$ for $t > 0$ -- which is justified in view of the compact support property of the actual data $z^{\ast}$ and the local smoothing effect for the Schr\"odinger equation -- $Z$ is uniquely specified. Our approximate solution to the linear equation is taken to be the truncation $\wh{z}_{\lin}(t, x) := t^{\frac{\nu}{2}} Z(t^{-\frac{1}{2}} r) \chi(r)$. Employing tools from second-order ODE, we may obtain rather precise information on $Z$, and thus $\wh{z}_{\lin}(t, x)$, near $r = 0$ for $t > 0$. 

\smallskip
\textbf{Step 2.} Approximate solution $z$ to the nonlinear equation.

Given $\wh{z}_{\lin}$, one may attempt to use a perturbative argument to upgrade it to a solution to \eqref{eq:z-eqn-intro}. This idea is feasible but runs into many technical difficulties, especially when $\Re \nu > 0$ is small (i.e., $z^{\ast}$ is very singular) because the associated solution would only have limited global Sobolev regularity by propagation of singularity (see also Remark~\ref{rem:z-lin-td}). Instead, our simple observation, which leads to a significant technical simplification, is that \emph{we need not construct an exact solution to \eqref{eq:z-eqn-intro}, but only an approximate solution} with an error $\Psi_{z} := (\hbox{LHS of \eqref{eq:z-eqn-intro}})$ that is admissible in the blow-up construction. In particular, $\Psi_{z}$ need not be as regular as in the adequate control of an actual solution (for which we need to prove good local regularity near $r = 0$ despite the limited global Sobolev regularity). As a result, when $\Re \nu > 0$ is small enough (more precisely, $\Re \nu < 1$), it is sufficient to simply take $z = \wh{z}_{\lin}$ as the approximate radiation.

At the other extreme, when $\Re \nu$ is large, we expect $z(t, 0)$ to vanish rapidly as $t \to 0$ and correspondingly we need the error $\Psi_{z}$ to also vanish rapidly as $t \to 0$. We achieve this by performing an expansion of $z$ in powers of $t$ near $t = 0$ up to the order needed; see \eqref{eq:z-form}--\eqref{eq:def-hn} for the precise definition of $z$.

\smallskip
We refer the reader to Sections~\ref{subsec:z-lin} and \ref{subsec:z-nonlin} for the detailed execution of Steps~1 and 2, respectively.

\medskip{}

\subsubsection{Blow-up construction}

With the (approximate) radiation $z(t, r)$ constructed as above, we now consider a solution $u(t, r)$ to \eqref{eq:CSS-m-equiv} with the decomposition
\[
u(t, r)=e^{i\gmm_{z}(t)} \left[ \frac{e^{i \gmm(t)}}{\lmb(t)}\left[ Q(\cdot)+\eps(t, \cdot)\right]\left(\frac{r}{\lmb(t)}\right)+z(t, r) \right],
\]
where $\gmm_{z}(t) \in \bbR$ is an extra phase rotation caused by the nonlocal effect of $z$ to the rescaled $Q$ (see \eqref{eq:Def-tht_z}). The scaling and phase rotation parameters $\lmb(t) > 0$ and $\gmm(t) \in \bbR$ are determined through the orthogonality conditions $(\eps,\calZ_{1})_{r}=(\eps,\calZ_{2})_{r}=0$, where $\calZ_{1},\calZ_{2}\in C_{c,m}^{\infty}$ satisfy a suitable nondegeneracy condition (see \eqref{eq:Z1Z2-transversality}). By initializing the construction (essentially) at the blow-up time $t = 0$ and solving backwards, our goal is to find a solution $u(t)$ of the above form for $t$ in some interval $[t_{0}^{\ast}, 0)$ such that
\begin{equation*}
	\nrm{\eps(t, \cdot)}_{L^{2}} \to 0, \quad \hbox{ and } \lmb(t), \gmm(t) \hbox{ asymptotically behave as \eqref{eq:lmb-gmm-formula} as } t \to 0^{-}.
\end{equation*} 

We first summarize the formal derivation of the above behavior, in particular the precise formula \eqref{eq:lmb-gmm-formula}. As is usual in modulation analysis, we pass to the renormalized variables $(s, y, u^{\flat})$ satisfying $\ud t = \lmb^{2} \ud s$, $r = \lmb y$, and $u = \frac{e^{i \gmm}}{\lmb} u^{\flat}$ -- this has the effect of fixing the scaling and phase rotation parameters, and we may view the flow $u^{\flat}$ in these variables as being a perturbation of $Q$. Let $\calL_{Q}$ be the Hessian of the energy functional $E[u]$ at $Q$, so that the linearization of \eqref{eq:CSS-m-equiv} at $Q$ (in the Hamiltonian form $\rd_{t} u + i \nb E[u] = 0$) becomes $(\rd_{t} + i \calL_{Q}) v = 0$. Clearly, $\Lmb Q$ and $i Q$ -- which are obtained by applying the scaling and phase rotation symmetries to $Q$ -- belong to $\ker i \calL_{Q}$. There also hold generalized kernel relations
\begin{equation*}
	i \calL_{Q} (i \tfrac{y^{2}}{4} Q) = \Lmb Q, \quad i \calL_{Q} \rho = i Q,
\end{equation*}
where $i \tfrac{y^{2}}{4} Q$ and $\rho$ are regular and $O(1)$ as $r \to \infty$ (see Sections~\ref{subsec:Linearization-of-CSS}--\ref{subsec:Invariant-subspace-decomposition}). Motivated by these relations, we introduce modulation parameters $(b(s), \eta(s))$ and take $\eps$ to be of the form
\begin{equation} \label{eq:eps-expect}
	\eps(s, y) \approx \left[ b(s) (-i \tfrac{y^{2}}{4} Q) + \eta(s) \rho \right] \chi_{B_{0}} + \hbox{(smaller terms)}.
\end{equation}
where $\chi_{B_{0}}$ is a truncation in $y$ near the self-similar scale $B_{0} := \lmb(t)^{-1} \abs{t}^{1/2}$, which is needed as the terms in square brackets do not decay as $y \to \infty$ (see also Remark~\ref{rem:mod-est-cutoff-radius} for a more precise but technical discussion of admissible cutoff radii).
In view of the generalized kernel relations (or, more specifically, comparing the terms involving $\Lmb Q$ and $i Q$ in the evolution equation for $\eps(s, y)$), one should expect that the following \emph{adiabatic ansatz} holds:
\begin{equation*}
\frac{\lmb_{s}}{\lmb}+b \approx 0,\qquad\gmm_{s}+\eta \approx 0.
\end{equation*}
The formal evolution laws for $b$ and $\eta$ may be derived by testing the evolution equation for $\eps$ by $i \Lmb(Q+\eps)$ and $Q+\eps$ in the spirit of Pohozaev-type
computations (see, for example, \cite[Section 3]{RaphaelRodnianski2012Publ.Math.}, as well as \cite[Section 4]{KimKwon2019arXiv} for a related computation for \eqref{eq:CSS-m-equiv}). We keep only the main terms, which turn out to be the quadratic nonlinearity in $(b, \eta)$ and the nonlinear interaction between $z$ and $Q$. Moreover, we compute the precise asymptotics of the interaction term using our precise description of $z$. In terms of the complexified variables $\bm{b} := b + i \eta$ and $\bm{\lmb} = \lmb e^{i \gmm}$, the resulting formal modulation ODE takes the form
\begin{equation*}
	\rd_{s} \bm{\lmb} = - \bm{\lmb} \bm{b}, \quad
	\rd_{s} \bm{b} + \abs{\bm{b}}^{2} + \lmb^{2} \frac{\br{\bm{\lmb}} \cdot 8 \sqrt{8} \pi p q (4 i t)^{\frac{\nu - 2}{2}}}{4 \pi \log B_{0}} = 0.
\end{equation*}
Here, the first equation is the adiabatic ansatz, and $p$ is a constant that arises in the description of the approximate radiation $z$. Integrating these ODEs gives rise to the formula \eqref{eq:lmb-gmm-formula}. For more details on this part, see Section~\ref{subsec:Formal-derivation-of-blow-up}. 

Next, we describe our overall strategy for rigorously constructing a solution $u$ with the above decomposition and formally derived behavior. We avoid the idea of modified profile and tail computation (in contrast to, say, the earlier works \cite{KimKwon2020arXiv, KimKwonOh2020arXiv} on forward construction) and instead use the following ideas as in \cite{JendrejLawrieRodriguez2019arXiv}:

\smallskip \textbf{1.} Refined modulation parameters $\bm{\zt}$ and $\bm{b}$.

We introduce \emph{refined modulation parameters} $\bm{\zt}$ and $\bm{b}$, which are corrections of $\bm{\lmb}$ and $- \bm{\lmb}^{-1} \rd_{s} \bm{\lmb}$, respectively. Their evolution equations better resemble the formal modulation ODE. The correction $\bm{\zt}$ of $\bm{\lmb}$ (see \eqref{eq:def-zeta}) is motivated by the invariant subspace decomposition of the linearized flow $(\rd_{s} + i \calL_{Q}) v = 0$. The definition of $\bm{b}$ (see \eqref{eq:def-b} and \eqref{eq:def-eta}) is nonlinear and thus more subtle; it is motivated by the nonlinear virial/mass identities for $u$. Such a careful selection of the nonlinear (in $\eps$) correction is necessary, in particular, to reveal the correct quadratic nonlinearity $\abs{\bm{b}}^{2}$ in the evolution equation for $\bm{b}$. Specifically, see \eqref{eq:b-diff-equality} and \eqref{eq:P-positivity} in Lemma~\ref{lem:Control-of-b-and-eta}, whose proof uses the nonlinear correction in a crucial way. We also refer the reader to the discussion around equations \eqref{eq:def-b}, \eqref{eq:def-zeta}, and \eqref{eq:def-eta} in Section~\ref{sec:Blow-up-Const} for a more detailed explanation of the motivation behind our choices of the refined modulation parameters.

\smallskip \textbf{2.} Nonlinear energy functional $\calE$.

To complete the construction, it remains to obtain control of $\eps$ (via backward propagation) and to justify that the refined parameters indeed evolve essentially according to the formal modulation equations. The key step is to construct a suitable nonlinear energy functional $\calE$ that (i)~controls the $\dot{H}^{1}$-norm of $\eps$, and (ii)~obeys the expected evolution equation $\rd_{t} \calE = (\hbox{explicit}) + (\hbox{error})$, where $(\hbox{explicit})$ refers to an explicitly computable term that captures the interaction between $Q$ and $z^{\flat}$ exactly as in the formal modulation ODE. See \eqref{eq:def-calE} for the definition of $\calE$ and the preceding discussion for its heuristic derivation. In this high-level discussion, we simply emphasize that the explicit interaction term in $\rd_{t} \calE$, which suitably matches the interaction term in the formal modulation ODE, is important for justifying the expected leading order behavior of $\eps$ (see \eqref{eq:eps-expect}) and thus $\bm{b}$ to close the argument. 

\smallskip

Having explained the overall strategy, a final (somewhat technical) comparison with \cite{JendrejLawrieRodriguez2019arXiv} is in order. Aside from the additional technical difficulties arising from the nonlocal nonlinearity of \eqref{eq:CSS-m-equiv} (which makes our argument lengthier), some key differences between the present problem and that of \cite{JendrejLawrieRodriguez2019arXiv} are (i)~the presence of a quadratic nonlinearity $\abs{\bm{b}}^{2}$ in the formal modulation ODE and (ii)~the complex-valued nature of $\bm{b}$ (especially when $\nu$ is not real). Item~(i) underscores the importance of carefully selecting the quadratic correction in $\eps$ to $\bm{b}$ in our problem (which is done in accordance with nonlinear virial/mass identities), as discussed above. Item~(ii) seems to preclude the argument of \cite[Section~5]{JendrejLawrieRodriguez2019arXiv} that combines a sharp \emph{one-sided} lower bound for $b$ with a sharp upper bound for (the kinetic energy part of) the energy to close the main bootstrap; see also Remark~\ref{rem:complex-valued}. Instead, our approach is to use the energy estimate (and backward integration with suitably chosen data) to precisely identify the main part of the energy as
\begin{equation*}
	\calE = 2 \pi \log B_{0} \cdot \abs{\bm{b} / \bm{\zt}}^{2} + (\hbox{smaller}),
\end{equation*}
which is sufficient to justify the evolution equation for $\bm{b}$ up to acceptable errors and to close our argument. 

\medskip{}

\noindent \mbox{\textbf{Organization of the paper.} }
In Section~\ref{sec:Basic-preliminaries}, we collect the notation and preliminaries for our overall analysis. In Section~\ref{sec:ApprxRadiation}, we construct the radiation $z(t, r)$ from $z^{\ast}(r)$. Then, after making more preparations for the blow-up construction in Section~\ref{sec:blow-up-prelim}, in Section~\ref{sec:Blow-up-Const}, we put everything together, carry out the blow-up construction, and prove Theorem~\ref{thm:MainThm} and Corollary~\ref{cor:InfiniteTimeBlowUp}.
\medskip{}

\noindent \mbox{\textbf{Acknowledgements.} }
Part of this work was done while K.~Kim was supported by the Huawei Young Talents Programme
at IHES. 
K.~Kim was supported by the New Faculty Startup Fund from Seoul National University, the POSCO Science Fellowship of POSCO TJ Park Foundation, and the National Research Foundation of Korea (NRF) grant funded by the Korea government (MSIT) RS-2025-00523523.
S.~Kwon was partially supported by the National Research Foundation of Korea, RS-2019-NR040050 and NRF-2022R1A2C1091499. 
S.-J. Oh was partially supported by the Samsung Science
\& Technology Foundation under Project Number BA1702-02, a Sloan Research Fellowship, an NSF CAREER Grant under NSF-DMS-1945615, and an NSF Grant under NSF-DMS-2452760.

\section{\label{sec:Basic-preliminaries}Basic preliminaries}

In this section, we introduce the basic notation, equivariant Sobolev
spaces, and a decomposition of the nonlinearity that will be used
throughout this paper. Other preliminary facts that are more relevant
to the blow-up analysis will be collected in Section~\ref{sec:blow-up-prelim}.

\subsection{\label{subsec:Notations}Notation}

For $A\in\bbC$ and $B\geq0$, we use the standard asymptotic notation
$A\lesssim B$ or $A=O(B)$ if there exists a constant $C>0$ such that
$|A|\leq CB$. $A\sim B$ means that $A\lesssim B$ and $B\lesssim A$.
The dependencies of $C$ are specified by subscripts, e.g., $A\lesssim_{E}B\Leftrightarrow A=O_{E}(B)\Leftrightarrow|A|\leq C(E)B$.
In this paper, any dependencies on the equivariance index $m$ and
the parameters $q$ and $\nu$ (recall $z^{\ast}(r)=qr^{\nu}\chi(r)$)
will be omitted.

We also use the notation $\langle x\rangle$, $\log_{+}x$, $\log_{-}x$
defined by 
\[
\langle x\rangle\coloneqq(|x|^{2}+1)^{\frac{1}{2}},\quad\log_{+}x\coloneqq\max\{\log x,0\},\quad\log_{-}x\coloneqq\max\{-\log x,0\}.
\]
We let $\chi=\chi(x)$ be a smooth spherically symmetric cutoff function
such that $\chi(x)=1$ for $|x|\leq1$ and $\chi(x)=0$ for $|x|\geq2$.
For $A>0$, we define its rescaled version by $\chi_{A}(x)\coloneqq\chi(x/A)$.

We mainly work with equivariant functions on $\bbR^{2}$, say $\phi:\bbR^{2}\to\bbC$,
or equivalently their radial part $u:\bbR_{+}\to\bbC$ with $\phi(x)=u(r)e^{im\tht}$,
where $\bbR_{+}\coloneqq(0,\infty)$ and $x_{1}+ix_{2}=re^{i\tht}$.
The flat Cauchy--Riemann operator and its adjoints are denoted by
$\rd_{\pm}=\rd_{1}\pm i\rd_{2}$. We define $\rd_{\pm}^{(m)}$ via
the relations $\rd_{\pm}\phi=[\rd_{\pm}^{(m)}u]e^{i(m\pm1)\tht}$.
Thus $\rd_{\pm}$ sends $m$-equivariant functions to $(m\pm1)$-equivariant
functions and its radial part is given by $\rd_{\pm}^{(m)}=\rd_{r}\mp\frac{m}{r}$.
We denote by $\Delta^{(m)}=\rd_{rr}+\tfrac{1}{r}\rd_{r}-\tfrac{m^{2}}{r^{2}}$
the Laplacian acting on $m$-equivariant functions.

The integral symbol $\int$ means 
\[
\int=\int_{\bbR^{2}}\,dx=2\pi\int\,rdr.
\]
For complex-valued functions $f$ and $g$, we define their \emph{real
inner product} by 
\[
(f,g)_{r}\coloneqq\int\Re(\br fg).
\]
For a real-valued functional $F$ and a function $u$, we denote by
$\nabla F[u]$ the \emph{functional derivative} of $F$ at $u$ with
respect to this real inner product.

We denote by $\Lmb$ the $L^{2}$-scaling generator: 
\[
\Lmb\coloneqq r\rd_{r}+1.
\]
For a cutoff parameter $A>0$, its truncated version is defined by
\begin{equation}
\Lmb_{A}\coloneqq\chi_{A}r\rd_{r}+(\chi_{A}+\tfrac{1}{2}r\chi_{A}')=\chi_{A}\Lmb+\tfrac{1}{2}r\chi_{A}',\label{eq:def-truncated-scaling-gen}
\end{equation}
whose zeroth order term is determined to make $\Lmb_{A}$ antisymmetric.

Given a scaling parameter $\lmb\in\bbR_{+}$ and a phase rotation parameter
$\gmm\in\bbR/2\pi\bbZ$, we denote by 
\[
f_{\lmb,\gmm}(r)\coloneqq\frac{e^{i\gmm}}{\lmb}f\Big(\frac{r}{\lmb}\Big)
\]
for a function $f$. When $\lmb$ and $\gmm$ are clear from the context,
we will also denote the above by $f^{\sharp}$ as in \cite{KimKwon2019arXiv},
i.e., 
\[
f^{\sharp}\coloneqq f_{\lmb,\gmm}.
\]
Similarly, we define its inverse by $\flat$: 
\[
g^{\flat}(y)\coloneqq\lmb e^{-i\gmm}g(\lmb y).
\]
When a time-dependent scaling parameter $\lmb(t)$ is given, we define
a rescaled spacetime variables $s,y$ by the relations 
\begin{equation}
\frac{ds}{dt}=\frac{1}{\lmb^{2}(t)}\qquad\text{and}\qquad y=\frac{r}{\lmb(t)}.\label{eq:def-renormalized-var}
\end{equation}
The raising operation $\sharp$ converts a function $f=f(y)$ to a
function of $r$: $f^{\sharp}=f^{\sharp}(r)$. Similarly, the lowering
operation $\flat$ converts a function $g=g(r)$ to a function of
$y$: $g^{\flat}=g^{\flat}(y)$. In modulation analysis, the dynamical
parameters such as $\lmb,\gmm,b,\eta$ are functions of either the
variables $t$ or $s$ under $\frac{ds}{dt}=\frac{1}{\lmb^{2}}$.

For $k\in\bbN$, we define 
\begin{align*}
|f|_{k} & \coloneqq\max\{|f|,|r\rd_{r}f|,\dots,|(r\rd_{r})^{k}f|\},\\
|f|_{-k} & \coloneqq\max\{|\rd_{r}^{k}f|,|\tfrac{1}{r}\rd_{r}^{k-1}f|,\dots,|\tfrac{1}{r^{k}}f|\}.
\end{align*}
We note that $|f|_{k}\sim r^{k}|f|_{-k}$. The following Leibniz rules
hold: 
\[
|fg|_{k}\aleq|f|_{k}|g|_{k},\qquad|fg|_{-k}\aleq|f|_{k}|g|_{-k}.
\]

For normed function spaces $X$ and $Y$ and a scalar $s\in\bbR$,
we define 
\begin{align*}
r^{s}X & \coloneqq\{r^{s}f:f\in X\}, & \|f\|_{r^{s}X} & \coloneqq\|r^{-s}f\|_{X},\\
X+Y & \coloneqq\{f+g:f\in X,\ g\in Y\}, & \|f\|_{X+Y} & \coloneqq\inf_{f=f_{X}+f_{Y}}\|f_{X}\|_{X}+\|f_{Y}\|_{Y},
\end{align*}
and 
\[
\|f\|_{X\cap Y}\coloneqq\|f\|_{X}+\|f\|_{Y}.
\]
The relevant function spaces will be discussed in Section~\ref{subsec:Adapted-function-spaces}.

\subsection{\label{subsec:function-spaces}Equivariant Sobolev spaces}

Let $m\in\bbZ$. For $s\geq0$, we denote by $H_{m}^{s}$ the restriction
of the usual Sobolev space $H^{s}(\bbR^{2})$ on $m$-equivariant
functions. The sets $C_{c,m}^{\infty}$ of $m$-equivariant smooth
compactly supported functions as well as $\calS_{m}$ of $m$-equivariant
Schwartz functions, are dense in $H_{m}^{s}$. We use the $H_{m}^{s}$-norms
and $\dot{H}_{m}^{s}$-norms to mean the usual $H^{s}$-norms and
$\dot{H}^{s}$-norms, but the subscript $m$ is used to emphasize
that we are applying these norms to $m$-equivariant functions.

One of the advantages of using equivariant Sobolev spaces is the generalized
Hardy's inequality \cite[Lemma A.7]{KimKwon2019arXiv}: whenever $0\leq k\leq|m|$,
we have 
\begin{equation}
\||f|_{-k}\|_{L^{2}}\sim\|f\|_{\dot{H}_{m}^{k}},\qquad\forall f\in\calS_{m}.\label{eq:GenHardySection2}
\end{equation}
In particular, when $0\leq k\leq|m|$, we can define the \emph{homogeneous
equivariant Sobolev space} $\dot{H}_{m}^{k}$ by taking the completion
of $\calS_{m}$ under the $\dot{H}_{m}^{k}$-norm and have the embeddings
\[
\calS_{m}\embed H_{m}^{k}\embed\dot{H}_{m}^{k}\embed L_{\mathrm{loc}}^{2}.
\]
Specializing this to $k=1$ and applying the fundamental theorem of
calculus to $\rd_{r}|f|^{2}=2\Re(\br f\rd_{r}f)$, we have the \emph{Hardy-Sobolev
inequality} \cite[Lemma A.6]{KimKwon2019arXiv}: whenever $|m|\geq1$,
we have 
\begin{equation}
\|r^{-1}f\|_{L^{2}}+\|f\|_{L^{\infty}}=\|f\|_{rL^{2}\cap L^{\infty}}\aleq\|f\|_{\dot{H}_{m}^{1}}.\label{eq:HardySobolevSection2}
\end{equation}
Note in general that $H_{0}^{1}\embed L^{\infty}$ is \emph{false}.

Finally, we define the weighted Sobolev space $H_{m}^{1,1}\coloneqq H_{m}^{1}\cap r^{-1}L^{2}$.
Thus
\[
\|f\|_{H_{m}^{1,1}}=\|f\|_{H_{m}^{1}}+\|rf\|_{L^{2}}.
\]

\subsection{\label{subsec:nonlinearity}Decomposition of the nonlinearity}

In this subsection, we introduce a notation for the nonlinearity of
\eqref{eq:CSS-m-equiv} following \cite{KimKwon2019arXiv}. 

Denote by $\calN(u)$ the nonlinearity of \eqref{eq:CSS-m-equiv},
i.e., 
\[
i\rd_{t}u+\rd_{rr}u+\frac{1}{r}\rd_{r}u-\frac{m^{2}}{r^{2}}u=\calN(u).
\]
Therefore, the nonlinearity $\calN(u)$ (and its potential term $V_{u}$)
is explicitly given by 
\[
\calN(u)\coloneqq V_{u}u\coloneqq(-|u|^{2}+\tfrac{2m}{r^{2}}A_{\tht}[u]+\tfrac{1}{r^{2}}A_{\tht}^{2}[u]+A_{t}[u])u.
\]
This nonlinearity decomposes into the sum of the cubic and quintic
nonlinearities
\[
\calN=\calN_{3,0}+m(\calN_{3,1}+\calN_{3,2})+\calN_{5,1}+\calN_{5,2},
\]
where we abbreviate $\calN_{\ast}(u)\coloneqq\calN_{\ast}(u,\dots,u)$
(where $\ast$ is a place-holder) and denote the cubic nonlinearities
by 
\begin{align*}
\calN_{3,0}(u_{1},u_{2},u_{3}) & \coloneqq V_{2,0}[u_{1},u_{2}]u_{3}\coloneqq-\Re(\br{u_{1}}u_{2})u_{3},\\
\calN_{3,1}(u_{1},u_{2},u_{3}) & \coloneqq V_{2,1}[u_{1},u_{2}]u_{3}\coloneqq\tfrac{2}{r^{2}}A_{\tht}[u_{1},u_{2}]u_{3},\\
\calN_{3,2}(u_{1},u_{2},u_{3}) & \coloneqq V_{2,2}[u_{1},u_{2}]u_{3}\coloneqq-(\tint r{\infty}\Re(\br{u_{1}}u_{2})\tfrac{dr'}{r'})u_{3},
\end{align*}
and the quintic nonlinearities by 
\begin{align*}
\calN_{5,1}(u_{1},\dots,u_{5}) & \coloneqq V_{4,1}[u_{1},\dots,u_{4}]u_{5}\coloneqq\tfrac{1}{r^{2}}A_{\tht}[u_{1},u_{2}]A_{\tht}[u_{3},u_{4}]u_{5},\\
\calN_{5,2}(u_{1},\dots,u_{5}) & \coloneqq V_{4,2}[u_{1},\dots,u_{4}]u_{5}\coloneqq-(\tint r{\infty}A_{\tht}[u_{1},u_{2}]\Re(\br{u_{3}}u_{4})\tfrac{dr'}{r'})u_{5}.
\end{align*}
We also denote 
\[
\calN_{3}\coloneqq\calN_{3,0}+m(\calN_{3,1}+\calN_{3,2})\qquad\text{and}\qquad\calN_{5}\coloneqq\calN_{5,1}+\calN_{5,2}.
\]
We remark that $\calN_{3,1}$ and $\calN_{3,2}$ do not appear in
the case $m=0$.

In view of the Hamiltonian structure of \eqref{eq:CSS-m-equiv}, the
nonlinearity of \eqref{eq:CSS-m-equiv} appears as a part of the functional
derivative of energy, i.e., 
\[
\nabla E[u]=-\Delta^{(m)}u+\calN(u).
\]
In order to relate each $\calN_{\ast}$ to each component of energy,
we decompose 
\[
E[u]=\tfrac{1}{2}\tint{}{}(|\rd_{r}u|^{2}+\tfrac{m^{2}}{r^{2}}|u|^{2})+\calM_{4,0}[u]+m\calM_{4,1}[u]+\calM_{6}[u],
\]
where we abbreviate $\calM_{\ast}(u)\coloneqq\calM_{\ast}(u,\dots,u)$
and denote the multilinear forms by 
\begin{align*}
\calM_{4,0}(u_{1},\dots,u_{4}) & \coloneqq-\tfrac{1}{4}\Re(\br{u_{1}}u_{2})\Re(\br{u_{3}}u_{4}),\\
\calM_{4,1}(u_{1},\dots,u_{4}) & \coloneqq\tint{}{}\tfrac{1}{r^{2}}A_{\tht}[u_{1},u_{2}]\Re(\br{u_{3}}u_{4}),\\
\calM_{6}(u_{1},\dots,u_{6}) & \coloneqq\tfrac{1}{2}\tint{}{}\tfrac{1}{r^{2}}A_{\tht}[u_{1},u_{2}]A_{\tht}[u_{3},u_{4}]\Re(\br{u_{5}}u_{6}).
\end{align*}
It is then easy to verify that 
\begin{equation}
\left\{ \begin{aligned}(\calN_{3,0}(u_{1},u_{2},u_{3}),u_{4})_{r} & =4\calM_{4,0}(u_{1},u_{2},u_{3},u_{4}),\\
(m\calN_{3,1}(u_{1},u_{2},u_{3}),u_{4})_{r} & =2m\calM_{4,1}(u_{1},u_{2},u_{3},u_{4}),\\
(m\calN_{3,2}(u_{1},u_{2},u_{3}),u_{4})_{r} & =2m\calM_{4,1}(u_{3},u_{4},u_{1},u_{2}),\\
(\calN_{5,1}(u_{1},u_{2},u_{3},u_{4},u_{5}),u_{6})_{r} & =2\calM_{6}(u_{1},u_{2},u_{3},u_{4},u_{5},u_{6}),\\
(\calN_{5,2}(u_{1},u_{2},u_{3},u_{4},u_{5}),u_{6})_{r} & =4\calM_{6}(u_{1},u_{2},u_{5},u_{6},u_{3},u_{4}).
\end{aligned}
\right.\label{eq:duality-relations}
\end{equation}
We remark that $\calM_{4,1}$ does not appear in the case $m=0$.

Finally, we introduce a notation for the nonlinearity of the \emph{phase-rotated
CSS}, i.e., \eqref{eq:CSS-m-equiv} with the outermost integral $-\int_{r}^{\infty}\cdot\frac{dr'}{r'}$
(for $A_{t}[u]$) replaced by $\int_{0}^{r}\cdot\frac{dr'}{r'}$.
We set 
\begin{align*}
\rN_{3,0} & =\calN_{3,0}, & \rN_{3,1} & =\calN_{3,1}, & \rN_{5,1} & =\calN_{5,1},\\
\rV_{2,0} & =V_{2,0}, & \rV_{2,1} & =V_{2,1}, & \rV_{4,1} & =V_{4,1},
\end{align*}
whereas we set 
\begin{align*}
\rN_{3,2}(u_{1},u_{2},u_{3}) & =\rV_{2,2}[u_{1},u_{2}]u_{3}, & \rV_{2,2} & =\tint 0r\Re(\br{u_{1}}u_{2})\tfrac{dr'}{r'},\\
\rN_{5,2}(u_{1},\dots,u_{5}) & =\rV_{4,2}[u_{1},\dots,u_{4}]u_{5}, & \rV_{4,2} & =\tint 0rA_{\tht}[u_{1},u_{2}]\Re(\br{u_{3}}u_{4})\tfrac{dr'}{r'}.
\end{align*}
We denote 
\[
\rN_{3}=\rN_{3,0}+m(\rN_{3,1}+\rN_{3,2}),\qquad\rN_{5}=\rN_{5,1}+\rN_{5,2},\qquad\rN=\rN_{3}+\rN_{5}.
\]
We remark that 
\[
m\rV_{2,2}[u]+\rV_{4,2}[u]=mV_{2,2}[u]+V_{4,2}[u]+\tint 0{\infty}(m+A_{\tht}[u])|u|^{2}\tfrac{dr}{r}.
\]

\section{\label{sec:ApprxRadiation}Asymptotic profile and approximate radiation}

Our goal here in this section is to associate a given asymptotic
profile $z^{\ast}(r)=qr^{\nu}\chi(r)$ with an adequate nonlinear evolution
$z(t,r)$, which we call the \emph{radiation} associated with $z^{\ast}$,
having sufficiently precise asymptotics near where singularity would
form. As we shall see in Section~\ref{sec:Blow-up-Const}, the strong
interaction between the radiation $z$ and $Q_{\lmb,\gmm}$ will drive
the blow-up dynamics.

\subsection{Equations for the radiation} \label{subsec:z-eqn}

We begin by discussing the PDE that the radiation should solve (see
also \cite{KimKwon2019arXiv}). For this general discussion, we may
assume that the equivariance index $m$ is merely nonnegative instead
of $0$.

Consider an $m$-equivariant solution $u$ to \eqref{eq:CSS-m-equiv}
admitting the formal decomposition $u=Q_{\lmb,\gmm}+\zt$ and dropping
all terms that vanish weakly as $\lmb\to0$ as $\zt$ is fixed (and
assumed to be sufficiently smooth). The equation solved by $\zt$
after such a formal manipulation is 
\begin{equation}
i\partial_{t}\zt+\rd_{rr}\zeta+\frac{1}{r}\rd_{r}\zeta-\Big(\frac{\frkm+A_{\tht}[\zeta]}{r}\Big)^{2}\zeta-A_{t}^{(\frkm)}[\zeta]\zeta+|\zeta|^{2}\zeta=0,\label{eq:m-equiv-CSS-zt}
\end{equation}
where $\frkm=-m-2$ and $A_{t}^{(\frkm)}[\zeta]$ are defined by the
formula \eqref{eq:def-A} with $m$ replaced by $\frkm$.

Observe that \eqref{eq:m-equiv-CSS-zt} differs from the original
equation \eqref{eq:CSS-m-equiv} only in its equivariance index. Moreover,
unlike \eqref{eq:CSS-m-equiv}, the equivariant index $\frkm$ associated
with \eqref{eq:m-equiv-CSS-zt} is \emph{strictly negative}. This
remarkable property of the nonlocal nonlinearity of the self-dual
Chern--Simons--Schrödinger equation was the basis of the proof of
the soliton resolution theorems (Theorem~\ref{thm:asymptotic-description})
in \cite{KimKwonOh2022arXiv1}. In the derivation of \eqref{eq:m-equiv-CSS-zt},
this phenomenon arises concretely due to the following computation:
As $\lmb\to0$, 
\begin{align*}
A_{\tht}[Q_{\lmb,\gmm}+\zt](r) & =-\frac{1}{2}\int_{0}^{r}\abs{Q_{\lmb,\gmm}+\zt}^{2}r'dr'\\
 & \to-\frac{1}{2}\int_{0}^{\infty}Q^{2}r'dr'-\frac{1}{2}\int_{0}^{r}\abs\zt^{2}r'dr'=-2(m+1)+A_{\tht}[\zt]
\end{align*}
where we used \eqref{eq:Q-formula} to compute $\int_{0}^{\infty}Q^{2}r'dr'=4(m+1)$.
As a result, observe that $\bfD_{u}^{(m)}u\weakto\bfD_{\zt}^{(-m-2)}\zt$
and $\bfD_{u}^{(m)\ast}\bfD_{u}^{(m)}u\weakto\bfD_{\zt}^{(-m-2)\ast}\bfD_{\zt}^{(-m-2)}\zt$
as $\lmb\to0$. Afterward, the derivation of \eqref{eq:m-equiv-CSS-zt}
is straightforward.

By performing a suitable phase rotation, we now derive another formulation of \eqref{eq:m-equiv-CSS-zt}, in which the domain of integration $(r, \infty)$ in $A_{t}^{(\frkm)}[\zeta]$ is altered to $(0, r)$. As we shall see, this formulation will
be more convenient in Section~\ref{sec:Blow-up-Const} (it will capture
the phase correction on the soliton $Q_{\lmb,\gmm}$ due to the radiation).
Introducing $z$ through the relation 
\[
\zt(t,r)=e^{i\gmm_{z}(t)}z(t,r),
\]
where $\gmm_{z}(t)$ is given by (note that $|z|=|\zeta|$) 
\begin{equation}
\gmm_{z}(t)\coloneqq-\int_{0}^{t}\tht_{z}(t')dt'\quad\text{with}\quad\tht_{z}\coloneqq-\int_{0}^{\infty}(\frkm+A_{\tht}[z])|z|^{2}\frac{dr'}{r'},\label{eq:Def-tht_z}
\end{equation}
we arrive at the \emph{phase-rotated CSS}: 
\begin{equation}
i\partial_{t}z+\rd_{rr}z+\frac{1}{r}\rd_{r}z-\Big(\frac{\frkm+A_{\tht}[z]}{r}\Big)^{2}z-\Big(\int_{0}^{r}(\frkm+A_{\tht}[z])|z|^{2}\frac{dr'}{r'}\Big)z+|z|^{2}z=0.\label{eq:m-equiv-PhaseRotatedCSS}
\end{equation}
As $\tht_{z}$ only depends on the modulus of $z$, it is easy to
pass back and forth between $z$ and $\zt$.

\subsection{Statement of the main result}

We now state the main result of this section. We formulate it in terms
of the phase-rotated variable $z$, which approximately solves \eqref{eq:m-equiv-PhaseRotatedCSS}
with $\frkm=-2$ as discussed in the previous subsection. 
\begin{prop}[Properties of the radiation]
\label{prop:ApprxRadiation}Let $q,\nu\in\bbC$ with $\Re(\nu)>0$
and $q\neq0$. Then, there exist $\dlt_{z}>0$ and an approximate
radiation $z(t,r)$ that solves the equation 
\begin{equation}
i\partial_{t}z+\rd_{rr}z+\frac{1}{r}\rd_{r}z-\Big(\frac{\frkm+A_{\tht}[z]}{r}\Big)^{2}z-\Big(\int_{0}^{r}(\frkm+A_{\tht}[z])|z|^{2}\frac{dr'}{r'}\Big)z+|z|^{2}z=\Psi_{z}\label{eq:eqn-aprx-radiation}
\end{equation}
with $\frkm=-2$ on the time interval $[-1,1]$ and satisfies the
following estimates: 
\begin{enumerate}
\item (Initial data and smooth radiation) We have 
\begin{equation}
\lim_{t\to0}\|z(t,r)-qr^{\nu}\chi(r)\|_{H_{\frkm}^{1,1}}=0.\label{eq:z-initial-justification}
\end{equation}
Moreover, $z(t)$ is a smooth $\frkm$-equivariant function for any
$t\neq0$. 
\item (Pointwise bounds in the self-similar zone) For any $k\in\bbN$, we
have 
\begin{align}
\chf_{r\leq|t|^{\frac{1}{2}}}|z(t,r)|_{k} & \aleq_{k}|t|^{\frac{1}{2}(\Re(\nu)-2)}r^{2},\label{eq:z-ss-bound}\\
\chf_{r\leq|t|^{\frac{1}{2}}}|z(t,r)-qp(4it)^{\frac{\nu-2}{2}}r^{2}|_{k} & \aleq_{k}|t|^{\frac{1}{2}(\Re(\nu)-4)}r^{4},\label{eq:z-ss-leading}\\
\chf_{r\leq|t|^{\frac{1}{2}}}|z_{1}(t,r)-4qp(4it)^{\frac{\nu-2}{2}}r|_{k} & \aleq_{k}|t|^{\frac{1}{2}(\Re(\nu)-4)}r^{3},\label{eq:z1-ss-leading}\\
\chf_{r\leq|t|^{\frac{1}{2}}}|\rd_{-}^{(\frkm)}z(t,r)|_{k} & \aleq_{k}|t|^{\frac{1}{2}(\Re(\nu)-2)}r^{3},\label{eq:rd-z-ss-degen}
\end{align}
where $z_{1}=\bfD_{z}^{(\frkm)}z$ and 
\begin{equation}
p=\frac{1}{2}\Gmm\Big(\frac{\nu}{2}+2\Big).\label{eq:def-p}
\end{equation}
\item (Pointwise bounds outside the self-similar zone) For any $k\in\bbN$,
we have 
\begin{equation}
\chf_{r\geq|t|^{\frac{1}{2}}}|z(t,r)|_{k}\aleq_{k}\chf_{|t|^{\frac{1}{2}}\leq r\leq2}\{r^{\Re(\nu)}+|t|^{\Re(\nu)+1-k}r^{-\Re(\nu)-2+2k}\}.\label{eq:z-ext-bound}
\end{equation}
\item (Estimates for $\Psi_{z}$) We need 
\begin{align}
\|\Psi_{z}\|_{L^{2}} & \aleq|t|^{\frac{1}{2}(\Re(\nu)-1)+\dlt_{z}},\label{eq:Psi_z-L2}\\
\|r^{-s}|\Psi_{z}|_{-1}\|_{L^{2}} & \aleq|t|^{\frac{1}{2}(\Re(\nu)-2)+\dlt_{z}},\qquad s\in[0,\dlt_{z}].\label{eq:Psi_z-Hdot1}
\end{align}
\end{enumerate}
\end{prop}

The remainder of this section is devoted to the proof of this proposition.
For an overview of the strategy, we refer to Section~\ref{subsec:Strategy-of-the-proof}.

\subsection{Approximate solution for the linear evolution} \label{subsec:z-lin}

We begin by constructing an explicit approximate solution to the linear
equation 
\begin{equation}
(i\rd_{t}+\lap^{(\frkm)})v=\left(i\rd_{t}+\rd_{rr}+\frac{1}{r}\rd_{r}-\frac{\frkm^{2}}{r^{2}}~\right)v(t,r)=0,\label{eq:m-equiv-lin}
\end{equation}
with initial data $z^{\ast}(r)=qr^{\nu}\chi(r)$. To clarify the role
of $\frkm$ in this construction, \textbf{we shall simply assume that
$\frkm$ is an integer such that $\frkm\neq0$ in this subsection},
although we shall soon specialize to the case $\frkm=-2$ in the next
subsection. Note that a similar argument may be carried out for $\frkm=0$,
but logarithmic corrections would appear. Moreover, \textbf{we assume
that $t>0$}; the case $t<0$ is handled by time reversal symmetry.

Our idea is to use the scaling property of the untruncated initial
data $qr^{\nu}$. We introduce the self-similar variable 
\[
Y=t^{-\frac{1}{2}}r
\]
and the ansatz 
\[
v(t,r)=t^{\frac{\nu}{2}}V(Y).
\]
We note that the factor $t^{\frac{\nu}{2}}$ is motivated by the scaling
property of the untruncated initial data $qr^{\nu}$. Note that 
\[
t^{-\frac{\nu}{2}-1}\left(i\rd_{t}+\rd_{rr}+\frac{1}{r}\rd_{r}-\frac{\frkm^{2}}{r^{2}}\right)v=\calA_{\frkm,\nu}V.
\]
where 
\[
\calA_{\frkm,\nu}\coloneqq\rd_{YY}+\frac{1}{Y}\rd_{Y}-\frac{\frkm^{2}}{Y^{2}}-\frac{i}{2}Y\rd_{Y}+\frac{i\nu}{2}.
\]

The ODE $\calA_{\frkm,\nu}V=0$ has a regular singularity at the origin
and an irregular singularity at infinity. By the standard theory of
ODE's, we have the following asymptotic series for a fundamental system
near each singularity: 
\begin{lem}[Fundamental systems for $\calA_{\frkm,\nu}$]
\label{lem:lin-ss-fundsys}~
\begin{itemize}
\item (Fundamental system adapted for large $Y$) There exists a fundamental
system $\set{f_{1},f_{2}}$ for $\calA_{\frkm,\nu}$ as $Y\to\infty$
that admits the following asymptotic series expansion: for every $K\geq0$,
\begin{align}
f_{1}(Y) & =Y^{\nu}\left(1+\sum_{n=1}^{K-1}c_{f_{1},2n}Y^{-2n}+e_{f_{1},2K}(Y)\right),\label{eq:f1-expn}\\
f_{2}(Y) & =e^{i\frac{Y^{2}}{4}}Y^{-(\nu+2)}\left(1+\sum_{n=1}^{K-1}c_{f_{2},2n}Y^{-2n}+e_{f_{1},2K}(Y)\right),\label{eq:f2-expn}
\end{align}
where 
\begin{equation}
|e_{f_{1},2K}|_{k}+|e_{f_{2},2K}|_{k}\aleq_{k}Y^{-2K}\qquad\text{for }Y\geq\frac{1}{100}\text{ and }k\in\bbN.\label{eq:e_f-bound}
\end{equation}
\item (Fundamental system adapted for small $Y$) There exists a fundamental
system $\set{e_{1},e_{2}}$ for $\calA_{\frkm,\nu}$ as $Y\to0$ that
admits the following asymptotic series expansion: for every $K\geq0$,
\begin{align*}
e_{1}(Y) & =Y^{\abs\frkm}\left(1+\sum_{n=1}^{K-1}c_{e_{1},2n}Y^{2n}+e_{e_{1},2K}(Y)\right),\\
e_{2}(Y) & =Y^{-\abs\frkm}\left(1+\sum_{n=1}^{K-1}c_{e_{2},2n}Y^{2n}+e_{e_{2},2K}(Y)\right),
\end{align*}
where 
\begin{equation}
|e_{e_{1},2K}|_{k}+|e_{e_{2},2K}|_{k}\aleq_{k}Y^{2K}\qquad\text{for }Y\leq100\text{ and }k\in\bbN.\label{eq:e_e-bound}
\end{equation}
\end{itemize}
\end{lem}

We are now ready to motivate the choice of our explicit approximate
solution $\td z_{\lin}$ to \eqref{eq:m-equiv-lin} with initial data
$z^{\ast}(r)=qr^{\nu}\chi(r)$. First, observe that as $r$ is fixed
and $t\to0$, we have $Y=t^{-\frac{1}{2}}r\to\infty$. Moreover, $t^{\frac{\nu}{2}}f_{1}\to r^{\nu}$
whereas $f_{2}\to0$ in the same limit. Finally, to achieve the initial
data $z^{\ast}=qr^{\nu}\chi(r)$, we consider the truncated ansatz (see also Remark~\ref{rem:z-lin-td} below)
\begin{equation} \label{eq:lin-rad-hat}
\wh z_{\lin} \coloneqq qt^{\frac{\nu}{2}}\{f_{1}(t^{-\frac{1}{2}}r)+\alp(\nu)f_{2}(t^{-\frac{1}{2}}r)\}\chi(r),
\end{equation}
where it remains to specify $\alp(\nu)\in\bbC$.


Since $z^{\ast}(r)=qr^{\nu}\chi(r)$ is compactly supported, in view
of the local smoothing property of the Schrödinger equation, the corresponding
exact linear evolution $e^{it\lap^{(\frkm)}}z^{\ast}$ must be smooth
for $t>0$. Since $e_{2}(Y)$ is singular at $Y=0$, this motivates
choosing $\alp(\nu)$ so that $f_{1}(Y)+\alp(\nu)f_{2}(Y)$ is proportional
to $e_{1}(Y)$ and hence the linear evolution becomes smooth near
$r=0$ for all $t>0$. The following lemma guaranties the existence
of such an $\alp(\nu)$ and, moreover, computes the proportionality
constant.
\begin{lem}[Connection formula]
\label{lem:lin-ss-connect} For every $\nu\in\bbC$ with $\Re(\nu)>-\abs\frkm-2$,
there exists a unique $\alp=\alp(\nu)$ such that $f_{1}+\alp f_{2}$
is proportional to $e_{1}$. In fact, 
\begin{equation}
f_{1}+\alp f_{2}=(4i)^{\frac{\nu-\abs\frkm}{2}}pe_{1},\label{eq:f1-to-e1}
\end{equation}
where 
\[
p=\frac{\Gmm(\tfrac{\nu+\abs\frkm}{2}+1)}{\Gmm(\abs\frkm+1)}.
\]
\end{lem}

\begin{proof}
We present here a proof that uses distribution theory and the fundamental
solution to explicitly construct a solution to the free Schrödinger
equation with initial data $r^{\nu}$; for an alternative ODE proof,
see Remark~\ref{rem:lin-ss-connect-ODE} below.

Our starting point is the formal relation 
\[
r^{-\abs\frkm}(i\rd_{t}+\lap^{(\frkm)})f(t,r)=(i\rd_{t}+\lap_{\bbR^{2\abs\frkm+2}})(r^{-\abs\frkm}f(t,r)),
\]
where for each $t$, $r^{-\abs\frkm}f(t,r)$ is viewed as a radial
function on $\bbR^{2\abs\frkm+2}$. We shall therefore consider $r^{-\abs\frkm}r^{\nu}$
as a radial function on $\bbR^{2\abs\frkm+2}$ and try to understand
its evolution under the unitary group $e^{it\lap_{\bbR^{2\abs\frkm+2}}}$.
Observe that $r^{\nu-\abs\frkm}$ is locally integrable on $\bbR^{2\abs\frkm+2}$
if and only if $\Re(\nu)>-\abs\frkm-2$; since it (as well as its
derivatives) grows only polynomially as $r\to\infty$, $r^{\nu-\abs\frkm}\in\calS'(\bbR^{2\abs\frkm+2})$
as well for the same range of $\nu$. Hence, $\bfv(t,x)=e^{it\lap_{\bbR^{2\abs\frkm+2}}}(r^{\nu-\abs\frkm})$
is well-defined. A useful way to represent $\bfv$ is to consider
the approximating sequence (defined for each positive integer $N$)
\[
\bfv_{N}=e^{it\lap_{\bbR^{2\abs\frkm+2}}}(r^{\nu-\abs\frkm}\chi_{\aleq1}(\tfrac{r}{N}))=\int_{\bbR^{2\abs\frkm+2}}\frac{1}{(4\pi it)^{\abs\frkm+1}}e^{i\frac{\abs{x-y}^{2}}{4t}}\abs y^{\nu-\abs\frkm}\chi_{\aleq1}(\tfrac{\abs y}{N})\,dy,
\]
where the last integral formula makes sense since $\abs y^{\nu-\abs\frkm}\chi_{\aleq1}(\tfrac{\abs y}{N})\in L^{1}(\bbR^{2\abs\frkm+2})$.
Note that $\bfv_{N}\weakto\bfv$ in $\calS'(\bbR^{2\abs\frkm+2})$.

We now claim that for every fixed $t>0$, $\bfv(t,x)$ is smooth in
$x$. Fix $N_{0}\in2^{\bbZ_{\geq0}}$, and consider the decomposition
\[
\bfv(t,x)=\bfv_{N_{0}}(t,x)+\sum_{N\in2^{\bbZ_{\geq0}},N\geq N_{0}}(\bfv_{2N}-\bfv_{N})(t,x),
\]
where the series converges in $\calS'(\bbR^{2\abs\frkm+2})$. Observe
first that $\bfv_{N_{0}}(t,x)$ is smooth in $x$ by the smoothness
of $\frac{1}{(4\pi it)^{\abs\frkm+1}}e^{i\frac{\abs x^{2}}{4t}}$
and the compact support of $\abs y^{\nu-\abs\frkm}\chi_{\aleq1}(\tfrac{\abs y}{N_{0}})$
(this is the local smoothing phenomenon). To handle the remaining
sum, consider 
\begin{align*}
(\bfv_{2N}-\bfv_{N})(t,x) & =\int_{\bbR^{2\abs\frkm+2}}\frac{1}{(4\pi it)^{\abs\frkm+1}}e^{i\frac{\abs{x-y}^{2}}{4t}}\abs y^{\nu-\abs\frkm}\left(\chi_{\aleq1}(\tfrac{\abs y}{2N})-\chi_{\aleq1}(\tfrac{\abs y}{N})\right)\,dy.
\end{align*}
Let $x$ lie in the ball $\abs x<\tfrac{1}{2}N_{0}\leq\tfrac{1}{2}N$.
Integrating by parts repeatedly using $-\frac{t(x-y)^{j}}{2i\abs{x-y}^{2}}\rd_{y^{j}}e^{i\frac{\abs{x-y}^{2}}{4t}}=e^{i\frac{\abs{x-y}^{2}}{4t}}$,
and observing that $\abs y\aeq\abs{x-y}\aeq N$ on the support of
the integrand, we have 
\begin{align*}
\abs{\rd_{x}^{\alp}(\bfv_{2N}-\bfv_{N})(t,x)} & \aleq_{k}t^{-\abs\frkm-1}(\tfrac{t}{N^{2}})^{k}(\tfrac{N}{t})^{\abs\alp}N^{\Re(\nu)+\abs\frkm+2}
\end{align*}
for any $k=0,1,2,\ldots$. Choosing $k$ sufficiently large depending
on $\abs\alp+\Re(\nu)+\abs\frkm$, we see that $\sum_{N\in2^{\bbZ_{\geq0}},\,N\geq N_{0}}\rd_{x}^{\alp}(\bfv_{2N}-\bfv_{N})(t,x)$
is uniformly convergent on $\set{\abs x\leq\frac{1}{2}N_{0}}$. Since
$\alp$ and $N_{0}$ were arbitrary, the claim follows.

Next, observe from the rotational and scaling invariances of $e^{it\lap_{2\abs\frkm+2}}$
that $\bfv(t,x)$ is radial and self-similar in the sense that $\bfv(t,x)=t^{\frac{\nu-\abs\frkm}{2}}\bfV(t^{-\frac{1}{2}}r)$
for some smooth function $\bfV$ and $r=\abs x$. Moreover, $V(Y)=Y^{\abs\frkm}\bfV(Y)$
solves the ODE $\calA V=0$. By the smoothness of $\bfv(1,x)=\bfV(r)$
near the origin, it follows that $V$ is proportional to $e_{1}$.
On the other hand, since $\set{f_{1},f_{2}}$ is a fundamental system
for the ODE $\calA V=0$, there exists a unique pair $(\alp',\alp)$
such that $V=\alp'f_{1}+\alp f_{2}$. Moreover, $\alp'=1$ by comparing
$\lim_{t\to0}t^{\frac{\nu}{2}}V(t^{-\frac{1}{2}}r)=r^{\nu}$ with
the behaviors of $t^{\frac{\nu}{2}}f_{1}$ and $t^{\frac{\nu}{2}}f_{2}$
as $t\to0$.

It remains to compute $p$. We fix $t=1$ and observe that 
\begin{align*}
p & =(4i)^{-\frac{\nu-\abs\frkm}{2}}\bfv(1,0)=(4i)^{-\frac{\nu-\abs\frkm}{2}}\lim_{N\to\infty}\bfv_{N}(1,0)\\
 & =(4i)^{-\frac{\nu-\abs\frkm}{2}}\lim_{N\to\infty}\int_{\bbR^{2\abs\frkm+2}}\frac{1}{(4\pi i)^{\abs\frkm+1}}e^{i\frac{\abs y^{2}}{4}}\abs y^{\nu-\abs\frkm}\chi_{\aleq1}(\tfrac{\abs y}{N})\,dy\\
 & =(4i)^{-\frac{\nu-\abs\frkm}{2}}\abs{\bbS^{2\abs\frkm+1}}\lim_{N\to\infty}\int_{0}^{\infty}\frac{1}{(4\pi i)^{\abs\frkm+1}}e^{i\frac{r^{2}}{4}}r^{\nu+\abs\frkm+1}\chi_{\aleq1}(\tfrac{r}{N})\,dr\\
 & =(4i)^{-\frac{\nu-\abs\frkm}{2}}\frac{2^{\nu-\abs\frkm+1}i^{-\abs\frkm-1}}{\Gmm(\abs\frkm+1)}\lim_{N\to\infty}\int_{0}^{\infty}e^{i\rho^{2}}\rho^{\nu+\abs m+1}\chi_{\aleq1}(2N^{-1}\rho)\,d\rho.
\end{align*}
Performing a contour integral (and integrating by parts repeatedly
using $\frac{1}{2i\rho}\rd_{\rho}e^{i\rho^{2}}=e^{i\rho^{2}}$), we
easily see that 
\begin{align*}
\int_{0}^{\infty}e^{i\rho^{2}}\rho^{\nu+\abs m+1}\chi_{\aleq1}(2N^{-1}\rho)\,d\rho=\frac{i^{\frac{\nu+\abs\frkm+2}{2}}}{2}\Gmm(\tfrac{\nu+\abs\frkm}{2}+1)+o_{N\to\infty}(1).
\end{align*}
Hence, $p=\frac{\Gmm(\tfrac{\nu+\abs\frkm}{2}+1)}{\Gmm(\abs\frkm+1)}$,
as desired. \qedhere 
\end{proof}
\begin{rem}
\label{rem:lin-ss-connect-ODE} Here we sketch another way to prove
Lemma~\ref{lem:lin-ss-connect} via the theories of ODEs and special
functions. First observe that under the change of variables $z=i\tfrac{Y^{2}}{4}$
and $V=Y^{\abs\frkm}w$, the ODE $\calA_{\frkm,\nu}V=0$ becomes the
\emph{confluent hypergeometric equation} \cite[Ch.~7, Eq.~(9.01)]{Olver}
\[
z\rd_{zz}w+(c-z)\rd_{z}w-aw=0
\]
with $c=\abs\frkm+1$ and $a=-\frac{\nu-\abs\frkm}{2}$. It follows
that $e_{1}=Y^{\abs\frkm}M(a,c,i\tfrac{Y^{2}}{4})$, where $M(a,c,z)$
is Kummer's function \cite[Ch.~7, Eq.~(9.03)]{Olver}. Moreover, we
have $f_{1}=(\tfrac{i}{4})^{a}Y^{\abs\frkm}U(a,c,i\tfrac{Y^{2}}{4})$
and $f_{2}=(-\tfrac{i}{4})^{-a+c}V(a,c,z)$ \cite[Ch.~7, Eq.~(10.01) and (10.02)]{Olver}.
By the connection formula \cite[Ch.~7, Eq.~(9.04) and (10.09)]{Olver}
\[
M(a,c,z)=\Gmm(c)\left(\frac{e^{a\pi i}}{\Gmm(c-a)}U(a,c,z)+\frac{e^{(a-c)\pi i}}{\Gmm(a)}V(a,c,z)\right),
\]
it follows that $(4i)^{\frac{\nu-\abs\frkm}{2}}p=(\tfrac{i}{4})^{a}e^{-a\pi i}\frac{\Gmm(c-a)}{\Gmm(c)}$,
which implies $p=\frac{\Gmm(\frac{\nu+\abs\frkm}{2}+1)}{\Gmm(\abs\frkm+1)}$. 
\end{rem}


With $\alp(\nu)$ chosen according to Lemma~\ref{lem:lin-ss-connect} (see also Remark~\ref{rem:lin-ss-connect-ODE}), the resulting $\wh z_{\lin}$ has the following properties.
\begin{prop}[Properties of $\wh z_{\lin}$]
Let $\wh z_{\lin}$ be defined as in \eqref{eq:lin-rad-hat} with
$\alp(\nu)$ given by Lemma~\ref{lem:lin-ss-connect}. 
\begin{itemize}
\item (Expansion in the self-similar zone) For every $K\geq0$ and any $r\leq t^{\frac{1}{2}}$,
we have 
\begin{equation}
\wh z_{\lin}=qp(4it)^{\frac{\nu-|\frkm|}{2}}r^{|\frkm|}\Big(1+\sum_{n=1}^{K-1}c_{e_{1},2n}(t^{-\frac{1}{2}}r)^{2n}+e_{e_{1},2K}(t^{-\frac{1}{2}}r)\Big),\label{eq:z-lin-hat-ss}
\end{equation}
where $c_{e_{1},2n}$ and $e_{e_{1},2K}$ are defined in Lemma~\ref{lem:lin-ss-fundsys}.
\item (Global-in-space expansion) For every $K\geq0$, we have 
\begin{equation}
\wh z_{\lin}(t,r)=\sum_{n=0}^{K-1}t^{n}\cdot qc_{f_{1},2n}r^{\nu-2n}\chi(r)+e_{\lin,K}(t,r)\label{eq:z-lin-hat-expn}
\end{equation}
with the bound for all $k\in\bbN$
\begin{equation}
\begin{aligned}|e_{\lin,K}|_{k} & \aleq_{k}\chf_{r\leq t^{\frac{1}{2}}}\{t^{\frac{1}{2}(\Re(\nu)-|\frkm|)}r^{|\frkm|}+\chf_{K\geq1}r^{\Re(\nu)-2(K-1)}\}\\
 & \qquad+\chf_{t^{\frac{1}{2}}\leq r\leq2}\{t^{K}r^{\Re(\nu)-2K}+t^{\Re(\nu)+1-k}r^{-(\Re(\nu)+2)+2k}\}.
\end{aligned}
\label{eq:elin-bound}
\end{equation}
\item (Expansion for the linear error) For every $K\geq0$, we have 
\begin{equation}
i\rd_{t}\wh z_{\lin}+\lap^{(\frkm)}\wh z_{\lin}=\sum_{n=0}^{K-1}t^{n}P_{\lin,n}(r)+E_{\lin,K}(t,r)\label{eq:z-lin-hat-eqn-expn}
\end{equation}
with the following bounds for all $k\in\bbN$ 
\begin{align}
|P_{\lin,n}|_{k} & \aleq_{k}\chf_{1\leq r\leq2},\label{eq:Plin-bound}\\
|E_{\lin,K}|_{k} & \aleq_{k}\chf_{1\leq r\leq2}\{t^{K}+t^{\Re(\nu)-k}\}.\label{eq:Elin-bound}
\end{align}
\end{itemize}
\end{prop}

\begin{proof}
\textbf{Step 1.} Expansion in the self-similar zone.

The expansion \eqref{eq:z-lin-hat-ss} immediately follows from the
definition \eqref{eq:lin-rad-hat}, $r\leq t^{\frac{1}{2}}$, and
Lemma~\ref{lem:lin-ss-connect}.

\smallskip
\textbf{Step 2.} Global-in-space expansion.

In this step, we show that $\wh z_{\lin}$ admits the expansion \eqref{eq:z-lin-hat-expn}
with the error bound \eqref{eq:elin-bound}. In the self-similar zone
$r\leq t^{\frac{1}{2}}$, we use \eqref{eq:f1-to-e1} to have 
\begin{align*}
e_{\lin,K}(t,r) & =\wh z_{\lin}(t,r)-\tsum{n=0}{K-1}t^{n}\cdot qc_{f_{1},2n}r^{\nu-2n}\chi(r)\\
 & =t^{\frac{\nu}{2}}(4i)^{\frac{\nu-\abs\frkm}{2}}pe_{1}(t^{-\frac{1}{2}}r)\chi(r)-\tsum{n=0}{K-1}t^{n}\cdot qc_{f_{1},2n}r^{\nu-2n}\chi(r).
\end{align*}
Estimating each term, we have the desired bound 
\[
\chf_{r\leq t^{\frac{1}{2}}}|e_{\lin,K}|_{k}\aleq_{k}\chf_{r\leq t^{\frac{1}{2}}}\{t^{\frac{1}{2}(\Re(\nu)-|\frkm|)}r^{|\frkm|}+\chf_{K\geq1}r^{\Re(\nu)-2(K-1)}\}.
\]
Outside the self-similar zone, i.e., when $r\geq t^{\frac{1}{2}}$,
we have 
\begin{align*}
e_{\lin,K}(t,r) & =\wh z_{\lin}(t,r)-\tsum{n=0}{K-1}t^{n}\cdot qc_{f_{1},2n}r^{\nu-2n}\chi(r)\\
 & =q\{t^{\frac{\nu}{2}}f_{1}(t^{-\frac{1}{2}}r)-\tsum{n=0}{K-1}t^{n}\cdot c_{f_{1},2n}r^{\nu-2n}\}\chi(r)+qt^{\frac{\nu}{2}}f_{2}(t^{-\frac{1}{2}}r)\chi(r)\\
 & =qt^{\frac{\nu}{2}}\{e_{f_{1},2K}(t^{-\frac{1}{2}}r)+f_{2}(t^{-\frac{1}{2}}r)\}\chi(r).
\end{align*}
Next, we apply \eqref{eq:e_f-bound}. We note from \eqref{eq:f2-expn}
that $f_{2}(Y)$ has a pseudoconformal phase $e^{i\frac{Y^{2}}{4}}$.
Then, we have the desired bound 
\begin{align*}
\chf_{r\geq t^{\frac{1}{2}}}|e_{\lin,K}|_{k} & \aleq_{k}\chf_{t^{\frac{1}{2}}\leq r\leq2}t^{\frac{1}{2}\Re(\nu)}\{(t^{-\frac{1}{2}}r)^{\Re(\nu)-2K}+(t^{-\frac{1}{2}}r)^{-(\Re(\nu)+2)+2k}\}\\
 & \aleq_{k}\chf_{t^{\frac{1}{2}}\leq r\leq2}\{t^{K}r^{\Re(\nu)-2K}+t^{\Re(\nu)+1-k}r^{-(\Re(\nu)+2)+2k}\}.
\end{align*}
This completes the proof of \eqref{eq:elin-bound}.

\smallskip
\textbf{Step 3.} Expansion for the linear error.

In this step, we prove the expansion \eqref{eq:z-lin-hat-eqn-expn}
with the bounds \eqref{eq:Plin-bound} and \eqref{eq:Elin-bound}.
By the definition \eqref{eq:lin-rad-hat} of $\wh z_{\lin}$, we have
\[
i\rd_{t}\wh z_{\lin}+\Dlt^{(\frkm)}\wh z_{\lin}=qt^{\frac{\nu}{2}}[\Dlt^{(\frkm)},\chi]\{f_{1}(t^{-\frac{1}{2}}r)+\alp f_{2}(t^{-\frac{1}{2}}r)\}.
\]
Using the expansion \eqref{eq:f1-expn} of $f_{1}(Y)$, we obtain
\[
i\rd_{t}\wh z_{\lin}+\Dlt^{(\frkm)}\wh z_{\lin}=\sum_{n=0}^{K-1}t^{n}P_{\lin,n}(r)+E_{\lin,K}(t,r)
\]
with (set $c_{f_{1},0}=1$)
\begin{align*}
P_{\lin,n}(r) & =qc_{f_{1},2n}[\Dlt^{(\frkm)},\chi]r^{\nu},\\
E_{\lin,K}(t,r) & =[\Dlt^{(\frkm)},\chi]\{qr^{\nu}e_{f_{1},2K}(t^{-\frac{1}{2}}r)+\alp t^{\frac{\nu}{2}}f_{2}(t^{-\frac{1}{2}}r)\}.
\end{align*}
Now, \eqref{eq:Plin-bound} is immediate from the previous display.
We show \eqref{eq:Elin-bound}. The contribution of the first term
of $E_{\lin,K}$ is estimated using \eqref{eq:e_f-bound}: 
\[
\big|[\Dlt^{(\frkm)},\chi]qr^{\nu}e_{f_{1},2K}(t^{-\frac{1}{2}}r)\big|_{k}\aleq_{k}\chf_{1\leq r\leq2}Y^{-2K}\aleq_{k}\chf_{1\leq r\leq2}t^{K}.
\]
For the contribution of the second term of $E_{\lin,K}$, recall from
\eqref{eq:f2-expn} that we have the pseudoconformal phase factor
$e^{i\frac{Y^{2}}{4}}$ for $f_{2}$. Using the bound \eqref{eq:e_f-bound}
for $f_{2}$ with $K=0$, we obtain 
\begin{align*}
\big|[\Dlt^{(\frkm)},\chi]qt^{\frac{\nu}{2}}f_{2}(t^{-\frac{1}{2}}r)\big|_{k} & \aleq_{k}\chf_{1\leq r\leq2}|t^{\frac{\nu}{2}}f_{2}(t^{-\frac{1}{2}}r)|_{k+1}\\
 & \aleq_{k}\chf_{1\leq r\leq2}t^{\frac{1}{2}\Re(\nu)}Y^{-\Re(\nu)+2k}\aleq_{k}\chf_{1\leq r\leq2}t^{\Re(\nu)-k}.
\end{align*}
This completes the proof.
\end{proof}
\begin{rem}[A more precise ansatz and propagation of singularity]
\label{rem:z-lin-td} Instead of $\wh{z}_{\lin}$, one may consider the ansatz
\[
\td z_{\lin}=q\{f_{1}(t^{-\frac{1}{2}}r)\chi(r)+\alp(\nu)f_{2}(t^{-\frac{1}{2}}r)\},
\]
where $\alp(\nu)$ is given by Lemma~\ref{lem:lin-ss-connect}. This
$\td z_{\lin}(t,r)$ is also smooth for all $t>0$ and converges to
$z^{\ast}(r)$ as $t\to0$, but it has no cutoff for the $f_{2}$-term.
This property makes two notable differences compared to $\wh z_{\lin}$:
(i) $f_{2}$ has the same Sobolev regularity (and no better)
as $qr^{\nu}\chi(r)$ (namely, $H^{\Re(\nu)+1-}$); and (ii) the linear equation error is smoother
than \eqref{eq:Elin-bound} in the sense that it has bounded derivatives
of all orders: 
\[
i\rd_{t}\td z_{\lin}+\lap^{(\frkm)}\td z_{\lin}=\sum_{n=0}^{K-1}t^{n}P_{\lin,n}(r)+\td E_{\lin,K}(t,r),
\]
where $P_{\lin,n}(r)$ is the same profile as in \eqref{eq:z-lin-hat-eqn-expn}
and $\td E_{\lin,K}$ satisfies 
\[
|\td E_{\lin,K}|_{k}\aleq_{k}\chf_{1\leq r\leq2}t^{K}.
\]
The error estimate is more favorable because the cutoff error
involving $f_{2}$ (hence involving the pseudoconformal phase) does
not appear. 

It turns out that the more naive (and regular) ansatz $\wh{z}_{\lin}$ leads to some technical simplification and still produces an approximate radiation $z$ with error bounds sufficient for our blow-up construction, hence our choice. Nevertheless, the properties of $\td{z}_{\lin}$ provides some insights about the regularity of $u$ (which we do not precisely prove in this paper). In view of (i) and (ii), $\td{z}_{\lin}$ may be regarded as a more accurate reflection of propagation of singularity (of $q r^{\nu} \chi(r)$ at $r = 0$) under the Schr\"odinger equation. The limited Sobolev regularity of $\td{z}_{\lin}$ (due to its behavior in the self-similar zone $r \sim t^{\frac{1}{2}}$) is also expected to manifest in the full solution $u$ we construct below, because the self-similar zone is well-separated from the blow-up scale (i.e., $\lmb(t) \ll t^{\frac{1}{2}}$). This feature is the Schr\"odinger analogue of the propagation of singularity across the backward light cone in \cite{KriegerSchlagTataru2008Invent, JendrejLawrieRodriguez2019arXiv}. 
\end{rem}

\begin{rem}[More general class of initial data $z^{\ast}(r)$]
Our construction may be generalized to a larger class of initial
data with some additional ideas. One useful idea is to differentiate
in the parameter $\nu$; this idea would lead to the construction
of $\td z_{\lin}$ corresponding to $z^{\ast}=r^{\nu}\log r\chi(r)$.
More generally, we may construct $\td z_{\lin}$ corresponding to
any data $z_{sing}^{\ast}(r)\chi(r)$ where $z_{sing}^{\ast}$ is
expressible as $\int r^{\nu}\varphi(\nu)\,\ud\nu$ for some $\varphi\in\calD'(\bbC)$,
provided that $\supp\varphi\subseteq\set{\Re(\nu)>-\abs\frkm-2}$
and $\varphi$ obeys suitable decay conditions as $\abs\nu\to\infty$. 
\end{rem}

\subsection{Proof of Proposition~\ref{prop:ApprxRadiation}} \label{subsec:z-nonlin}

As stated before, \textbf{we assume that $m=0$, hence $\frkm=-2$,
unless otherwise stated}. Also, \textbf{we assume that $0<t<1$};
the case $-1<t<0$ is handled by time reversal symmetry. Recall that
implicit constants may depend on $q$ and $\nu$.

The challenge is to ensure a sufficient rate of vanishing in $t$
for the error. The idea is to perform a formal expansion of the solution
in $t$ and include them in the approximate radiation; we will set
the approximate radiation of the form
\begin{equation}
z(t,r)=\wh z_{\lin}(t,r)+\chi_{r\ageq t^{\frac{1}{2}}}\sum_{n=1}^{N}t^{n}g_{n}(r)\label{eq:z-form}
\end{equation}
with $\wh z_{\lin}$ defined in \eqref{eq:lin-rad-hat} and some appropriate
choices of $N\in\bbN$ and functions $g_{n}(r)$. Explicitly, we choose
\begin{equation}
N=\Big\lfloor\frac{\Re(\nu)+1}{2}\Big\rfloor\label{eq:def-N}
\end{equation}
and define $g_{n}$, $n=0,1,\dots,N$, by the recursion relation 
\begin{equation}
\left\{ \begin{aligned}g_{0}(r) & \coloneqq0,\\
g_{n+1}(r) & \coloneqq\frac{i}{n+1}\Big(\Dlt^{(\frkm)}g_{n}+P_{\lin,n}-\sum_{k\in\{3,5\}}\sum_{n_{1}+\dots+n_{k}=n}\rN_{k}(h_{n_{1}},\dots,h_{n_{k}})\Big),
\end{aligned}
\right.\label{eq:def-gn}
\end{equation}
where we denoted 
\begin{equation}
h_{n}(r)\coloneqq qc_{f_{1},2n}r^{\nu-2n}\chi(r)+g_{n}(r).\label{eq:def-hn}
\end{equation}
The goal is then to show that this $z(t,r)$ satisfies the statements
in Proposition~\ref{prop:ApprxRadiation}.

We first motivate the choices \eqref{eq:def-N} and \eqref{eq:def-gn}.
In the exterior zone $r\ageq t^{\frac{1}{2}}$, \eqref{eq:z-form}
, and \eqref{eq:z-lin-hat-expn} imply that $z(t,r)$ is approximated
by 
\[
z(t,r)\approx\sum_{n=0}^{N}t^{n}h_{n}(r),\qquad h_{n}(r)=qc_{f_{1},2n}r^{\nu-2n}\chi(r)+g_{n}(r),
\]
where $g_{1},\dots,g_{N}$ are yet to be constructed. A rough computation
, together with \eqref{eq:z-lin-hat-eqn-expn}, gives 
\begin{align*}
\Psi_{z} & =i\rd_{t}z+\Dlt^{(\frkm)}z-\rN(z)\\
 & \approx(i\rd_{t}\wh z_{\lin}+\Dlt^{(\frkm)}\wh z_{\lin})+(i\rd_{t}+\Dlt^{(\frkm)})\sum_{n=1}^{N}t^{n}g_{n}(r)-\rN\Big(\sum_{n=0}^{N}t^{n}h_{n}(r)\Big)\\
 & \approx\sum_{n=0}^{N-1}t^{n}\Big(P_{\lin,n}+i(n+1)g_{n+1}+\Dlt^{(\frkm)}g_{n}-\sum_{\ell\in\{3,5\}}\sum_{n_{1}+\dots+n_{\ell}=n}\rN_{\ell}(h_{n_{1}},\dots,h_{n_{\ell}})\Big).
\end{align*}
In order to ensure a sufficient rate of vanishing in $t$ for $\Psi_{z}$,
we choose $g_{n}$s such that the right hand vanishes; this motivates 
the recursive definition \eqref{eq:def-gn} of $g_{n}$. Formally,
this expansion method up to $N$th order makes the error term $\Psi_{z}$
of size $\aeq t^{N}$. Comparing this with one of the necessary estimates
\eqref{eq:Psi_z-L2} for $\Psi_{z}$ implies that $N>\frac{1}{2}(\Re(\nu)-1)$
is necessary; the choice \eqref{eq:def-N} of $N$ is the minimal
one with this requirement. 

We will now proceed with the proof of Proposition~\ref{prop:ApprxRadiation}
. 
\begin{lem}[Pointwise estimates for the approximate radiation $z(t,r)$]
\label{lem:ptwise-est-z}~
\begin{itemize}
\item (Pointwise bounds in the self-similar zone) The estimates \eqref{eq:z-ss-bound}--\eqref{eq:rd-z-ss-degen}
hold.
\item (Estimates for the expansion of $z(t,r)$) For all $n=0,1,\dots,N$
and $k\in\bbN$, we have the pointwise estimates
\begin{align}
|g_{n}|_{k} & \aleq_{k}\chf_{r\leq2}r^{3\Re(\nu)+2-2n},\label{eq:gn-bounds}\\
|h_{n}|_{k} & \aleq_{k}\chf_{r\leq2}r^{\Re(\nu)-2n}.\label{eq:hn-bounds}
\end{align}
The estimate \eqref{eq:z-ext-bound} also holds. Finally, we have
\begin{equation}
z(t,r)=\sum_{n=0}^{K-1}t^{n}h_{n}(r)+E_{z,K}(t,r),\qquad\forall K=0,1,\dots,N,\label{eq:z-expn}
\end{equation}
where 
\begin{equation}
|E_{z,K}|_{1}\aleq\chf_{r\leq t^{\frac{1}{2}}}t^{\frac{1}{2}(\Re(\nu)-1)}r+\chf_{t^{\frac{1}{2}}\leq r\leq2}t^{K}r^{\Re(\nu)-2K}.\label{eq:z-expn-err-bound}
\end{equation}
\end{itemize}
\end{lem}

\begin{proof}
\textbf{Step 1.} Pointwise bounds in the self-similar zone.

In this step, we show \eqref{eq:z-ss-bound}--\eqref{eq:rd-z-ss-degen}.
In the self-similar zone $r\leq t^{\frac{1}{2}}$, we have $z(t,r)=\wh z_{\lin}(t,r)$.
Applying \eqref{eq:z-lin-hat-ss} and \eqref{eq:e_e-bound} with $|\frkm|=2$
and $K=0,1$ implies 
\begin{align*}
|z(t,r)|_{k} & \aleq_{k}t^{\frac{1}{2}(\Re(\nu)-2)}r^{2},\\
|z(t,r)-qp(4it)^{\frac{\nu-2}{2}}r^{2}|_{k} & \aleq_{k}t^{\frac{1}{2}(\Re(\nu)-4)}r^{4},
\end{align*}
which are \eqref{eq:z-ss-bound} and \eqref{eq:z-ss-leading}. To
show \eqref{eq:z1-ss-leading}, we begin with 
\[
z_{1}=\bfD_{z}^{(\frkm)}z=\rd_{r}z-\frac{-2+A_{\tht}[z]}{r}z=\Big(\rd_{r}+\frac{2}{r}\Big)z-\frac{A_{\tht}[z]}{r}z.
\]
By \eqref{eq:z-ss-bound} and \eqref{eq:z-ss-leading}, we have 
\[
|(\rd_{r}z+\tfrac{2}{r}z)-4qp(4it)^{\frac{\nu-2}{2}}r|_{k}\aleq_{k}t^{\frac{1}{2}(\Re(\nu)-4)}r^{3}
\]
whereas 
\[
|r^{-1}A_{\tht}[z]z|_{k}\aleq r^{-1}\cdot t^{\Re(\nu)-2}r^{6}\cdot t^{\frac{1}{2}(\Re(\nu)-2)}r^{2}\aeq t^{\frac{3}{2}(\Re(\nu)-2)}r^{7}\aleq t^{\frac{3}{2}\Re(\nu)-1}r^{3}.
\]
As $\frac{3}{2}\Re(\nu)-1>\frac{1}{2}(\Re(\nu)-4)$, the nonlinear
term $r^{-1}A_{\tht}[z]z$ is absorbed into the error term, completing
the proof of \eqref{eq:z1-ss-leading}. Finally, \eqref{eq:rd-z-ss-degen}
follows from $\rd_{-}^{(\frkm)}=\rd_{r}-\frac{2}{r}$ and \eqref{eq:z-ss-leading}.

\smallskip
\textbf{Step 2.} Bounds for $g_{n}$ and $h_{n}$.

In this step, we show \eqref{eq:gn-bounds} and \eqref{eq:hn-bounds}
by induction on $n$. For $n=0$, these bounds are immediate from
$g_{0}\equiv0$ and $h_{n}=qr^{\nu}\chi(r)$. Now assume the bounds
\eqref{eq:gn-bounds} and \eqref{eq:hn-bounds} for all $n=0,\dots,K$
for some $0\leq K\leq N-1$. Then, the formulas \eqref{eq:def-gn}
and \eqref{eq:def-hn} imply
\begin{align*}
|g_{K+1}|_{k} & \aleq_{k}r^{-2}|g_{K}|_{k+2}+|P_{\lin,K}|_{k}+\sum_{\ell\in\{3,5\}}\sum_{n_{1}+\dots+n_{\ell}=K}|\rN_{\ell}(h_{n_{1}},\dots,h_{n_{\ell}})|_{k},\\
|h_{K+1}|_{k} & \aleq_{k}\chf_{r\leq2}r^{\Re(\nu)-2(K+1)}+|g_{K+1}|_{k}.
\end{align*}
Using the inductive bound for $g_{K}$, \eqref{eq:Plin-bound} for
$P_{\lin,K}$, 
\[
\sum_{n_{1}+\dots+n_{\ell}=K}|\rN_{\ell}(h_{n_{1}},\dots,h_{n_{\ell}})|_{k}\aleq_{k}\begin{cases}
\chf_{r\leq2}r^{3\Re(\nu)-2K} & \text{if }\ell=3,\\
\chf_{r\leq2}r^{5\Re(\nu)+2-2K} & \text{if }\ell=5
\end{cases}
\]
for the nonlinearity and $\Re(\nu)>0$, we obtain
\begin{align*}
|g_{K+1}|_{k} & \aleq_{k}\chf_{r\leq2}\{r^{3\Re(\nu)-2K}+r^{5\Re(\nu)+2-2K}\}\aleq_{k}\chf_{r\leq2}r^{3\Re(\nu)-2K},\\
|h_{K+1}|_{k} & \aleq_{k}\chf_{r\leq2}\{r^{\Re(\nu)-2(K+1)}+r^{3\Re(\nu)-2K}\}\aleq_{k}\chf_{r\leq2}r^{\Re(\nu)-2(K+1)},
\end{align*}
which are \eqref{eq:gn-bounds} and \eqref{eq:hn-bounds} for $n=K+1$.

\smallskip
\textbf{Step 3.} Proof of \eqref{eq:z-ext-bound}.

As $z(t,r)=\wh z_{\lin}(t,r)+\chi_{r\ageq t^{\frac{1}{2}}}\sum_{n=0}^{N-1}t^{n}g_{n}(r)$,
we use the bounds \eqref{eq:elin-bound} for $K=0$ and \eqref{eq:gn-bounds}
to obtain 
\[
\chf_{r\geq t^{\frac{1}{2}}}|z|_{k}\aleq_{k}\chf_{t^{\frac{1}{2}}\leq r\leq2}\{r^{\Re(\nu)}+t^{\Re(\nu)+1-k}r^{-(\Re(\nu)+2)+2k}\},
\]
which is \eqref{eq:z-ext-bound}.

\smallskip
\textbf{Step 4.} Estimates for the expansion of $z(t,r)$.

In this step, we show that $z(t,r)$ admits the expansion \eqref{eq:z-expn}
with the estimate \eqref{eq:z-expn-err-bound}. Writing $z(t,r)$
as in \eqref{eq:z-expn}, we have 
\begin{align*}
E_{z,K}(t,r) & =z(t,r)-\tsum{n=0}{K-1}t^{n}h_{n}(r)\\
 & =\{\wh z_{\lin}(t,r)-\tsum{n=0}{K-1}t^{n}\cdot qc_{f_{1},2n}r^{\nu-2n}\chi(r)\}+\chi_{r\aleq t^{\frac{1}{2}}}\tsum{n=0}{K-1}t^{n}g_{n}(r).
\end{align*}
For the first term, we apply \eqref{eq:elin-bound} with $k=1$, $|\frkm|=2$,
$\Re(\nu)-2(K-1)\geq1$, and $K\leq\Re(\nu)$ to have 
\begin{align*}
 & \big|\wh z_{\lin}(t,r)-\tsum{n=0}{K-1}t^{n}\cdot qc_{f_{1},2n}r^{\nu-2n}\chi(r)\big|_{1}\\
 & \quad\aleq\chf_{r\leq t^{\frac{1}{2}}}t^{\frac{1}{2}(\Re(\nu)-1)}r+\chf_{t^{\frac{1}{2}}\leq r\leq2}t^{K}r^{\Re(\nu)-2K}.
\end{align*}
For the second term, we apply \eqref{eq:gn-bounds} and $K-1\leq\frac{1}{2}(\Re(\nu)-1)$
to have 
\begin{align*}
 & \big|\chi_{r\aleq t^{\frac{1}{2}}}\tsum{n=0}{K-1}t^{n}g_{n}(r)\big|_{1}\\
 & \quad\aleq\chf_{K\geq1}\cdot\chf_{r\leq2t^{\frac{1}{2}}}t^{K-1}r^{3\Re(\nu)+2-2(K-1)}\aleq\chf_{r\leq2t^{\frac{1}{2}}}t^{\frac{1}{2}(\Re(\nu)-1)}r^{2\Re(\nu)+3}.
\end{align*}
Previous two displays imply \eqref{eq:z-expn-err-bound}.
\end{proof}
\begin{lem}[Estimates for the expansion of $i\rd_{t}z+\Dlt^{(\frkm)}z$]
We have 
\begin{equation}
i\rd_{t}z+\Dlt^{(\frkm)}z=\chi_{r\ageq t^{\frac{1}{2}}}\sum_{n=0}^{N-1}t^{n}\Big(i(n+1)\cdot g_{n+1}(r)+P_{\lin,n}(r)+\Dlt^{(\frkm)}g_{n}(r)\Big)+\Psi_{\lin,N}(t,r)\label{eq:z-lin-evol-error}
\end{equation}
with the estimates for all $k\in\bbN$
\begin{equation}
|\Psi_{\lin,N}|_{k}\aleq_{k}\chf_{t^{\frac{1}{2}}\leq r\leq2}t^{N}r^{3\Re(\nu)-2N}+\chf_{1\leq r\leq2}t^{\Re(\nu)-k}.\label{eq:Psi-lin-bound}
\end{equation}
\end{lem}

\begin{proof}
By the definition \eqref{eq:z-form} of $z$, we have 
\begin{align*}
i\rd_{t}z+\Dlt^{(\frkm)}z & =(i\rd_{t}\wh z_{\lin}+\Dlt^{(\frkm)}\wh z_{\lin})+\chi_{r\ageq t^{\frac{1}{2}}}(i\rd_{t}+\Dlt^{(\frkm)})\Big(\sum_{n=1}^{N}t^{n}g_{n}(r)\Big)\\
 & \peq+[i\rd_{t}+\Dlt^{(\frkm)},\chi_{r\ageq t^{\frac{1}{2}}}]\Big(\sum_{n=1}^{N}t^{n}g_{n}(r)\Big).
\end{align*}
Applying \eqref{eq:z-lin-hat-eqn-expn} for the first term, using
$P_{\lin,n}(r)=\chi_{r\ageq t^{\frac{1}{2}}}P_{\lin,n}(r)$, and rearranging
the second term, we arrive at 
\begin{align*}
 & i\rd_{t}z+\Dlt^{(\frkm)}z\\
 & =\chi_{r\ageq t^{\frac{1}{2}}}\sum_{n=0}^{N-1}t^{n}\Big(i(n+1)\cdot g_{n+1}(r)+P_{\lin,n}(r)+\Dlt^{(\frkm)}g_{n}(r)\Big)\\
 & \peq+E_{\lin,N}(t,r)+\chi_{r\ageq t^{\frac{1}{2}}}t^{N}\Dlt^{(\frkm)}g_{N}(r)+[i\rd_{t}+\Dlt^{(\frkm)},\chi_{r\ageq t^{\frac{1}{2}}}]\Big(\sum_{n=1}^{N}t^{n}g_{n}(r)\Big).
\end{align*}
Therefore, $\Psi_{\lin,N}$ collects the terms in the last line, and
it suffices to estimate each term. First, applying \eqref{eq:Elin-bound}
with $K=N$, we have 
\[
|E_{\lin,N}|_{k}\aleq_{k}\chf_{1\leq r\leq2}\{t^{N}+t^{\Re(\nu)-k}\}.
\]
Next, we apply \eqref{eq:gn-bounds} to have 
\[
\Big|\chi_{r\ageq t^{\frac{1}{2}}}t^{N}\Dlt^{(\frkm)}g_{N}(r)+[i\rd_{t}+\Dlt^{(\frkm)},\chi_{r\ageq t^{\frac{1}{2}}}](\tsum{n=1}Nt^{n}g_{n}(r))\Big|_{k}\aleq_{k}\chf_{t^{\frac{1}{2}}\leq r\leq2}t^{N}r^{3\Re(\nu)-2N}.
\]
This completes the proof.
\end{proof}
\begin{lem}[Estimates for the expansion of nonlinearity]
We can write 
\begin{align}
\rN_{3}(z)(t,r) & =\sum_{n=0}^{N-1}t^{n}P_{\rN_{3},n}(r)+E_{\rN_{3},N}(t,r),\label{eq:rN-ext-cubic}\\
\rN_{5}(z)(t,r) & =\sum_{n=0}^{N-1}t^{n}P_{\rN_{5},n}(r)+E_{\rN_{5},N}(t,r),\label{eq:rN-ext-quintic}
\end{align}
where 
\begin{align*}
P_{\rN_{3},n}(r) & =\sum_{n_{1}+n_{2}+n_{3}=n}\rN_{3}(h_{n_{1}},h_{n_{2}},h_{n_{3}}),\\
P_{\rN_{5},n}(r) & =\sum_{n_{1}+\dots+n_{5}=n}\rN_{5}(h_{n_{1}},\dots,h_{n_{5}}),
\end{align*}
together with the estimates 
\begin{align}
|P_{\rN_{3},n}(r)|_{k} & \aleq_{k}\chf_{r\leq2}r^{3\Re(\nu)-2n},\label{eq:PN3-bound}\\
|P_{\rN_{5},n}(r)|_{k} & \aleq_{k}\chf_{r\leq2}r^{5\Re(\nu)+2-2n},\label{eq:PN5-bound}
\end{align}
and 
\begin{align}
\chf_{r\geq t^{\frac{1}{2}}}|E_{\rN_{3},N}(t,r)|_{1} & \aleq\chf_{t^{\frac{1}{2}}\leq r\leq2}t^{N}r^{3\Re(\nu)-2N},\label{eq:EN3-bound}\\
\chf_{r\geq t^{\frac{1}{2}}}|E_{\rN_{5},N}(t,r)|_{1} & \aleq\chf_{t^{\frac{1}{2}}\leq r\leq2}t^{N}r^{5\Re(\nu)+2-2N}.\label{eq:EN5-bound}
\end{align}
\end{lem}

\begin{proof}
\textbf{Step 1.} Estimates for the cubic nonlinearity $\rN_{3}$.

First, we claim that 
\begin{equation}
|z|^{2}=\sum_{n=0}^{K-1}t^{n}P_{|z|^{2},n}(r)+E_{|z|^{2},K}(t,r),\label{eq:|z|^2-expn}
\end{equation}
where (recall the expansion \eqref{eq:z-expn} of $z(t,r)$) 
\begin{align*}
P_{|z|^{2},n} & =\tsum{n_{1}=0}nh_{n_{1}}\br h_{n-n_{1}},\\
|P_{|z|^{2},n}(r)|_{k} & \aleq_{k}\chf_{r\leq2}r^{2\Re(\nu)-2n},\\
|E_{|z|^{2},K}(t,r)|_{1} & \aleq\chf_{r\leq t^{\frac{1}{2}}}t^{\Re(\nu)-1}r^{2}+\chf_{t^{\frac{1}{2}}\leq r\leq2}t^{K}r^{2\Re(\nu)-2K}.
\end{align*}
The bounds for $P_{|z|^{2},n}(r)$ is immediate from \eqref{eq:hn-bounds}.
To estimate $E_{|z|^{2},K}(t,r)$, the expansion \eqref{eq:z-expn}
for $z(t,r)$ implies that (note that $E_{z,0}(t,r)=z(t,r)$)
\[
E_{|z|^{2},K}=\tsum{n=0}{K-1}t^{n}h_{n}\br E_{z,K-n}+E_{z,K}\br E_{z,0}.
\]
Using \eqref{eq:hn-bounds} and \eqref{eq:z-expn-err-bound}, we have
\begin{align*}
 & |E_{|z|^{2},K}(r)|_{1}\\
 & \aleq\tsum{n=0}{K-1}\{\chf_{r\leq t^{\frac{1}{2}}}t^{n}r^{\Re(\nu)-2n}\cdot t^{\frac{1}{2}(\Re(\nu)-1)}r+\chf_{t^{\frac{1}{2}}\leq r\leq2}t^{n}r^{\Re(\nu)-2n}\cdot t^{K-n}r^{\Re(\nu)-2(K-n)}\}\\
 & \peq+\{\chf_{r\leq t^{\frac{1}{2}}}t^{\Re(\nu)-1}r^{2}+\chf_{t^{\frac{1}{2}}\leq r\leq2}t^{K}r^{2\Re(\nu)-2K}\}\\
 & \aleq\chf_{r\leq t^{\frac{1}{2}}}\{\tsum{n=0}{K-1}t^{n+\frac{1}{2}(\Re(\nu)-1)}r^{\Re(\nu)-2n+1}+t^{\Re(\nu)-1}r^{2}\}+\chf_{t^{\frac{1}{2}}\leq r\leq2}t^{K}r^{2\Re(\nu)-2K}.
\end{align*}
Since $\Re(\nu)-2n\geq1$ for $n=0,\dots,K-1$, we have $t^{n}r^{\Re(\nu)-2n}\aleq t^{\frac{1}{2}(\Re(\nu)-1)}r$
in the zone $r\leq t^{\frac{1}{2}}$. Therefore, 
\[
|E_{|z|^{2},K}(r)|_{1}\aleq\chf_{r\leq t^{\frac{1}{2}}}t^{\Re(\nu)-1}r^{2}+\chf_{t^{\frac{1}{2}}\leq r\leq2}t^{K}r^{2\Re(\nu)-2K}
\]
as desired.

Next, we take the integrals $\int_{0}^{r}|z|^{2}r'dr'$ and $\int_{0}^{r}|z|^{2}\frac{dr'}{r'}$
to obtain 
\begin{align}
A_{\tht}[z](t,r) & =\sum_{n=0}^{K-1}t^{n}P_{A_{\tht},n}(r)+E_{A_{\tht},K}(t,r),\label{eq:A_tht-expn}\\
\rV_{2,2}[z](t,r) & =\sum_{n=0}^{K-1}t^{n}P_{\rV_{2,2},n}(r)+E_{\rV_{2,2},K}(t,r),\label{eq:rV_2,2-expn}
\end{align}
with the estimates (using $N\leq\Re(\nu)$) 
\begin{align*}
\chf_{r\leq2}|P_{A_{\tht},n}|_{k} & \aleq_{k}\chf_{r\leq2}r^{2\Re(\nu)+2-2n},\\
\chf_{r\leq2}|P_{\rV_{2,2},n}|_{k} & \aleq_{k}\chf_{r\leq2}r^{2\Re(\nu)-2n},
\end{align*}
and 
\begin{align*}
\chf_{r\leq2}|E_{A_{\tht},K}|_{1} & \aleq\chf_{r\leq t^{\frac{1}{2}}}t^{\Re(\nu)-1}r^{4}+\chf_{t^{\frac{1}{2}}\leq r\leq2}t^{K}r^{2\Re(\nu)+2-2K},\\
\chf_{r\leq2}|E_{\rV_{2,2},K}|_{1} & \aleq\chf_{r\leq t^{\frac{1}{2}}}t^{\Re(\nu)-1}r^{2}+\chf_{t^{\frac{1}{2}}\leq r\leq2}t^{K}r^{2\Re(\nu)-2K}.
\end{align*}

Now we consider the cubic nonlinearity $\rN_{3}$. Recall that $\rN_{3}$
is a linear combination of $|z|^{2}z$, $\frac{1}{r^{2}}A_{\tht}[z]z$,
and $\int_{0}^{r}|z|^{2}\frac{dr'}{r'}z$. Therefore, we simply multiply
the expansions \eqref{eq:|z|^2-expn}, \eqref{eq:A_tht-expn} (times
$\frac{1}{r^{2}}$), and \eqref{eq:rV_2,2-expn}, with \eqref{eq:z-expn}.
This easily implies the bound \eqref{eq:PN3-bound} for $P_{\rN_{3},n}(r)$.
The bound \eqref{eq:EN3-bound} for $E_{\rN_{3},N}(t,r)$ follows
from 
\begin{align*}
 & \chf_{r\geq t^{\frac{1}{2}}}|E_{\rN_{3},N}(t,r)|_{1}\\
 & \aleq\chf_{t^{\frac{1}{2}}\leq r\leq2}\tsum{n=0}{N-1}\{t^{n}r^{2\Re(\nu)-2n}\cdot t^{N-n}r^{\Re(\nu)-2(N-n)}\}+t^{N}r^{2\Re(\nu)-2N}\cdot r^{\Re(\nu)}\\
 & \aleq\chf_{t^{\frac{1}{2}}\leq r\leq2}t^{N}r^{3\Re(\nu)-2N}.
\end{align*}

\smallskip
\textbf{Step 2.} Estimates for the quintic nonlinearity $\rN_{5}$.

The proof for the quintic nonlinearity is similar. Multiplying the
expansions \eqref{eq:A_tht-expn} and \eqref{eq:|z|^2-expn}, we obtain
\begin{align}
\rV_{4,1}[z]=\frac{1}{r^{2}}(A_{\tht}[z])^{2} & =\sum_{n=0}^{K-1}t^{n}P_{\rV_{4,1},n}(r)+E_{\rV_{4,1},K}(t,r),\label{eq:rV_4,1-expn}\\
A_{\tht}[z]|z|^{2} & =\sum_{n=0}^{K-1}t^{n}P_{A_{\tht}|z|^{2},n}(r)+E_{A_{\tht}|z|^{2},K}(t,r),\nonumber 
\end{align}
with the estimates 
\begin{align*}
\chf_{r\leq2}\{|P_{\rV_{4,1},n}|_{k}+|P_{A_{\tht}|z|^{2},n}|_{k}\} & \aleq_{k}\chf_{r\leq2}r^{4\Re(\nu)+2-2n},\\
\chf_{r\leq2}\{|E_{\rV_{4,1},K}|_{1}+|E_{A_{\tht}|z|^{2},K}|_{1}\} & \aleq\chf_{r\leq t^{\frac{1}{2}}}t^{2\Re(\nu)-2}r^{6}+\chf_{t^{\frac{1}{2}}\leq r\leq2}t^{K}r^{4\Re(\nu)+2-2K}.
\end{align*}
Taking the integral $\int_{0}^{r}\cdot\frac{dr'}{r'}$ for $A_{\tht}[z]|z|^{2}$
and using $N\leq\Re(\nu)$ again, we further have
\begin{equation}
\rV_{4,2}[z]=\sum_{n=0}^{K-1}t^{n}P_{\rV_{4,2},n}(r)+E_{\rV_{4,2},K}(t,r)\label{eq:rV-4,2-expn}
\end{equation}
with the estimates 
\begin{align*}
\chf_{r\leq2}|P_{\rV_{4,2},n}|_{k} & \aleq_{k}\chf_{r\leq2}r^{4\Re(\nu)+2-2n},\\
\chf_{r\leq2}|E_{\rV_{4,2},K}|_{1} & \aleq\chf_{r\leq t^{\frac{1}{2}}}t^{2\Re(\nu)-2}r^{6}+\chf_{t^{\frac{1}{2}}\leq r\leq2}t^{K}r^{4\Re(\nu)+2-2K}.
\end{align*}
Now recall that $\rN_{5}(z)$ is a linear combination of $\rV_{4,1}\cdot z$
and $\rV_{4,2}\cdot z$. Therefore, multiplying the expansions \eqref{eq:rV_4,1-expn}
and \eqref{eq:rV-4,2-expn} with \eqref{eq:z-expn}, we obtain \eqref{eq:PN5-bound}
and \eqref{eq:EN5-bound} as desired.
\end{proof}
We are now ready to finish the proof of Proposition~\ref{prop:ApprxRadiation}.
\begin{proof}[Proof of Proposition~\ref{prop:ApprxRadiation}]
~

\textbf{Step 1.} Justification of initial data.

Here we show \eqref{eq:z-initial-justification}. When $\Re(\nu)\geq1$,
we can apply \eqref{eq:z-expn} and \eqref{eq:z-expn-err-bound} with
$K=1$ to have 
\[
|z(t,r)-qr^{\nu}\chi(r)|_{1}\aleq\chf_{r\leq t^{\frac{1}{2}}}t^{\frac{1}{2}(\Re(\nu)-1)}r+\chf_{t^{\frac{1}{2}}\leq r\leq2}tr^{\Re(\nu)-2}.
\]
We also note that the above function is supported in $r\leq2$. Hence,
we have 
\[
\|z(t,r)-qr^{\nu}\chi(r)\|_{H_{\frkm}^{1,1}}\aleq\||z(t,r)-qr^{\nu}\chi(r)|_{-1}\|_{L^{2}}\aleq t^{\frac{1}{2}\Re(\nu)}+t^{\frac{1}{2}}\to0
\]
as $t\to0$. When $0<\Re(\nu)<1$, we write 
\[
z(t,r)-qr^{\nu}\chi(r)=q\{t^{\frac{\nu}{2}}f_{1}(t^{-\frac{1}{2}}r)+\alp t^{\frac{\nu}{2}}f_{2}(t^{-\frac{1}{2}}r)-r^{\nu}\}\chi(r)
\]
and estimate this using \eqref{eq:f1-to-e1} (in the self-similar
zone $r\leq t^{\frac{1}{2}}$) and \eqref{eq:e_f-bound} (in the exterior
zone $r\geq t^{\frac{1}{2}}$) as 
\begin{align*}
 & |z(t,r)-qr^{\nu}\chi(r)|_{1}\\
 & \aleq\chf_{r\leq t^{\frac{1}{2}}}\{t^{\frac{1}{2}(\Re(\nu)-2)}r^{2}+r^{\Re(\nu)}\}+\chf_{t^{\frac{1}{2}}\leq r\leq2}\{tr^{\Re(\nu)-2}+t^{\Re(\nu)}r^{-\Re(\nu)}\}\\
 & \aleq\chf_{r\leq t^{\frac{1}{2}}}r^{\Re(\nu)}+\chf_{t^{\frac{1}{2}}\leq r\leq2}t^{\Re(\nu)}r^{-\Re(\nu)}.
\end{align*}
Hence the $H_{\frkm}^{1,1}$-norm is bounded by 
\[
\|z(t,r)-qr^{\nu}\chi(r)\|_{H_{\frkm}^{1,1}}\aleq t^{\frac{1}{2}\Re(\nu)}\to0
\]
as $t\to0$. This completes the proof of \eqref{eq:z-initial-justification}.

\smallskip
\textbf{Step 2.} Pointwise bounds for $z$.

These are already proved in Lemma~\ref{lem:ptwise-est-z}.

\smallskip
\textbf{Step 3.} Estimates for $\Psi_{z}$.

In this step, we show \eqref{eq:Psi_z-L2} and \eqref{eq:Psi_z-Hdot1}.
Recall that $\Psi_{z}=i\rd_{t}z+\Dlt^{(\frkm)}z-\rN(z)$. Using \eqref{eq:z-lin-evol-error},
\eqref{eq:rN-ext-cubic}, and \ref{eq:rN-ext-quintic}, we obtain
\begin{align*}
\Psi_{z} & =\chi_{r\ageq t^{\frac{1}{2}}}\Big\{\sum_{n=0}^{N-1}t^{n}\Big(i(n+1)\cdot g_{n+1}(r)+\Dlt^{(\frkm)}g_{n}(r)+P_{\lin,n}(r)-P_{\rN_{3},n}(r)-P_{\rN_{5},n}(r)\Big)\Big\}\\
 & \peq+\Psi_{\lin,N}-\chi_{r\aleq t^{\frac{1}{2}}}\rN(z)-\chi_{r\ageq t^{\frac{1}{2}}}\{E_{\rN_{3},N}+E_{\rN_{5},N}\}.
\end{align*}
By the choice \eqref{eq:def-gn} of $g_{n}$, the first line vanishes.
Therefore, 
\[
\Psi_{z}=\Psi_{\lin,N}-\chi_{r\aleq t^{\frac{1}{2}}}\rN(z)-\chi_{r\ageq t^{\frac{1}{2}}}\{E_{\rN_{3},N}+E_{\rN_{5},N}\}.
\]

It remains to show that each term in the previous display satisfies
the bounds \eqref{eq:Psi_z-L2} and \eqref{eq:Psi_z-Hdot1}. First,
\eqref{eq:Psi-lin-bound}, $\Re(\nu)>0$, and $\frac{1}{2}(\Re(\nu)-1)<N\leq\Re(\nu)$
imply the bounds 
\begin{align*}
\|\Psi_{\lin,N}\|_{L^{2}} & \aleq t^{N}+t^{\Re(\nu)}\aleq t^{\frac{1}{2}(\Re(\nu)-1)+\dlt_{z}},\\
\|r^{-\dlt_{z}}|\Psi_{\lin,N}|_{-1}\|_{L^{2}} & \aleq t^{N}+t^{\Re(\nu)-1}\aleq t^{\frac{1}{2}(\Re(\nu)-2)+\dlt_{z}},
\end{align*}
provided that $\dlt_{z}>0$ is sufficiently small. Next, \eqref{eq:z-ss-bound}
implies 
\begin{align*}
\chf_{r\leq2t^{\frac{1}{2}}}|\rN(z)|_{1} & \aleq\chf_{r\leq2t^{\frac{1}{2}}}\{|\rN_{3}(z)|_{1}+|\rN_{5}(z)|_{1}\}\\
 & \aleq\chf_{r\leq2t^{\frac{1}{2}}}\{t^{\frac{3}{2}(\Re(\nu)-2)}r^{6}+t^{\frac{5}{2}(\Re(\nu)-2)}r^{12}\},
\end{align*}
which in turn gives 
\begin{align*}
\|\chi_{r\aleq t^{\frac{1}{2}}}\rN(z)\|_{L^{2}} & \aleq t^{\frac{3}{2}\Re(\nu)+\frac{1}{2}}+t^{\frac{5}{2}\Re(\nu)+\frac{3}{2}}\aleq t^{\frac{1}{2}(\Re(\nu)-1)+\dlt_{z}},\\
\|r^{-\dlt_{z}}|\chi_{r\aleq t^{\frac{1}{2}}}\rN(z)|_{-1}\|_{L^{2}} & \aleq t^{\frac{3}{2}\Re(\nu)-\frac{1}{2}\dlt_{z}}+t^{\frac{5}{2}\Re(\nu)+1-\frac{1}{2}\dlt_{z}}\aleq t^{\frac{1}{2}(\Re(\nu)-2)+\dlt_{z}}.
\end{align*}
Next, \eqref{eq:EN3-bound}, \eqref{eq:EN5-bound}, $\Re(\nu)>0$,
and $\frac{1}{2}(\Re(\nu)-1)<N\leq\Re(\nu)$ imply 
\[
\|E_{\rN_{3},N}+E_{\rN_{5},N}\|_{L^{2}}+\|r^{-\dlt_{z}}|E_{\rN_{3},N}+E_{\rN_{5},N}|_{-1}\|_{L^{2}}\aleq t^{N}\aleq t^{\frac{1}{2}(\Re(\nu)-1)+\dlt_{z}},
\]
provided that $\dlt_{z}>0$ is sufficiently small. This completes
the proof.
\end{proof}

\section{\label{sec:blow-up-prelim}Preliminaries for the blow-up construction}

In this section, we collect some preliminary facts on linearization,
adapted function spaces, duality estimates for \eqref{eq:CSS-m-equiv},
and a decomposition lemma near modulated solitons, which we shall need
to carry out the blow-up construction.

\subsection{\label{subsec:Linearization-of-CSS}Linearization of \eqref{eq:CSS-m-equiv}}

\subsubsection*{Linearization of the Bogomol'nyi operator}

Consider the (radial Coulomb-gauge) Bogomol'nyi operator $w\mapsto\bfD_{w}w$.
We can write 
\begin{equation}
\bfD_{w+\eps}(w+\eps)=\bfD_{w}w+L_{w}\eps+N_{w}(\eps),\label{eq:LinearizationBogomolnyi}
\end{equation}
where 
\begin{align*}
L_{w}\eps & \coloneqq\bfD_{w}\eps+wB_{w}\eps,\\
N_{w}(\eps) & \coloneqq\eps B_{w}\eps+\tfrac{1}{2}wB_{\eps}\eps+\tfrac{1}{2}\eps B_{\eps}\eps,\\
B_{w}\eps & \coloneqq-\tfrac{2}{y}A_{\tht}[w,\eps]=\tfrac{1}{y}{\textstyle \int_{0}^{y}}\Re(\br w\eps)y'dy',
\end{align*}
and $A_{\tht}[\psi_{1},\psi_{2}]$ is defined through the polarization
\begin{equation}
A_{\tht}[\psi_{1},\psi_{2}]\coloneqq-\tfrac{1}{2}\tint 0r\Re(\br{\psi_{1}}\psi_{2})r'dr'.\label{eq:A-tht-def}
\end{equation}
The $L^{2}$-adjoint $L_{w}^{\ast}$ of $L_{w}$ takes the form 
\begin{align*}
L_{w}^{\ast}v & =\bfD_{w}^{\ast}v+B_{w}^{\ast}(\br wv),\\
B_{w}^{\ast}v & =w{\textstyle \int_{y}^{\infty}}\Re(v)\,dy'.
\end{align*}
We remark that the operators $L_{w}$, $B_{w}$, and their adjoints
are \emph{only} $\bbR$-linear.

In particular, when $m\geq0$ and $w=Q$, we use $\bfD_{Q}Q=0$ and
\eqref{eq:energy-self-dual-form} to have the following expansion
for the energy: 
\begin{equation}
E[Q+\eps]=\tfrac{1}{2}\|L_{Q}\eps\|_{L^{2}}^{2}+\text{(h.o.t.)}.\label{eq:linearized-energy-expn}
\end{equation}

\subsubsection*{Linearization of \eqref{eq:CSS-m-equiv}}

Next, we linearize \eqref{eq:CSS-m-equiv}, which we write in the
Hamiltonian form $\rd_{t}u+i\nabla E[u]=0$. We decompose 
\begin{equation}
\nabla E[w+\eps]=\nabla E[w]+\calL_{w}\eps+R_{w}(\eps),\label{eq:grad-energy-lin-nonlin}
\end{equation}
where $\calL_{w}\eps$ collects all the linear terms in $\eps$, and
$R_{w}(\eps)$ collects the remainders. Note that $\calL_{w}$ is
the Hessian of $E$, i.e., 
\[
\nabla^{2}E[w]=\calL_{w}.
\]
If one recalls the self-duality \eqref{eq:energy-self-dual-form},
i.e., $\nabla E[u]=L_{u}^{\ast}\bfD_{u}u$, then \eqref{eq:grad-energy-lin-nonlin}
and \eqref{eq:LinearizationBogomolnyi} yield 
\begin{align*}
\calL_{w}\eps & =L_{w}^{\ast}L_{w}\eps+\bfD_{w}w(B_{w}\eps)+B_{w}^{\ast}(\br{\eps}\bfD_{w}w)+B_{\eps}^{\ast}(\br w\bfD_{w}w).
\end{align*}
We remark again that the operator $\calL_{w}$ is only $\bbR$-linear.
Being the Hessian of the energy, $\calL_{w}$ is formally symmetric
with respect to the real inner products: 
\[
(\calL_{w}f,g)_{r}=(f,\calL_{w}g)_{r}.
\]

In particular, when $m\geq0$ and $w=Q$, from $\bfD_{Q}Q=0$, we
observe the self-dual factorization of $i\calL_{Q}$: 
\begin{equation}
i\calL_{Q}=iL_{Q}^{\ast}L_{Q}.\label{eq:calLQ}
\end{equation}
This identity was first observed in \cite{LawrieOhShashahani_unpub}.
Thus, the linearization of \eqref{eq:CSS-self-dual-form} at $Q$
is 
\begin{equation}
\rd_{t}\eps+i\calL_{Q}\eps=0,\quad\text{or}\quad\rd_{t}\eps+iL_{Q}^{\ast}L_{Q}\eps=0.\label{eq:lin-CSS-Q}
\end{equation}

Although it will not be used in this paper, let us note the linear
conjugation identity discovered in \cite{KimKwon2020arXiv}: 
\[
L_{Q}iL_{Q}^{\ast}=iA_{Q}^{\ast}A_{Q},
\]
where $A_{Q}\coloneqq\bfD_{Q}-\tfrac{1}{r}$ (which is $\bbC$-linear
and local). Remarkably enough, the above identity says that $L_{Q}iL_{Q}^{\ast}$
is now $\bbC$-linear and local. Moreover, $A_{Q}^{\ast}A_{Q}$ is
exactly the same as the linearized operator around the equivariant
harmonic maps into the two sphere. See \cite[Remark 2.3]{KimKwon2020arXiv}
for further discussions.

\subsection{\label{subsec:Invariant-subspace-decomposition}Formal invariant
subspace decomposition for $i\protect\calL_{Q}$}

Let $m\geq0$. In this subsection, we study the invariant subspace
decomposition for the linear operator $i\calL_{Q}$, which lies at
the heart of modulation analysis. Our discussion here is formal but
can be made rigorous when $m$ is large.

We first recall the formal generalized kernel of $i\calL_{Q}$: 
\[
N_{g}(i\calL_{Q})=\mathrm{span}_{\bbR}\{\Lmb Q,iQ,\tfrac{i}{4}r^{2}Q,\rho\}
\]
with the relations (see \cite[Proposition 3.4]{KimKwon2019arXiv})
\begin{equation}
\left\{ \begin{aligned}i\calL_{Q}(i\tfrac{r^{2}}{4}Q) & =\Lmb Q, & i\calL_{Q}\rho & =iQ,\\
i\calL_{Q}(\Lmb Q) & =0, & i\calL_{Q}(iQ) & =0,
\end{aligned}
\right.\label{eq:gen-kernel-rel}
\end{equation}
where the existence of $\rho$ is given in \cite[Lemma 3.6]{KimKwon2019arXiv}.
In fact, we have: 
\begin{equation}
\left\{ \begin{aligned}L_{Q}(i\tfrac{r^{2}}{4}Q) & =i\tfrac{r}{2}Q, & L_{Q}\rho & =\tfrac{1}{2(m+1)}rQ,\\
L_{Q}^{\ast}(i\tfrac{r}{2}Q) & =-i\Lmb Q, & L_{Q}^{\ast}(\tfrac{1}{2(m+1)}rQ) & =Q,\\
L_{Q}(\Lmb Q) & =0, & L_{Q}(iQ) & =0.
\end{aligned}
\right.\label{eq:gen-kernel-rel-LQ}
\end{equation}
We will need the following properties of $\rho$. 
\begin{lem}[{{The generalized null mode $\rho$; see \cite[Lemma 3.6]{KimKwon2019arXiv}
and \cite[Lemma 2.1]{KimKwon2020arXiv}}}]
\label{lem:rho}Let $m\geq0$. There exists a unique smooth function
$\rho:(0,\infty)\to\bbR$ satisfying the following properties: 
\begin{enumerate}
\item (Smoothness on the ambient space) The $m$-equivariant extension $\rho(x)\coloneqq\rho(r)e^{im\tht}$,
$x=re^{i\tht}$, is smooth on $\bbR^{2}$. 
\item (Relations) $\rho(r)$ satisfies 
\[
L_{Q}\rho=\tfrac{1}{2(m+1)}rQ\qquad\text{and}\qquad\calL_{Q}\rho=Q.
\]
\item (Pointwise bounds and matched spatial asymptotics) For any $k\in\bbN$,
we have 
\begin{align}
|\rho|_{k} & \aleq_{k}r^{2}Q,\label{eq:rho-ptwise-bdd}\\
\chf_{r\geq2}|\rho-\tfrac{1}{4(m+1)}r^{2}Q|_{k} & \aleq_{k}Q\cdot(\log r)^{2}.\label{eq:matched-spatial-asymp}
\end{align}
\end{enumerate}
\end{lem}

\begin{proof}
All the statements except \eqref{eq:matched-spatial-asymp} are proved
in \cite[Lemma 2.1]{KimKwon2020arXiv}. To prove \eqref{eq:matched-spatial-asymp},
by the proof of \cite[Lemma 3.6]{KimKwon2019arXiv}, we recall that
$\td{\rho}\coloneqq Q^{-1}\rho$ is a unique solution to the integral
equation 
\[
\rd_{r}\td{\rho}+\frac{1}{r}\int_{0}^{r}Q^{2}\td{\rho}r'dr'=\frac{r}{2(m+1)}
\]
in the class $|\td{\rho}(r)|\aleq r^{2}$. Next, we set $\wh{\rho}\coloneqq\td{\rho}-\frac{r^{2}}{4(m+1)}$
and observe that 
\begin{equation}
\rd_{r}\wh{\rho}=-\frac{1}{r}\int_{0}^{r}Q^{2}\td{\rho}r'dr'\aleq\begin{cases}
\chf_{r\leq2}r+\chf_{r\geq2}\frac{1}{r} & \text{if }m\geq1,\\
\chf_{r\leq2}r+\chf_{r\geq2}\frac{\log r}{r} & \text{if }m=0.
\end{cases}\label{eq:rho-property-tmp1}
\end{equation}
Integrating the above from the origin yields 
\[
|\wh{\rho}(r)|\aleq\begin{cases}
\chf_{r\leq2}r^{2}+\chf_{r\geq2}\log r & \text{if }m\geq1,\\
\chf_{r\leq2}r^{2}+\chf_{r\geq2}(\log r)^{2} & \text{if }m=0.
\end{cases}
\]
The bounds for $|\wh{\rho}(r)|_{k}$ follow from repeated differentiations
of \eqref{eq:rho-property-tmp1} and the bound \eqref{eq:rho-ptwise-bdd}.
This completes the proof. 
\end{proof}
Now, $L^{2}$ has a formal invariant subspace decomposition under
the linearized dynamics $(\rd_{t}+i\calL_{Q})w=0$: 
\begin{equation}
L^{2}=N_{g}(i\calL_{Q})\oplus N_{g}(\calL_{Q}i)^{\perp},\label{eq:invariant-subspace-decomp}
\end{equation}
where\footnote{Note that $y^{2}Q,\rho\notin L^{2}$ when $m\leq1$. This formal invariant
subspace decomposition makes sense when $m\geq2$.} 
\begin{align*}
N_{g}(i\calL_{Q}) & =\mathrm{span}_{\bbR}\{\Lmb Q,iQ,i\tfrac{r^{2}}{4}Q,\rho\},\\
N_{g}(\calL_{Q}i)^{\perp} & =\{i\Lmb Q,Q,\tfrac{1}{4}r^{2}Q,i\rho\}^{\perp}.
\end{align*}
One might understand that $\Lmb Q$- and $iQ$-coordinates of $w$
correspond to the scaling and phase rotation parameters for $u=Q+w$.
These coordinates can be detected by the inner products $(w,\tfrac{1}{4}r^{2}Q)_{r}$
and $(w,i\rho)_{r}$.\footnote{More precisely, one needs to consider $(w,\tfrac{1}{4}r^{2}Q-c_{1}Q)_{r}$
and $(w,i\rho-c_{2}\Lmb Q)_{r}$ such that $\tfrac{1}{4}r^{2}Q-c_{1}Q$
and $i\rho-c_{2}\Lmb Q$ are orthogonal to $\rho$ and $i\tfrac{r^{2}}{4}Q$,
respectively. Anyway, such inner products will not make sense when
$m\leq1$ and our discussion is at the formal level.} By \eqref{eq:gen-kernel-rel}, their time evolutions are detected
by 
\begin{align*}
\rd_{t}(w,\tfrac{1}{4}y^{2}Q)_{r} & +(w,i\Lmb Q)_{r}=0,\\
\rd_{t}(w,i\rho)_{r} & +(w,Q)_{r}=0.
\end{align*}
These relations will motivate our use of the inner products $(\eps,i\Lmb Q)_{r}$
and $(\eps,Q)_{r}$ in later sections.

When $m\in\{0,1\}$, due to the slow decay of $Q$ (and hence all
the generalized kernel elements), the inner products involved in the
invariant subspace decomposition \eqref{eq:invariant-subspace-decomp}
do not make sense. For our blow-up construction in the $m=0$ case,
we will need their truncated versions. 
\begin{lem}[Truncated generalized kernel relations]
Let $m=0$; let $R>1$ be a cutoff radius. 
\begin{enumerate}
\item (Logarithmic divergence) We have 
\begin{equation}
(\tfrac{1}{2}rQ,\chi_{R}\tfrac{1}{2}rQ)_{r}=4\pi\log R+O(1).\label{eq:yQ-not-in-L2}
\end{equation}
\item (Generalized kernel relations) We have 
\begin{align}
\|L_{Q}(\chi_{R}\tfrac{1}{4}ir^{2}Q)-\chi_{R}i\tfrac{r}{2}Q\|_{L^{2}}+\|L_{Q}(\chi_{R}\rho)-\chi_{R}\tfrac{r}{2}Q\|_{L^{2}} & \aleq1,\label{eq:truncated-est-1}\\
\|\langle r\rangle(L_{Q}^{\ast}(\chi_{R}\tfrac{1}{2}irQ)+i\Lmb_{R}Q)\|_{L^{2}}+\|\langle r\rangle(L_{Q}^{\ast}(\chi_{R}\tfrac{1}{2}rQ)-\chi_{R}Q)\|_{L^{2}} & \aleq1,\label{eq:truncated-est-2}
\end{align}
\item (Transversality) We have 
\begin{equation}
\begin{pmatrix}(\Lmb Q,-\tfrac{r^{2}}{4}Q\chi_{R})_{r} & (iQ,-\tfrac{r^{2}}{4}Q\chi_{R})_{r}\\
(\Lmb Q,-i\rho\chi_{R})_{r} & (iQ,-i\rho\chi_{R})_{r}
\end{pmatrix}=\begin{pmatrix}4\pi\log R+O(1) & 0\\
0 & -4\pi\log R+O(1)
\end{pmatrix}.\label{eq:transversality-radial-case}
\end{equation}
\end{enumerate}
\end{lem}

\begin{proof}
The proof of \eqref{eq:yQ-not-in-L2} is immediate from \eqref{eq:Q-formula}.
The proof of \eqref{eq:truncated-est-1} follows from the pointwise
bounds 
\begin{equation}
L_{Q}(\chi_{R}i\tfrac{r^{2}}{4}Q)-\chi_{R}i\tfrac{r}{2}Q=\bfD_{Q}(\chi_{R}i\tfrac{r^{2}}{4}Q)-\chi_{R}i\tfrac{r}{2}Q=\chi_{R}'\cdot(i\tfrac{r^{2}}{4}Q)\aleq\chf_{[R,2R]}rQ\label{eq:truncated-est-ptwise1}
\end{equation}
and 
\begin{align}
L_{Q}(\chi_{R}\rho) & -\chi_{R}\tfrac{r}{2}Q=[L_{Q},\chi_{R}]\rho\label{eq:truncated-est-ptwise2}\\
 & =\chi_{R}'\cdot\rho+(1-\chi_{R})\tfrac{Q}{r}\tint 0rQ\rho r'dr'-\tfrac{Q}{r}\tint 0r(1-\chi_{R})Q\rho r'dr'\nonumber \\
 & \aleq Q\cdot(\chf_{[R,2R]}r+\chf_{[R,\infty)}\tfrac{1}{r}\langle\log r\rangle).\nonumber 
\end{align}
The proof of \eqref{eq:truncated-est-2} follows from 
\[
L_{Q}^{\ast}(\chi_{R}\tfrac{1}{2}irQ)=i\bfD_{Q}^{\ast}(\chi_{R}\tfrac{1}{2}rQ)=-\chi_{R}i\Lmb Q-\chi_{R}'i\tfrac{1}{2}rQ=-i\Lmb_{R}Q
\]
and 
\begin{align*}
 & L_{Q}^{\ast}(\tfrac{1}{2}r\chi_{R}Q)-\chi_{R}Q=[L_{Q}^{\ast},\chi_{R}](\tfrac{1}{2}rQ)\\
 & \qquad=-\chi_{R}'(\tfrac{1}{2}rQ)+(1-\chi_{R})Q\tint r{\infty}Q^{2}r'dr'-Q\tint r{\infty}(1-\chi_{R})Q^{2}r'dr'\\
 & \qquad\aleq Q\cdot(\chf_{[R,2R]}+\chf_{[R,\infty)}\langle y\rangle^{-2}+\chf_{(0,2R]}R^{-2}).
\end{align*}
This completes the proof.

Finally, the proof of \eqref{eq:transversality-radial-case} for the
diagonal entries follows from 
\begin{align*}
(\Lmb Q,-\tfrac{r^{2}}{4}Q\chi_{R})_{r} & =(L_{Q}^{\ast}(\tfrac{1}{2}irQ),\chi_{R}\tfrac{r^{2}}{4}Q)_{r}=(\tfrac{1}{2}irQ,L_{Q}\chi_{R}\tfrac{r^{2}}{4}Q)_{r},\\
(iQ,-i\rho\chi_{R})_{r} & =(L_{Q}^{\ast}(\tfrac{1}{2}irQ),\chi_{R}\tfrac{r^{2}}{4}Q)_{r}=(\tfrac{1}{2}rQ,L_{Q}\chi_{R}\rho)_{r},
\end{align*}
and applying \eqref{eq:truncated-est-ptwise1}, \eqref{eq:truncated-est-ptwise2},
and \eqref{eq:yQ-not-in-L2}. Note that the off-diagonal entries of
\eqref{eq:transversality-radial-case} are zero by considering the
real and imaginary parts. This completes the proof. 
\end{proof}

\subsection{\label{subsec:Adapted-function-spaces}Adapted function spaces}

In this subsection, we introduce the adapted function spaces $\dot{\calH}_{m}^{1}$
and a linear coercivity estimate associated to energy. All the following
facts have already appeared in \cite{KimKwon2019arXiv,KimKwonOh2020arXiv}.

Let $m\geq0$. We recall the linear coercivity of energy, namely,
the coercivity estimates for the linearized Bogomol'nyi operator $L_{Q}$
at the $\dot{H}^{1}$-level. The available Hardy-type controls on
$f$ from $\|L_{Q}f\|_{L^{2}}$ are different for the cases $m=0$
and $m\geq1$. When $m\geq1$, we have a coercivity of $L_{Q}$ in
terms of the usual $\dot{H}_{m}^{1}$-norm. Thus we let\footnote{We note that $\dot{\calH}_{m}^{1}$ in this paper does not coincide
with $\dot{\calH}_{2}^{1}$ in \cite[Section 3.3]{KimKwonOh2020arXiv}
when $m=2$. This is because the former is defined through the coercivity
estimates of $L_{Q}$, but the latter was defined through that of
the operator $A_{Q}^{\ast}=(\bfD_{Q}-\tfrac{1}{r})^{\ast}$.

For the coercivity estimates of higher order (adapted) derivatives,
one needs to introduce spaces such as $\dot{\calH}_{m}^{k}$ which
differ from the usual Sobolev space $\dot{H}_{m}^{k}$; see \cite[Section 2.3]{KimKwon2020arXiv}.} 
\[
\dot{\calH}_{m}^{1}\coloneqq\dot{H}_{m}^{1}\quad\text{when }m\geq1.
\]
When $m=0$, the adapted function space $\dot{\calH}_{0}^{1}$ is
defined by the completion of the space $\calS_{0}$ under the $\dot{\calH}_{0}^{1}$-norm
\[
\|f\|_{\dot{\calH}_{0}^{1}}\coloneqq\|\rd_{r}f\|_{L^{2}}+\|\langle\log_{-}r\rangle^{-1}r^{-1}f\|_{L^{2}}.
\]
We remark that the logarithmic loss near the origin $r=0$ is introduced
due to the failure of Hardy's inequality when $m=0$. Let us note
$\dot{\calH}_{0}^{1}\embed\dot{H}_{0}^{1}$ and $\dot{\calH}_{0}^{1}\cap L^{2}=H_{0}^{1}$.
Moreover, we have a logarithmically weakened version of \eqref{eq:HardySobolevSection2}:
\begin{equation}
\|\langle\log_{-}r\rangle^{-1}f\|_{L^{\infty}\cap rL^{2}}\aleq\|f\|_{\dot{\calH}_{0}^{1}}.\label{eq:m=00003D0-HardySobolev}
\end{equation}

We now state the coercivity estimates of $L_{Q}$ at the $\dot{H}^{1}$-level.
To obtain the coercivity of $L_{Q}$, it is necessary to preclude
the kernel elements $\Lmb Q$ and $iQ$ of $L_{Q}$. We do this by
imposing suitable otrhogonality conditions. First, we require the
profiles $\calZ_{1},\calZ_{2}\in C_{c,m}^{\infty}$ for orthogonality
conditions to satisfy \emph{the transversality condition} 
\begin{equation}
\det\begin{pmatrix}(\Lmb Q,\calZ_{1})_{r} & (iQ,\calZ_{1})_{r}\\
(\Lmb Q,\calZ_{2})_{r} & (iQ,\calZ_{2})_{r}
\end{pmatrix}\neq0.\label{eq:Z1Z2-transversality}
\end{equation}

\begin{lem}[Coercivity of $L_{Q}$; \cite{KimKwon2019arXiv,KimKwonOh2020arXiv}]
\label{lem:linear-coercivity}Let $m\geq0$. Let $\calZ_{1},\calZ_{2}\in C_{c,m}^{\infty}$
satisfy \eqref{eq:Z1Z2-transversality}. Then, 
\begin{equation}
\|L_{Q}f\|_{L^{2}}\sim\|f\|_{\dot{\calH}_{m}^{1}},\qquad\forall f\in\dot{\calH}_{m}^{1}\text{ with }(f,\calZ_{1})_{r}=(f,\calZ_{2})_{r}=0.\label{eq:coercivity}
\end{equation}
\end{lem}

\begin{rem}[Choice of $\calZ_{1}$ and $\calZ_{2}$ in this paper]
In this paper, we not only require $\calZ_{1}$ and $\calZ_{2}$
to satisfy the transversality condition \eqref{eq:Z1Z2-transversality}
but also require $\calZ_{1}$ and $\calZ_{2}$ to satisfy the gauge
condition 
\begin{equation}
(-i\tfrac{y^{2}}{4}Q,\calZ_{k})_{r}=(\rho,\calZ_{k})_{r}=0,\qquad\forall k\in\{1,2\}.\label{eq:Z1Z2-gauge-condition}
\end{equation}
\end{rem}

\subsection{Weighted $L^{1}$-estimate}

We record the following weighted $L^{1}$-estimate. This will mostly
be used to estimate $\|Qz^{\flat}\|_{L^{1}}$ and $\|Q\eps\|_{L^{1}}$
in our blow-up analysis. 
\begin{lem}[Weighted $L^{1}$-estimate]
\label{lem:L1-log-bound}For any $f\in L^{2}$ with $\|f\|_{L^{2}}\aleq1$,
we have 
\begin{equation}
\|\langle r\rangle^{-2}f\|_{L^{1}}\aleq\langle\log_{-}\|\langle r\rangle^{-1}f\|_{L^{2}}\rangle^{\frac{1}{2}}\|\langle r\rangle^{-1}f\|_{L^{2}}.\label{eq:weighted-L1}
\end{equation}
We also note that the function $x\mapsto\langle\log_{-}x\rangle^{\frac{1}{2}}x$
from $[0,\infty)\to[0,\infty)$ is increasing. 
\end{lem}

\begin{proof}
If $\|\langle r\rangle^{-1}f\|_{L^{2}}\geq1$, then the inequality
holds since $\|\langle y\rangle^{-2}f\|_{L^{1}}\aleq\|f\|_{L^{2}}\aleq1$.
If $f=0$, the inequality is obvious. Henceforth, we assume $0<\|\langle r\rangle^{-1}f\|_{L^{2}}\leq1$.
For some $R_{0}\geq1$ to be chosen, we estimate 
\begin{align*}
\|\langle y\rangle^{-2}f\|_{L^{1}} & \aleq\|\chf_{y\leq R_{0}}\langle y\rangle^{-2}f\|_{L^{1}}+\|\chf_{y\geq R_{0}}\langle y\rangle^{-2}f\|_{L^{1}}\\
 & \aleq\langle\log_{+}R_{0}\rangle^{\frac{1}{2}}\|\langle y\rangle^{-1}f\|_{L^{2}}+R_{0}^{-1}.
\end{align*}
Choosing $R_{0}=1/\|\langle r\rangle^{-1}f\|_{L^{2}}$ completes the
proof. 
\end{proof}

\subsection{Integral estimates}

The following is a key lemma to estimate various integral operators
appearing in $\calN_{\ast}$ and $\calM_{\ast}$. 
\begin{lem}[{{Mapping properties for integral operators \cite[Lemma 2.2]{KimKwonOh2022arXiv1}}}]
\label{lem:mapping-integral-op}Let $1\leq p,q\leq\infty$ and $s\in[0,2]$
be such that $(p,q,s)=(1,\infty,0)$ or $\frac{1}{q}+1=\frac{1}{p}+\frac{s}{2}$
with $p>1$. Then, we have 
\begin{equation}
\Big\|\frac{1}{r^{s}}\int_{0}^{r}f(r')r'dr'\Big\|_{L^{q}}\lesssim_{p}\|f\|_{L^{p}}.\label{eq:integral-operator-bdd1}
\end{equation}
Denoting by $p'$ and $q'$ the Hölder conjugates of $p$ and $q$,
respectively, we have the dual estimate 
\begin{equation}
\Big\|\int_{r}^{\infty}g(r')(r')^{1-s}dr'\Big\|_{L^{p'}}\aleq\|g\|_{L^{q'}}.\label{eq:integral-operator-bdd2}
\end{equation}
\end{lem}

\begin{rem}
\label{rem:decreasing-weight}If $w:(0,\infty)\to\bbR_{+}$ is a decreasing
function, then we have pointwise bounds 
\begin{align*}
|w(r)\tint 0rf(r')dr'| & \leq\tint 0rw(r')|f(r')|dr',\\
|\tint r{\infty}w(r')f(r')dr'| & \leq w(r)\tint r{\infty}|f(r')|dr'.
\end{align*}
These bounds, combined with Lemma~\ref{lem:mapping-integral-op},
will be used frequently in later sections without referring to this
remark. 
\end{rem}

\subsection{Duality estimates}

\label{subsec:duality-ests}

In this subsection, we recall the estimates for the nonlinearity of
\eqref{eq:CSS-m-equiv} from \cite{KimKwon2019arXiv,KimKwonOh2022arXiv1},
which we shall use in the blow-up construction. 

Recall the definition of $\calN_{\ast}$ and $\calM_{\ast}$ from
Section~\ref{subsec:nonlinearity}. The following Hölder- and weighted
$L^{1}$-type estimates were proved in \cite{KimKwonOh2022arXiv1}. 
\begin{lem}[{Duality estimates (Hölder-type) \cite[Lemma 2.3]{KimKwonOh2022arXiv1}}]
\label{lem:duality-estimates-Holder}The following estimates hold. 
\begin{itemize}
\item (For $\calM_{4,\ast}$) Let $1\leq p,q\leq\infty$ be such that $\frac{1}{p}+\frac{1}{q}=1$.
Then, we have 
\begin{align*}
|\calM_{4,0}(\psi_{1},\psi_{2},\psi_{3},\psi_{4})| & \aleq\|\psi_{1}\psi_{2}\|_{L^{p}}\|\psi_{3}\psi_{4}\|_{L^{q}},\\
|\calM_{4,1}(\psi_{1},\psi_{2},\psi_{3},\psi_{4})| & \aleq\|\psi_{1}\psi_{2}\|_{L^{p}}\|\psi_{3}\psi_{4}\|_{L^{q}},\quad\text{if }(p,q)\neq(1,\infty).
\end{align*}
\item (For $\calM_{6}$) Let $1\leq p,q,r\leq\infty$ be such that $\frac{1}{p}+\frac{1}{q}+\frac{1}{r}=2$
and $(p,q,r)\neq(1,1,\infty)$. Then, we have 
\[
|\calM_{6}(\psi_{1},\dots,\psi_{6})|\aleq\|\psi_{1}\psi_{2}\|_{L^{p}}\|\psi_{3}\psi_{4}\|_{L^{q}}\|\psi_{5}\psi_{6}\|_{L^{r}}.
\]
\end{itemize}
\end{lem}

\begin{lem}[{Duality estimates (weighted $L^{1}$-type) \cite[Lemma 2.4]{KimKwonOh2022arXiv1}}]
\label{lem:duality-estimates-weighted-L1}The following estimates
hold. 
\begin{itemize}
\item (For $\calM_{4,1}$) Let $w_{12},w_{34}:(0,\infty)\to\bbR_{+}$ be
decreasing functions such that $w_{12}(r)w_{34}(r)=\frac{1}{r^{2}}$.
Then, we have 
\begin{align*}
|\calM_{4,1}(\psi_{1},\psi_{2},\psi_{3},\psi_{4})| & \aleq\|w_{12}\psi_{1}\psi_{2}\|_{L^{1}}\|w_{34}\psi_{3}\psi_{4}\|_{L^{1}}.
\end{align*}
\item (For $\calM_{6}$) Let $w_{12},w_{34},w_{56}:(0,\infty)\to\bbR_{+}$
be decreasing functions such that $w_{12}(r)w_{34}(r)w_{56}(r)=\frac{1}{r^{2}}$.
Then, we have 
\[
|\calM_{6}(\psi_{1},\dots,\psi_{6})|\aleq\|w_{12}\psi_{1}\psi_{2}\|_{L^{1}}\|w_{34}\psi_{3}\psi_{4}\|_{L^{1}}\|w_{56}\psi_{5}\psi_{6}\|_{L^{1}}.
\]
\end{itemize}
\end{lem}

The following nonlinear estimates follow from the duality relations
\eqref{eq:duality-relations} and the two lemmas mentioned above. 
\begin{cor}[{Nonlinear estimates (Hölder-type) \cite[Corollary 2.5]{KimKwonOh2022arXiv1}}]
\label{cor:nonlinear-estimates-holder}The following estimates hold. 
\begin{itemize}
\item (For $\calN_{3,k}$) For any $1\leq p_{1},\dots,p_{4}\leq\infty$
with $\sum_{j=1}^{4}\frac{1}{p_{j}}=1$ and $\#\{j:p_{j}=\infty\}\leq1$,
we have 
\[
\|\calN_{3,k}(\psi_{1},\psi_{2},\psi_{3})\|_{L^{p_{4}'}}\aleq\|\psi_{1}\|_{L^{p_{1}}}\|\psi_{2}\|_{L^{p_{2}}}\|\psi_{3}\|_{L^{p_{3}}}.
\]
\item (For $\calN_{5,k}$) For any $1\leq p_{1},\dots,p_{6}\leq\infty$
with $\sum_{j=1}^{6}\frac{1}{p_{j}}=2$ and $\#\{j:p_{j}=\infty\}\leq1$,
we have 
\[
\|\calN_{5,k}(\psi_{1},\dots,\psi_{5})\|_{L^{p_{6}'}}\aleq\prod_{j=1}^{5}\|\psi_{j}\|_{L^{p_{j}}}.
\]
\end{itemize}
\end{cor}

\begin{cor}[{Nonlinear estimates (weighted $L^{2}$-type) \cite[Corollary 2.6]{KimKwonOh2022arXiv1}}]
\label{cor:nonlinear-estimates-weightedL2}The following estimates
hold. 
\begin{itemize}
\item (For $\calN_{3,1}$ and $\calN_{3,2}$) Let $w_{1},\dots,w_{3}:(0,\infty)\to\bbR_{+}$
be decreasing functions such that $\prod_{j=1}^{3}w_{3}(r)=\frac{1}{r^{2}}$.
Then, for any $k\in\{1,2\}$, we have 
\[
\|\calN_{3,k}(\psi_{1},\psi_{2},\psi_{3})\|_{L^{2}}\aleq\prod_{j=1}^{3}\|w_{j}\psi_{j}\|_{L^{2}}.
\]
\item (For $\calN_{5,1}$ and $\calN_{5,2}$) Let $w_{1},\dots,w_{5}:(0,\infty)\to\bbR_{+}$
be decreasing functions such that $\prod_{j=1}^{5}w_{j}(r)=\frac{1}{r^{2}}$.
Then, for any $k\in\{1,2\}$, we have 
\[
\|\calN_{5,k}(\psi_{1},\psi_{2},\psi_{3})\|_{L^{2}}\aleq\prod_{j=1}^{5}\|w_{j}\psi_{j}\|_{L^{2}}.
\]
\end{itemize}
\end{cor}

\begin{rem}
$\calM_{4,1}$, $\calN_{3,1}$, $\calN_{3,2}$ do not appear in the
$m=0$ case. 
\end{rem}

\begin{rem}
The multilinear form $\calM_{4,0}$ and the nonlinearity $\calN_{3,0}$
are always estimated by Hölder's inequality. Note that the restriction
of the form $\#\{j:p_{j}=\infty\}\leq1$ (cf. Corollary~\ref{cor:nonlinear-estimates-holder})
do not apply to $\calM_{4,0}$ and $\calN_{3,0}$. 
\end{rem}

\subsection{Decomposition near modulated soliton}

We will need a standard decomposition lemma for a function $u$ near
a modulated soliton $Q_{\lmb,\gmm}$. 
\begin{lem}[Decomposition near $Q_{\lmb,\gmm}$]
\label{lem:dec-near-Q}For $\delta>0$, let us denote by $\calT_{\delta}$
the set of $w\in L^{2}$ satisfying 
\begin{equation}
\inf_{\lmb'\in\bbR_{+},\,\gmm'\in\bbR/2\pi\bbZ}\|w-Q_{\lmb',\gmm'}\|_{L^{2}}<\delta.\label{eq:dec-tmp}
\end{equation}
Then, there exists\footnote{Indeed, one can prove that, for any sufficiently small $\eta_{\mathrm{dec}}>0$,
there exists $0<\delta_{\mathrm{dec}}<1$ satisfying the properties
in this lemma.} $0<\delta_{\mathrm{dec}},\eta_{\mathrm{dec}}\ll1$ such that any
$w\in\calT_{\delta_{\mathrm{dec}}}$ admits the unique decomposition
\[
w=[Q+\eps]_{\lmb,\gmm}
\]
satisfying the orthogonality conditions 
\[
(\eps,\calZ_{1})_{r}=(\eps,\calZ_{2})_{r}=0
\]
and smallness 
\[
\|\eps\|_{L^{2}}<\eta_{\mathrm{dec}}.
\]
\end{lem}

This is a standard consequence of the implicit function theorem, and
we omit the proof. See Lemma~4.3 of \cite{KimKwonOh2022arXiv1} for
a related statement. We note that Lemma~4.3 of \cite{KimKwonOh2022arXiv1}
is, in some sense, stronger than Lemma~\ref{lem:dec-near-Q} above,
because it can deal with solutions $u$ which are far from $Q$ in
the $L^{2}$-topology.

\section{\label{sec:Blow-up-Const}Blow-up construction of strongly interacting
regime}

In this section, we prove Theorem~\ref{thm:MainThm}. \textbf{We
always assume $m=0$} (and hence $\frkm=-2$). Given the asymptotic
profile $z^{\ast}(r)=qr^{\nu}\chi(r)$ with $q\in\bbC\setminus\{0\}$
and $\Re(\nu)>0$, we have constructed in Section~\ref{sec:ApprxRadiation}
the approximate radiation $z(t,r)$ to the $\frkm$-equivariant phase-rotated
CSS; see \eqref{eq:eqn-aprx-radiation}. Here, we construct a finite-time
blow-up solution $u(t)$, which blows up due to the strong interaction
between the modulated soliton $Q_{\lmb(t),\gmm(t)}$ (the scaling
parameter $\lmb(t)$ and the phase rotation parameter $\gmm(t)$) and
the radiation $z(t)$. We perform modulation analysis with backward
construction.

We begin with a formal derivation of the blow-up regime given in Theorem~\ref{thm:MainThm}.
Next, in view of the method of backward construction, we reduce the
proof of Theorem~\ref{thm:MainThm} to the proof of the main bootstrap
Proposition~\ref{prop:bootstrap} described in terms of modulation
parameters and the remainder. Then, the rest of this section is devoted
to the proof of Proposition~\ref{prop:bootstrap} using modulation
analysis.

In our analysis, we do not introduce additional (blow-up) profile
modifications. We only use $Q_{\lmb(t),\gmm(t)}$ as the blow-up part.
Instead, we will keep track of two leading directions of the remainder
part $\eps$ of the solution $u$, which provide refined controls on
the modulation parameters $\lmb$ and $\gmm$ and justify the blow-up
dynamics. Our approach is inspired by the work \cite{JendrejLawrieRodriguez2019arXiv}
of Jendrej, Lawrie, and Rodriguez in the context of corotational wave
maps.

\subsection{\label{subsec:Formal-derivation-of-blow-up}Formal derivation of
the blow-up regime}

In this subsection, we formally derive the blow-up regime described
in Theorem~\ref{thm:MainThm}. The asymptotic profile $z^{\ast}(r)=qr^{\nu}\chi(r)$
is prescribed and the approximate radiation $z(t,r)$ is constructed
in Section~\ref{sec:ApprxRadiation}. We are interested in the \emph{strongly
interacting regime}, where the interaction between the radiation $z(t)$
and the modulated soliton $Q_{\lmb(t),\gmm(t)}$ plays a nontrivial
role in the blow-up dynamics.

Consider a finite-time blow-up solution $u(t)$ that admits the decomposition
\[
u(t)=e^{i\gmm_{z}(t)}[(Q+\eps)^{\sharp}+z](t),
\]
where $\gmm_{z}(t)$ is the extra phase rotation defined by \eqref{eq:Def-tht_z}
and the $\sharp$-operation is a notation for dynamically rescaled functions, defined in Section~\ref{subsec:Notations}, through some scaling and phase
rotation parameters $\lmb(t)$ and $\gmm(t)$.

We derive evolution equations for $\eps^{\sharp}$ and $\eps$. Rewrite
\eqref{eq:CSS-m-equiv-ham} and \eqref{eq:eqn-aprx-radiation} as
\begin{align}
\rd_{t}(e^{-i\gmm_{z}}u) & =-i\nabla E[e^{-i\gmm_{z}}u]+i\tht_{z}e^{-i\gmm_{z}}u,\label{eq:u-phase-rotated-eqn}\\
\rd_{t}z & =-i\nabla\td E[z]+i\tht_{z}z-i\Psi_{z},\label{eq:z-eqn}
\end{align}
where $\td E[z]$ denotes the energy functional of the $\frkm$-equivariant
CSS (recall $\frkm=-2$) and $\tht_{z}=-\rd_{t}\gmm_{z}$ from \eqref{eq:Def-tht_z}.
Subtracting \eqref{eq:z-eqn} from \eqref{eq:u-phase-rotated-eqn},
we obtain 
\[
\rd_{t}(Q+\eps)^{\sharp}=-i(\nabla E[(Q+\eps)^{\sharp}+z]-\nabla\td E[z]-\tht_{z}(Q+\eps)^{\sharp})+i\Psi_{z}.
\]
Using the decomposition \eqref{eq:grad-energy-lin-nonlin}, the above
display yields the equation for $\eps^{\sharp}$: 
\begin{align}
\rd_{t}\eps^{\sharp} & =-\rd_{t}Q^{\sharp}-i(\calL_{Q^{\sharp}+z}\eps^{\sharp}+R_{Q^{\sharp}+z}(\eps^{\sharp})-\tht_{z}\epsilon^{\sharp})-iR_{Q^{\sharp},z}+i\Psi_{z},\label{eq:eps-sharp-eq}\\
R_{Q^{\sharp},z} & \coloneqq\nabla E[Q^{\sharp}+z]-\nabla\td E[z]-\tht_{z}Q^{\sharp}.\label{eq:def-RQsharp,z}
\end{align}
Here, $R_{Q^{\sharp},z}$ is the interaction term between the modulated
soliton $Q^{\sharp}$ and the radiation $z$.\footnote{Compared to that of \cite{KimKwon2019arXiv}, our $R_{Q^{\sharp},z}$
already subtracted out the extra phase rotation on $Q^{\sharp}$ coming
from $z$; it corresponds to $\td R_{Q^{\sharp},z}$ in \cite{KimKwon2019arXiv}.} Taking the $\flat$-operation, we also obtain the equation for $\epsilon$:
\begin{equation}
(\rd_{s}-\tfrac{\lmb_{s}}{\lmb}\Lmb+\gmm_{s}i)\epsilon=(\tfrac{\lmb_{s}}{\lmb}\Lmb Q-\gmm_{s}iQ)-i(\calL_{Q+z^{\flat}}\epsilon+R_{Q+z^{\flat}}(\epsilon)-\tht_{z^{\flat}}\epsilon)-iR_{Q,z^{\flat}}+i\Psi_{z^{\flat}},\label{eq:eps-eq-s,y}
\end{equation}
where 
\begin{align}
R_{Q,z^{\flat}} & \coloneqq\nabla E[Q+z^{\flat}]-\nabla\td E[z^{\flat}]-\tht_{z^{\flat}}Q=\lmb^{2}[R_{Q^{\sharp},z}]^{\flat},\label{eq:def-RQ,zflat}\\
\tht_{z^{\flat}} & \coloneqq\lmb^{2}\tht_{z}\aleq\lmb^{2}\|\tfrac{1}{r}z\|_{L^{2}}\|z_{1}\|_{L^{2}}\aleq\lmb^{2},\label{eq:def-tht-z-flat}\\
\Psi_{z^{\flat}} & \coloneqq\lmb^{2}[\Psi_{z}]^{\flat}.\label{eq:def-Psi-zflat}
\end{align}
Note that 
\begin{equation}
\tht_{z^{\flat}}=\lmb^{2}\tht_{z}\aleq\lmb^{2}(1+\|z\|_{L^{2}}^{2})\|\tfrac{1}{r}z\|_{L^{2}}^{2}\aleq\lmb^{2}.\label{eq:tht-z-flat-bound}
\end{equation}

In order to formally derive the blow-up, we first assume the \emph{adiabatic
ansatz}\footnote{We flipped the sign of $\eta$ compared to the previous works \cite{KimKwon2019arXiv,KimKwon2020arXiv,KimKwonOh2020arXiv}
of the authors. Correspondingly, the sign of $\eta(s)\rho$ in \eqref{eq:eps-ansatz}
is also flipped.} 
\begin{equation}
\frac{\lmb_{s}}{\lmb}+b=0,\qquad\gmm_{s}+\eta=0.\label{eq:adiabatic-ansatz}
\end{equation}
Motivated by the generalized kernel relations \eqref{eq:gen-kernel-rel},
we assume that $\eps$ takes the form 
\begin{equation}
\eps(s,y)=b(s)(-i\tfrac{y^{2}}{4}Q)+\eta(s)\rho+\text{(higher order)}.\label{eq:eps-ansatz}
\end{equation}
The formal blow-up dynamics is now encoded in the evolution laws for
$b$ and $\eta$. We substitute the adiabatic ansatz \eqref{eq:adiabatic-ansatz}
into the $\eps$-equation \eqref{eq:eps-eq-s,y}, ignore $\tht_{z^{\flat}}$
and $\Psi_{z^{\flat}}$, and make an approximation $\nabla E[(Q+z^{\flat})+\eps]\approx\nabla E[Q+\eps]$
to simplify the $\eps$-equation as 
\begin{equation}
\rd_{s}\eps=-b\Lmb(Q+\eps)-\eta i(Q+\eps)-i\nabla E[Q+\eps]-iR_{Q,z^{\flat}}.\label{eq:formal-eps-eq}
\end{equation}

We derive formal evolution laws for $b$ and $\eta$. Motivated by
the invariant subspace decomposition \eqref{eq:invariant-subspace-decomp},
we can detect the $b_{s}$-equation and $\eta_{s}$-equation by testing
the formal $\eps$-equation \eqref{eq:formal-eps-eq} against $i\Lmb Q$
and $Q$, respectively. However, a cleverer way is to take an inner
product with $i\Lmb(Q+\eps)$ and $Q+\eps$, in the spirit of Pohozaev-type
computations (see, for example, \cite[Section 3]{RaphaelRodnianski2012Publ.Math.}
and also \cite[Section 4]{KimKwon2019arXiv} for a related computation
for \eqref{eq:CSS-m-equiv}). Using the formal ansatz \eqref{eq:eps-ansatz},
we formally have\footnote{Strictly speaking, the inner product $(-i\tfrac{y^{2}}{4}Q,i\Lmb Q)_{r}$
is not well-defined because $yQ\notin L^{2}$. We are proceeding formally
here, and any rigorous derivation should include truncations (see the
later sections).} 
\begin{align*}
(\rd_{s}\eps,i\Lmb(Q+\eps))_{r} & \approx b_{s}(-i\tfrac{y^{2}}{4}Q,i\Lmb Q)_{r}.
\end{align*}
On the other hand, we formally have 
\begin{align*}
(-b\Lmb(Q+\eps)-\eta i(Q+\eps),i\Lmb(Q+\eps))_{r} & =0,\\
(-i\nabla E[Q+\eps],i\Lmb(Q+\eps))_{r} & =-2E[Q+\eps]\approx-\|L_{Q}\eps\|_{L^{2}}^{2},\\
(-iR_{Q,z^{\flat}},i\Lmb(Q+\eps))_{r} & \approx-(R_{Q,z^{\flat}},\Lmb Q)_{r},
\end{align*}
where in the second row, we used the scaling property of the energy,
which is equivalent to the virial identity \eqref{eq:virial-2}. Formally
applying the relations \eqref{eq:gen-kernel-rel-LQ}, we have 
\begin{align*}
(-i\tfrac{y^{2}}{4}Q,i\Lmb Q)_{r} & =(i\tfrac{y^{2}}{4}Q,L_{Q}^{\ast}(i\tfrac{1}{2}yQ))_{r}=(L_{Q}(i\tfrac{y^{2}}{4}Q),i\tfrac{1}{2}yQ)_{r}=\|\tfrac{1}{2}yQ\|_{L^{2}}^{2},\\
\|L_{Q}\eps\|_{L^{2}}^{2} & \approx\|(\eta-ib)(\tfrac{1}{2}yQ)\|_{L^{2}}^{2}=(b^{2}+\eta^{2})\|\tfrac{1}{2}yQ\|_{L^{2}}^{2}.
\end{align*}
Therefore, we have arrived at 
\begin{equation}
b_{s}+b^{2}+\eta^{2}+\frac{(R_{Q,z^{\flat}},\Lmb Q)_{r}}{\|\tfrac{1}{2}yQ\|_{L^{2}}^{2}}=0.\label{eq:formal-b-without-trunc}
\end{equation}
A similar procedure using (the first row being equivalent to mass
conservation)
\begin{align*}
(-i\nabla E[Q+\eps],Q+\eps)_{r} & =0,\\
(\rho,Q)_{r}=(\rho,L_{Q}^{\ast}(\tfrac{1}{2}yQ))_{r} & =\|\tfrac{1}{2}yQ\|_{L^{2}}^{2}
\end{align*}
yields 
\begin{equation}
\eta_{s}-\frac{(R_{Q,z^{\flat}},iQ)_{r}}{\|\tfrac{1}{2}yQ\|_{L^{2}}^{2}}=0.\label{eq:formal-eta-without-trunc}
\end{equation}

Notice that the denominators of \eqref{eq:formal-b-without-trunc}
and \eqref{eq:formal-eta-without-trunc}, i.e., $\|\tfrac{1}{2}yQ\|_{L^{2}}^{2}$,
are not well-defined due to $yQ\notin L^{2}$. A careful and rigorous
analysis in later sections requires a truncation near the self-similar
scale $y=B_{0}$, where 
\begin{equation}
B_{0}(t)\coloneqq\frac{|t|^{1/2}}{\lmb(t)}.\label{eq:def-B0}
\end{equation}
Note by \eqref{eq:yQ-not-in-L2} that 
\begin{equation}
\|\tfrac{1}{2}yQ\chi_{B_{0}}\|_{L^{2}}^{2}\approx4\pi\log B_{0}.\label{eq:B0-1}
\end{equation}

As we will consider the strongly interacting regime, the inner products
$(R_{Q,z^{\flat}},\Lmb Q)_{r}$ and $(R_{Q,z^{\flat}},iQ)_{r}$ should
be among the dominant terms in \eqref{eq:formal-b-without-trunc} and \eqref{eq:formal-eta-without-trunc}. The two inner products determine the modulation laws for $b$ and $\eta$. This is in contrast to the pseudoconformal blow-up
studied in \cite{KimKwon2019arXiv}, where the soliton-radiation interaction
is weak and does not affect the blow-up speed. In the next lemma,
we extract the leading order contributions of these inner products
to the modulation equations \eqref{eq:formal-b-without-trunc} and
\eqref{eq:formal-eta-without-trunc}. These are derived from the asymptotics
of $z$ in the self-similar region \eqref{eq:z-ss-leading}-\eqref{eq:z1-ss-leading}.
\begin{lem}[Interaction term $R_{Q,z^{\flat}}$]
\label{lem:inn-prod-interaction}Assume $\lmb(t)\aleq|t|$. Then,
there exists $\delta>0$ such that for all small negative $t$ 
\begin{align}
(R_{Q,z^{\flat}},\Lmb Q)_{r} & =8\sqrt{8}\pi\lmb^{3}\Re(e^{-i\gmm}pq(4it)^{\frac{\nu-2}{2}})+O(\lmb^{4-}+\lmb^{3}|t|^{\frac{\Re(\nu)-2}{2}+\delta}),\label{eq:RQ,z-1}\\
(R_{Q,z^{\flat}},iQ)_{r} & =-8\sqrt{8}\pi\lmb^{3}\Im(e^{-i\gmm}pq(4it)^{\frac{\nu-2}{2}})+O(\lmb^{4-}+\lmb^{3}|t|^{\frac{\Re(\nu)-2}{2}+\delta}).\label{eq:RQ,z-2}
\end{align}
Here, we recall $p=\frac{1}{2}\Gmm(\frac{\nu}{2}+2)$ from Proposition~\ref{prop:ApprxRadiation}.
\end{lem}

Note that the assumption $\lmb(t)\aleq|t|$ is natural in view of
Theorem~\ref{thm:asymptotic-description}. The proof of Lemma~\ref{lem:inn-prod-interaction}
is postponed to the end of this subsection.

Let us briefly discuss why the term $\lmb^{3}e^{-i\gmm}\cdot8\sqrt{8}\pi pq(4it)^{\frac{\nu-2}{2}}$
in Lemma~\ref{lem:inn-prod-interaction} is the dominating term.
Applying Lemma~\ref{lem:inn-prod-interaction} to \eqref{eq:formal-b-without-trunc},
the modulation equations \eqref{eq:formal-b-without-trunc}-\eqref{eq:formal-eta-without-trunc}
have a schematic form (where we ignored $\gmm$, $\eta$, any logarithmic
factors, and assumed $\nu>0$ and $t>0$)
\[
\frac{\lmb_{s}}{\lmb}+b=0,\qquad b_{s}+b^{2}+\lmb^{3}t^{\frac{\nu-2}{2}}+\lmb^{4-}=0,
\]
where $\lmb^{3}t^{\frac{\nu-2}{2}}$ and $\lmb^{4-}$ come from $R_{Q,z^{\flat}}$.
Suppose that $\lmb^{4-}$ is a dominating term, say $b^{2}+\lmb^{3}t^{\frac{\nu-2}{2}}\aleq\lmb^{4-}$.
Let us substitute the ansatz $\lmb(t)\sim t{}^{p}$, so that $b\sim\lmb\lmb_{t}\sim t^{2p-1}$.
On one hand, $b_{t}\sim\rd_{t}t^{2p-1}\sim t^{2p-2}$. On the other
hand, as $\lmb^{4-}$ is the dominating term, we have $b_{t}\sim\lmb^{-2}b_{s}\sim\lmb^{2-}\sim t{}^{2p-}$.
This says $2p-2=2p-$, i.e., $p$ must be arbitrarily large. This
contradicts to $\lmb^{3}t^{\frac{\nu-2}{2}}\aleq\lmb^{4-}$. Therefore,
$\lmb^{3}t^{\frac{\nu-2}{2}}$ must be the dominating term in the
$b_{s}$-equation.

With the help of Lemma~\ref{lem:inn-prod-interaction} and \eqref{eq:B0-1},
our formal modulation equations become: 
\[
\left\{ \begin{aligned}\frac{\lmb_{s}}{\lmb}+b=0,\qquad\gmm_{s}+\eta=0,\\
b_{s}+b^{2}+\eta^{2}+\Re\Big(\lmb^{3}e^{-i\gmm}\cdot\frac{8\sqrt{8}\pi pq(4it)^{\frac{\nu-2}{2}}}{4\pi\log B_{0}}\Big) & =0,\\
\eta_{s}+\Im\Big(\lmb^{3}e^{-i\gmm}\cdot\frac{8\sqrt{8}\pi pq(4it)^{\frac{\nu-2}{2}}}{4\pi\log B_{0}}\Big) & =0.
\end{aligned}
\right.
\]
Integrating these together with the boundary conditions $\lmb(t),\frac{b(t)}{\lmb(t)},\frac{\eta(t)}{\lmb(t)}\to0$
as $t\to0^{-}$, (see Section~\ref{subsec:Closing-bootstrap} for
more details), one obtains a solution $(\lmb,\gmm,b,\eta)$ that is
equal to $(\lmb_{q,\nu},\gmm_{q,\nu},b_{q,\nu},\eta_{q,\nu})$ at
the leading order, where 
\begin{align}
\bm{\lmb}_{q,\nu}(t)\coloneqq\lmb_{q,\nu}(t)e^{i\gmm_{q,\nu}(t)} & \coloneqq-\frac{\sqrt{2}}{4}\frac{\Gmm(\tfrac{\nu}{2})}{\Re(\nu)+1}\cdot q\frac{(4it)^{\frac{\nu}{2}+1}}{|\log|t||},\label{eq:def-lmb-gmm-exact-Sect5}\\
\bm{b}_{q,\nu}(t)\coloneqq b_{q,\nu}(t)+i\eta_{q,\nu}(t) & \coloneqq\frac{1}{2}\Big(\frac{\nu}{2}+1\Big)\Big|\frac{\Gmm(\tfrac{\nu}{2})}{\Re(\nu)+1}\Big|^{2}\cdot|q|^{2}\frac{|4it|^{\Re(\nu)+1}}{|\log|t||^{2}}.\label{eq:def-b-eta-exact-Sect5}
\end{align}
We end this subsection with the proof of Lemma~\ref{lem:inn-prod-interaction}.
\begin{proof}[Proof of Lemma~\ref{lem:inn-prod-interaction}]
The heart of the proof is to extract the main terms of $R_{Q,z^{\flat}}$
that contribute to the inner products \eqref{eq:RQ,z-1} and \eqref{eq:RQ,z-2}.

\smallskip
\textbf{Step 1.} Decomposition of $R_{Q,z^{\flat}}$.

We begin by expanding $R_{Q,z^{\flat}}$ using the self-dual form
\eqref{eq:CSS-self-dual-form}: 
\begin{align*}
R_{Q,z^{\flat}} & =\nabla E[Q+z^{\flat}]-\nabla\td E[z^{\flat}]-\tht_{z^{\flat}}Q\\
 & =L_{Q+z^{\flat}}^{\ast}\bfD_{Q+z^{\flat}}(Q+z^{\flat})-L_{z^{\flat}}^{(\frkm)\ast}z_{1}^{\flat_{-1}}-\tht_{z^{\flat}}Q,
\end{align*}
where $z_{1}^{\flat_{-1}}=\bfD_{z^{\flat}}^{(\frkm)}z^{\flat}=\lmb^{2}e^{-i\gmm}z_{1}(\lmb y)$
is an $\dot{H}^{-1}$-scaling of $z_{1}$. We rearrange this as
\[
R_{Q,z^{\flat}}=\{(L_{Q+z^{\flat}}^{\ast}-L_{z^{\flat}}^{(\frkm)\ast})z_{1}^{\flat_{-1}}-\tht_{z^{\flat}}Q\}+L_{Q+z^{\flat}}^{\ast}(\bfD_{Q+z^{\flat}}(Q+z^{\flat})-z_{1}^{\flat_{-1}}).
\]
We expand the definitions of $L_{w}^{\ast}$ and $L_{z^{\flat}}^{(\frkm)\ast}$
and use $\tht_{z^{\flat}}=\int_{0}^{\infty}\Re(\br{z^{\flat}}z_{1}^{\flat_{-1}})dy$
to obtain 
\begin{equation}
\begin{aligned}R_{Q,z^{\flat}} & =-\tfrac{1}{y}(2+A_{\tht}[Q])z_{1}^{\flat_{-1}}+(\tint y{\infty}\Re(Qz_{1}^{\flat_{-1}})dy')Q\\
 & \quad-\tfrac{2}{y}A_{\tht}[Q,z^{\flat}]z_{1}^{\flat_{-1}}+(\tint y{\infty}\Re(Qz_{1}^{\flat_{-1}})dy')z^{\flat}-(\tint 0y\Re(\br{z^{\flat}}z_{1}^{\flat_{-1}})dy')Q\\
 & \quad+L_{Q+z^{\flat}}^{\ast}(\bfD_{Q+z^{\flat}}(Q+z^{\flat})-z_{1}^{\flat_{-1}}).
\end{aligned}
\label{eq:RQ,z-expansion}
\end{equation}
Notice that we have the $\tint 0y\cdot dy'$ integral for the fifth
term, which exploits the phase correction term $\tht_{z^{\flat}}Q$.\footnote{In fact, the phase correction term $\tht_{z^{\flat}}Q$ has no effect
when computing the inner products with $\Lmb Q$ and $iQ$. However,
this fact will be crucially used in the $L^{2}$-estimate of $R_{Q,z^{\flat}}$
in Lemma~\ref{lem:RQ,z}.}

Next, write 
\[
R_{Q,z^{\flat}}=\wh R_{Q,z^{\flat}}+\td R_{Q,z^{\flat}},
\]
where $\wh R_{Q,z^{\flat}}$ denotes the first line of RHS\eqref{eq:RQ,z-expansion}
and $\td R_{Q,z^{\flat}}$ collects the remaining terms. We will show
that $\wh R_{Q,z^{\flat}}$ contains the main contribution to the
inner products and $\td R_{Q,z^{\flat}}$ is an error. More precisely,
we claim: 
\begin{align}
(\wh R_{Q,z^{\flat}},\Lmb Q)_{r} & =8\sqrt{8}\pi\lmb^{3}\Re(e^{-i\gmm}pq(4it)^{\frac{\nu-2}{2}})+O(\lmb^{4-}+\lmb^{3}|t|^{\frac{\Re(\nu)-2}{2}+\delta})\label{eq:RQ,z-claim1}\\
(\wh R_{Q,z^{\flat}},iQ)_{r} & =-8\sqrt{8}\pi\lmb^{3}\Im(e^{-i\gmm}pq(4it)^{\frac{\nu-2}{2}})+O(\lmb^{4-}+\lmb^{3}|t|^{\frac{\Re(\nu)-2}{2}+\delta})\label{eq:RQ,z-claim2}
\end{align}
and 
\begin{equation}
|(\td R_{Q,z^{\flat}},\Lmb Q)_{r}|+|(\td R_{Q,z^{\flat}},iQ)_{r}|\aleq\lmb^{4-}+\lmb^{3}|t|^{\frac{\Re(\nu)-2}{2}+\delta}.\label{eq:RQ,z-claim3}
\end{equation}
These claims immediately imply our lemma. It remains to show \eqref{eq:RQ,z-claim1}-\eqref{eq:RQ,z-claim3}.

\smallskip
\textbf{Step 2.} Preliminary claims.

Before we get into the proof of \eqref{eq:RQ,z-claim1}-\eqref{eq:RQ,z-claim3},
we claim the following estimates for $z^{\flat}$ and $z_{1}^{\flat_{-1}}$:
\begin{align}
\chf_{y\leq B_{0}}|z^{\flat}|_{1} & \aleq\lmb^{3}|t|^{\frac{\Re(\nu)-2}{2}}y^{2},\label{eq:RQ,z-claim4-1}\\
\chf_{y\leq B_{0}}|z_{1}^{\flat_{-1}}-\lmb^{3}\cdot4qp(4it)^{\frac{\nu-2}{2}}y|_{1} & \aleq\lmb^{3}|t|^{\frac{\Re(\nu)-2}{2}}y^{3}B_{0}^{-2},\label{eq:RQ,z-claim4-2}\\
\|y^{-1}|z^{\flat}|_{-1}\|_{L^{2}} & \aleq\lmb^{2}|\log|t||^{\frac{1}{2}}+\lmb^{3}|t|^{\frac{\Re(\nu)-2}{2}}B_{0},\label{eq:RQ,z-claim4-3}\\
\|\chf_{y\geq B_{0}}y^{-2}|z^{\flat}|_{-1}\|_{L^{2}} & \aleq\lmb^{3}|\log|t||^{\frac{1}{2}}+\lmb^{3}|t|^{\frac{\Re(\nu)-2}{2}},\label{eq:RQ,z-claim4-4}\\
\|y^{-s}|z^{\flat}|_{1}\|_{L^{2}\cap yL^{1}\cap y^{-1}L^{\infty}} & \aleq\lmb^{s},\qquad\forall s\in[0,1].\label{eq:RQ,z-claim4-5}
\end{align}
The estimates \eqref{eq:RQ,z-claim4-1} and \eqref{eq:RQ,z-claim4-2}
easily follow from rescaling \eqref{eq:z-ss-bound} and \eqref{eq:z1-ss-leading}
together with $|t|^{-\frac{1}{2}}r=B_{0}^{-1}y$. The estimates \eqref{eq:RQ,z-claim4-3}
and \eqref{eq:RQ,z-claim4-4} follow from \eqref{eq:RQ,z-claim4-1}
and rescaling \eqref{eq:z-ext-bound}: 
\begin{align*}
\|\chf_{y\leq B_{0}}y^{-1}|z^{\flat}|_{-1}\|_{L^{2}} & \aleq\lmb^{3}|t|^{\frac{\Re(\nu)-2}{2}}B_{0},\\
\|\chf_{y\geq B_{0}}y^{-1}|z^{\flat}|_{-1}\|_{L^{2}} & \aleq\lmb^{2}\|\chf_{|t|^{\frac{1}{2}}\leq r\leq1}r^{\Re(\nu)-2}\|_{L^{2}}\aleq\lmb^{2}(|\log|t||^{\frac{1}{2}}+|t|^{\frac{\Re(\nu)-1}{2}}),\\
\|\chf_{y\geq B_{0}}y^{-2}|z^{\flat}|_{-1}\|_{L^{2}} & \aleq\lmb^{3}\|\chf_{|t|^{\frac{1}{2}}\leq r\leq1}r^{\Re(\nu)-3}\|_{L^{2}}\aleq\lmb^{3}(|\log|t||^{\frac{1}{2}}+|t|^{\frac{\Re(\nu)-2}{2}}).
\end{align*}
Finally, the estimate \eqref{eq:RQ,z-claim4-5} follows from 
\begin{align*}
\|\chf_{y\leq B_{0}}y^{-s}|z^{\flat}|_{1}\|_{L^{2}\cap yL^{1}\cap y^{-1}L^{\infty}} & \aleq\lmb^{3}|t|^{\frac{\Re(\nu)-2}{2}}B_{0}^{3-s}\aeq\lmb^{s}|t|^{\frac{\Re(\nu)+1-s}{2}},\\
\|\chf_{y\geq B_{0}}y^{-s}|z^{\flat}|_{1}\|_{L^{2}\cap yL^{1}\cap y^{-1}L^{\infty}} & \aleq\lmb^{s}\|\chf_{|t|^{\frac{1}{2}}\leq r\leq1}r^{\Re(\nu)-s}\|_{L^{2}\cap yL^{1}\cap y^{-1}L^{\infty}}\aleq\lmb^{s}.
\end{align*}

\smallskip
\textbf{Step 3.} The main contribution from $\wh R_{Q,z^{\flat}}$.

In this step, we prove \eqref{eq:RQ,z-claim1} and \eqref{eq:RQ,z-claim2}.
We begin by noticing that 
\[
\wh R_{Q,z^{\flat}}=(L_{Q}^{\ast}+\rd_{y}-\tfrac{1}{y})z_{1}^{\flat_{-1}}=(L_{Q}-\rd_{y}-\tfrac{2}{y})^{\ast}z_{1}^{\flat_{-1}}.
\]
Since $\Lmb Q$ and $iQ$ belong to the kernel of $L_{Q}$, we have\footnote{We note that $L_{Q}-\rd_{y}$ does not contain any derivatives. One
only needs to use Fubini to justify the computation $((L_{Q}-\rd_{y}-\tfrac{2}{y})^{\ast}f,g)_{r}=(f,(L_{Q}-\rd_{y}-\tfrac{2}{y})g)_{r}$.} 
\begin{align*}
(\wh R_{Q,z^{\flat}},\Lmb Q)_{r} & =(z_{1}^{\flat_{-1}},(L_{Q}-\rd_{y}-\tfrac{2}{y})\Lmb Q)_{r}=(z_{1}^{\flat_{-1}},(-\rd_{y}-\tfrac{2}{y})\Lmb Q)_{r},\\
(\wh R_{Q,z^{\flat}},iQ)_{r} & =(z_{1}^{\flat_{-1}},(L_{Q}-\rd_{y}-\tfrac{2}{y})iQ)_{r}=(z_{1}^{\flat_{-1}},(-\rd_{y}-\tfrac{2}{y})iQ)_{r}.
\end{align*}
We then observe that $(-\rd_{y}-\tfrac{2}{y})$ taken on $\Lmb Q$
or $iQ$ enjoys extra \emph{spatial decay}: 
\[
|(-\rd_{y}-\tfrac{2}{y})\Lmb Q|+|(-\rd_{y}-\tfrac{2}{y})iQ|\aleq y^{-1}\langle y\rangle^{-4}.
\]
Thanks to this decay, we are able to replace $z_{1}^{\flat_{-1}}$
by its leading term; we observe for any $\psi$ with $|\psi|\aleq y^{-1}\langle y\rangle^{-4}$
that 
\begin{align*}
 & |(z_{1}^{\flat_{-1}}-\chf_{y\leq B_{0}}4qp(4it)^{\frac{\nu-2}{2}}\lmb^{3}y,\psi)_{r}|\\
 & \aleq\|\chf_{y\leq B_{0}}|z_{1}^{\flat_{-1}}-4qp(4it)^{\frac{\nu-2}{2}}\lmb^{3}y|\cdot y^{-1}\lan y\ran^{-4}\|_{L^{1}}+\|\chf_{y\geq B_{0}}|z^{\flat}|_{-1}\cdot y^{-5}\|_{L^{1}}\\
 & \aleq\|\chf_{y\leq B_{0}}|z_{1}^{\flat_{-1}}-4qp(4it)^{\frac{\nu-2}{2}}\lmb^{3}y|\cdot y^{-1}\lan y\ran^{-4}\|_{L^{1}}+\|\chf_{y\geq B_{0}}y^{-2}|z^{\flat}|_{-1}\|_{L^{2}}B_{0}^{-2},
\end{align*}
and we then apply \eqref{eq:RQ,z-claim4-2} and \eqref{eq:RQ,z-claim4-4}
to continue the previous display as 
\begin{align*}
 & |(z_{1}^{\flat_{-1}}-\chf_{y\leq B_{0}}4qp(4it)^{\frac{\nu-2}{2}}\lmb^{3}y,\psi)_{r}|\\
 & \aleq\lmb^{3}|t|^{\frac{\Re(\nu)-2}{2}}B_{0}^{-2}\log B_{0}+(\lmb^{3}|\log|t||^{\frac{1}{2}}+\lmb^{3}|t|^{\frac{\Re(\nu)-2}{2}})B_{0}^{-2}\\
 & \aleq\lmb^{5}|t|^{-1}|\log|t||^{\frac{1}{2}}+\lmb^{3}|t|^{\frac{\Re(\nu)-2}{2}+\dlt}\\
 & \aleq\lmb^{4-}+\lmb^{3}|t|^{\frac{\Re(\nu)-2}{2}+\dlt},
\end{align*}
where we used $\lmb\aleq|t|$ (which also implies $B_{0}^{-1}\aleq|t|^{\frac{1}{2}}$).
Therefore, we have shown that 
\begin{align*}
(\wh R_{Q,z^{\flat}},\Lmb Q)_{r} & =(\chf_{y\leq B_{0}}4qp(4it)^{\frac{\nu-2}{2}}\lmb^{3}y,(-\rd_{y}-\tfrac{2}{y})\Lmb Q)_{r}+O(\lmb^{4-}+\lmb^{3}|t|^{\frac{\Re(\nu)-2}{2}+\delta}),\\
(\wh R_{Q,z^{\flat}},iQ)_{r} & =(\chf_{y\leq B_{0}}4qp(4it)^{\frac{\nu-2}{2}}\lmb^{3}y,(-\rd_{y}-\tfrac{2}{y})iQ)_{r}+O(\lmb^{4-}+\lmb^{3}|t|^{\frac{\Re(\nu)-2}{2}+\delta}).
\end{align*}
Additional integration by parts using $(-\rd_{y}-\tfrac{2}{y})^{\ast}y=(\rd_{y}-\tfrac{1}{y})y=0$
then shows that, we have for any $c\in\bbC$ 
\begin{align*}
(\chf_{y\leq B_{0}}cy,(-\rd_{y}-\tfrac{2}{y})\Lmb Q)_{r} & =-2\pi\cdot\Re(c)y^{2}\Lmb Q\Big|_{y=0}^{B_{0}},\\
(\chf_{y\leq B_{0}}cy,(-\rd_{y}-\tfrac{2}{y})iQ)_{r} & =-2\pi\cdot\Im(c)y^{2}Q\Big|_{y=0}^{B_{0}}.
\end{align*}
The proof of \eqref{eq:RQ,z-claim1} and \eqref{eq:RQ,z-claim2} then
follows from substituting into the above the asymptotics 
\[
|y^{2}Q-\sqrt{8}|+|y^{2}\Lmb Q+\sqrt{8}|\aleq y^{-2}\qquad\text{as }y\to\infty
\]
and $c=4qp(4it)^{\frac{\nu-2}{2}}\lmb^{3}$.

\smallskip
\textbf{Step 4.} Error estimates for $\td R_{Q,z^{\flat}}$.

In this step, we prove \eqref{eq:RQ,z-claim3}. We will often use
Lemma~\ref{lem:mapping-integral-op} and Remark~\ref{rem:decreasing-weight}
without mentioning.

First, we claim that all the contributions of the second line of \eqref{eq:RQ,z-expansion}
to the inner products are bounded by 
\[
\aleq\|y^{-1}|z^{\flat}|_{-1}\|_{L^{2}}^{2}.
\]
This bound is admissible due to \eqref{eq:RQ,z-claim4-3}: 
\begin{align*}
\|y^{-1}|z^{\flat}|_{-1}\|_{L^{2}}^{2} & \aleq(\lmb^{2}|\log|t||^{\frac{1}{2}}+\lmb^{3}|t|^{\frac{\Re(\nu)-2}{2}}B_{0})^{2}\\
 & \aleq\lmb^{4}|\log|t||+\lmb^{4}|t|^{\Re(\nu)-1}\aleq\lmb^{4-}+\lmb^{3}|t|^{\frac{\Re(\nu)-2}{2}+\delta}.
\end{align*}
For the proof of this claim, we observe that 
\begin{align*}
 & |(-\tfrac{2A_{\tht}[Q,z^{\flat}]}{y}z_{1}^{\flat_{-1}},\Lmb Q)_{r}|+|(-\tfrac{2A_{\tht}[Q,z^{\flat}]}{y}z_{1}^{\flat_{-1}},iQ)_{r}|\\
 & \quad\aleq\|\tfrac{1}{y^{2}}A_{\tht}[Q,z^{\flat}]\|_{L^{2}}\|\langle y\rangle^{-1}z_{1}^{\flat_{-1}}\|_{L^{2}}\aleq\|y^{-1}|z^{\flat}|_{-1}\|_{L^{2}}^{2},
\end{align*}
and 
\begin{align*}
 & |((\tint y{\infty}\Re(Qz_{1}^{\flat_{-1}})dy')z^{\flat},\Lmb Q)_{r}|+|((\tint y{\infty}\Re(Qz_{1}^{\flat_{-1}})dy')z^{\flat},iQ)_{r}|\\
 & \quad\aleq\|\tint y{\infty}|y'Qz_{1}^{\flat_{-1}}|\tfrac{dy'}{y'}\|_{L^{2}}\|Qz^{\flat}\|_{L^{2}}\aleq\|y^{-1}|z^{\flat}|_{-1}\|_{L^{2}}^{2},
\end{align*}
and by integration by parts 
\begin{align*}
 & |((\tint 0y\Re(\br{z^{\flat}}z_{1}^{\flat_{-1}})dy')Q,\Lmb Q)_{r}|+|((\tint 0y\Re(\br{z^{\flat}}z_{1}^{\flat_{-1}})dy')Q,iQ)_{r}|\\
 & \quad=|-\tfrac{1}{2}(Q\Re(\br{z^{\flat}}z_{1}^{\flat_{-1}}),yQ)_{r}|+0\aleq\|y^{-1}|z^{\flat}|_{-1}\|_{L^{2}}^{2}.
\end{align*}

Now, it suffices to estimate the contribution of the last term of
\eqref{eq:RQ,z-expansion}. We prove the following claims 
\begin{align}
\||\bfD_{Q+z^{\flat}}(Q+z^{\flat})-z_{1}^{\flat_{-1}}|_{-1}\|_{L^{2}} & \aleq\lmb^{3-}+\lmb^{3}|t|^{\frac{\Re(\nu)-2}{2}},\label{eq:RQ,z-claim6}\\
\|yL_{Q+z^{\flat}}\Lmb Q\|_{L^{2}}+\|yL_{Q+z^{\flat}}iQ\|_{L^{2}} & \aleq\lmb,\label{eq:RQ,z-claim7}
\end{align}
in the next paragraph. Assuming these claims, we obtain for $\psi\in\{\Lmb Q,iQ\}$
\begin{align*}
 & (L_{Q+z^{\flat}}^{\ast}(\bfD_{Q+z^{\flat}}(Q+z^{\flat})-z_{1}^{\flat_{-1}}),\psi)_{r}=(\bfD_{Q+z^{\flat}}(Q+z^{\flat})-z_{1}^{\flat_{-1}},L_{Q+z^{\flat}}\psi)_{r}\\
 & \qquad\aleq\|\tfrac{1}{y}(\bfD_{Q+z^{\flat}}(Q+z^{\flat})-z_{1}^{\flat_{-1}})\|_{L^{2}}\|yL_{Q+z^{\flat}}\psi\|_{L^{2}}\aleq(\lmb^{3-}+\lmb^{3}|t|^{\frac{\Re(\nu)-2}{2}})\cdot\lmb.
\end{align*}
This completes the proof of \eqref{eq:RQ,z-claim3}, assuming \eqref{eq:RQ,z-claim6}
and \eqref{eq:RQ,z-claim7}.

\emph{Proof of \eqref{eq:RQ,z-claim6}.} We estimate 
\begin{align*}
 & \||\bfD_{Q+z^{\flat}}(Q+z^{\flat})-z_{1}^{\flat_{-1}}|_{-1}\|_{L^{2}}\\
 & \leq\||\tfrac{2A_{\tht}[Q,z^{\flat}]+A_{\tht}[z^{\flat}]}{y}Q|_{-1}\|_{L^{2}}+\||\tfrac{2+A_{\tht}[Q]+A_{\tht}[Q,z^{\flat}]}{y}z^{\flat}|_{-1}\|_{L^{2}}\\
 & \aleq\big(\|Q^{2}z^{\flat}\|_{L^{2}}+\|\tfrac{2A_{\tht}[Q,z^{\flat}]}{y^{2}}Q\|_{L^{2}}+\|\tfrac{2+A_{\tht}[Q]}{y}|z^{\flat}|_{-1}\|_{L^{2}}\big)\\
 & \peq+\big(\|Q(z^{\flat})^{2}\|_{L^{2}}+\|\tfrac{A_{\tht}[z^{\flat}]}{y^{2}}Q\|_{L^{2}}+\|\tfrac{2A_{\tht}[Q,z^{\flat}]}{y}|z^{\flat}|_{-1}\|_{L^{2}}\big)\\
 & \aleq\|y^{-1}\langle y\rangle^{-2}|z^{\flat}|_{-1}\|_{L^{2}}+\|\langle y\rangle^{-2}z^{\flat}\|_{L^{2}}(\|z^{\flat}\|_{L^{\infty}}+\||z^{\flat}|_{-1}\|_{L^{2}}).
\end{align*}
We estimate the first term of the RHS of the above display using \eqref{eq:RQ,z-claim4-1}
and \eqref{eq:RQ,z-claim4-4}: 
\begin{align*}
\|y^{-1}\langle y\rangle^{-2}|z^{\flat}|_{-1}\|_{L^{2}} & \aleq\|\chf_{y\leq B_{0}}y^{-1}\langle y\rangle^{-2}|z^{\flat}|_{-1}\|_{L^{2}}+\|\chf_{y\geq B_{0}}y^{-2}|z^{\flat}|_{-1}\|_{L^{2}}\cdot B_{0}^{-1}\\
 & \aleq\lmb^{3}|t|^{\frac{\Re(\nu)-2}{2}}+\lmb^{3-}B_{0}^{-1}\aleq\lmb^{3}|t|^{\frac{\Re(\nu)-2}{2}}+\lmb^{3}.
\end{align*}
For the second term, we use \eqref{eq:RQ,z-claim4-3} to have 
\[
\|\langle y\rangle^{-2}z^{\flat}\|_{L^{2}}(\|z^{\flat}\|_{L^{\infty}}+\||z^{\flat}|_{-1}\|_{L^{2}})\aleq(\lmb^{2-}+\lmb^{3}|t|^{\frac{\Re(\nu)-2}{2}}B_{0})\cdot\lmb.
\]
Collecting the above bounds yields 
\[
\||\bfD_{Q+z^{\flat}}(Q+z^{\flat})-z_{1}^{\flat_{-1}}|_{-1}\|_{L^{2}}\aleq\lmb^{3-}+\lmb^{3}|t|^{\frac{\Re(\nu)-2}{2}},
\]
completing the proof of \eqref{eq:RQ,z-claim6}.

\emph{Proof of \eqref{eq:RQ,z-claim7}.} We only show \eqref{eq:RQ,z-claim7}
for $\psi=\Lmb Q$ and leave the case $\psi=iQ$ to the readers. Using
$L_{Q}\Lmb Q=0$ and $\tint 0yQ\Lmb Qy'dy'=\frac{1}{2}y^{2}Q^{2}\aleq\langle y\rangle^{-2}$,
we have 
\begin{align*}
 & \|yL_{Q+z^{\flat}}\Lmb Q\|_{L^{2}}=\|y(L_{Q+z^{\flat}}-L_{Q})\Lmb Q\|_{L^{2}}\\
 & \leq\|-2A_{\tht}[Q,z^{\flat}]\Lmb Q+\tfrac{1}{2}y^{2}Q^{2}z^{\flat}+(\tint 0y\Re(z^{\flat}\Lmb Q)y'dy')Q\|_{L^{2}}\\
 & \peq+\|-A_{\tht}[z^{\flat}]\Lmb Q\|_{L^{2}}+\|(\tint 0y\Re(z^{\flat}\Lmb Q)y'dy')z^{\flat}\|_{L^{2}}\\
 & \aleq\|\langle y\rangle^{-2}z^{\flat}\|_{L^{2}}+\||z^{\flat}|^{2}\|_{L^{2}}+\|\langle y\rangle^{-2}z^{\flat}\|_{L^{1}}\|z^{\flat}\|_{L^{2}}\aleq\lmb,
\end{align*}
where in the last inequality, we used \eqref{eq:RQ,z-claim4-5}. This
completes the proof of \eqref{eq:RQ,z-claim7}.
\end{proof}

\subsection{\label{subsec:reduction}Reduction of the proof of Theorem~\ref{thm:MainThm}
and Corollary~\ref{cor:InfiniteTimeBlowUp}}

For the proof of Theorem~\ref{thm:MainThm}, we follow a general
strategy of the backward construction with modulation analysis. By
a limiting argument, the construction of the desired blow-up solution
$u(t)$ reduces to the construction of a sequence of solutions $u^{(\tau)}(t)$
on some time interval $[t_{0}^{\ast},\tau]$ for all small $\tau<0$
that follow the blow-up regime uniformly in $\tau$.

To analyze each $u^{(\tau)}=u$ (we suppress the dependence on $\tau$
at this moment), we prescribe suitable initial data at time $t=\tau$
and evolve it backward in time. We then introduce a scaling parameter
$\lmb(t)$, phase parameter $\gmm(t)$, and a small remainder term
$\eps(t)$ to decompose $u(t)$ in the form 
\begin{equation}
u(t)=e^{i\gmm_{z}(t)}[(Q+\eps(t))_{\lmb(t),\gmm(t)}+z(t)],\label{eq:u-decomp}
\end{equation}
where $z(t)$ is the approximate radiation in Proposition~\ref{prop:ApprxRadiation}
and $\gmm_{z}(t)$ is the extra phase correction defined by \eqref{eq:Def-tht_z},
so that the dynamics of $u(t)$ can be read off from those of $\lmb(t)$,
$\gmm(t)$, and $\eps(t)$. Note that $\gmm_{z}(t)$ captures the
extra phase rotation caused by the nonlocal interaction between $Q_{\lmb(t),\gmm(t)}$
and $z(t)$ through $A_{t}[u]$, which was observed in \cite{KimKwon2019arXiv}.

We will not modify the blow-up profile into, for example, $Q_{b,\eta}=Q+b(-i\tfrac{y^{2}}{4}Q)+\eta\rho+\cdots$.
We only use $Q$ itself as a blow-up profile. Instead, we will read
off these parameters $b$ and $\eta$ from the inner products $(\eps,i\Lmb Q)_{r}$
and $(\eps,Q)_{r}$ together with suitable corrections. This approach
is inspired by the work \cite{JendrejLawrieRodriguez2019arXiv} and
we view our approach as its adaptation to the Schrödinger case.

The goal of this subsection is to reduce the proof of Theorem~\ref{thm:MainThm}
to the proof of main bootstrap Proposition~\ref{prop:bootstrap},
which quantifies the blow-up regime in terms of $\lmb(t)$, $\gmm(t)$,
and $\eps(t)$.
\begin{proof}[Proof of Theorem~\ref{thm:MainThm} assuming Proposition~\ref{prop:bootstrap}]
Let $m=0$ and $\frkm=-2$; let $q,\nu\in\bbC$ with $\Re(\nu)>0$
and $q\neq0$; let $z^{\ast}(r)=qr^{\nu}\chi(r)$ be the asymptotic
profile. Let $z(t,r)$ be the approximate radiation constructed in
Proposition~\ref{prop:ApprxRadiation}; it, in particular, satisfies
$z(t)\to z^{\ast}$ in $L^{2}$ as $t\to0$ and approximately solves
the $\frkm$-equivariant phase-rotated CSS as \eqref{eq:eqn-aprx-radiation}. 

We will consider a negative time $-1<t_{0}^{\ast}<0$ that will be
chosen sufficiently small during the course of the proof. 

\smallskip
\textbf{Step 1.} Family of initial data at time $t=\tau$.

For each $\tau\in(t_{0}^{\ast},0)$, we fix the initial data for $u^{(\tau)}$
at time $t=\tau$ by 
\[
u^{(\tau)}(\tau)=e^{i\gmm_{z}(\tau)}[(Q+\eps_{q,\nu}(\tau))_{\lmb_{q,\nu}(\tau),\gmm_{q,\nu}(\tau)}+z(\tau)],
\]
where $\lmb_{q,\nu}(\tau)$ and $\gmm_{q,\nu}(\tau)$ are defined
via the formula \eqref{eq:lmb-gmm-formula}, $\eps_{q,\nu}(\tau)$
is defined by 
\[
\eps_{q,\nu}(\tau)\coloneqq b_{q,\nu}(\tau)\cdot(-i\tfrac{y^{2}}{4}Q\chi_{B_{0}})+\eta_{q,\nu}(\tau)\cdot\rho\chi_{B_{0}}-(1-\chi_{B_{0}})Q
\]
with $B_{0}=|\tau|^{\frac{1}{2}}/\lmb_{q,\nu}(\tau)$ at time $t=\tau$
(see \eqref{eq:def-B0}) and $b_{q,\nu}(\tau)$ and $\eta_{q,\nu}(\tau)$
defined by \eqref{eq:def-b-eta-exact-Sect5}. As $t_{0}^{\ast}$ is
small, $B_{0}$ is sufficiently large so that \eqref{eq:Z1Z2-gauge-condition}
ensures the orthogonality conditions at time $t=\tau$
\[
(\eps_{q,\nu}(\tau),\calZ_{1})_{r}=(\eps_{q,\nu}(\tau),\calZ_{2})_{r}=0.
\]
We remark that the last term $-(1-\chi_{B_{0}})Q$ is introduced to
ensure that $Q+\eps_{q,\nu}(\tau)$ is compactly supported and, hence
, $ru^{(\tau)}(\tau)\in L^{2}$. 

For later purposes, we note using \eqref{eq:transversality-radial-case}
and $(B_{0})^{-2}\sim b_{q,\nu}\sim|\bm{b}_{q,\nu}|$ that 
\begin{equation}
\left\{ \begin{aligned}(\eps_{q,\nu}(\tau),i\Lmb_{B_{0}}Q)_{r} & =(4\pi\log B_{0})b_{q,\nu}(\tau)+O(b_{q,\nu}(\tau)),\\
(\eps_{q,\nu}(\tau),\chi_{B_{0}}Q)_{r} & =(4\pi\log B_{0})\eta_{q,\nu}(\tau)+O(b_{q,\nu}(\tau)),\\
\|L_{Q}\eps_{q,\nu}(\tau)\|_{L^{2}}^{2} & =(4\pi\log B_{0}+O(1))|\bm{b}_{q,\nu}(\tau)|^{2},\\
\|\eps_{q,\nu}(\tau)\|_{L^{2}}^{2} & \aleq b_{q,\nu}(\tau).
\end{aligned}
\right.\label{eq:initial-data-prop}
\end{equation}
We also note using $z^{\ast}\in H_{\frkm}^{1,1}$ (use \eqref{eq:linearized-energy-expn},
the coercivity estimate \eqref{eq:coercivity} for $\|\eps_{q,\nu}(\tau)\|_{\dot{\calH}_{0}^{1}}$,
and the duality estimate to estimate $E[u^{(\tau)}(\tau)]$) that
\begin{equation}
\left\{ \begin{aligned}\|ru^{(\tau)}(\tau)\|_{L^{2}} & \aleq(\lmb_{q,\nu}(\tau))^{2}\|y(Q+\eps_{q,\nu}(\tau))\|_{L^{2}}+1\aleq1,\\
|(u^{(\tau)}(\tau),i\Lmb u^{(\tau)})_{r}| & \aleq\||Q+\eps_{q,\nu}(\tau)|_{1}\|_{L^{2}}^{2}+1\aleq1,\\
E[u^{(\tau)}(\tau)] & \aleq(\lmb_{q,\nu}(\tau))^{-2}\|\eps_{q,\nu}(\tau)+z^{\flat}(\tau)\|_{\dot{\calH}_{0}^{1}}^{2}\aleq1.
\end{aligned}
\right.\label{eq:initial-data-prop-1}
\end{equation}

\smallskip
\textbf{Step 2.} Backward-in-time evolution and main bootstrap.

We evolve each $u^{(\tau)}$ backward in time. Let $(T_{-}^{(\tau)},\tau]$
be its backward-in-time maximal interval of existence and set $t_{-}^{(\tau)}\coloneqq\max(T_{-}^{(\tau)},t_{0}^{\ast})$.
Next, we apply Lemma~\ref{lem:dec-near-Q} to define the time $t_{\mathrm{dec}}^{(\tau)}\in[t_{-}^{(\tau)},\tau)$
by 
\[
t_{\mathrm{dec}}^{(\tau)}\coloneqq\inf\{t\in(t_{-}^{(\tau)},\tau]:e^{-i\gmm_{z}(t)}u^{(\tau)}(t)-z(t)\in\calT_{\delta_{\mathrm{dec}}}\},
\]
where $\calT_{\delta}$ and $\delta_{\mathrm{dec}}$ are defined in
Lemma~\ref{lem:dec-near-Q}. If $|\tau|>0$ is sufficiently small,
then the set in the above display is nonempty and $t_{\mathrm{dec}}^{(\tau)}\in[t_{-}^{(\tau)},\tau)$
is well-defined. For each $t\in(t_{\mathrm{dec}}^{(\tau)},\tau]$,
we can uniquely decompose $e^{-i\gmm_{z}(t)}u^{(\tau)}(t)-z(t)=[Q+\eps]_{\lmb^{(\tau)}(t),\gmm^{(\tau)}(t)}$
according to Lemma~\ref{lem:dec-near-Q}. Equivalently, we have uniquely
decomposed $u^{(\tau)}(t)$ as 
\begin{equation}
u^{(\tau)}(t)=e^{i\gmm_{z}(t)}[(Q+\eps^{(\tau)}(t))_{\lmb^{(\tau)}(t),\gmm^{(\tau)}(t)}+z(t)]\label{eq:u^tau-decomp}
\end{equation}
with the orthogonality conditions 
\begin{equation}
(\eps^{(\tau)}(t),\calZ_{1})_{r}=(\eps^{(\tau)}(t),\calZ_{2})_{r}=0\label{eq:eps^tau-orthog}
\end{equation}
for each $t\in(t_{\mathrm{dec}}^{(\tau)},\tau]$. We note that $\lmb^{(\tau)}(\tau)=\lmb_{q,\nu}(\tau)$,
$\gmm^{(\tau)}(\tau)=\gmm_{q,\nu}(\tau)$, and $\eps^{(\tau)}(\tau)=\eps_{q,\nu}(\tau)$,
thanks to the uniqueness of the decomposition.

Having defined the modulation parameters, we are now ready to state
the main bootstrap proposition. We will assume this proposition and
finish the proof of Theorem~\ref{thm:MainThm}.
\begin{prop}[Bootstrap]
\label{prop:bootstrap}There exists a small negative time $t_{0}^{\ast}\in(-1,0)$
with the following property. For any $\tau\in(t_{0}^{\ast},0)$, consider
$u^{(\tau)}$ that admits the decomposition \eqref{eq:u^tau-decomp}
with \eqref{eq:eps^tau-orthog} on $(t_{\mathrm{dec}}^{(\tau)},\tau]$
as described above. If $u^{(\tau)}$ satisfies the bootstrap hypothesis
\begin{equation}
\left\{ \begin{aligned}|\lmb^{(\tau)}(t)-\lmb_{q,\nu}(t)| & \leq\tfrac{1}{2}\lmb_{q,\nu}(t),\\
\Big|\frac{(\eps^{(\tau)},i\Lmb_{B_{0}^{(\tau)}}Q)_{r}}{4\pi\log B_{0}^{(\tau)}}-b_{q,\nu}(t)\Big| & \leq\tfrac{1}{2}b_{q,\nu}(t),\\
\|L_{Q}\eps^{(\tau)}(t)\|_{L^{2}} & \leq2\cdot\{2\pi(\Re(\nu)+1)\}^{\frac{1}{2}}|\bm{b}_{q,\nu}(t)|,\\
\|\eps^{(\tau)}(t)\|_{L^{2}} & \leq(b_{q,\nu}(t))^{\frac{1}{2}}|\log b_{q,\nu}(t)|
\end{aligned}
\right.\label{eq:bootstrap-hypothesis}
\end{equation}
on the time interval $[s,\tau]$ for some $s>\max\{t_{\mathrm{dec}}^{(\tau)},t_{0}^{\ast}\}$,
then the bootstrap conclusion 
\begin{equation}
\left\{ \begin{aligned}|\lmb^{(\tau)}(t)e^{i\gmm^{(\tau)}(t)}-\lmb_{q,\nu}(t)e^{i\gmm_{q,\nu}(t)}| & \leq o_{t\to0}(\lmb_{q,\nu}(t)),\\
\Big|\frac{(\eps^{(\tau)},i\Lmb_{B_{0}^{(\tau)}}Q)_{r}+i(\eps^{(\tau)},\chi_{B_{0}^{(\tau)}}Q)_{r}}{4\pi\log B_{0}^{(\tau)}}-\bm{b}_{q,\nu}(t)\Big| & \leq o_{t\to0}(b_{q,\nu}(t)),\\
\Big|\|L_{Q}\eps^{(\tau)}(t)\|_{L^{2}}-\{2\pi(\Re(\nu)+1)\}^{\frac{1}{2}}|\bm{b}_{q,\nu}(t)|\Big| & \leq o_{t\to0}(b_{q,\nu}(t)),\\
\|\eps^{(\tau)}(t)\|_{L^{2}} & \leq o_{t\to0}((b_{q,\nu}(t))^{\frac{1}{2}}|\log b_{q,\nu}(t)|)
\end{aligned}
\right.\label{eq:bootstrap-conclusion}
\end{equation}
hold on $[s,\tau]$. Furthermore, we may shrink $t_{0}^{\ast}$ such
that the bootstrap conclusion implies $\|\eps^{(\tau)}(t)\|_{\dot{\calH}_{0}^{1}}<\tfrac{1}{2}\delta_{\mathrm{dec}}$.
\end{prop}
Here, the notation $o_{t \to 0}(f(t))$ means an $o_{t \to 0}(1)$ function that is independent of $\tau$ times $f(t)$. In many cases, the function $o_{t \to 0}(1)$ can be made explicit, but we choose to avoid it for the simplicity of the exposition.

Since the bootstrap conclusion implies $\|\eps^{(\tau)}(t)\|_{\dot{\calH}_{0}^{1}}<\tfrac{1}{2}\delta_{\mathrm{dec}}$,
we have $e^{-i\gmm_{z}(t)}u^{(\tau)}(t)-z(t)=[Q+\eps^{(\tau)}(t)]_{\lmb^{(\tau)}(t),\gmm^{(\tau)}(t)}\in\calT_{\delta_{\mathrm{dec}}/2}$.
Then the above bootstrap implies that, for any $\tau\in(t_{0}^{\ast},0)$,
we have $t_{\mathrm{dec}}^{(\tau)}<t_{0}^{\ast}$ and the bootstrap
conclusion holds on the interval $[t_{0}^{\ast},\tau]$.

\smallskip
\textbf{Step 3.} Construction of the blow-up solution by a limiting
argument.

In the previous step, we constructed a family of solutions $\{u^{(\tau)}\}_{\tau\in(t_{0}^{\ast},0)}$
(each $u^{(\tau)}$ defined on $[t_{0}^{\ast},\tau]$) satisfying
the bootstrap conclusion, uniformly in $\tau$. The main goal of this
step is to construct a desired blow-up solution $u$ defined on $[t_{0}^{\ast},0)$
by taking a weak limit as $\tau\to0^{-}$.

For a limiting argument, we need compactness properties of $\lmb^{(\tau)}$,
$\gmm^{(\tau)}$, and $u^{(\tau)}(t_{0}^{\ast})$ in various topologies.
First, we claim that $(\lmb^{(\tau)},\gmm^{(\tau)})$ is precompact
in the topology $C_{\mathrm{loc}}^{0}([t_{0}^{\ast},0)\to\bbR_{+}\times\bbR/2\pi\bbZ)$.
To see this, by the bootstrap conclusion, $|\log\lmb^{(\tau)}|$ is
uniformly bounded on $[t_{0}^{\ast},c]$ for any $c\in(t_{0}^{\ast},0)$.
On the other hand, by \eqref{eq:rough-control-lmb-gmm} with $\rd_{s}=(\lmb^{(\tau)})^{2}\rd_{t}$
and \eqref{eq:proximity-b-bqnu}, we see that $|\rd_{t}\log\lmb^{(\tau)}|$
and $|\rd_{t}\gmm^{(\tau)}|$ are also uniformly bounded on $[t_{0}^{\ast},c]$.
Now the claim follows by the Arzelà--Ascoli theorem.

Next, we claim that $\{u^{(\tau)}(t_{0}^{\ast})\}$ is $H_{0}^{1,1}$-bounded.
The $H_{0}^{1}$-boundedness is immediate from \eqref{eq:bootstrap-conclusion}
and the coercivity estimate \eqref{eq:coercivity}, so it suffices
to show that $\{\|ru^{(\tau)}(t_{0}^{\ast})\|_{L^{2}}\}_{\tau}$ is
bounded. This follows from \eqref{eq:initial-data-prop-1}, the nonlinear
virial identities \eqref{eq:virial-1}--\eqref{eq:virial-2}, and
energy conservation. Note in particular that the $H_{0}^{1,1}$-boundedness
implies the $L^{2}$-precompactness.

Therefore, there exist a sequence $\tau_{n}\to0^{-}$, continuous
parameter functions $\lmb(t)$ and $\gmm(t)$ on $[t_{0}^{\ast},0)$,
and a profile $u(t_{0}^{\ast})\in H_{0}^{1,1}$ such that $(\lmb^{(\tau_{n})},\gmm^{(\tau_{n})})\to(\lmb,\gmm)$
in $C_{\mathrm{loc}}^{0}([t_{0}^{\ast},0)\to\bbR_{+}\times\bbR/2\pi\bbZ)$,
$u^{(\tau_{n})}(t_{0}^{\ast})\weakto u(t_{0}^{\ast})$ in $H_{0}^{1,1}$,
and $u^{(\tau_{n})}(t_{0}^{\ast})\to u(t_{0}^{\ast})$ in $H_{0}^{1-}$.
In particular, $\lmb(t)$ and $\gmm(t)$ satisfies \eqref{eq:bootstrap-conclusion}
on $[t_{0}^{\ast},0)$.

Let $u(t)$ be the forward-in-time maximal evolution of $u(t_{0}^{\ast})$
on $[t_{0}^{\ast},T_{+})$. We first claim that $T_{+}\geq0$. Suppose
not, i.e., $T_{+}<0$. By the standard $H_{0}^{1-}$-Cauchy theory,
$u^{(\tau_{n})}\to u$ in $C_{\mathrm{loc}}^{0}([t_{0}^{\ast},T_{+})\to H_{0}^{1-})$.
Since $\{u^{(\tau_{n})}(t)\}$ is $H_{0}^{1}$-bounded uniformly in
$n$ and $t\in[t_{0}^{\ast},T_{+}]$, Fatou property of weak limits
implies that $\{u(t)\}$ is also $H_{0}^{1}$-bounded as $t\to T_{+}$.
This contradicts to the blow-up criterion, i.e., $\|u(t)\|_{H_{0}^{1}}\to\infty$
as $t\to T_{+}$. Therefore, $T_{+}\geq0$.

For each $t\in[t_{0}^{\ast},0)$, since $u(t),\lmb(t),\gmm(t)$ are
defined, we can define $\eps(t)$ via the formula \eqref{eq:u-decomp}.
Due to the convergences of $\lmb^{(\tau_{n})}$, $\gmm^{(\tau_{n})}$,
and $u^{(\tau_{n})}$, we have $\eps^{(\tau_{n})}(t)\to\eps(t)$ in
$H_{0}^{1-}$ for any $t\in[t_{0}^{\ast},0)$. It is then clear that
$\lmb$, $\gmm$, and $\eps$ satisfy \eqref{eq:bootstrap-conclusion}
on $[t_{0}^{\ast},0)$. In particular, $\lmb(t)\to0$ as $t\to0^{-}$
and $T_{+}=0$. Moreover, as $u(t_{0}^{\ast})\in H_{0}^{1,1}$, $u$
is a $H_{0}^{1,1}$-solution. Therefore, the properties stated in
Theorem~\ref{thm:MainThm} hold.
\end{proof}
\begin{proof}[Proof of Corollary~\ref{cor:InfiniteTimeBlowUp}]
Let $v$ be a $H_{0}^{1,1}$-solution on the time interval $[-\frac{1}{t_{0}},0)$
constructed in Theorem~\ref{thm:MainThm}; let us rewrite \eqref{eq:thm-decomp}
as 
\[
v(t)-Q_{\lmb_{q,\nu}(t),\gmm_{q,\nu}(t)}-z_{\mathrm{lin}}(t)\to0\text{ in }L^{2}\text{ as }t\to0^{-},
\]
where $z_{\mathrm{lin}}(t)\coloneqq e^{it\Delta^{(-2)}}z^{\ast}$ with $\Dlt^{(-2)} = \rd_{r}^{2} + \frac{1}{r} \rd_{r} - \frac{4}{r^{2}}$.
Now, let $u$ be the pseudoconformal transform of $v$, i.e., 
\[
u(t,r)\coloneqq[\calC v](t,r)=\frac{e^{i\frac{r^{2}}{4t}}}{t}v\Big(-\frac{1}{t},\frac{r}{t}\Big),\qquad\forall t\in[t_{0},\infty),
\]
which satisfies 
\[
u(t,r)-e^{i\frac{r^{2}}{4t}}Q_{\wh{\lmb}_{q,\nu}(t),\wh{\gmm}_{q,\nu}(t)}(r)-\calC z_{\mathrm{lin}}(t,r)\to0\text{ in }L^{2}\text{ as }t\to\infty.
\]
To complete the proof, we show that the pseudoconformal phase $e^{i\frac{r^{2}}{4t}}$
can be dropped and $\calC z_{\mathrm{lin}}=e^{it\Delta^{(-2)}}u^{\ast}$ with $u^{\ast} \in H^{1, 1}_{0} \cap r L^{2}$.
For the first assertion, we note that 
\[
\|(e^{i\frac{r^{2}}{4t}}-1)Q_{\wh{\lmb}_{q,\nu}(t),\wh{\gmm}_{q,\nu}(t)}(r)\|_{L^{2}}=\|(e^{i(\wh{\lmb}_{q,\nu}(t)r)^{2}/4t}-1)Q\|_{L^{2}}\to0,
\]
where we used $(\wh{\lmb}_{q,\nu}(t))^{2}/t\to0$ as $t\to\infty$
and the DCT. 

For the second assertion, we formally view $z^{\ast}(r)$ and $z_{\mathrm{lin}}(t, r)$ as the
radial part of $(-2)$-equivariant functions $\bfz^{\ast}(x) = z^{\ast}(r) e^{-2 i \tht}$ and $\bfz_{\mathrm{lin}}(t, x) = z_{\mathrm{lin}}(t, r) e^{- 2 i \tht}$, respectively, so that $\lap_{\bbR^{2}} \bfz^{\ast}(x) = (\lap^{(-2)} z^{\ast}(r)) e^{- 2 i \tht}$ and similarly for $z_{\mathrm{lin}}$, $\bfz_{\mathrm{lin}}$.
Therefore, by the Fourier representation of $\bfz_{\mathrm{lin}}$,
\begin{equation*}
\bfz_{\mathrm{lin}}(t, x) = \frac{1}{(2 \pi)^{2}} \int_{\bbR^{2}} e^{- i t \abs{\xi}^{2} + i \xi \cdot x} \wh{\bfz^{\ast}}(\xi) \, d \xi
\end{equation*}
where we denoted $\wh \bff(\xi)=\int_{\bbR^{2}}e^{-i\xi\cdot x} \bff(x)dx$. Applying the pseudoconformal transform, we have 
\begin{align*}
[\calC \bfz_{\mathrm{lin}}](t,x) & =\frac{1}{(2\pi)^{2}}\frac{1}{t}\int_{\bbR^{2}}e^{i(\frac{1}{t})|\xi|^{2}+i\xi\cdot(\frac{x}{t})+i\frac{|x|^{2}}{4t}}\wh{\bfz^{\ast}}(\xi)d\xi\\
 & =\frac{1}{(2\pi)^{2}}\frac{1}{t}\int_{\bbR^{2}}e^{\frac{i}{4t}|x+2\xi|^{2}}\wh{\bfz^{\ast}}(\xi)d\xi\\
 & =\frac{1}{4\pi it}\int_{\bbR^{2}}e^{i\frac{|x-y|^{2}}{4t}}\cdot\frac{i}{4\pi}\wh{\bfz^{\ast}}(-y/2)dy.
\end{align*}
Using the formula of the fundamental solution to the propagator $e^{it\Delta_{\bbR^{2}}}$
in 2D, we obtain 
\[
\calC \bfz_{\mathrm{lin}}=e^{it\Delta_{\bbR^{2}}}[\frac{i}{4\pi}\wh{\bfz^{\ast}}(-\frac{\cdot}{2})].
\]
Finally using the definition of the Bessel functions, it follows that
\[
\frac{i}{4\pi}\wh{\bfz^{\ast}}(-\frac{x}{2}) e^{2 i \tht} = \left[ - \frac{i}{2}\int_{0}^{\infty}J_{-2}(|\frac{x}{2}|r)z^{\ast}(r)rdr\right] e^{- 2 i \tht} = u^{\ast}(r) e^{- 2 i \tht}.
\]
Define $\bfu^{\ast}(x) := u^{\ast}(r) e^{- 2 i \tht}$. Since $\bfz^{\ast} \in H^{1, 1}(\bbR^{2})$, we also have $\bfu^{\ast} \in H^{1, 1}$. Clearly, $\rd_{r} u^{\ast}(r) \in L^{2}$ and $r u^{\ast}(r) \in L^{2}$. Moreover, by Hardy's inequality for $(-2)$-equivariant functions, $r^{-1} u^{\ast}(r) \in L^{2}$; hence $u^{\ast} \in H^{1, 1}_{0} \cap r L^{2}$. Finally,  
\begin{equation*}
\calC z_{\lin}(t, r) = e^{2 i \tht} \calC \bfz_{\mathrm{lin}}(x) = e^{2 i \tht} \left(e^{i t \lap_{\bbR^{2}}} \bfu^{\ast}(\cdot)\right)(x) = e^{i t \lap^{(-2)}} u^{\ast}(r),
\end{equation*}
which completes the proof. \qedhere
\end{proof}

\subsection{Bounds for $z^{\flat}$, $\Psi_{z^{\flat}}$, and $R_{Q,z^{\flat}}$
under bootstrap}

From now on, we prove the main bootstrap Proposition~\ref{prop:bootstrap}.
For the sake of notational convenience, we will drop the superscripts
$(\tau)$. We recall that $u$ admits the decomposition 
\[
u(t)=e^{i\gmm_{z}(t)}[(Q+\eps)^{\sharp}+z](t),
\]
where $\lmb(t)$, $\gmm(t)$, and $\eps(t)$ are determined by the
orthogonality condition 
\begin{equation}
(\eps,\calZ_{1})_{r}=(\eps,\calZ_{2})_{r}=0\label{eq:eps-orthog}
\end{equation}
and satisfy the bootstrap hypothesis \eqref{eq:bootstrap-hypothesis}.
Note that the $\sharp$-operation is defined through these $\lmb$
and $\gmm$. The extra phase rotation $\gmm_{z}(t)$ is defined by
\eqref{eq:Def-tht_z}. \textbf{We will always assume the bootstrap
hypothesis \eqref{eq:bootstrap-hypothesis} from now on, unless otherwise
stated.}

In this section, assuming the bootstrap hypothesis, we record several
convenient (though not sharp) bounds for $z^{\flat}$ as well as for
the interaction term $R_{Q,z^{\flat}}$, in terms of the parameter
$b$. In our context, we define $b$ by 
\begin{equation}
b\coloneqq\frac{(\eps,i\Lmb_{B}^{\mathrm{(avg)}}Q)_{r}+\tfrac{1}{2}(\eps,i\Lmb_{B}^{\mathrm{(avg)}}\eps)_{r}}{4\pi\log B_{0}},\label{eq:def-b}
\end{equation}
where $B_{0}=|t|^{\frac{1}{2}}/\lmb(t)$ as in \eqref{eq:def-B0},
$\Lmb_{B}$ is a truncated $L^{2}$-scaling generator defined by \eqref{eq:def-truncated-scaling-gen},
and $T_{B}^{\avg}$ denotes the average of a family of operators (or
functions) $\{T_{B}\}_{B}$: 
\begin{equation}
T_{B}^{\avg}\coloneqq\frac{10}{\log|\log|t||}\int_{B_{-1/5}}^{B_{-1/10}}T_{B}\frac{dB}{B}\label{eq:def-averaging}
\end{equation}
with 
\[
B_{-s}\coloneqq\frac{B_{0}}{|\log|t||^{s}}=\frac{1}{\lmb(t)}\cdot\frac{|t|^{\frac{1}{2}}}{|\log|t||^{s}}.
\]
One can view $b$ as a modification of $\frac{(\eps,i\Lmb_{B_{0}}Q)_{r}}{4\pi\log B_{0}}$
or the ``$b$-parameter'' in our formal derivation of blow-up (Section~\ref{subsec:Formal-derivation-of-blow-up}).
Notice that we truncated strictly inside the self-similar scale $B_{0}$.
This is because we want to regard $(\eps,i\Lmb_{B}\eps)_{r}$ as an
error (or, a correction term). The choices of $B_{-1/5}$ and $B_{-1/10}$,
and averaging over $B$ are introduced to gain an additional small
parameter in the modulation estimate of $b$; see Remark~\ref{rem:mod-est-cutoff-radius}.

The fact that we need at this stage is just the proximity of $b$
to $b_{q,\nu}$:
\begin{lem}[Proximity of $b$ to $b_{q,\nu}$]
We have 
\begin{equation}
|b(t)-b_{q,\nu}(t)|\aleq|\log|t||^{-\frac{1}{5}}\cdot b_{q,\nu}\aleq o_{t\to0}(b_{q,\nu}).\label{eq:proximity-b-bqnu}
\end{equation}
\end{lem}

\begin{proof}
From \eqref{eq:bootstrap-hypothesis} and \eqref{eq:coercivity},
we note that 
\[
\|\eps\|_{\dot{\calH}_{0}^{1}}\sim\|L_{Q}\eps\|_{L^{2}}\aleq b_{q,\nu}|\log b_{q,\nu}|^{\frac{1}{2}}.
\]
Now we write 
\[
b-b_{q,\nu}=\Big(\frac{(\eps,i\Lmb_{B_{0}}^{\avg}Q)_{r}}{4\pi\log B_{0}}-b_{q,\nu}\Big)+\frac{\big(\eps,i(\Lmb_{B}^{\avg}-\Lmb_{B_{0}})Q\big)_{r}+\tfrac{1}{2}(\eps,i\Lmb_{B}^{\avg}\eps)_{r}}{4\pi\log B_{0}}.
\]
For $B\in[B_{-1/5},B_{-1/10}]$, we observe that 
\begin{align*}
|(\eps,i(\Lmb_{B}-\Lmb_{B_{0}})Q)_{r}| & \aleq\|\eps\cdot\chf_{[B,2B_{0}]}\langle y\rangle^{-2}\|_{L^{1}}\\
 & \aleq(\log(B_{0}/B))^{\frac{1}{2}}\|\eps\|_{\dot{\calH}_{0}^{1}}\aleq(\log|\log|t||)^{\frac{1}{2}}b_{q,\nu}|\log b_{q,\nu}|^{\frac{1}{2}}
\end{align*}
and 
\begin{align*}
|(\eps,i\Lmb_{B}\eps)_{r}| & \aleq\|\chf_{(0,2B]}|\eps|\cdot|\eps|_{1}\|_{L^{1}}\\
\aleq B^{2} & \|\eps\|_{\dot{\calH}_{0}^{1}}^{2}\aleq(B_{-1/10})^{2}\cdot b_{q,\nu}^{2}|\log b_{q,\nu}|\aleq|\log|t||^{-\frac{1}{5}}\cdot b_{q,\nu}|\log b_{q,\nu}|,
\end{align*}
where we used $(B_{0})^{-2}\sim b_{q,\nu}$ and the bootstrap hypothesis
\eqref{eq:bootstrap-hypothesis}. Substituting $4\pi\log B_{0}\sim|\log b_{q,\nu}|$
and the above two displays into the formula of $b-b_{q,\nu}$ completes
the proof.
\end{proof}
\begin{rem}[Convenient relations under the bootstrap hypothesis]
\label{rem:convenient-rel}Under the bootstrap hypothesis \eqref{eq:bootstrap-hypothesis},
we have 
\begin{equation}
\lmb(t)\sim\frac{|t|^{\frac{\Re(\nu)}{2}+1}}{|\log|t||}\quad\text{and}\quad b(t)\sim\frac{|t|^{\Re(\nu)+1}}{|\log|t||^{2}}\label{eq:conv-rel-1}
\end{equation}
by \eqref{eq:proximity-b-bqnu}. Thus we have the following convenient
relations: 
\begin{equation}
\left\{ \begin{aligned}\lmb & \sim b^{\frac{1}{2}}\cdot|t|^{\frac{1}{2}},\\
B_{0} & \sim b^{-\frac{1}{2}},\\
\lmb^{3}|t|^{\frac{\Re(\nu)-2}{2}} & \sim b^{2}|\log b|.
\end{aligned}
\right.\label{eq:conv-rel-2}
\end{equation}
Moreover, the bootstrap hypothesis \eqref{eq:bootstrap-hypothesis}
combined with \eqref{eq:coercivity} and \eqref{eq:proximity-b-bqnu}
says that 
\begin{equation}
\left\{ \begin{aligned}\|\eps\|_{L^{2}} & \aleq b^{\frac{1}{2}}|\log b|,\\
\|\eps\|_{\dot{\calH}_{0}^{1}} & \aleq b|\log b|^{\frac{1}{2}}.
\end{aligned}
\right.\label{eq:conv-rel-3}
\end{equation}
Keeping in mind these relations will help us identify the main terms
and error terms in our analysis.
\end{rem}

Combining Proposition~\ref{prop:ApprxRadiation} and Remark~\ref{rem:convenient-rel},
we are able to obtain the following simple (but not sharp) bounds
for $z^{\flat}$ in terms of $b$.
\begin{cor}[Energy bounds for $z^{\flat}$ and $\Psi_{z^{\flat}}$]
\label{cor:bounds-z-bootstrap}There exists\footnote{See Proposition~\ref{prop:ApprxRadiation} for $\delta_{z}$.}
$0<\delta<\min(\frac{1}{2},\delta_{z})$ such that the following estimates
hold.
\begin{itemize}
\item (Energy bounds at $L^{2}$ and $\dot{H}^{1}$-scalings) We have 
\begin{align}
\|\langle\log_{-}y\rangle^{2}z^{\flat}\|_{L^{2}} & \aleq1,\label{eq:z-flat-L2-bound}\\
\|\langle\log_{-}y\rangle^{2}|z^{\flat}|_{-1}\|_{L^{2}}+\|\langle\log_{-}y\rangle^{2}z^{\flat}\|_{L^{\infty}} & \aleq\lmb,\label{eq:z-flat-H1-bound-1}\\
\|Qz^{\flat}\|_{L^{1}} & \aleq b^{\frac{1}{2}}\cdot|t|^{\delta},\label{eq:z-flat-H1-bound-2}
\end{align}
\item (Energy bounds at $\dot{H}^{2}$ and $\dot{H}^{3}$-scalings) We have
\begin{align}
\|\langle\log_{-}y\rangle^{2}\tfrac{1}{y}|z^{\flat}|_{-1}\|_{L^{2}\cap y^{-1}L^{1}\cap yL^{\infty}} & \aleq b\cdot|t|^{\delta},\label{eq:z-flat-H2-bound}\\
\|\tfrac{1}{y^{2}\langle\log_{-}y\rangle^{2}}|z^{\flat}|_{-1}\|_{L^{2}\cap y^{-1}L^{1}\cap yL^{\infty}} & \aleq b^{\frac{3}{2}}\cdot|t|^{\delta}.\label{eq:z-flat-H3-bound}
\end{align}
\item (Energy bound for $\rd_{-}^{(\frkm)}z^{\flat}$ at $\dot{H}^{4}$-scalings)
We have 
\begin{align}
\|\langle\log_{-}y\rangle^{2}\tfrac{1}{\langle y\rangle y}|\rd_{-}^{(\frkm)}z^{\flat}|_{-1}\|_{L^{2}\cap yL^{1}} & \aleq b^{2}\cdot|t|^{\delta}.\label{eq:z-flat-degenerate-H4-bound}
\end{align}
\item (Energy bounds for $\Psi_{z^{\flat}}$) We have 
\begin{align}
\|\Psi_{z^{\flat}}\|_{L^{2}} & \aleq b^{\frac{3}{2}}\cdot|t|^{\delta},\label{eq:Psi_zflat-L2}\\
\|\langle\log_{-}y\rangle|\Psi_{z^{\flat}}|_{-1}\|_{L^{2}} & \aleq b^{2}\cdot|t|^{\delta}.\label{eq:Psi_zflat-Hdot1}
\end{align}
\end{itemize}
\end{cor}

\begin{rem}
In general, $z^{\flat}$ shows an $y^{2}$-behavior near the origin;
see \eqref{eq:z-ss-leading}. Thus, the logarithmic factor in \eqref{eq:z-flat-H3-bound}
is necessary to make the LHS finite.
\end{rem}

\begin{proof}
We first show the estimates for $z^{\flat}$, namely, \eqref{eq:z-flat-L2-bound}--\eqref{eq:z-flat-degenerate-H4-bound}.

\smallskip
\textbf{Step 1.} Proof of \eqref{eq:z-flat-L2-bound}--\eqref{eq:z-flat-degenerate-H4-bound}
in the self-similar region $y\leq B_{0}$.

We use the rescaled versions of \eqref{eq:z-ss-bound} and \eqref{eq:rd-z-ss-degen}
and apply \eqref{eq:conv-rel-2} to obtain 
\begin{align*}
\chf_{y\leq B_{0}}|z^{\flat}|_{1} & \aleq\chf_{y\leq B_{0}}\lmb^{3}|t|^{\frac{\Re(\nu)-2}{2}}y^{2}\aleq b^{2}|\log b|\cdot\chf_{y\leq B_{0}}y^{2},\\
\chf_{y\leq B_{0}}|\rd_{-}^{(\frkm)}z^{\flat}|_{1} & \aleq\chf_{y\leq B_{0}}\lmb^{3}|t|^{\frac{\Re(\nu)-2}{2}}y^{3}B_{0}^{-2}\aleq b^{3}|\log b|\cdot\chf_{y\leq B_{0}}y^{3}.
\end{align*}
These estimates immediately imply all the estimates \eqref{eq:z-flat-L2-bound}-\eqref{eq:z-flat-degenerate-H4-bound}
in the self-similar region $y\leq B_{0}$.

\smallskip
\textbf{Step 2.} Proof of \eqref{eq:z-flat-L2-bound}--\eqref{eq:z-flat-degenerate-H4-bound}
outside the self-similar region, i.e., $y\geq B_{0}$.

We restrict all the estimates outside the self-similar region. The
bounds \eqref{eq:z-flat-L2-bound}, \eqref{eq:z-flat-H1-bound-1},
and \eqref{eq:z-flat-H1-bound-2} (outside the self-similar region)
are already shown in \eqref{eq:RQ,z-claim4-5}. Next, following a
similar line of the proof of \eqref{eq:RQ,z-claim4-3} and \eqref{eq:RQ,z-claim4-4}
yields the bounds 
\begin{align*}
\|\chf_{y\geq B_{0}}y^{-1}|z^{\flat}|_{-1}\|_{L^{2}\cap y^{-1}L^{1}\cap yL^{\infty}} & \aleq\lmb^{2}|\log|t||+\lmb^{3}|t|^{\frac{\Re(\nu)-2}{2}}B_{0},\\
\|\chf_{y\geq B_{0}}y^{-2}|z^{\flat}|_{-1}\|_{L^{2}\cap y^{-1}L^{1}\cap yL^{\infty}} & \aleq\lmb^{3}|\log|t||+\lmb^{3}|t|^{\frac{\Re(\nu)-2}{2}},
\end{align*}
and applying \eqref{eq:conv-rel-2} to the above display gives \eqref{eq:z-flat-H2-bound}
and \eqref{eq:z-flat-H3-bound}. Finally, the estimate \eqref{eq:z-flat-degenerate-H4-bound}
follows from rescaling, \eqref{eq:z-ext-bound}, and \eqref{eq:conv-rel-2}:
\begin{align*}
 & \|\chf_{y\geq B_{0}}y^{-2}|\rd_{-}^{(\frkm)}z^{\flat}|_{-1}\|_{L^{2}\cap yL^{1}}\aleq\lmb^{4}\|\chf_{r\geq|t|^{\frac{1}{2}}}r^{-2}|z|_{-2}\|_{L^{2}\cap rL^{1}}\\
 & \qquad\aleq\lmb^{4}\|\chf_{|t|^{\frac{1}{2}}\leq r\leq2}\{r^{\Re(\nu)-4}+|t|^{\Re(\nu)-1}r^{-\Re(\nu)-2}\}\|_{L^{2}\cap rL^{1}}\\
 & \qquad\aleq\lmb^{4}\{|\log|t||+|t|^{\frac{1}{2}(\Re(\nu)-3)}+|t|^{\Re(\nu)-1}\}\\
 & \qquad\aleq\lmb^{4-}+\lmb|t|^{-\frac{1}{2}}\cdot\lmb^{3}|t|^{\frac{\Re(\nu)-2}{2}}\aleq b^{2}\cdot|t|^{\dlt}.
\end{align*}
This completes the proof of \eqref{eq:z-flat-L2-bound}--\eqref{eq:z-flat-degenerate-H4-bound}.

\smallskip
\textbf{Step 3.} Proof of \eqref{eq:Psi_zflat-L2}--\eqref{eq:Psi_zflat-Hdot1}.

The estimates \eqref{eq:Psi_zflat-L2} and \eqref{eq:Psi_zflat-Hdot1}
follow from \eqref{eq:Psi_z-L2}, \eqref{eq:Psi_z-Hdot1}, and \eqref{eq:conv-rel-2}:
\[
\|\Psi_{z^{\flat}}\|_{L^{2}}=\lmb^{2}\|\Psi_{z}\|_{L^{2}}\aleq\lmb^{2}|t|^{\frac{1}{2}(\Re(\nu)-1)+\dlt_{z}}\aleq b^{\frac{3}{2}}\cdot|t|^{\delta},
\]
and 
\begin{align*}
 & \|\langle\log_{-}y\rangle|\Psi_{z^{\flat}}|_{-1}\|_{L^{2}}\aleq\|(y^{-\dlt_{z}}+1)|\Psi_{z^{\flat}}|_{-1}\|_{L^{2}}\\
 & \qquad\aleq\lmb^{3+\dlt_{z}}\|r^{-\dlt_{z}}|\Psi_{z}|_{-1}\|_{L^{2}}+\lmb^{3}\||\Psi_{z}|_{-1}\|_{L^{2}}\aleq\lmb^{3}|t|^{\frac{1}{2}(\Re(\nu)-2)+\dlt_{z}}\aleq b^{2}\cdot|t|^{\delta}.
\end{align*}
This completes the proof.
\end{proof}
We turn our attention to the interaction term $R_{Q,z^{\flat}}$.
We have already computed the leading terms of $(R_{Q,z^{\flat}},\Lmb Q)_{r}$
and $(R_{Q,z^{\flat}},iQ)_{r}$ in Lemma~\ref{lem:inn-prod-interaction}.
Here, we further simplify the errors of Lemma~\ref{lem:inn-prod-interaction}
under the bootstrap \eqref{eq:bootstrap-hypothesis}, using Corollary~\ref{cor:bounds-z-bootstrap}.
Moreover, we also estimate the $L^{2}$-norm of $R_{Q,z^{\flat}}$
for later purposes.
\begin{lem}[Interaction term $R_{Q,z^{\flat}}$ under bootstrap hypothesis]
\label{lem:RQ,z}We have 
\begin{align}
\|R_{Q,z^{\flat}}\|_{L^{2}} & \aleq b^{\frac{3}{2}}\cdot|t|^{\delta},\label{eq:RQ,z-L2-est}\\
(R_{Q,z^{\flat}},\Lmb Q)_{r} & =8\sqrt{8}\pi\lmb^{3}\Re(e^{-i\gmm}pq(4it)^{\frac{\nu-2}{2}})+O(b^{2}\cdot|t|^{\delta}),\label{eq:RQ,z-inner-prod-1}\\
(R_{Q,z^{\flat}},iQ)_{r} & =-8\sqrt{8}\pi\lmb^{3}\Im(e^{-i\gmm}pq(4it)^{\frac{\nu-2}{2}})+O(b^{2}\cdot|t|^{\delta}).\label{eq:RQ,z-inner-prod-2}
\end{align}
\end{lem}

\begin{proof}
Note that \eqref{eq:RQ,z-inner-prod-1} and \eqref{eq:RQ,z-inner-prod-2}
are immediate consequences of \eqref{eq:RQ,z-1}, \eqref{eq:RQ,z-2},
and Remark~\ref{rem:convenient-rel}. It remains to show \eqref{eq:RQ,z-L2-est}.
Recall the expansion \eqref{eq:RQ,z-expansion} of $R_{Q,z^{\flat}}$:
\[
\begin{aligned}R_{Q,z^{\flat}} & =-\tfrac{2+A_{\tht}[Q]}{y}z_{1}^{\flat_{-1}}+(\tint y{\infty}\Re(Qz_{1}^{\flat_{-1}})dy')Q\\
 & \quad-\tfrac{2A_{\tht}[Q,z^{\flat}]}{y}z_{1}^{\flat_{-1}}+(\tint y{\infty}\Re(Qz_{1}^{\flat_{-1}})dy')z^{\flat}-(\tint 0y\Re(\br{z^{\flat}}z_{1}^{\flat_{-1}})dy')Q\\
 & \quad+L_{Q+z^{\flat}}^{\ast}(\bfD_{Q+z^{\flat}}(Q+z^{\flat})-z_{1}^{\flat_{-1}}).
\end{aligned}
\tag{\ref{eq:RQ,z-expansion}}
\]
The $L^{2}$-norm of the first line of RHS\eqref{eq:RQ,z-expansion}
is estimated by 
\begin{align*}
 & \aleq\|-\tfrac{2+A_{\tht}[Q]}{y}z_{1}^{\flat_{-1}}+(\tint y{\infty}\Re(Qz_{1}^{\flat_{-1}})dy')Q\|_{L^{2}}\\
 & \quad\aleq\|y^{-1}\langle y\rangle^{-2}|z^{\flat}|_{-1}\|_{L^{2}}+\|y^{-1}\langle y\rangle^{-2}|z^{\flat}|_{-1}\|_{L^{1}}\aleq b^{3/2}\cdot|t|^{\delta}.
\end{align*}
where in the last inequality we used \eqref{eq:z-flat-H3-bound}.
For the second line of \eqref{eq:RQ,z-expansion}, we use \eqref{eq:z-flat-H1-bound-1}
and \eqref{eq:z-flat-H2-bound} to have 
\[
\|-\tfrac{2A_{\tht}[Q,z^{\flat}]}{y}z_{1}^{\flat_{-1}}\|_{L^{2}}\aleq\|Qz^{\flat}\|_{L^{2}}\|z_{1}^{\flat_{-1}}\|_{L^{2}}\aleq b|t|^{\delta}\cdot\lmb\aleq b^{3/2}\cdot|t|^{\delta},
\]
and 
\[
\|(\tint y{\infty}\Re(Qz_{1}^{\flat_{-1}})dy')z^{\flat}\|_{L^{2}}\aleq\|Qz_{1}^{\flat_{-1}}\|_{L^{1}}\|\tfrac{1}{y}z^{\flat}\|_{L^{2}}\aleq b|t|^{\delta}\cdot\lmb\aleq b^{3/2}\cdot|t|^{\delta},
\]
and 
\begin{align*}
\|(\tint 0y\Re(\br{z^{\flat}}z_{1}^{\flat_{-1}})dy')Q\|_{L^{2}} & \aleq\|\tfrac{1}{y}\br{z^{\flat}}z_{1}^{\flat_{-1}}\|_{L^{2}}\aleq\|\tfrac{1}{y}z^{\flat}\|_{L^{\infty}}\||z|_{-1}\|_{L^{2}}\aleq b^{3/2}\cdot|t|^{\delta}.
\end{align*}
The last term of \eqref{eq:RQ,z-expansion} can be estimated using
\eqref{eq:RQ,z-claim6} in the proof of Lemma~\ref{lem:inn-prod-interaction}:
\begin{align*}
 & \|L_{Q+z^{\flat}}^{\ast}(\bfD_{Q+z^{\flat}}(Q+z^{\flat})-z_{1}^{\flat_{-1}})\|_{L^{2}}\\
 & \quad\aleq\||\bfD_{Q+z^{\flat}}(Q+z^{\flat})-z_{1}^{\flat_{-1}}|_{-1}\|_{L^{2}}\aleq(\lmb^{3-}+\lmb^{3}|t|^{\frac{\Re(\nu)-2}{2}})\aleq b^{3/2}\cdot|t|^{\delta}.
\end{align*}
This completes the proof.
\end{proof}

\subsection{\label{subsec:Modulation-estimates}Modulation estimates}

In this subsection, we control the modulation parameters $\lmb$ and
$\gmm$. For this purpose, we study their evolution laws. Our $\lmb$
and $\gmm$ are fixed via the orthogonality conditions \eqref{eq:eps-orthog},
from which the governing equations of $\lmb$ and $\gmm$ are derived.
However, as our choice of orthogonality conditions is quite arbitrary,
the equations for $\frac{\lmb_{s}}{\lmb}$ and $\gmm_{s}$ are too
rough to be analyzed.

In view of the invariant subspace decomposition \eqref{eq:invariant-subspace-decomp},
the inner products of $\eps$ with $y^{2}Q$, $i\rho$, $i\Lmb Q$,
and $Q$ (after suitable truncations) will play key roles. The dynamics
of the former two inner products, combined with differentiations by
parts, represent the dynamics of a new complex-valued parameter $\bm{\zeta}$
satisfying $\bm{\zeta}\approx\bm{\lmb}=\lmb e^{i\gmm}$. The latter
two inner products describe the velocity of $\bm{\zeta}$ and their
evolutions are affected by the soliton-radiation interaction ($R_{Q,z^{\flat}}$
affects the acceleration for $\zeta$). Moreover, we will make further
corrections on the inner products $(\eps,i\Lmb Q)_{r}$ and $(\eps,Q)_{r}$
in the spirit of \cite{RaphaelSzeftel2011JAMS,JendrejLawrieRodriguez2019arXiv}.
These corrections will reflect the nonlinear virial identity and mass
conservation of \eqref{eq:CSS-m-equiv}; their technical advantages
will be discussed on the way.

We begin with the following rough control of $\lmb$ and $\gmm$,
obtained by differentiating the orthogonality conditions.
\begin{lem}[Rough control of $\lmb$ and $\gmm$]
We have 
\begin{equation}
\Big|\frac{\lmb_{s}}{\lmb}\Big|+|\gmm_{s}|\aleq\|\eps\|_{\dot{\calH}_{0}^{1}}+b^{\frac{3}{2}}\cdot|t|^{\delta}\aleq b|\log b|^{\frac{1}{2}}.\label{eq:rough-control-lmb-gmm}
\end{equation}
\end{lem}

\begin{proof}
The proof follows from differentiating the orthogonality conditions
\eqref{eq:eps-orthog}. Indeed, using the $\epsilon$-equation \eqref{eq:eps-eq-s,y},
we have 
\begin{align*}
0=\rd_{s}(\epsilon,\calZ_{k})_{r} & =\tfrac{\lmb_{s}}{\lmb}\{(\Lmb Q,\calZ_{k})_{r}+(\Lmb\epsilon,\calZ_{k})_{r}\}-\gmm_{s}\{(iQ,\calZ_{k})_{r}+(i\epsilon,\calZ_{k})_{r}\}\\
 & \quad+(\calL_{Q+z^{\flat}}\epsilon+R_{Q+z^{\flat}}(\epsilon)-\tht_{z^{\flat}}\epsilon,i\calZ_{k})_{r}+(R_{Q,z^{\flat}},i\calZ_{k})_{r}-(\Psi_{z^{\flat}},i\calZ_{k})_{r}.
\end{align*}
Using the antisymmetry of $\Lmb$, $i$, and $\calL_{Q}=L_{Q}^{\ast}L_{Q}$,
we rewrite the above as 
\begin{align}
 & -\tfrac{\lmb_{s}}{\lmb}\{(\Lmb Q,\calZ_{k})_{r}-(\epsilon,\Lmb\calZ_{k})_{r}\}+\gmm_{s}\{(iQ,\calZ_{k})_{r}-(\epsilon,i\calZ_{k})_{r}\}\nonumber \\
 & =(L_{Q}\epsilon,L_{Q}i\calZ_{k})_{r}-\tht_{z^{\flat}}(\epsilon,i\calZ_{k})_{r}+((\calL_{Q+z^{\flat}}-\calL_{Q})\epsilon,i\calZ_{k})_{r}\label{eq:rough-control-tmp1}\\
 & \quad+(R_{Q+z^{\flat}}(\epsilon),i\calZ_{k})_{r}+(R_{Q,z^{\flat}},i\calZ_{k})_{r}-(\Psi_{z^{\flat}},i\calZ_{k})_{r}.\nonumber 
\end{align}
By \eqref{eq:Z1Z2-transversality}, we have 
\begin{align*}
\Big|\frac{\lmb_{s}}{\lmb}\Big|+|\gmm_{s}| & \aleq\sup_{k\in\{1,2\}}|{\rm RHS}\eqref{eq:rough-control-tmp1}|\\
 & \aleq\|L_{Q}\eps\|_{L^{2}}+|\tht_{z^{\flat}}|\|\eps\|_{\dot{\calH}_{0}^{1}}+\|(\calL_{Q+z^{\flat}}-\calL_{Q})\eps+R_{Q+z^{\flat}}(\eps)+R_{Q,z^{\flat}}-\Psi_{z^{\flat}}\|_{L^{2}}.
\end{align*}
Applying \eqref{eq:tht-z-flat-bound} for $\tht_{z^{\flat}}$, \eqref{eq:conv-rel-3},
\eqref{eq:RQ,z-L2-est} for $R_{Q,z^{\flat}}$, and \eqref{eq:Psi_zflat-L2}
for $\Psi_{z^{\flat}}$, the proof of \eqref{eq:rough-control-lmb-gmm}
is completed once we show that 
\begin{align}
\|(\calL_{Q+z^{\flat}}-\calL_{Q})\epsilon\|_{L^{2}} & \aleq\lmb\|\epsilon\|_{\dot{\calH}_{0}^{1}},\label{eq:rough-control-tmp2}\\
\|R_{Q+z^{\flat}}(\epsilon)\|_{L^{2}} & \aleq\|\epsilon\|_{\dot{\calH}_{0}^{1}}^{2}.\label{eq:rough-control-tmp3}
\end{align}

\emph{Proof of \eqref{eq:rough-control-tmp2}.} We note that $(\calL_{Q+z^{\flat}}-\calL_{Q})\epsilon$
is a linear combination of $\calN_{\ast}(\psi_{1},\dots,\psi_{\ast})$
such that $\psi_{j}\in\{Q,z^{\flat},\eps\}$ and $\psi_{j}$' contain
exactly one $\eps$ and at least one $z^{\flat}$. Applying the duality
estimates (Corollary~\ref{cor:nonlinear-estimates-weightedL2}),
\eqref{eq:z-flat-H1-bound-1}, and \eqref{eq:m=00003D0-HardySobolev},
we obtain 
\begin{align*}
\|(\calL_{Q+z^{\flat}}\eps-\calL_{Q})\epsilon\|_{L^{2}} & \aleq\|z^{\flat}\eps\|_{L^{\infty}}+\|\tfrac{\langle\log_{-}y\rangle}{y}z^{\flat}\|_{L^{2}}\|\tfrac{1}{y\langle\log_{-}y\rangle}\eps\|_{L^{2}}\\
 & \aleq\|\langle\log_{-}y\rangle z^{\flat}\|_{L^{\infty}\cap yL^{2}}\|\tfrac{1}{\langle\log_{-}y\rangle}\eps\|_{L^{\infty}\cap yL^{2}}\aleq\lmb\|\eps\|_{\dot{\calH}_{0}^{1}}.
\end{align*}
This completes the proof of \eqref{eq:rough-control-tmp2}.

\emph{Proof of \eqref{eq:rough-control-tmp3}.} We note that $R_{Q+z^{\flat}}(\epsilon)$
is a linear combination of $\calN_{\ast}(\psi_{1},\dots,\psi_{\ast})$
such that $\psi_{j}\in\{Q+z^{\flat},\eps\}$ and $\psi_{j}$'s contain
at least two $\epsilon$'s. If there is at least one $Q+z^{\flat}$
in $\psi_{j}$'s, such a contribution is estimated by (using Corollary~\ref{cor:nonlinear-estimates-weightedL2},
\eqref{eq:z-flat-H1-bound-1}, and \eqref{eq:m=00003D0-HardySobolev})
\begin{align*}
 & \aleq\|(Q+z^{\flat})\eps^{2}\|_{L^{2}}+\|\langle\log_{-}y\rangle^{2}(Q+z^{\flat})\|_{L^{2}}\|\tfrac{1}{y\langle\log_{-}y\rangle}\eps\|_{L^{2}}^{2}\\
 & \aleq\|\langle\log_{-}y\rangle^{2}(Q+z^{\flat})\|_{L^{2}}\|\tfrac{1}{\langle\log_{-}y\rangle}\eps\|_{L^{\infty}\cap yL^{2}}^{2}\aleq\|\eps\|_{\dot{\calH}_{0}^{1}}^{2}.
\end{align*}
If all $\psi_{j}$'s are $\eps$, then this contribution is simply
$\calN(\eps)$ and is estimated (using Corollary~\eqref{cor:nonlinear-estimates-holder})
\[
\|\calN(\eps)\|_{L^{2}}\aleq\|\eps\|_{L^{6}}^{3}(1+\|\eps\|_{L^{2}}^{2})\aleq\|\eps\|_{\dot{H}_{0}^{1}}^{2}\|\eps\|_{L^{2}}\aleq\|\eps\|_{\dot{\calH}_{0}^{1}}^{2}.
\]
This completes the proof of \eqref{eq:rough-control-tmp3}.
\end{proof}
Next, we study the evolution of the inner products 
\[
(\eps,\tfrac{y^{2}}{4}Q\chi_{B_{0}})_{r}\quad\text{and}\quad(\eps,i\rho\chi_{B_{0}})_{r}.
\]
Computing their time derivatives and differentiating by parts will
yield good control over the dynamics of the complex-valued quantity
\begin{equation}
\bm{\zeta}\coloneqq\bm{\lmb}\Big\{1+\frac{(\eps,\tfrac{y^{2}}{4}Q\chi_{B_{0}})_{r}+i(\eps,i\rho\chi_{B_{0}})_{r}}{4\pi\log B_{0}}\Big\},\label{eq:def-zeta}
\end{equation}
where $\bm{\lmb}\coloneqq\lmb e^{i\gmm}$. Indeed, we have the following.
\begin{lem}[Control of $\bm{\zeta}$]
\label{lem:control-zeta}We have 
\begin{align}
\Big|\frac{\bm{\zeta}}{\bm{\lmb}}-1\Big| & \aleq\frac{1}{|\log b|^{\frac{1}{2}}},\label{eq:control-zeta-bound}\\
\Big|\frac{\rd_{s}\bm{\zeta}}{\bm{\lmb}}+\frac{(\eps,i\Lmb_{B_{0}}Q)_{r}+i(\eps,Q\chi_{B_{0}})_{r}}{4\pi\log B_{0}}\Big| & \aleq\frac{b}{|\log b|^{\frac{1}{2}}}.\label{eq:control-zeta-diff-eq}
\end{align}
\end{lem}

\begin{proof}
\textbf{Step 1.} Bound of $\bm{\zeta}$.

We prove \eqref{eq:control-zeta-bound}. By \eqref{eq:def-zeta} and
\eqref{eq:conv-rel-3}, we have 
\[
\frac{\bm{\zeta}}{\bm{\lmb}}-1=\frac{(\eps,\tfrac{y^{2}}{4}Q\chi_{B_{0}})_{r}+i(\eps,i\rho\chi_{B_{0}})_{r}}{4\pi\log B_{0}}\aleq\frac{\|\chf_{(0,2B_{0}]}\eps\|_{L^{1}}}{|\log b|}\aleq\frac{(B_{0})^{2}}{|\log b|}\|\eps\|_{\dot{\calH}_{0}^{1}}\aleq|\log b|^{\frac{1}{2}}.
\]
This completes the proof of \eqref{eq:control-zeta-bound}.

\smallskip
\textbf{Step 2.} Computation of $\rd_{s}(\eps,\psi\chi_{B_{0}})_{r}$
and treatment of easy error terms.

In Steps 2-4, we prove \eqref{eq:control-zeta-diff-eq}. In this
step, we compute $\rd_{s}(\eps,\psi\chi_{B_{0}})_{r}$ for a fixed
time-independent profile $\psi$ and estimate some easy error terms.
Later, we will substitute $\psi\in\{\tfrac{y^{2}}{4}Q,i\rho\}$, whose
choice is motivated by the formal invariant subspace decomposition
for $i\calL_{Q}$; see Subsection \ref{subsec:Invariant-subspace-decomposition}.

We note that $\frac{(B_{0})_{s}}{B_{0}}=-\frac{\lmb_{s}}{\lmb}-\frac{\lmb^{2}}{2|t|}$
has the oscillatory term $-\frac{\lmb_{s}}{\lmb}$, whose bound \eqref{eq:rough-control-lmb-gmm}
is logarithmically worse than $b$. This logarithmic loss might be
dangerous because the term $(\eps,\psi\rd_{s}\chi_{B_{0}})_{r}$ which
possibly appears in the computation of $\rd_{s}(\eps,\psi\chi_{B_{0}})_{r}$
cannot be treated as an error. However, such a term is an artifact of
the computation in the $(s,y)$-variables, and we can get around this
problem by computing instead in the original $(t,r)$-variables: 
\begin{align*}
\rd_{s}(\eps,\psi\chi_{B_{0}})_{r} & =\lmb^{2}\rd_{t}(\eps^{\sharp},\psi^{\sharp}\chi_{|t|^{1/2}})_{r}\\
 & =\lmb^{2}\{(\rd_{t}\eps^{\sharp},\psi^{\sharp}\chi_{|t|^{1/2}})_{r}+(\eps^{\sharp},(\rd_{t}\psi^{\sharp})\chi_{|t|^{1/2}})_{r}\\
 & \qquad+\tfrac{1}{2|t|}(\eps^{\sharp},\psi^{\sharp}(r\rd_{r}\chi_{|t|^{1/2}}))_{r}\}.
\end{align*}
We substitute the $\eps^{\sharp}$-equation \eqref{eq:eps-sharp-eq}
into the above to continue it as 
\begin{align*}
 & =\lmb^{2}\big\{(-\rd_{t}Q^{\sharp},\psi^{\sharp}\chi_{|t|^{1/2}})_{r}+(\eps^{\sharp},(\rd_{t}\psi^{\sharp})\chi_{|t|^{1/2}})_{r}\\
 & \qquad\quad+(\calL_{Q^{\sharp}+z}\eps^{\sharp}+R_{Q^{\sharp}+z}(\eps^{\sharp})-\tht_{z}\epsilon^{\sharp}+R_{Q^{\sharp},z}-\Psi_{z},i\psi^{\sharp}\chi_{|t|^{1/2}})_{r}\big\}\\
 & \quad+\tfrac{1}{2(B_{0})^{2}}(\eps^{\sharp},\psi^{\sharp}(r\rd_{r}\chi_{|t|^{1/2}}))_{r}.
\end{align*}
Going back to the $(s,y)$-variables, we have arrived at 
\begin{equation}
\begin{aligned} & \rd_{s}(\eps,\psi\chi_{B_{0}})_{r}\\
 & =\tfrac{\lmb_{s}}{\lmb}\{(\Lmb Q,\psi\chi_{B_{0}})_{r}-(\eps,(\Lmb\psi)\chi_{B_{0}})_{r}\}-\gmm_{s}\{(iQ,\psi\chi_{B_{0}})_{r}-(\eps,i\psi\chi_{B_{0}})_{r}\}\\
 & \quad+(\calL_{Q+z^{\flat}}\eps+R_{Q+z^{\flat}}(\eps)-\tht_{z^{\flat}}\eps+R_{Q,z^{\flat}}-\Psi_{z^{\flat}},i\psi\chi_{B_{0}})_{r}+\tfrac{1}{2(B_{0})^{2}}(\eps,\psi(y\rd_{y}\chi_{B_{0}}))_{r}.
\end{aligned}
\label{eq:control-zeta-step2-1}
\end{equation}

Next, we assume 
\[
|\psi|\aleq1,
\]
which is the case for $\psi\in\{\tfrac{y^{2}}{4}Q,i\rho\}$, and estimate
some easy error terms in RHS\eqref{eq:control-zeta-step2-1}. First,
by \eqref{eq:rough-control-tmp2}, \eqref{eq:rough-control-tmp3},
\eqref{eq:RQ,z-L2-est}, and \eqref{eq:Psi_zflat-L2}, we have 
\begin{align*}
|((\calL_{Q+z^{\flat}}-\calL_{Q})\eps,i\psi\chi_{B_{0}})_{r}| & \aleq\lmb\|\eps\|_{\dot{\calH}_{0}^{1}}\cdot B_{0}\aleq|t|^{\frac{1}{2}}b|\log b|^{\frac{1}{2}}\aleq b\cdot|t|^{\frac{1}{2}-},\\
|(R_{Q+z^{\flat}}(\eps),i\psi\chi_{B_{0}})_{r}| & \aleq\|\eps\|_{\dot{\calH}_{0}^{1}}^{2}\cdot B_{0}\aleq b^{\frac{3}{2}-},\\
|(R_{Q,z^{\flat}}-\Psi_{z^{\flat}},i\psi\chi_{B_{0}})_{r}| & \aleq(\|R_{Q,z^{\flat}}\|_{L^{2}}+\|\Psi_{z^{\flat}}\|_{L^{2}})\cdot B_{0}\aleq b\cdot|t|^{\delta}.
\end{align*}
Next, by \eqref{eq:tht-z-flat-bound}, we have 
\[
|\tht_{z^{\flat}}(\eps,i\psi\chi_{B_{0}})_{r}|\aleq\lmb^{2}\|\eps\|_{\dot{\calH}_{0}^{1}}(B_{0})^{2}\aleq b|\log b|^{\frac{1}{2}}|t|\aleq b\cdot|t|^{1-}.
\]
Finally, 
\[
|\tfrac{1}{2(B_{0})^{2}}(\eps,\psi(y\rd_{y}\chi_{B_{0}}))_{r}|\aleq\|\eps\|_{\dot{\calH}_{0}^{1}}\aleq b|\log b|^{\frac{1}{2}}.
\]
Substituting the above error estimates into \eqref{eq:control-zeta-step2-1}
and using $\calL_{Q}=L_{Q}^{\ast}L_{Q}$, we have shown that 
\begin{align}
\rd_{s}(\eps, & \psi\chi_{B_{0}})_{r}-\tfrac{\lmb_{s}}{\lmb}\{(\Lmb Q,\psi\chi_{B_{0}})_{r}-(\eps,(\Lmb\psi)\chi_{B_{0}})_{r}\}\label{eq:control-zeta-step2-2}\\
+\gmm_{s} & \{(iQ,\psi\chi_{B_{0}})_{r}-(\eps,i\psi\chi_{B_{0}})_{r}\}=(L_{Q}\eps,L_{Q}i\psi\chi_{B_{0}})_{r}+O(b|\log b|^{\frac{1}{2}}).\nonumber 
\end{align}
We remark that $\frac{\lmb_{s}}{\lmb}(\eps,(\Lmb\psi)\chi_{B_{0}})_{r}$
and $\gmm_{s}(\eps,i\psi\chi_{B_{0}})_{r}$ cannot be viewed as error
terms because the control \eqref{eq:rough-control-lmb-gmm} is rough.
However, we will see in Step 4 that these terms can be integrated
into a correction term when $\psi\in\{\tfrac{y^{2}}{4}Q,i\rho\}$
motivating the definition of $\bm{\zeta}$.

\smallskip
\textbf{Step 3.} Substitution of $\psi\in\{\tfrac{y^{2}}{4}Q,i\rho\}$.

We substitute $\psi\in\{\tfrac{y^{2}}{4}Q,i\rho\}$ into \eqref{eq:control-zeta-step2-2}.
We use \eqref{eq:transversality-radial-case} and \eqref{eq:rough-control-lmb-gmm}
for LHS\eqref{eq:control-zeta-step2-2}, and we use \eqref{eq:truncated-est-1}
and \eqref{eq:truncated-est-2} for RHS\eqref{eq:control-zeta-step2-2}
to have the following two displays: 
\begin{align}
\rd_{s}(\eps,\tfrac{y^{2}}{4}Q\chi_{B_{0}})_{r} & +\tfrac{\lmb_{s}}{\lmb}\{4\pi\log B_{0}+(\eps,\Lmb(\tfrac{y^{2}}{4}Q)\chi_{B_{0}})_{r}\}\label{eq:control-zeta-step3-1}\\
 & -\gmm_{s}(\eps,i\tfrac{y^{2}}{4}Q\chi_{B_{0}})_{r}=-(\eps,i\Lmb_{B_{0}}Q)_{r}+O(b|\log b|^{\frac{1}{2}})\nonumber 
\end{align}
and 
\begin{align}
\rd_{s}(\eps,-i\rho\chi_{B_{0}})_{r} & +\tfrac{\lmb_{s}}{\lmb}(\eps,i(\Lmb\rho)\chi_{B_{0}})_{r}\label{eq:control-zeta-step3-2}\\
 & +\gmm_{s}\{4\pi\log B_{0}+(\eps,\rho\chi_{B_{0}})_{r}\}=-(\eps,\chi_{B_{0}}Q)_{r}+O(b|\log b|^{\frac{1}{2}}).\nonumber 
\end{align}

\smallskip
\textbf{Step 4.} Differentiation by parts and conclusion.

We can merge the displays \eqref{eq:control-zeta-step3-1} and \eqref{eq:control-zeta-step3-2}
into 
\begin{equation}
\begin{aligned} & (\rd_{s}+\tfrac{\lmb_{s}}{\lmb}+\gmm_{s}i)4\pi\log B_{0}+\rd_{s}\{(\eps,\tfrac{y^{2}}{4}Q\chi_{B_{0}})_{r}+i(\eps,i\rho\chi_{B_{0}})_{r}\}\\
 & +\tfrac{\lmb_{s}}{\lmb}\{(\eps,\Lmb(\tfrac{y^{2}}{4}Q)\chi_{B_{0}})_{r}+i(\eps,i(\Lmb\rho)\chi_{B_{0}})_{r}\}\\
 & +\gmm_{s}\{-(\eps,i\tfrac{y^{2}}{4}Q\chi_{B_{0}})_{r}+i(\eps,\rho\chi_{B_{0}})_{r}\}\\
 & =-\{(\eps,i\Lmb_{B_{0}}Q)_{r}+i(\eps,\chi_{B_{0}}Q)_{r}\}+O(b|\log b|^{\frac{1}{2}}),
\end{aligned}
\label{eq:control-zeta-step3-3}
\end{equation}
where we used 
\begin{equation}
|\rd_{s}(4\pi\log B_{0})|\aleq|\tfrac{\lmb_{s}}{\lmb}|+\tfrac{\lmb^{2}}{|t|}\aleq b|\log b|^{\frac{1}{2}}.\label{eq:B0-deriv-est}
\end{equation}

We mention again that the last two terms of LHS\eqref{eq:control-zeta-step3-3}
cannot be treated as errors. However, a crucial observation is that
these terms can be integrated into a correction term under the flow
of $\rd_{s}+\frac{\lmb_{s}}{\lmb}+\gmm_{s}i=\bm{\lmb}^{-1}\rd_{s}(\bm{\lmb}\cdot)$.
More precisely, we use the following \emph{matched spatial asymptotics}
(see \eqref{eq:matched-spatial-asymp})
\[
|\Lmb(\tfrac{y^{2}}{4}Q)-\tfrac{y^{2}}{4}Q|+|\Lmb\rho-\rho|+|\tfrac{y^{2}}{4}Q-\rho|\aleq\langle y\rangle^{-2}\langle\log_{+}y\rangle^{2}
\]
to have 
\begin{align*}
 & \tfrac{\lmb_{s}}{\lmb}\{(\eps,\Lmb(\tfrac{y^{2}}{4}Q)\chi_{B_{0}})_{r}+i(\eps,i(\Lmb\rho)\chi_{B_{0}})_{r}\}\\
 & =\tfrac{\lmb_{s}}{\lmb}\{(\eps,\tfrac{y^{2}}{4}Q\chi_{B_{0}})_{r}+i(\eps,i\rho\chi_{B_{0}})_{r}\}+O(|\tfrac{\lmb_{s}}{\lmb}|\cdot\|\eps\cdot\chf_{(0,2B_{0}]}\langle y\rangle^{-2}\langle\log_{+}y\rangle^{2}\|_{L^{1}})
\end{align*}
and 
\begin{align*}
 & \gmm_{s}\{-(\eps,i\tfrac{y^{2}}{4}Q\chi_{B_{0}})_{r}+i(\eps,\rho\chi_{B_{0}})_{r}\}\\
 & =\gmm_{s}i\{(\eps,\tfrac{y^{2}}{4}Q\chi_{B_{0}})_{r}+i(\eps,i\rho\chi_{B_{0}})_{r}\}+O(|\gmm_{s}|\cdot\|\eps\cdot\chf_{(0,2B_{0}]}\langle y\rangle^{-2}\langle\log_{+}y\rangle^{2}\|_{L^{1}}).
\end{align*}
We also note that 
\[
(|\tfrac{\lmb_{s}}{\lmb}|+|\gmm_{s}|)\cdot\|\eps\cdot\chf_{(0,2B_{0}]}\langle y\rangle^{-2}\langle\log_{+}y\rangle^{2}\|_{L^{1}}\aleq b|\log b|^{\frac{1}{2}}\cdot\|\eps\|_{\dot{\calH}_{0}^{1}}|\log B_{0}|^{\frac{3}{2}}\aleq b^{2-}.
\]
Therefore, 
\begin{align*}
\text{LHS}\eqref{eq:control-zeta-step3-3} & =(\rd_{s}+\tfrac{\lmb_{s}}{\lmb}+\gmm_{s}i)\{4\pi\log B_{0}+(\eps,\tfrac{y^{2}}{4}Q\chi_{B_{0}})_{r}+i(\eps,i\rho\chi_{B_{0}})_{r}\}+O(b^{2-})\\
 & =\bm{\lmb}^{-1}\rd_{s}(4\pi\log B_{0}\cdot\bm{\zeta})+O(b^{2-}).
\end{align*}
Substituting this into \eqref{eq:control-zeta-step3-3} gives 
\[
\bm{\lmb}^{-1}\rd_{s}(4\pi\log B_{0}\cdot\bm{\zeta})=-\{(\eps,i\Lmb_{B_{0}}Q)_{r}+i(\eps,\chi_{B_{0}}Q)_{r}\}+O(b|\log b|^{\frac{1}{2}}).
\]
Dividing both sides by $4\pi\log B_{0}$ and using \eqref{eq:B0-deriv-est}
again, we obtain \eqref{eq:control-zeta-diff-eq}. This completes
the proof.
\end{proof}
From the previous lemma, we have seen that the dynamics of the complex-valued
scaling parameter $\bm{\zeta}$ is governed by the inner products
$(\eps,i\Lmb_{B_{0}}Q)_{r}$ and $(\eps,Q\chi_{B_{0}})_{r}$. Our
next goal is to track the dynamics of these inner products. As in
the definition of $\bm{\zeta}$, we need proper modifications of these
inner products. More precisely, we will track the following quantities:
\begin{align}
b & \coloneqq\frac{(\eps,i\Lmb_{B}^{\mathrm{(avg)}}Q)_{r}+\tfrac{1}{2}(\eps,i\Lmb_{B}^{\mathrm{(avg)}}\eps)_{r}}{4\pi\log B_{0}},\tag{\ref{eq:def-b}}\nonumber \\
\eta & \coloneqq\frac{(\eps,\chi_{B}^{\mathrm{(avg)}}Q)_{r}+\tfrac{1}{2}(\eps,\chi_{B}^{\mathrm{(avg)}}\eps)_{r}}{4\pi\log B_{0}}.\label{eq:def-eta}
\end{align}
Note that we have already seen the definition of $b$ in \eqref{eq:def-b}.
In view of \eqref{eq:control-zeta-diff-eq}, it is natural to consider
also the complex-valued scalar 
\begin{equation}
\bm{b}\coloneqq b+i\eta.\label{eq:def-bm-b}
\end{equation}

Recall that, in Section~\ref{subsec:Formal-derivation-of-blow-up},
we tested the $\rd_{s}\eps$-equation against $i\Lmb(Q+\eps)$ and
$(Q+\eps)$ to derive the formal parameter equations \eqref{eq:formal-b-without-trunc}-\eqref{eq:formal-eta-without-trunc}.
This motivates the definitions of \eqref{eq:def-b} and \eqref{eq:def-eta},
thanks to the formal identities 
\begin{align*}
(\rd_{s}\eps,i\Lmb(Q+\eps))_{r} & =\rd_{s}\{(\eps,i\Lmb Q)_{r}+\tfrac{1}{2}(\eps,i\Lmb\eps)_{r}\},\\
(\rd_{s}\eps,(Q+\eps))_{r} & =\rd_{s}\{(\eps,Q)_{r}+\tfrac{1}{2}(\eps,\eps)_{r}\},
\end{align*}
and \eqref{eq:transversality-radial-case}. We remark that $b$ and
$\eta$ can be viewed as truncated nonlinear virial and mass functionals
at $Q+\eps$. As we have already seen in Section~\ref{subsec:Formal-derivation-of-blow-up},
we can expect ``nice'' structures in the computations of $b_{s}$
and $\eta_{s}$ (i.e., the appearance of $-\|L_{Q}\eps\|_{L^{2}}^{2}$
in $b_{s}$ and the absence of quadratic terms in $\eta_{s}$) in
the spirit of the nonlinear virial identity and mass conservation.

In addition, we averaged over the cutoff parameter $B$. As in \cite{KimKwon2019arXiv,KimKwon2020arXiv},
this is introduced to deal with errors of (possibly having) critical
size but \emph{localized} in annuli. We note that the admissible range
of $B$ seems to be very limited; see Remark~\ref{rem:mod-est-cutoff-radius}.
\begin{lem}[Control of $\bm{b}$]
\label{lem:Control-of-b-and-eta}We have 
\begin{align}
\Big|\bm{b}-\frac{(\eps,i\Lmb_{B_{0}}Q)_{r}+i(\eps,Q\chi_{B_{0}})_{r}}{4\pi\log B_{0}}\Big| & \aleq\frac{b}{|\log b|^{\frac{1}{5}}},\label{eq:b-bound}\\
\Big|\rd_{s}\bm{b}+|\bm{b}|^{2}+\lmb^{2}\frac{\br{\bm{\lmb}}\cdot8\sqrt{8}\pi pq(4it)^{\frac{\nu-2}{2}}}{4\pi\log B_{0}}+P\Big| & \aleq o_{t\to0}(b^{2})\label{eq:b-diff-equality}
\end{align}
for some almost positive $P=P(t)$ satisfying 
\begin{equation}
-o_{t\to0}(b^{2})\leq P\leq\Big(\frac{\|L_{Q}\eps\|_{L^{2}}^{2}}{4\pi\log B_{0}}-|\bm{b}|^{2}\Big)+o_{t\to0}(b^{2}).\label{eq:P-positivity}
\end{equation}
\end{lem}

\begin{rem}
Roughly speaking, $P$ is an averaged version of $\|(L_{Q}\eps)^{\perp}\|_{L^{2}}^{2}/4\pi\log B_{0}$,
where $\perp$ is the orthogonal projection onto the orthogonal complement
of $\chi_{B}yQ$ and $\chi_{B}iyQ$. Thus, the almost positivity of
$P$ is natural.

The upper bound \eqref{eq:P-positivity} for $P$ might have a critical
size $b^{2}$ at first glance. However, our expectation $\eps\approx b\cdot(-i\tfrac{y^{2}}{4}Q\chi_{B_{0}})+\eta\cdot\rho\chi_{B_{0}}$
says that we may hope to have an improved bound $P\ll b^{2}$, which
holds at the initial time $t=\tau$. In the past times, this indeed
happens and will be key to close our bootstrap. Later in the energy
estimates (Lemma~\ref{lem:H1-energy-est}), we will introduce a monotonicity
that propagates the initial smallness of $P$ backward in time, justifying
$P\ll b^{2}$.
\end{rem}

\begin{rem}
\label{rem:mod-est-cutoff-radius}It seems that the range of admissible
cutoff radii $B$ is very limited; $B$ must be an almost self-similar
scale. We have chosen the radii $B_{-\delta}$ for $\delta\in[\frac{1}{10},\frac{1}{5}]$
in this work, but inspecting its proof, it suffices to use $\delta\in[\delta_{1},\delta_{2}]\subset(0,\frac{1}{4})$.
The restriction $\delta>0$ is made to ensure that the size of the
correction term of $\bm{b}$ is less than the size of its main term.
The restriction $\delta<\frac{1}{4}$ comes from dealing with cutoff
errors when computing time derivatives of $\bm{b}$; see, for example
, \eqref{eq:control-b-eta-step1-5} and \eqref{eq:control-b-eta-step4-3}.
\end{rem}

\begin{proof}
We have already shown the real part of \eqref{eq:b-bound}, i.e., the
estimate for $b$, in \eqref{eq:proximity-b-bqnu}. The bound for
$\eta$ can be proved in a similar manner, so we omit its proof. Henceforth,
we focus on the proof of \eqref{eq:b-diff-equality}. We first prove
analogous statements for the \emph{unaveraged} versions (and multiplied
by $4\pi\log B_{0}$) of $\bm{b}$, i.e., for the quantities 
\begin{equation}
\left\{ \begin{aligned}\td b & \coloneqq(\eps,i\Lmb_{B}Q)_{r}+\tfrac{1}{2}(\eps,i\Lmb_{B}\eps)_{r},\\
\td{\eta} & \coloneqq(\eps,\chi_{B}Q)_{r}+\tfrac{1}{2}(\eps,\chi_{B}\eps)_{r},\\
\td{\bm{b}} & \coloneqq\td b+i\td{\eta}.
\end{aligned}
\right.\label{control-b-eta-step0-1}
\end{equation}
where $B=B_{0}c$ with $c\in[\frac{1}{|\log|t||^{1/5}},\frac{1}{|\log|t||^{1/10}}]$
 \emph{fixed} (time-independent). After that, we finish the proof
using an averaging argument.

\smallskip
\textbf{Step 1.} Computation of $\rd_{s}\td b$ and $\rd_{s}\td{\eta}$
and estimates for easy error terms.

In this step, we show that 
\begin{equation}
\begin{aligned}\rd_{s}\td b=-(\calL_{Q+z^{\flat}}\epsilon+R_{Q+z^{\flat}}(\epsilon),\Lmb_{B}(Q+\epsilon))_{r} & -(R_{Q,z^{\flat}},\Lmb_{B}(Q+\epsilon))_{r}\\
 & +O(|\log|t||^{-\frac{1}{10}}\cdot b^{2}|\log b|),
\end{aligned}
\label{eq:control-b-eta-step1-1}
\end{equation}
and 
\begin{equation}
\begin{aligned}\rd_{s}\td{\eta}=(\calL_{Q+z^{\flat}}\epsilon+R_{Q+z^{\flat}}(\epsilon),\chi_{B}(iQ+i\epsilon))_{r} & +(R_{Q,z^{\flat}},\chi_{B}(iQ+i\epsilon))_{r}\\
 & +O(|\log|t||^{-\frac{1}{10}}\cdot b^{2}|\log b|),
\end{aligned}
\label{eq:control-b-eta-step1-2}
\end{equation}
In fact, we will only show \eqref{eq:control-b-eta-step1-1} because
the proof of \eqref{eq:control-b-eta-step1-2} is essentially the
same.

As in the proof of Lemma~\ref{lem:control-zeta}, we begin by computing
in the original $(t,r)$-variables to avoid the oscillatory term $-\frac{\lmb_{s}}{\lmb}$
in $\frac{B_{s}}{B}$: 
\begin{align*}
 & \rd_{s}\{(\eps,i\Lmb_{B}Q)_{r}+\tfrac{1}{2}(\eps,i\Lmb_{B}\eps)_{r}\}\\
 & =\lmb^{2}\rd_{t}\{(\eps^{\sharp},i\Lmb_{\lmb B}Q^{\sharp})_{r}+\tfrac{1}{2}(\eps^{\sharp},i\Lmb_{\lmb B}\eps^{\sharp})_{r}\}\\
 & =\lmb^{2}\{(\rd_{t}\eps^{\sharp},i\Lmb_{\lmb B}(Q^{\sharp}+\eps^{\sharp}))_{r}+(\eps^{\sharp},i\Lmb_{\lmb B}\rd_{t}Q^{\sharp})_{r}\}\\
 & \peq+\lmb^{2}\{(\eps^{\sharp},i(\rd_{t}\Lmb_{\lmb B})Q^{\sharp})_{r}+\tfrac{1}{2}(\eps^{\sharp},i(\rd_{t}\Lmb_{\lmb B})\eps^{\sharp})_{r}\}.
\end{align*}
Notice that we avoided time derivatives falling directly on $B$.
Applying \eqref{eq:eps-sharp-eq}, the above display continues as
\begin{align*}
 & =-\lmb^{2}\{(\calL_{Q^{\sharp}+z}\eps^{\sharp}+R_{Q^{\sharp}+z}(\eps^{\sharp})-\tht_{z}\eps^{\sharp}+R_{Q^{\sharp},z}-\Psi_{z},\Lmb_{\lmb B}(Q^{\sharp}+\eps^{\sharp}))_{r}\}\\
 & \peq+\lmb^{2}\{-(\rd_{t}Q^{\sharp},i\Lmb_{\lmb B}(Q^{\sharp}+\eps^{\sharp}))_{r}+(\eps^{\sharp},i\Lmb_{\lmb B}\rd_{t}Q^{\sharp})_{r}\}\\
 & \peq+\lmb^{2}\{(\eps^{\sharp},i(\rd_{t}\Lmb_{\lmb B})Q^{\sharp})_{r}+\tfrac{1}{2}(\eps^{\sharp},i(\rd_{t}\Lmb_{\lmb B})\eps^{\sharp})_{r}\}.
\end{align*}
We rescale the first term and exploit the antisymmetricity of $\Lmb_{\lmb B}$
in the above display. As a result, we have arrived at 
\begin{equation}
\begin{aligned} & \rd_{s}\{(\eps,i\Lmb_{B}Q)_{r}+\tfrac{1}{2}(\eps,i\Lmb_{B}\eps)_{r}\}\\
 & =-(\calL_{Q+z^{\flat}}\epsilon+R_{Q+z^{\flat}}(\epsilon),\Lmb_{B}(Q+\epsilon))_{r}-(R_{Q,z^{\flat}},\Lmb_{B}(Q+\epsilon))_{r}\\
 & \peq+(\Psi_{z^{\flat}},\Lmb_{B}(Q+\eps))_{r}\\
 & \peq+\lmb^{2}\{-(\rd_{t}Q^{\sharp},i\Lmb_{\lmb B}Q^{\sharp})_{r}+\tht_{z}(\eps^{\sharp},\Lmb_{\lmb B}Q^{\sharp})_{r}\}\\
 & \peq+\lmb^{2}\{(\eps^{\sharp},i(\rd_{t}\Lmb_{\lmb B})Q^{\sharp})_{r}+\tfrac{1}{2}(\eps^{\sharp},i(\rd_{t}\Lmb_{\lmb B})\eps^{\sharp})_{r}\}.
\end{aligned}
\label{eq:control-b-eta-step1-3}
\end{equation}

It suffices to show that the second to the fourth lines of RHS\eqref{eq:control-b-eta-step1-3}
are errors. First, we apply \eqref{eq:Psi_zflat-Hdot1} to have 
\begin{align*}
(\Psi_{z^{\flat}},\Lmb_{B}(Q+\eps))_{r} & \aleq\|\tfrac{1}{y}\Psi_{z^{\flat}}\|_{L^{2}}\|y\Lmb_{B}(Q+\eps)\|_{L^{2}}\\
 & \aleq b^{2}|t|^{\delta}\cdot(|\log B|+B^{2}\|\eps\|_{\dot{\calH}_{0}^{1}})\aleq b^{2}|t|^{\delta-}.
\end{align*}
Next, we have 
\begin{equation}
\lmb^{2}(\rd_{t}Q^{\sharp},i\Lmb_{\lmb B}Q^{\sharp})_{r}=0.\label{eq:control-b-eta-step1-4}
\end{equation}
Next, we have 
\[
\lmb^{2}\tht_{z}(\eps^{\sharp},\Lmb_{\lmb B}Q^{\sharp})_{r}\aleq\lmb^{2}\|\eps\cdot\Lmb_{B}Q\|_{L^{1}}\aleq\lmb^{2}\|\eps\|_{\dot{\calH}_{0}^{1}}(\log B)^{\frac{1}{2}}\aleq b^{2}\cdot|t|^{1-}.
\]
Next, we use $|(A\rd_{A}\Lmb_{A})f|\aleq\chf_{[A,2A]}|f|_{1}$ and
$\frac{(\lmb B)_{t}}{\lmb B}=\frac{(|t|^{1/2}c)_{t}}{|t|^{1/2}c}=-\frac{1}{2|t|}$
to have 
\begin{align*}
 & \lmb^{2}|(\eps^{\sharp},i(\rd_{t}\Lmb_{\lmb B})Q^{\sharp})_{r}+\tfrac{1}{2}(\eps^{\sharp},i(\rd_{t}\Lmb_{\lmb B})\eps^{\sharp})_{r}|\\
 & \aleq|\tfrac{(\lmb B)_{s}}{\lmb B}|\cdot\|\chf_{[\lmb B,2\lmb B]}|\eps^{\sharp}|\cdot(|Q^{\sharp}|_{1}+|\eps^{\sharp}|_{1})\|_{L^{1}}\\
 & \aleq\tfrac{\lmb^{2}}{|t|}\cdot\|\chf_{[B,2B]}|\eps|\cdot(|Q|_{1}+|\eps|_{1})\|_{L^{1}}\\
 & \aleq b\cdot\|\eps\|_{\dot{\calH}_{0}^{1}}(1+B^{2}\|\eps\|_{\dot{\calH}_{0}^{1}})\aleq|\log|t||^{-\frac{1}{5}}\cdot b^{2}|\log b|.
\end{align*}
This completes the proof of \eqref{eq:control-b-eta-step1-1}.

As mentioned earlier, the proof of \eqref{eq:control-b-eta-step1-2}
is very similar, but this time \eqref{eq:control-b-eta-step1-4} should
be replaced by 
\begin{align}
|\lmb^{2}(\rd_{t}Q^{\sharp},\chi_{\lmb B}Q^{\sharp})_{r}| & =|\gmm_{s}|\cdot|(\Lmb Q,\chi_{B}Q)_{r}|\label{eq:control-b-eta-step1-5}\\
 & =|\gmm_{s}|\cdot|(Q,\tfrac{1}{2}[\chi_{B},\Lmb]Q)_{r}|\nonumber \\
 & \aleq b|\log b|^{\frac{1}{2}}\cdot B^{-2}\aleq|\log|t||^{-\frac{1}{10}}\cdot b^{2}|\log b|.\nonumber 
\end{align}
We leave the details to the readers. 

\smallskip
\textbf{Step 2.} Contribution of the interaction term $R_{Q,z^{\flat}}$.

In this step, we show that 
\begin{align}
-(R_{Q,z^{\flat}},\Lmb_{B}(Q+\epsilon))_{r} & =-8\sqrt{8}\pi\Re(\lmb^{3}e^{-i\gmm}pq(4it)^{\frac{\nu-2}{2}})+O(b^{2}|t|^{\delta-}),\label{eq:control-b-eta-step2-1}\\
(R_{Q,z^{\flat}},\chi_{B}(iQ+i\eps))_{r} & =-8\sqrt{8}\pi\Im(\lmb^{3}e^{-i\gmm}pq(4it)^{\frac{\nu-2}{2}})+O(b^{2}|t|^{\delta-}).\label{eq:control-b-eta-step2-2}
\end{align}
For the proof of \eqref{eq:control-b-eta-step2-1}, we use Lemma~\ref{lem:RQ,z}
to observe that 
\begin{align*}
 & -(R_{Q,z^{\flat}},\Lmb_{B}(Q+\epsilon))_{r}\\
 & =-(R_{Q,z^{\flat}},\Lmb Q)_{r}+O(\|R_{Q,z^{\flat}}\|_{L^{2}}\|(\Lmb Q-\Lmb_{B}Q)+\Lmb_{B}\eps\|_{L^{2}})\\
 & =-(R_{Q,z^{\flat}},\Lmb Q)_{r}+O(b^{\frac{3}{2}}|t|^{\delta}\cdot(B^{-1}+B\|\eps\|_{\dot{\calH}_{0}^{1}}))\\
 & =-8\sqrt{8}\pi\Re(\lmb^{3}e^{-i\gmm}\cdot pq(4it)^{\frac{\nu-2}{2}})+O(b^{2}|t|^{\delta-}).
\end{align*}
For the proof of \eqref{eq:control-b-eta-step2-2}, one proceeds similarly
and uses \eqref{eq:RQ,z-inner-prod-2} instead. We omit the proof.

\smallskip
\textbf{Step 3.} Nonlinear structure of $\calL_{Q+z^{\flat}}\epsilon+R_{Q+z^{\flat}}(\epsilon)$.

In this step, we claim for $\varphi\in\{\Lmb_{B}(Q+\eps),\chi_{B}(iQ+i\eps)\}$
that 
\begin{equation}
(\calL_{Q+z^{\flat}}\epsilon+R_{Q+z^{\flat}}(\epsilon),\varphi)_{r}=(\calL_{Q}\epsilon+R_{Q}(\epsilon),\varphi)_{r}+O(b^{2}|t|^{\delta-}).\label{eq:control-b-eta-step3-1}
\end{equation}
Let us begin by dealing with easy error terms. At first, we show that
we can replace $R_{Q+z^{\flat}}(\eps)$ with $R_{Q}(\eps)$. Indeed,
since $R_{Q+z^{\flat}}(\eps)-R_{Q}(\eps)$ is a linear combination
of $\calN_{\ast}(\psi_{1},\dots,\psi_{\ast})$ such that $\#\{j:\psi_{j}=\eps\}\geq2$
and $\#\{j:\psi_{j}=z^{\flat}\}\geq1$, we apply the duality estimates
(Lemma~\ref{lem:duality-estimates-weighted-L1}) to obtain 
\begin{align*}
 & |(R_{Q+z^{\flat}}(\eps)-R_{Q}(\eps),\varphi)_{r}|\\
 & \aleq\|\langle\log_{-}y\rangle z^{\flat}\|_{yL^{2}\cap L^{\infty}}\|\tfrac{1}{\langle\log_{-}y\rangle}\eps\|_{yL^{2}\cap L^{\infty}}\|\eps\|_{L^{2}}\|\varphi\|_{L^{2}}\\
 & \aleq\lmb\|\eps\|_{\dot{\calH}_{0}^{1}}\|\eps\|_{L^{2}}\cdot(\|\Lmb_{B}(Q+\epsilon)\|_{L^{2}}+\|\chi_{B}(iQ+i\epsilon)\|_{L^{2}})\\
 & \aleq\lmb\|\eps\|_{\dot{\calH}_{0}^{1}}\|\eps\|_{L^{2}}\cdot(1+B\|\eps\|_{\dot{\calH}_{0}^{1}})\aleq b^{2}|t|^{\frac{1}{2}-}.
\end{align*}
Next, we show that we can replace $\calL_{Q+z^{\flat}}\eps$ by $\calL_{Q}\eps$
when we take inner products with $\Lmb_{B}\eps$ or $\chi_{B}i\eps$.
Indeed, we use the duality estimates again to obtain 
\begin{align*}
 & |(\calL_{Q+z^{\flat}}\eps-\calL_{Q}\eps,\Lmb_{B}\eps)_{r}|+|(\calL_{Q+z^{\flat}}\eps-\calL_{Q}\eps,\chi_{B}i\eps)_{r}|\\
 & \aleq\|\langle\log_{-}y\rangle z^{\flat}\|_{yL^{2}\cap L^{\infty}}\|_{yL^{2}\cap L^{\infty}}\|\tfrac{1}{\langle\log_{-}y\rangle}\eps\|_{yL^{2}\cap L^{\infty}}(\|\Lmb_{B}\eps\|_{L^{2}}+\|\chi_{B}i\eps\|_{L^{2}})\\
 & \aleq\lmb\|\eps\|_{\dot{\calH}_{0}^{1}}\cdot B\|\eps\|_{\dot{\calH}_{0}^{1}}\aleq b^{2}|t|^{\frac{1}{2}-}.
\end{align*}

By the previous two displays, it suffices to show that 
\begin{equation}
|((\calL_{Q+z^{\flat}}-\calL_{Q})\eps,\td{\varphi})_{r}|\aleq b^{2}|t|^{\delta-},\qquad\td{\varphi}\in\{\Lmb_{B}Q,\chi_{B}iQ\}.\label{eq:control-b-eta-step3-2}
\end{equation}
These inner products require more careful analysis. We use the duality
formulation that $((\calL_{Q+z^{\flat}}-\calL_{Q})\eps,\td{\varphi})_{r}$
is a linear combination of $\calM_{\ast}(\psi_{1},\dots,\psi_{\ast})$,
where $\#\{j:\psi_{j}=\eps\}=1=\#\{j:\psi_{j}=\td{\varphi}\}$ and
$\#\{j:\psi_{j}=z^{\flat}\}\geq1$. We note that there is no $\calM_{4,1}$-term
since $m=0$. It suffices to consider the following three cases: for
some $j\in\{1,3,5\}$, (Case A) $\Re(\br{\psi_{j}}\psi_{j+1})=\Re(\br{\td{\varphi}}\eps)$,
(Case B) $\Re(\br{\psi_{j}}\psi_{j+1})=\Re(\br{\td{\varphi}}z^{\flat})$,
(Case C) $\Re(\br{\psi_{j}}\psi_{j+1})=\Re(\br{\td{\varphi}}Q)$.

\uline{Case A}. We use the duality estimates ($L^{\infty}$-$L^{1}$
Hölder for $\calM_{4,0}$ and Lemma~\ref{lem:duality-estimates-weighted-L1}
with putting the weight $\frac{1}{y^{2}}$ to some $\psi_{j}=z^{\flat}$
for $\calM_{6}$) and \eqref{eq:z-flat-H2-bound} to see that this
contribution can be bounded by 
\begin{align*}
 & \aleq(\||z^{\flat}|^{2}\|_{y^{2}L^{1}\cap L^{\infty}}+\|Qz^{\flat}\|_{y^{2}L^{1}\cap L^{\infty}})\|\td{\varphi}\eps\|_{L^{1}}\\
 & \aleq(\|z^{\flat}\|_{yL^{2}\cap L^{\infty}}^{2}+\|\tfrac{1}{y}z^{\flat}\|_{L^{\infty}})\|\td{\varphi}\eps\|_{L^{1}}\aleq b|t|^{\delta}\cdot\|\eps\|_{\dot{\calH}_{0}^{1}}|\log B_{0}|^{\frac{1}{2}}\aleq b^{2}|t|^{\delta-}.
\end{align*}

\uline{Case B}. We use the duality estimates (Lemma~\ref{lem:duality-estimates-Holder})
and \eqref{eq:z-flat-H2-bound} to see that this contribution is bounded
by 
\begin{align*}
 & \aleq(\|Q\eps\|_{L^{2}}+\|z^{\flat}\eps\|_{L^{2}})(\|(\Lmb_{B'}Q)z^{\flat}\|_{L^{2}}+\|(\chi_{B'}iQ)z^{\flat}\|_{L^{2}})\\
 & \aleq\|\eps\|_{\dot{\calH}_{0}^{1}}\|Qz^{\flat}\|_{L^{2}}\aleq b^{2}|t|^{\delta-}.
\end{align*}

\uline{Case C}. We note the following pointwise bound:
\begin{equation}
|\Re(\br{\td{\varphi}}Q)|+\tfrac{1}{y^{2}}|\tint 0y\Re(\br{\td{\varphi}}Q)y'dy'|\aleq\langle y\rangle^{-4},\label{eq:control-b-eta-step3-3}
\end{equation}
where the integral bound comes from integration by parts when $\td{\varphi}=\Lmb_{B}Q$.
Therefore, we have (there is no $\calM_{4,1}$-term since $m=0$)
\begin{align*}
|\calM_{4,0}(\eps,z^{\flat},\td{\varphi},Q)| & \aleq\|\tfrac{1}{\langle y\rangle^{3}}z^{\flat}\|_{L^{2}}\|\tfrac{1}{\langle y\rangle}\eps\|_{L^{2}},\\
|\calM_{6}(\td{\varphi},Q,\psi_{3},\psi_{4},\psi_{5},\psi_{6})| & \aleq\|\tfrac{1}{\langle y\rangle^{3}}z^{\flat}\|_{L^{2}}\|\tfrac{1}{\langle y\rangle}\eps\|_{L^{2}},\\
|\calM_{6}(\psi_{1},\psi_{2},\psi_{3},\psi_{4},\td{\varphi},Q)| & \aleq\|\tfrac{1}{y^{2}}z^{\flat}\|_{L^{2}}\|\tfrac{1}{\langle y\rangle}\eps\|_{L^{2}},
\end{align*}
where in the second row we used \eqref{eq:control-b-eta-step3-3}
and distributed the weight $\frac{1}{\langle y\rangle^{4}}$ in the
spirit of Lemma~\ref{lem:duality-estimates-weighted-L1}, and in
the third row we simply put $\langle y\rangle\td{\varphi}Q\in L^{1}$
and distributed the remaining weight $\frac{1}{y^{2}\langle y\rangle}$
into the integrals $\int_{0}^{y}\cdot y'dy'$. Further, using \eqref{eq:z-flat-H2-bound},
we see that the contribution of Case C is bounded by 
\[
\aleq b|t|^{\delta}\cdot\|\eps\|_{\dot{\calH}_{0}^{1}}\aleq b^{2}|t|^{\delta-}.
\]
This completes the proof of \eqref{eq:control-b-eta-step3-2} and
hence the claim \eqref{eq:control-b-eta-step3-1}.

\smallskip
\textbf{Step 4.} Truncated nonlinear virial/mass identities.

In this step, we claim that 
\begin{equation}
\begin{aligned}(\calL_{Q}\epsilon+R_{Q}(\epsilon),\Lmb_{B}(Q+\epsilon))_{r} & =\int\chi_{B}|L_{Q}\eps|^{2}\\
 & \peq+O\Big(b^{2}|\log b|^{\frac{4}{5}}+\int_{B}^{2B}|\eps|_{-1}^{2}ydy\Big)
\end{aligned}
\label{eq:control-b-eta-step4-1}
\end{equation}
and 
\begin{equation}
(\calL_{Q}\epsilon+R_{Q}(\epsilon),\chi_{B}(iQ+i\epsilon))_{r}=O\Big(b^{2}|\log b|^{\frac{4}{5}}+\int_{B}^{2B}|\eps|_{-1}^{2}ydy\Big).\label{eq:control-b-eta-step4-2}
\end{equation}
If there were no truncations for $\Lmb$ and $i$, then a formal application
of the scaling and phase rotation symmetries yields the identities
\begin{align*}
L_{Q+\eps}(\Lmb Q+\Lmb\epsilon) & =-\tfrac{d}{d\lmb}|_{\lmb=1}\bfD_{(Q+\eps)_{\lmb}}(Q+\eps)_{\lmb}=\Lmb_{-1}\bfD_{Q+\eps}(Q+\eps),\\
L_{Q+\eps}(iQ+i\epsilon) & =\tfrac{d}{d\tht}|_{\tht=0}\bfD_{e^{i\tht}(Q+\eps)}e^{i\tht}(Q+\eps)=i\bfD_{Q+\eps}(Q+\eps),
\end{align*}
which, in turn, gives 
\begin{align*}
(\bfD_{Q+\eps}(Q+\eps),L_{Q+\eps}(\Lmb Q+\Lmb\epsilon))_{r} & =2E[Q+\eps]\approx\|L_{Q}\eps\|_{L^{2}}^{2},\\
(\bfD_{Q+\eps}(Q+\eps),L_{Q+\eps}(iQ+i\epsilon))_{r} & =0.
\end{align*}
In this step we will derive truncated versions of the above.

For this purpose, we use the self-dual expression
\[
\calL_{Q}\eps+R_{Q}(\eps)=\nabla E[Q+\eps]=L_{Q+\eps}^{\ast}\bfD_{Q+\eps}(Q+\eps)
\]
to write 
\begin{align*}
(\calL_{Q}\epsilon+R_{Q}(\epsilon),\Lmb_{B}(Q+\epsilon))_{r} & =(\bfD_{Q+\eps}(Q+\eps),L_{Q+\eps}\Lmb_{B}(Q+\epsilon))_{r},\\
(\calL_{Q}\epsilon+R_{Q}(\epsilon),\chi_{B}(iQ+i\epsilon))_{r} & =(\bfD_{Q+\eps}(Q+\eps),L_{Q+\eps}\chi_{B}(iQ+i\epsilon))_{r}.
\end{align*}
With the help of the identities 
\begin{align*}
\bfD_{w}\Lmb_{A} & =(\Lmb_{A}+\chi_{A})\bfD_{w}+(y\chi_{A}'\rd_{y}+(\chi_{A}+\tfrac{1}{2}y\chi_{A}')')-\tfrac{1}{2}y\chi_{A}|w|^{2},\\
wB_{w}\Lmb_{A}w & =\tfrac{w}{2}y\chi_{A}|w|^{2},
\end{align*}
we have 
\begin{align*}
L_{Q+\eps}\Lmb_{B}(Q+\epsilon) & =(\Lmb_{B}+\chi_{B})\bfD_{Q+\eps}(Q+\eps)+(y\chi_{B}'\rd_{y}+(\chi_{B}+\tfrac{1}{2}y\chi_{B}')')(Q+\eps)\\
 & =(\Lmb_{B}+\chi_{B})\bfD_{Q+\eps}(Q+\eps)+O(\chf_{[B,2B]}|Q+\eps|_{-1}),\\
L_{Q+\eps}\chi_{B}(iQ+i\eps) & =i\chi_{B}\bfD_{Q+\eps}(Q+\eps)+\chi_{B}'(iQ+i\eps)\\
 & =i\chi_{B}\bfD_{Q+\eps}(Q+\eps)+O(\chf_{[B,2B]}|Q+\eps|_{-1}).
\end{align*}
We then take the inner product with $\bfD_{Q+\eps}(Q+\eps)$ and use
the antisymmetricity of $\Lmb_{B}$ and $\chi_{B}i$ to have 
\begin{align}
 & (\bfD_{Q+\eps}(Q+\eps),L_{Q+\eps}\Lmb_{B}(Q+\eps))_{r}\label{eq:control-b-eta-step4-3}\\
 & ={\textstyle \int}\chi_{B}|\bfD_{Q+\eps}(Q+\eps)|^{2}+O(\tint B{2B}|Q+\eps|_{-1}^{2}ydy)\nonumber \\
 & ={\textstyle \int}\chi_{B}|\bfD_{Q+\eps}(Q+\eps)|^{2}+O(B^{-4}+\tint B{2B}|\eps|_{-1}^{2}ydy)\nonumber 
\end{align}
and similarly 
\[
(\bfD_{Q+\eps}(Q+\eps),L_{Q+\eps}\chi_{B}(iQ+i\eps))_{r}=O(B^{-4}+\tint B{2B}|\eps|_{-1}^{2}ydy).
\]
The error term localized in the region $B\leq y\leq2B$ requires an
averaging argument and will be treated in Step 6. Finally, in view
of \eqref{lem:duality-estimates-Holder}, we have 
\[
\int\chi_{B}|\bfD_{Q+\eps}(Q+\eps)|^{2}=\int\chi_{B}|L_{Q}\eps|^{2}+O(\|\eps\|_{L^{2}}\|\eps\|_{\dot{\calH}_{0}^{1}}^{2})=\int\chi_{B}|L_{Q}\eps|^{2}+O(b^{\frac{5}{2}-}).
\]
This completes the proof of \eqref{eq:control-b-eta-step4-1} and
\eqref{eq:control-b-eta-step4-2}.

\smallskip
\textbf{Step 5.} Orthogonal decomposition of $L_{Q}\eps$.

In this step, we show that the main term of \eqref{eq:control-b-eta-step4-1}
can be written as 
\begin{equation}
\int\chi_{B}|L_{Q}\eps|^{2}=\frac{|\td{\bm{b}}|^{2}}{4\pi\log B_{0}}+\td P,\label{eq:control-b-eta-step5-1}
\end{equation}
where $\td P$ satisfies 
\begin{equation}
-O(b^{2}|\log b|^{\frac{4}{5}})\leq\td P\leq\Big\{\|L_{Q}\eps\|_{L^{2}}^{2}-\frac{|\td{\bm{b}}|^{2}}{4\pi\log B_{0}}\Big\}+O(b^{2}|\log b|^{\frac{4}{5}}).\label{eq:control-b-eta-step5-2}
\end{equation}
To see this, we perform an orthogonal decomposition of $L_{Q}\eps$
as 
\[
L_{Q}\eps=\frac{(L_{Q}\eps,\chi_{B/2}\tfrac{y}{2}Q)_{r}\cdot\chi_{B/2}\tfrac{y}{2}Q+(L_{Q}\eps,\chi_{B/2}i\tfrac{y}{2}Q)_{r}\cdot\chi_{B/2}i\tfrac{y}{2}Q}{\|\chi_{B/2}\tfrac{y}{2}Q\|_{L^{2}}^{2}}+(L_{Q}\eps)^{\perp}.
\]
Note that 
\begin{align*}
\|\chi_{B/2}\tfrac{y}{2}Q\|_{L^{2}}^{2} & =4\pi\log B+O(1)=4\pi\log B_{0}+O(\log|\log|t||),\\
(L_{Q}\eps,\chi_{B/2}i\tfrac{y}{2}Q)_{r} & =-\td b+O(\|\eps\|_{\dot{\calH}_{0}^{1}}+B^{2}\|\eps\|_{\dot{\calH}_{0}^{1}}^{2})=-\td b+O(b|\log b|^{\frac{4}{5}}),\\
(L_{Q}\eps,\chi_{B/2}\tfrac{y}{2}Q)_{r} & =\td{\eta}+O(\|\eps\|_{\dot{\calH}_{0}^{1}}+B^{2}\|\eps\|_{\dot{\calH}_{0}^{1}}^{2})=\td{\eta}+O(b|\log b|^{\frac{4}{5}}),
\end{align*}
where the second and third rows can be easily proved by combining
\eqref{eq:truncated-est-2} and the proof of \eqref{eq:proximity-b-bqnu}.
Thus, we have arrived at 
\[
\int\chi_{B}|L_{Q}\eps|^{2}=\frac{|\td{\bm{b}}|^{2}}{4\pi\log B_{0}}+\int\chi_{B}|(L_{Q}\eps)^{\perp}|^{2}+O(b^{2}|\log b|^{\frac{4}{5}}).
\]
Letting $\td P$ be the sum of the last two terms of the above, \eqref{eq:control-b-eta-step5-1}
and \eqref{eq:control-b-eta-step5-2} follow.

\smallskip
\textbf{Step 6.} Conclusion via averaging argument.

So far, we have shown that 
\begin{equation}
\Big|\td{\bm{b}}_{s}+\frac{|\td{\bm{b}}|^{2}}{4\pi\log B_{0}}+\lmb^{2}\br{\bm{\lmb}}\cdot8\sqrt{8}\pi pq(4it)^{\frac{\nu-2}{2}}+\td P\Big|\aleq\frac{b^{2}|\log b|}{|\log|t||^{\frac{1}{10}}}+\int_{B}^{2B}|\eps|_{-1}^{2}ydy,\label{eq:mod-est-step6-1}
\end{equation}
where $\td{\bm{b}}$ is defined by \eqref{control-b-eta-step0-1}
and $\td P$ was defined in Step 5.

In this step, we finish the proof of \eqref{eq:b-diff-equality} and
\eqref{eq:P-positivity} by averaging over $B$ and dividing by $4\pi\log B_{0}$.
For the sake of notational convenience, we denote $c_{1}(t)=|\log|t||^{-\frac{1}{5}}$
and $c_{2}(t)=|\log|t||^{-\frac{1}{10}}$. By changing variables,
we note that 
\begin{align*}
\bm{b} & =\frac{10}{\log|\log|t||}\int_{B_{-1/5}}^{B_{-1/10}}\frac{\td{\bm{b}}}{4\pi\log B_{0}}\frac{dB}{B}=\frac{10}{\log|\log|t||}\int_{c_{1}(t)}^{c_{2}(t)}\frac{\td{\bm{b}}|_{B=B_{0}c}}{4\pi\log B_{0}}\frac{dc}{c}.
\end{align*}
Thus, we have 
\begin{align}
\rd_{s}\bm{b} & =\frac{10}{\log|\log|t||}\int_{c_{1}(t)}^{c_{2}(t)}\frac{\rd_{s}\{\td{\bm{b}}|_{B=B_{0}c}\}}{4\pi\log B_{0}}\frac{dc}{c}-\Big(\frac{\rd_{s}(\log|\log|t||)}{\log|\log|t||}+\frac{\rd_{s}\log B_{0}}{\log B_{0}}\Big)\bm{b}\label{eq:control-b-eta-step6-2}\\
 & \quad+\frac{10}{\log|\log|t||}\cdot\frac{1}{4\pi\log B_{0}}\Big(\frac{\rd_{s}c_{2}(t)}{c_{2}(t)}\td{\bm{b}}|_{B=B_{0}c_{2}(t)}-\frac{\rd_{s}c_{1}(t)}{c_{1}(t)}\td{\bm{b}}|_{B=B_{0}c_{1}(t)}\Big).\nonumber 
\end{align}

Next, we show that all the terms except the first term of RHS\eqref{eq:control-b-eta-step6-2}
are errors: 
\begin{equation}
\rd_{s}\bm{b}=\frac{10}{\log|\log|t||}\int_{c_{1}(t)}^{c_{2}(t)}\frac{\rd_{s}\{\td{\bm{b}}|_{B=B_{0}c}\}}{4\pi\log B_{0}}\frac{dc}{c}+O\Big(\frac{b^{2}}{|\log b|^{\frac{1}{2}}}\Big).\label{eq:control-b-eta-step6-3}
\end{equation}
Indeed, from 
\begin{align*}
\Big|\frac{\rd_{s}(\log|\log|t||)}{\log|\log|t||}\Big| & =\Big|\frac{\rd_{s}|\log|t||}{|\log|t||\cdot\log|\log|t||}\Big|\aleq\frac{\lmb^{2}}{|t|}\frac{1}{|\log|t||}\sim\frac{b}{|\log b|},\\
\Big|\frac{\rd_{s}\log B_{0}}{\log B_{0}}\Big| & \aleq\frac{1}{|\log B_{0}|}\Big|\frac{(B_{0})_{s}}{B_{0}}\Big|\aleq\frac{b}{|\log b|^{\frac{1}{2}}},
\end{align*}
we have 
\[
\Big|\Big(\frac{\rd_{s}(\log|\log|t||)}{\log|\log|t||}+\frac{\rd_{s}\log B_{0}}{\log B_{0}}\Big)\bm{b}\Big|\aleq\frac{b^{2}}{|\log b|^{\frac{1}{2}}}.
\]
Next, from 
\begin{align*}
\Big|\frac{\rd_{s}c_{1}}{c_{1}}\Big|+\Big|\frac{\rd_{s}c_{2}}{c_{2}}\Big| & \aleq\frac{\lmb^{2}}{|t|}\frac{1}{|\log|t||}\sim\frac{b}{|\log b|},\\
|\td{\bm{b}}| & \aleq b|\log b|,
\end{align*}
the last line of RHS\eqref{eq:control-b-eta-step6-2} is bounded by
\[
\aleq\frac{1}{\log|\log|t||}\cdot\frac{b}{|\log b|^{2}}|\td{\bm{b}}|\aleq\frac{b^{2}}{|\log b|}.
\]
This completes the proof of \eqref{eq:control-b-eta-step6-3}.

To finish the proof of \eqref{eq:b-diff-equality} and \eqref{eq:P-positivity},
we substitute \eqref{eq:mod-est-step6-1} into \eqref{eq:control-b-eta-step6-3}
to obtain 
\begin{align*}
 & \Big|\rd_{s}\bm{b}+|\bm{b}|^{2}+\lmb^{2}\frac{\br{\bm{\lmb}}\cdot8\sqrt{8}\pi pq(4it)^{\frac{\nu-2}{2}}}{4\pi\log B_{0}}+P\Big|\\
 & \quad\aleq\frac{b^{2}}{|\log|t||^{\frac{1}{10}}}+\frac{1}{\log|\log|t||}\frac{1}{|\log b|}\int_{c_{1}(t)}^{c_{2}(t)}\Big(\int_{B}^{2B}|\eps|_{-1}^{2}ydy\Big)\frac{dc}{c},
\end{align*}
where 
\[
P\coloneqq\frac{10}{\log|\log|t||}\int_{c_{1}(t)}^{c_{2}(t)}\frac{\td P}{4\pi\log B_{0}}\frac{dc}{c}.
\]
The proof of \eqref{eq:P-positivity} now follows from \eqref{eq:control-b-eta-step5-2}.
The proof of \eqref{eq:b-diff-equality} follows from 
\begin{align*}
\frac{1}{\log|\log|t||}\frac{1}{|\log b|}\int_{c_{1}(t)}^{c_{2}(t)}\int_{B}^{2B}|\eps|_{-1}^{2}ydy\frac{dc}{c} & \aleq\frac{1}{\log|\log|t||}\frac{1}{|\log b|}\int_{B_{0}c_{1}(t)}^{2B_{0}c_{2}(t)}|\eps|_{-1}^{2}ydy\\
 & \aleq\frac{1}{\log|\log|t||}\frac{\|\eps\|_{\dot{\calH}_{0}^{1}}^{2}}{|\log b|}\aleq\frac{b^{2}}{\log|\log|t||},
\end{align*}
where we used Fubini in the first inequality. This completes the proof.
\end{proof}
From Lemma~\ref{lem:control-zeta} and Lemma~\ref{lem:Control-of-b-and-eta},
we obtained the modulation equations \eqref{eq:control-zeta-diff-eq}
and \eqref{eq:b-diff-equality} of the dynamical parameters $\bm{\zeta}$
and $\bm{b}$. In order to actually integrate them, we rewrite them
in terms of $\bm{\zeta}$ and $\bm{b}/\br{\bm{\zeta}}$, and in the
original time variable $t$.
\begin{cor}[Modulation estimates]
\label{cor:mod-est}We have 
\begin{align}
|\rd_{t}\bm{\zeta}+(\bm{b}/\br{\bm{\zeta}})| & \aleq\frac{1}{|\log|t||^{\frac{1}{2}}}\cdot\frac{b}{\lmb},\label{eq:mod-est-rd-zeta}\\
\Big|\rd_{t}(\bm{b}/\br{\bm{\zeta}})+\frac{8\sqrt{8}\pi pq(4it)^{\frac{\nu-2}{2}}}{4\pi\log B_{0}}+\frac{P}{\lmb^{2}\br{\bm{\zeta}}}\Big| & \aleq o_{t\to0}(1)\cdot\frac{b^{2}}{\lmb^{3}}.\label{eq:mod-est-rd-b/zeta}
\end{align}
\end{cor}

\begin{proof}
The proof of \eqref{eq:mod-est-rd-zeta} follows from writing 
\begin{align*}
\rd_{t}\bm{\zeta}+(\bm{b}/\br{\bm{\zeta}}) & =\frac{1}{\br{\bm{\lmb}}}\Big\{\Big(\frac{\rd_{s}\bm{\zeta}}{\bm{\lmb}}+\frac{(\eps,i\Lmb_{B_{0}}Q)_{r}+i(\eps,Q\chi_{B_{0}})_{r}}{4\pi\log B_{0}}\Big)\\
 & \qquad\qquad+\Big(\bm{b}-\frac{(\eps,i\Lmb_{B_{0}}Q)_{r}+i(\eps,Q\chi_{B_{0}})_{r}}{4\pi\log B_{0}}\Big)+\Big(\frac{\br{\bm{\lmb}}}{\br{\bm{\zeta}}}-1\Big)\bm{b}\Big\}
\end{align*}
and applying \eqref{eq:control-zeta-diff-eq}, \eqref{eq:control-zeta-bound},
and \eqref{eq:b-bound}.

For the proof of \eqref{eq:mod-est-rd-b/zeta}, we use \eqref{eq:control-zeta-bound}
and \eqref{eq:control-zeta-diff-eq} to have 
\begin{align*}
\rd_{t}(\bm{b}/\br{\bm{\zeta}}) & =\frac{1}{\br{\bm{\zeta}}}\Big\{\Big(\rd_{t}\bm{b}+\frac{|\bm{b}|^{2}}{\lmb^{2}}\Big)-\Big(\rd_{t}\br{\bm{\zeta}}+\frac{\br{\bm{b}}}{\bm{\zeta}}\Big)\frac{\bm{b}}{\br{\bm{\zeta}}}-\Big(\Big|\frac{\bm{\zeta}}{\bm{\lmb}}\Big|^{2}-1\Big)\Big|\frac{\bm{b}}{\br{\bm{\zeta}}}\Big|^{2}\Big\}\\
 & =\frac{1}{\br{\bm{\zeta}}}\Big\{\Big(\rd_{t}\bm{b}+\frac{|\bm{b}|^{2}}{\lmb^{2}}\Big)+O\Big(\frac{1}{|\log|t||^{\frac{1}{2}}}\cdot\frac{b^{2}}{\lmb^{2}}\Big)\Big\}.
\end{align*}
Further using \eqref{eq:b-diff-equality}, we have 
\begin{align*}
 & \Big|\rd_{t}(\bm{b}/\br{\bm{\zeta}})+\frac{8\sqrt{8}\pi pq(4it)^{\frac{\nu-2}{2}}}{4\pi\log B_{0}}+\frac{P}{\lmb^{2}\br{\bm{\zeta}}}\Big|\\
 & \aleq\frac{1}{\lmb}\Big|\rd_{t}\bm{b}+\frac{|\bm{b}|^{2}}{\lmb^{2}}+\frac{\br{\bm{\lmb}}\cdot8\sqrt{8}\pi pq(4it)^{\frac{\nu-2}{2}}}{4\pi\log B_{0}}+\frac{P}{\lmb^{2}\br{\bm{\zeta}}}\Big|+\Big|\frac{\bm{\zeta}}{\bm{\lmb}}-1\Big|\cdot\frac{|t|^{\frac{\Re(\nu)-2}{2}}}{|\log|t||}+\frac{1}{|\log|t||^{\frac{1}{2}}}\cdot\frac{b^{2}}{\lmb^{3}}\\
 & \aleq o_{t\to0}(1)\cdot\frac{b^{2}}{\lmb^{3}}.
\end{align*}
This completes the proof of \eqref{eq:mod-est-rd-b/zeta}.
\end{proof}

\subsection{$\dot{\protect\calH}_{0}^{1}$ energy estimates}

In this subsection, we control the $\dot{\calH}_{0}^{1}$-norm of
$\eps$. We will control it via a suitable energy functional that
is coercive and adapted to the nonlinear flow.

Our discussion here is formal and rough. We will ignore the phase
correction terms, assume that the radiation $z$ is smooth, and will
not even distinguish the $\dot{H}^{1}$-norms and $\dot{\calH}^{1}$-norms.
Moreover, we will use (or, \emph{assume}) the following set of simpler
equations: denoting $w=Q^{\sharp}+z$, 
\begin{equation}
\left\{ \begin{aligned}\rd_{t}z & =-i\nabla\td E[z],\\
\rd_{t}(Q^{\sharp}+\eps^{\sharp}) & =-i(\nabla E[w+\eps^{\sharp}]-\nabla E[w])-iR_{Q^{\sharp},z},\\
R_{Q^{\sharp},z} & =\nabla E[w]-\nabla\td E[z].
\end{aligned}
\right.\label{eq:formal-eqn}
\end{equation}

We hope to find some energy functional, say $\calE$, satisfying 
\[
\calE\sim\|\eps^{\sharp}\|_{\dot{H}^{1}}^{2}
\]
with a good propagation property (to be able to close the bootstrap).
Usually, such $\calE$ satisfies 
\[
|\rd_{t}\calE|\aleq\tfrac{1}{|t|}\calE,
\]
with the additional structure for the terms of size $\sim\tfrac{1}{|t|}\calE$
(which we will call \emph{the} \emph{critical size}) that they are
explicitly integrable.\footnote{This is equivalent to saying that there is some explicit functional
$\calF$ in terms of $\lmb,\gmm,t,\dots$ (but not with $\eps$) that
satisfies $|\rd_{t}(\calE-\calF)|\ll\tfrac{1}{|t|}\calE$. Integrating
this will yield $\calE\approx\calF$.} Such an additional structure is necessary because a general error
of size $\sim\tfrac{1}{|t|}\calE$ cannot close the bootstrap. Thus,
we hope to find $\calE$ such that we can explicitly know all terms
of critical size $\sim\frac{1}{|t|}\calE$ in $\rd_{t}\calE$. In
other words, 
\[
\rd_{t}\calE=\text{(explicit)}+o_{t\to0}(\tfrac{1}{|t|}\calE),
\]
where the term $\text{(explicit)}$ is of size $\sim\frac{1}{|t|}\calE$,
contains the leading order contribution from the interaction term
$R_{Q^{\sharp},z}$, and is explicitly integrable in time. The term
$\text{(explicit)}$ must be present because we expect $\calE\sim\frac{|\bm{b}|^{2}\log b}{\lmb^{2}}$
(or $\eps\approx b(-i\tfrac{y^{2}}{4}Q)+\eta\rho$ at leading 
order) and thus $\rd_{t}\calE$ should detect the strong interaction
from $R_{Q^{\sharp},z}$. The goal here is to find a candidate for
such $\calE$. The discussion in the following is inspired by the
work of Jendrej--Lawrie--Rodriguez \cite{JendrejLawrieRodriguez2019arXiv}.

A naive candidate for $\calE$ would be to consider 
\[
\tfrac{1}{2}(\eps^{\sharp},\calL_{Q^{\sharp}}\eps^{\sharp})_{r}=\tfrac{1}{2\lmb^{2}}\|L_{Q}\eps\|_{L^{2}}^{2}\sim\|\eps^{\sharp}\|_{\dot{H}^{1}}^{2}.
\]
This is adapted to the linear flow $i\rd_{t}-\calL_{Q^{\sharp}}$.
In the computation of $\rd_{t}(\eps^{\sharp},\calL_{Q^{\sharp}}\eps^{\sharp})_{r}$,
when one replaces $\rd_{t}\eps^{\sharp}$ in $(\rd_{t}\eps^{\sharp},\calL_{Q^{\sharp}}\eps^{\sharp})_{r}$
by \eqref{eq:eps-sharp-eq}, it generates a lot of error terms (e.g.,
$(iR_{w^{\flat}}(\eps),\calL_{Q}^{\sharp}\eps^{\sharp})_{r}$) of
critical sizes.\footnote{Compared to the previous work \cite{KimKwon2019arXiv}, our size assumption
on $\eps$ here is very weak because we did not make a profile modification.
In fact, $R_{w^{\flat}}(\eps)$ possibly contains terms of size $O(b^{2})$
(this is of critical size) if we substitute $\eps\approx b(-i\tfrac{y^{2}}{4}Q)+\eta\rho$.} It even contains terms that cannot be estimated in our energy method,
e.g., $\|(\nabla\eps^{\sharp})^{2}(\eps^{\sharp})^{2}\|_{L^{1}}$.

Instead, as in \cite{RaphaelSzeftel2011JAMS} and \cite{KimKwon2019arXiv},
we consider the functional 
\[
E_{w}^{\mathrm{(qd)}}[\eps^{\sharp}]\coloneqq E[w+\eps^{\sharp}]-E[w]-(\nabla E[w],\eps^{\sharp})_{r}.
\]
The functional $E_{w}^{\mathrm{(qd)}}[f]$ is adapted to the \emph{nonlinear
flow} $i\rd_{t}f=\nabla E[w+f]-\nabla E[w]$ in the sense that 
\begin{align*}
 & \rd_{t}E_{w}^{\mathrm{(qd)}}[f]\\
 & =(\nabla E[w+f]-\nabla E[w],\rd_{t}f)_{r}+(\nabla E[w+f]-\nabla E[w]-\nabla^{2}E[w]f,\rd_{t}w)_{r}\\
 & =(\nabla E[w+f]-\nabla E[w]-\nabla^{2}E[w]f,\rd_{t}w)_{r}.
\end{align*}
Notice the full cancellation of the first inner product. For the second
inner product, we can exploit the known structure of $\rd_{t}w$.

In our case, we have 
\begin{align}
 & \rd_{t}E_{w}^{\mathrm{(qd)}}[\eps^{\sharp}]\nonumber \\
 & =(\nabla E[w+\eps^{\sharp}]-\nabla E[w],\rd_{t}\eps^{\sharp})_{r}\nonumber \\
 & \peq+(\nabla E[w+\eps^{\sharp}]-\nabla E[w]-\nabla^{2}E[w]\eps^{\sharp},\rd_{t}w)_{r}\nonumber \\
 & =(\nabla E[w+\eps^{\sharp}]-\nabla E[w],\rd_{t}(Q^{\sharp}+\eps^{\sharp}))_{r}-(\nabla^{2}E[w]\eps^{\sharp},\rd_{t}Q^{\sharp})_{r}\label{eq:discussion-1}\\
 & \quad+(\nabla E[w+\eps^{\sharp}]-\nabla E[w]-\nabla^{2}E[w]\eps^{\sharp},\rd_{t}z)_{r}.\nonumber 
\end{align}
The second line of RHS\eqref{eq:discussion-1} is safe because we
roughly have 
\begin{align*}
 & (\nabla E[w+\eps^{\sharp}]-\nabla E[w]-\nabla^{2}E[w]\eps^{\sharp},\rd_{t}z)_{r}\\
 & \quad=(R_{w}(\eps^{\sharp}),-i\nabla\td E[z])_{r}\aleq\|\eps^{\sharp}\|_{\dot{H}^{1}}^{2}\|\nabla\td E[z]\|_{L^{2}}\aleq\|\eps^{\sharp}\|_{\dot{H}^{1}}^{2}\ll\frac{1}{|t|}\|\eps^{\sharp}\|_{\dot{H}^{1}}^{2}.
\end{align*}
The first line of RHS\eqref{eq:discussion-1} has terms with critical
size. Using \eqref{eq:formal-eqn}, we have 
\[
(\nabla E[w+\eps^{\sharp}]-\nabla E[w],\rd_{t}(Q^{\sharp}+\eps^{\sharp}))_{r}=(\nabla E[w+\eps^{\sharp}]-\nabla E[w],-iR_{Q^{\sharp},z})_{r}.
\]
Since $\nabla E[w+\eps^{\sharp}]-\nabla E[w]$ contains the linear
term $\nabla^{2}E[w]\eps^{\sharp}$ and $\eps$ is of size $O(b)$,
this term is of critical size. We can deal with this term, by noticing
that 
\begin{align}
 & (\nabla E[w+\eps^{\sharp}]-\nabla E[w],-iR_{Q^{\sharp},z})_{r}\nonumber \\
 & =(i\rd_{t}(Q^{\sharp}+\eps^{\sharp}),-iR_{Q^{\sharp},z})_{r}\nonumber \\
 & =-(\rd_{t}Q^{\sharp},R_{Q^{\sharp},z})_{r}-\rd_{t}(\eps^{\sharp},R_{Q^{\sharp},z})_{r}+(\eps^{\sharp},\rd_{t}R_{Q^{\sharp},z})_{r}.\label{eq:discussion-2}
\end{align}
The first term of \eqref{eq:discussion-2} is now \emph{explicit}
in terms of $\frac{\lmb_{t}}{\lmb}$, $\gmm_{t}$, and $R_{Q^{\sharp},z}$.
The second term of \eqref{eq:discussion-2} is in a time derivative,
so we absorb it as a \emph{correction} to $E_{w}^{\mathrm{(qd)}}[\eps^{\sharp}]$.
Finally, the last term of \eqref{eq:discussion-2} can be written
as 
\[
(\eps^{\sharp},\rd_{t}R_{Q^{\sharp},z})_{r}=((\nabla^{2}E[w]-\nabla^{2}\td E[z])\eps^{\sharp},\rd_{t}z)_{r}+(\eps^{\sharp},\nabla^{2}E[w]\rd_{t}Q^{\sharp})_{r},
\]
where the latter term \emph{cancels} with the second term of RHS\eqref{eq:discussion-1}
and the former term can be shown to be an admissible error.

The above discussion motivates us to consider the new functional 
\begin{align}
\calE & \coloneqq E_{Q^{\sharp}+z}^{{\rm (qd)}}[\epsilon^{\sharp}]+(\nabla E[Q^{\sharp}+z]-\nabla\td E[z],\epsilon^{\sharp})_{r}\label{eq:def-calE}\\
 & =E[u]-E[Q^{\sharp}+z]-(\nabla\td E[z],\epsilon^{\sharp})_{r},\nonumber 
\end{align}
from which we expect 
\[
\rd_{t}\calE=-(\rd_{t}Q^{\sharp},R_{Q^{\sharp},z})_{r}+\text{(error)}.
\]
In the next lemma, we justify the above heuristics.
\begin{lem}[$\dot{\calH}_{0}^{1}$ energy estimates]
\label{lem:H1-energy-est}We have 
\begin{align}
|\calE-\tfrac{1}{2\lmb^{2}}\|L_{Q}\epsilon\|_{L^{2}}^{2}| & \aleq\tfrac{1}{\lmb^{2}}\cdot b^{2}|\log b|\cdot|t|^{\delta}\label{eq:energy-coercivity}\\
\Big|\rd_{t}\calE-8\sqrt{8}\pi\Re(\rd_{t}(\lmb e^{-i\gmm})\cdot pq(4it)^{\frac{\nu-2}{2}})\Big| & \aleq\tfrac{b}{\lmb^{4}}\cdot b^{2}|\log b|\cdot|t|^{\delta}.\label{eq:energy-differential-equality}
\end{align}
Moreover, using the sign condition $b\geq0$ (from the bootstrap hypothesis),
we have a differential inequality 
\begin{equation}
\rd_{t}\{\calE-2\pi\log B_{0}\cdot|\bm{b}/\br{\bm{\zeta}}|^{2}+o_{t\to0}(\tfrac{1}{\lmb^{2}}b^{2}|\log b|)\}\geq-o_{t\to0}(\tfrac{b}{\lmb^{4}}\cdot b^{2}|\log b|).\label{eq:energy-differential-inequality}
\end{equation}
\end{lem}

\begin{rem} \label{rem:complex-valued}
One of the main diffculties compared to the case of corotational wave
maps is the complex-valued nature of the modulation parameters $\bm{\zeta}$
and $\bm{b}$.

If $\nu$ is real, (under the bootstrap hypothesis) the imaginary
part $\eta$ of $\bm{b}$ is negligible and $\gmm$ can be treated
as a constant. Hence one essentially considers the dynamics of real-valued
parameters $\lmb$ and $\frac{b}{\lmb}$ (in the spirit of Corollary~\ref{cor:mod-est})
and in particular the term $\frac{P}{\lmb}$ of $(\frac{b}{\lmb})_{t}$
has a good sign. Then the approach of \cite{JendrejLawrieRodriguez2019arXiv}
using only the first two estimates of Lemma~\ref{lem:H1-energy-est},
i.e., \eqref{eq:energy-coercivity}-\eqref{eq:energy-differential-equality},
seems to be sufficient to close the bootstrap.

For a complex-valued $\nu$, the phase $\gmm$ rotates infinitely
many times and $\eta$ is no longer negligible. In particular, the
contribution of $P/\br{\bm{\zeta}}$ to the integration of $\bm{b}/\br{\bm{\zeta}}$
cannot be treated easily using only the sign condition of $P$. To
overcome this issue, we introduce a more refined \emph{differential
inequality} for $\calE$ and $|\bm{b}/\br{\bm{\zeta}}|^{2}$, which
is \eqref{eq:energy-differential-inequality}. This propagates the
initial smallness of $P$ backwards in time, and directly shows that
$P/\br{\bm{\zeta}}$ in \eqref{eq:mod-est-rd-b/zeta} is in fact \emph{negligible}.
\end{rem}

\begin{proof}
\textbf{Step 1.} Coercivity of $\calE$.

In this step, we prove \eqref{eq:energy-coercivity}. We recall that
$\calE$ can be written as 
\[
\calE=\{E[u]-E[Q^{\sharp}+z]-(\nabla E[Q^{\sharp}+z],\eps^{\sharp})_{r}\}+(\nabla E[Q^{\sharp}+z]-\nabla\td E[z],\epsilon^{\sharp})_{r}.
\]
The terms in the curly bracket collects the quadratic and higher order
terms of $\eps^{\sharp}$ in the expression $E[(Q^{\sharp}+z)+\eps^{\sharp}]$;
we have 
\begin{align*}
 & E[u]-E[Q^{\sharp}+z]-(\nabla E[Q^{\sharp}+z],\eps^{\sharp})_{r}\\
 & =\tfrac{1}{\lmb^{2}}\{E[u^{\flat}]-E[Q+z^{\flat}]-(\nabla E[Q+z^{\flat}],\eps)_{r}\}\\
 & =\frac{1}{\lmb^{2}}\Big\{\frac{1}{2}\|L_{Q}\eps\|_{L^{2}}^{2}+\sum_{\#\{j:\psi_{j}=\eps\}\geq3}\calM_{\ast}(\psi_{1},\psi_{2},\dots)\Big\},
\end{align*}
where each $\psi_{j}\in\{Q+z^{\flat},\eps\}$ and $\sum\calM_{\ast}$
denotes a linear combination of $\calM_{\ast}$'s. Thus we have proved
that 
\[
\calE=\frac{1}{2\lmb^{2}}\|L_{Q}\eps\|_{L^{2}}^{2}+\frac{1}{\lmb^{2}}\sum_{\#\{j:\psi_{j}=\eps\}\geq3}\calM_{\ast}(\psi_{1},\dots,\psi_{\ast})+(\nabla E[Q^{\sharp}+z]-\nabla\td E[z],\epsilon^{\sharp})_{r}.
\]
Now it suffices to show that the last two terms of the above display
are errors. First, the higher order terms of $\eps$ can be easily
estimated by the duality estimates (Lemma~\ref{lem:duality-estimates-Holder}):
\[
\sum_{\#\{j:\psi_{j}=\eps\}\geq3}|\calM_{\ast}(\psi_{1},\dots,\psi_{\ast})|\aleq\|\eps\|_{L^{6}}^{3}\aleq\|\eps\|_{\dot{\calH}_{0}^{1}}^{2}\|\eps\|_{L^{2}}\aleq b^{2}|\log b|\cdot|t|^{\delta}.
\]
Next, by scaling, \eqref{eq:def-RQ,zflat}, \eqref{eq:RQ,z-L2-est},
$|\tht_{z^{\flat}}|\aleq\lmb^{2}$, and \eqref{eq:weighted-L1}, we
have 
\begin{align*}
|(\nabla E[Q^{\sharp}+z]-\nabla\td E[z],\epsilon^{\sharp})_{r}| & =\tfrac{1}{\lmb^{2}}|(\nabla E[Q+z^{\flat}]-\nabla\td E[z^{\flat}],\epsilon)_{r}|\\
 & =\tfrac{1}{\lmb^{2}}|(R_{Q,z^{\flat}}+\tht_{z^{\flat}}Q,\eps)_{r}|\\
 & \aleq\tfrac{1}{\lmb^{2}}(\|R_{Q,z^{\flat}}\|_{L^{2}}\|\eps\|_{L^{2}}+\lmb^{2}\|Q\eps\|_{L^{1}})\\
 & \aleq\tfrac{1}{\lmb^{2}}(b^{\frac{3}{2}}|t|^{\delta}\cdot b^{\frac{1}{2}}|\log b|+b|t|\cdot b|\log b|)\\
 & \aleq\tfrac{1}{\lmb^{2}}b^{2}|\log b|\cdot|t|^{\delta}.
\end{align*}
This completes the proof of \eqref{eq:energy-coercivity}.

\smallskip
\textbf{Step 2.} Computation of $\rd_{t}\calE$.

In this step, we show that 
\begin{align}
\rd_{t}\calE & =-(R_{Q^{\sharp},z},(\rd_{t}+\tht_{z}i)Q^{\sharp})_{r}-\tfrac{1}{\lmb^{4}}(i\nabla\td E[z^{\flat}],(\calL_{Q+z^{\flat}}-\calL_{z^{\flat}}^{(\frkm)})\eps+R_{Q+z^{\flat}}(\eps))_{r}\label{eq:energy-identity-step2-claim}\\
 & \peq+\tfrac{1}{\lmb^{4}}(\nabla E[Q+z^{\flat}]-\nabla\td E[z^{\flat}]+\nabla^{2}\td E[z^{\flat}]\eps,i\Psi_{z^{\flat}})_{r}.\nonumber 
\end{align}
Our computation will be based on the Hamiltonian structure of \eqref{eq:CSS-m-equiv},
so it is convenient to consider the variables 
\[
\wh Q^{\sharp}\coloneqq e^{i\gmm_{z}}Q^{\sharp},\quad\wh{\eps}^{\sharp}=e^{i\gmm_{z}}\eps^{\sharp},\quad\wh z\coloneqq e^{i\gmm_{z}}z
\]
so that $u=\wh Q^{\sharp}+\wh{\eps}^{\sharp}+\wh z$. Note that we
have $\rd_{t}u=-i\nabla E[u]$ and $\rd_{t}\wh z=-i\nabla\td E[\wh z]-i\Psi_{\wh z}$,
where $\Psi_{\wh z}=e^{i\gmm_{z}}\Psi_{z}$. Thus $\rd_{t}(\wh Q^{\sharp}+\wh{\eps}^{\sharp})=-i(\nabla E[u]-\nabla\td E[\wh z])+i\Psi_{\wh z}$.

We start computing $\rd_{t}\calE$. First, by energy conservation,
we have 
\[
\rd_{t}E[u]=0.
\]
Next, by the definition of the functional derivative, we have 
\[
\rd_{t}E[Q^{\sharp}+z]=\rd_{t}E[\wh Q^{\sharp}+\wh z]=(\nabla E[\wh Q^{\sharp}+\wh z],\rd_{t}\wh Q^{\sharp}+\rd_{t}\wh z)_{r}.
\]
Next, as $\nabla^{2}\td E[\wh z]$ is symmetric, we have 
\begin{align*}
\rd_{t}(\nabla\td E[z],\eps^{\sharp})_{r} & =\rd_{t}(\nabla\td E[\wh z],\wh{\eps}^{\sharp})_{r}\\
 & =(\nabla^{2}\td E[\wh z]\rd_{t}\wh z,\wh{\eps}^{\sharp})_{r}+(\nabla\td E[\wh z],\rd_{t}\wh{\eps}^{\sharp})_{r}\\
 & =(\rd_{t}\wh z,\nabla^{2}\td E[\wh z]\wh{\eps}^{\sharp})_{r}-(\nabla\td E[\wh z],\rd_{t}\wh Q^{\sharp})_{r}+(\nabla\td E[\wh z],\rd_{t}(\wh Q^{\sharp}+\wh{\eps}^{\sharp}))_{r}.
\end{align*}
Hence, we have 
\begin{align*}
\rd_{t}\calE & =-(\nabla E[\wh Q^{\sharp}+\wh z]-\nabla\td E[\wh z],\rd_{t}\wh Q^{\sharp})_{r}\\
 & \peq-(\nabla E[\wh Q^{\sharp}+\wh z]+\nabla^{2}\td E[\wh z]\wh{\eps}^{\sharp},\rd_{t}\wh z)_{r}-(\nabla\td E[\wh z],\rd_{t}(\wh Q^{\sharp}+\wh{\eps}^{\sharp}))_{r}.
\end{align*}
Next, we substitute $\nabla E[\wh Q^{\sharp}+\wh z]-\nabla\td E[\wh z]=R_{\wh Q^{\sharp},\wh z}+\tht_{\wh z}\wh Q^{\sharp}$
with $(\tht_{\wh z}\wh Q^{\sharp},\rd_{t}\wh Q^{\sharp})_{r}=0$,
$\rd_{t}\wh z=-i\nabla\td E[\wh z]-i\Psi_{\wh z}$, and $\rd_{t}(\wh Q^{\sharp}+\wh{\eps}^{\sharp})=-i(\nabla E[u]-\nabla\td E[\wh z])+i\Psi_{\wh z}$
into the above display to obtain 
\begin{align*}
\rd_{t}\calE & =-(R_{\wh Q^{\sharp},\wh z},\rd_{t}\wh Q^{\sharp})_{r}-(i\nabla\td E[\wh z],\nabla E[u]-\nabla E[\wh Q^{\sharp}+\wh z]-\nabla^{2}\td E[\wh z]\wh{\eps}^{\sharp})_{r}\\
 & \peq+(\nabla E[\wh Q^{\sharp}+\wh z]-\nabla\td E[\wh z]+\nabla^{2}\td E[\wh z]\wh{\eps}^{\sharp},i\Psi_{\wh z})_{r}.
\end{align*}
This can be rearranged as 
\begin{align*}
\rd_{t}\calE & =-(R_{Q^{\sharp},z},(\rd_{t}+\tht_{z}i)Q^{\sharp})_{r}-(i\nabla\td E[z],(\calL_{Q^{\sharp}+z}-\calL_{z}^{(\frkm)})\eps^{\sharp}+R_{Q^{\sharp}+z}(\eps^{\sharp}))_{r}\\
 & \peq+(\nabla E[Q^{\sharp}+z]-\nabla\td E[z]+\nabla^{2}\td E[z]\eps^{\sharp},i\Psi_{z})_{r}.
\end{align*}
Using the $\sharp/\flat$-operations, \eqref{eq:energy-identity-step2-claim}
follows.

The proof of \eqref{eq:energy-differential-equality} is divided into
four steps. In Step 3, we estimate the error term involving $\Psi_{z^{\flat}}$.
In Step 4, we deal with the term $R_{Q^{\sharp},z}$ and extract the
main term of \eqref{eq:energy-differential-equality}. Steps 5 and
6, which take up the bulk of the proof, are devoted to the error
estimates for $R_{Q+z^{\flat}}(\eps)$ and $(\calL_{Q+z^{\flat}}-\calL_{z^{\flat}}^{(\frkm)})\eps$.

\smallskip
\textbf{Step 3. }Estimates of the error term involving $\Psi_{z}$.

In this step, we claim that 
\begin{equation}
(\nabla E[Q+z^{\flat}]-\nabla\td E[z^{\flat}]+\nabla^{2}\td E[z^{\flat}]\eps,i\Psi_{z^{\flat}})_{r}\aleq b^{3}|\log b|^{\frac{1}{2}}\cdot|t|^{\delta}.\label{eq:energy-identity-step3-claim-0}
\end{equation}
To see this, we rewrite 
\[
\text{LHS\eqref{eq:energy-identity-step3-claim-0}}=(R_{Q,z^{\flat}}+\tht_{z^{\flat}}Q+\nabla^{2}\td E[z^{\flat}]\eps,i\Psi_{z^{\flat}})_{r}
\]
and estimate some easy error terms; we use \eqref{eq:RQ,z-L2-est}
and \eqref{eq:Psi_zflat-L2} to have 
\[
(R_{Q,z^{\flat}},i\Psi_{z^{\flat}})_{r}\aleq\|R_{Q,z^{\flat}}\|_{L^{2}}\|\Psi_{z^{\flat}}\|_{L^{2}}\aleq b^{3}|t|^{2\delta}
\]
and we use \eqref{eq:tht-z-flat-bound} and \eqref{eq:Psi_zflat-Hdot1}
to have 
\[
(\tht_{z^{\flat}}Q,i\Psi_{z^{\flat}})_{r}\aleq|\tht_{z^{\flat}}|\|y^{1-}Q\|_{L^{2}}\|\tfrac{1}{y^{1-}}\Psi_{z^{\flat}}\|_{L^{2}}\aleq\lmb^{2}\cdot b^{2-}|t|^{\delta}\aleq b^{3}|t|^{1+\delta-}.
\]
To estimate the last term, we recall that 
\[
\nabla^{2}\td E[z^{\flat}]\eps=\calL_{z^{\flat}}^{(\frkm)}\eps=-\Delta^{(\frkm)}\eps+\sum_{\#\{j:\psi_{j}=\eps\}=1}\calN_{\ast}^{(\frkm)}(\psi_{1},\psi_{2},\dots),
\]
where each $\psi_{j}\in\{z^{\flat},\eps\}$ and $\sum\calN_{\ast}^{(\frkm)}$
denote a linear combination of $\calN_{\ast}$'s. By the duality
relation \eqref{eq:duality-relations}, we have 
\[
(\nabla^{2}\td E[z^{\flat}]\eps,i\Psi_{z^{\flat}})_{r}=(-\Delta^{(\frkm)}\eps,i\Psi_{z^{\flat}})_{r}+\sum_{\substack{\#\{j:\psi_{j}=\eps\}=1\\
\#\{j:\psi_{j}=i\Psi_{z^{\flat}}\}=1
}
}\calM_{\ast}^{(\frkm)}(\psi_{1},\psi_{2},\dots),
\]
where each $\psi_{j}\in\{z^{\flat},\eps,i\Psi_{z^{\flat}}\}$ and
$\sum\calM_{\ast}^{(\frkm)}$ denotes a linear combination of $\calM_{\ast}$'s.
We now estimate using \eqref{eq:conv-rel-3} and \eqref{eq:Psi_zflat-Hdot1}:
\begin{align*}
(-\Delta^{(\frkm)}\eps,i\Psi_{z^{\flat}})_{r} & =(\rd_{y}\eps,\rd_{y}(i\Psi_{z^{\flat}}))_{r}+(\tfrac{4}{y^{2}}\eps,i\Psi_{z^{\flat}})_{r}\\
 & \aleq\|\eps\|_{\dot{\calH}_{0}^{1}}\|\langle\log_{-}y\rangle|\Psi_{z^{\flat}}|_{-1}\|_{L^{2}}\aleq b|\log b|^{\frac{1}{2}}\cdot b^{2}|t|^{\delta}\aleq b^{3}|\log b|^{\frac{1}{2}}\cdot|t|^{\delta}
\end{align*}
and (using Lemmas~\ref{lem:duality-estimates-Holder}--\ref{lem:duality-estimates-weighted-L1})
\begin{align*}
\calM_{\ast}^{(\frkm)}(\psi_{1},\psi_{2},\dots) & \aleq\|\tfrac{1}{\langle\log_{-}y\rangle}\eps\|_{L^{\infty}\cap yL^{2}}\|\langle\log_{-}y\rangle\Psi_{z^{\flat}}\|_{L^{\infty}\cap yL^{2}}(\|z^{\flat}\|_{L^{2}}^{2}+\|z^{\flat}\|_{L^{2}}^{4})\\
 & \aleq\|\eps\|_{\dot{\calH}_{0}^{1}}\|\langle\log_{-}y\rangle|\Psi_{z^{\flat}}|_{-1}\|_{L^{2}}\aleq b^{3}|\log b|^{\frac{1}{2}}\cdot|t|^{\delta}.
\end{align*}
This completes the proof of \eqref{eq:energy-identity-step3-claim-0}.

\smallskip
\textbf{Step 4.} Contribution from the interaction term $R_{Q^{\sharp},z}$.

In this step, we claim that 
\begin{equation}
\begin{aligned}-(R_{Q^{\sharp},z},(\rd_{t} & +\tht_{z}i)Q^{\sharp})_{r}\\
 & =8\sqrt{8}\pi\Re(\rd_{t}(\lmb e^{-i\gmm})\cdot pq(4it)^{\frac{\nu-2}{2}})+\tfrac{1}{\lmb^{4}}\cdot O(|t|^{\delta}\cdot b^{3}|\log b|).
\end{aligned}
\label{eq:energy-identity-step3-claim}
\end{equation}
Indeed, the leading order term can be extracted from: 
\begin{align*}
 & -(R_{Q^{\sharp},z},\rd_{t}Q^{\sharp})_{r}\\
 & =-\tfrac{1}{\lmb^{2}}(R_{Q,z^{\flat}},(-\tfrac{\lmb_{t}}{\lmb}\Lmb+\gmm_{t}i)Q)_{r}\\
 & =-8\sqrt{8}\pi(-\lmb_{t}\Re(e^{-i\gmm}pq(4it)^{\frac{\nu-2}{2}})-\lmb\gmm_{t}\Im(e^{-i\gmm}pq(4it)^{\frac{\nu-2}{2}}))\\
 & \quad+\tfrac{1}{\lmb^{2}}\cdot O((|\tfrac{\lmb_{t}}{\lmb}|+|\gmm_{t}|)\cdot b^{2}|t|^{\delta})\\
 & =8\sqrt{8}\pi\Re(\rd_{t}(\lmb e^{-i\gmm})\cdot pq(4it)^{\frac{\nu-2}{2}})+\tfrac{1}{\lmb^{4}}\cdot O(b^{3}|\log b|^{\frac{1}{2}}|t|^{\delta}),
\end{align*}
where we used Lemma~\ref{lem:RQ,z} and \eqref{eq:rough-control-lmb-gmm}.
Next, we use \eqref{eq:RQ,z-inner-prod-2} and $|\tht_{z^{\flat}}|\aleq\lmb^{2}$
to have 
\[
-(R_{Q^{\sharp},z},\tht_{z}iQ^{\sharp})_{r}=\tfrac{1}{\lmb^{4}}(R_{Q,z^{\flat}},\tht_{z^{\flat}}iQ)_{r}\aleq\tfrac{1}{\lmb^{4}}\cdot\lmb^{2}b^{2}|\log b|^{\frac{1}{2}}\aleq\tfrac{1}{\lmb^{4}}\cdot b^{2}|\log b|^{\frac{1}{2}}\cdot|t|^{\delta}.
\]
Combining the above two displays completes the proof of the claim
\eqref{eq:energy-identity-step3-claim}.

\smallskip
\textbf{Step 5.} Estimates for the quadratic error $R_{Q+z^{\flat}}(\eps)$.

In this step, we show 
\begin{equation}
(i\nabla\td E[z^{\flat}],R_{Q+z^{\flat}}(\eps))_{r}\aleq b^{3}|\log b|\cdot|t|^{\delta}.\label{eq:energy-identity-step4-claim}
\end{equation}
If we were to have sufficient regularity for the radiation $z$, say
$z^{\ast}\in H_{-2}^{2}$ (equivalently, $\Re(\nu)>1$), then the
following estimate

\[
(i\nabla\td E[z^{\flat}],R_{Q+z^{\flat}}(\eps))_{r}\aleq\|\nabla\td E[z^{\flat}]\|_{L^{2}}\|R_{Q+z^{\flat}}(\eps)\|_{L^{2}}\aleq\lmb^{2}\|\eps\|_{\dot{\calH}_{0}^{1}}^{2}
\]
gives a quick proof of \eqref{eq:energy-identity-step4-claim}. However,
in the full range $\Re(\nu)>0$, $\nabla\td E[z^{\flat}]$ no longer
belongs to $L^{2}$ due to the term $-\Delta^{(\frkm)}z^{\flat}$
of $\nabla\td E[z^{\flat}]$. Thus, we integrate $-\Delta^{(\frkm)}z^{\flat}$
by parts. This, in turn, requires several derivative estimates of the
nonlinearity $R_{Q+z^{\flat}}(\eps)$ and generates many technicalities
in the following estimates.

Recall the structure of $\nabla\td E[z^{\flat}]=-\Delta^{(\frkm)}z^{\flat}+\calN^{(\frkm)}(z^{\flat})$.
We first show that the nonlinear term $\calN^{(\frkm)}(z^{\flat})$
is safe. Indeed, $\calN^{(\frkm)}(z^{\flat})=V_{z^{\flat}}^{(\frkm)}z^{\flat}$
for some potential $V_{z^{\flat}}^{(\frkm)}$ (see \eqref{eq:def-V-2z}
below for the precise formula) and we have 
\[
\|V_{z^{\flat}}^{(\frkm)}\|_{L^{\infty}}\aleq\|z^{\flat}\|_{rL^{2}\cap L^{\infty}}^{2}(1+\|z^{\flat}\|_{L^{2}}^{2})\aleq\lmb^{2}
\]
in view of the duality estimates (Lemma~\ref{lem:duality-estimates-weighted-L1}).
Thus, we have the following pointwise bound: 
\begin{equation}
|\calN^{(\frkm)}(z^{\flat})|\aleq\lmb^{2}|z^{\flat}|.\label{eq:calN-2-ptwise-bound}
\end{equation}
Using this and \eqref{eq:rough-control-tmp3}, we see that the contribution
of $\calN^{(\frkm)}(z^{\flat})$ is safe: 
\[
|(i\calN^{(\frkm)}(z^{\flat}),R_{Q+z^{\flat}}(\eps))_{r}|\aleq\lmb^{2}\|z^{\flat}\|_{L^{2}}\|R_{Q+z^{\flat}}(\eps)\|_{L^{2}}\aleq\lmb^{2}\|\eps\|_{\dot{\calH}_{0}^{1}}^{2}\aleq b^{3}|\log b|\cdot|t|.
\]

The rest of this step is devoted to estimate 
\[
(-i\Delta^{(\frkm)}z^{\flat},R_{Q+z^{\flat}}(\eps))_{r}.
\]
We first decompose 
\[
R_{Q+z^{\flat}}(\eps)=V_{R,2}Q+V_{R,2}z^{\flat}+V_{R,1}\eps,
\]
where $V_{R,1}$ and $V_{R,2}$ are potential parts of $R_{Q+z^{\flat}}(\eps)$.
More precisely, $V_{R,1}$ is a linear combination of the potentials
$\Re(\br{\psi_{1}}\psi_{2})$, $\tfrac{1}{y^{2}}A_{\tht}[\psi_{1},\psi_{2}]A_{\tht}[\psi_{3},\psi_{4}]$,
and $\int_{y}^{\infty}A_{\tht}[\psi_{1},\psi_{2}]\Re(\br{\psi_{3}}\psi_{4})\frac{dy'}{y'}$,
where in each expression there is at least one $\psi_{j}=\eps$.
Similarly, each expression of $V_{R,2}$ contains at least two $\psi_{j}=\eps$.

\uline{Case A}. Estimate of $(-i\Delta^{(\frkm)}z^{\flat},V_{R,2}Q)_{r}$.

In this case, although $\Delta^{(\frkm)}z^{\flat}\notin L^{2}$, we
can hope for certain integrability of $Q\Delta^{(\frkm)}z^{\flat}$ in
view of $\Delta^{(\frkm)}=\rd_{-}^{(\frkm)}\rd_{+}^{(\frkm)}$ and
\eqref{eq:z-flat-degenerate-H4-bound}. Indeed, by the duality estimate
(see the proof of \eqref{eq:rough-control-tmp3}), we have 
\begin{align*}
|(-i\Delta^{(\frkm)}z^{\flat},V_{R,2}Q)_{r}| & \aleq\|\langle\log_{-}y\rangle Q\Delta^{(\frkm)}z^{\flat}\|_{L^{2}\cap yL^{1}}\|\eps\|_{L^{2}}\|\tfrac{1}{\langle\log_{-}y\rangle}\eps\|_{L^{\infty}\cap yL^{2}}\\
 & \aleq\|\tfrac{\langle\log_{-}y\rangle}{\langle y\rangle^{2}}|\rd_{-}^{(\frkm)}z^{\flat}|_{-1}\|_{L^{2}\cap yL^{1}}\|\eps\|_{L^{2}}\|\eps\|_{\dot{\calH}_{0}^{1}}\\
 & \aleq b^{\frac{3}{2}}|t|^{\delta}\cdot b^{\frac{1}{2}}\cdot b|\log b|^{\frac{1}{2}}\aleq b^{3}|\log b|\cdot|t|^{\delta}.
\end{align*}

\uline{Case B}. Estimate of $(-i\Delta^{(\frkm)}z^{\flat},V_{R,2}z^{\flat})_{r}$.

In this case, we need to integrate by parts $-i\Delta^{(\frkm)}z^{\flat}=-i\rd_{-}^{(\frkm)\ast}\rd_{-}^{(\frkm)}z^{\flat}$.
We start with writing
\[
(-i\Delta^{(\frkm)}z^{\flat},V_{R,2}z^{\flat})_{r}=(i\rd_{-}^{(\frkm)}z^{\flat},(\rd_{y}V_{R,2})z^{\flat})_{r}+(i\rd_{-}^{(\frkm)}z^{\flat},V_{R,2}(\rd_{-}^{(\frkm)}z^{\flat}))_{r}.
\]
We then notice that 
\[
y\rd_{y}V_{R,2}=V_{R,2}^{(I)}+V_{R,2}^{(II)},
\]
where $V_{R,2}^{(I)}$ is a linear combination of $\Re(\br{\psi_{1}}\psi_{2})$,
$\tfrac{1}{y^{2}}A_{\tht}[\psi_{1},\psi_{2}]A_{\tht}[\psi_{3},\psi_{4}]$
and $V_{R,2}^{(II)}$ is a linear combination of $y\rd_{y}\Re(\br{\psi_{1}}\psi_{2})$,
$A_{\tht}[\psi_{1},\psi_{2}]\Re(\br{\psi_{3}}\psi_{4})$, where each
expression in $V_{R,2}^{(I)}$ and $V_{R,2}^{(II)}$ has at least
two $\psi_{j}$'s equal to $\eps$. Now we have 
\begin{align*}
 & |(-i\Delta^{(\frkm)}z^{\flat},V_{R,2}z^{\flat})_{r}|\\
 & \quad\aleq\||\rd_{-}^{(\frkm)}z^{\flat}|\cdot(|V_{R,2}^{(I)}|+|V_{R,2}|)|z^{\flat}|_{-1}\|_{L^{1}}+\|(\rd_{-}^{(\frkm)}z^{\flat})\tfrac{1}{y}V_{R,2}^{(II)}z^{\flat}\|_{L^{1}}.
\end{align*}
To estimate the first term, since the terms in $V_{R,2}^{(I)}$ have already
appeared in the expression of $V_{R,2}$, it can be easily estimated
by the duality estimate (as in the proof of \eqref{eq:rough-control-tmp3}):
\begin{align*}
\||\rd_{-}^{(\frkm)}z^{\flat}|\cdot(|V_{R,2}^{(I)}|+|V_{R,2}|)|z^{\flat}|_{-1}\|_{L^{1}} & \aleq\|\langle\log_{-}y\rangle|z^{\flat}|_{-1}\|_{L^{2}}^{2}\|\eps\|_{\dot{\calH}_{0}^{1}}^{2}\\
 & \aleq\lmb^{2}\|\eps\|_{\dot{\calH}_{0}^{1}}^{2}\aleq b^{3}|\log b|\cdot|t|.
\end{align*}
For the second term, we note that $V_{R,2}^{(II)}$ does not appear
in the duality estimates and, hence, needs to be estimated separately.
We have 
\begin{align*}
 & |(i\rd_{-}^{(\frkm)}z^{\flat},\rd_{y}\Re(\br{\eps}\eps)\cdot z^{\flat})_{r}|\\
 & \aleq\|\langle\log_{-}y\rangle\rd_{-}^{(\frkm)}z^{\flat}\|_{L^{2}}\|\rd_{y}\eps\|_{L^{2}}\|\tfrac{1}{\langle\log_{-}y\rangle}\eps\|_{L^{\infty}}\|z^{\flat}\|_{L^{\infty}}\\
 & \aleq\lmb^{2}\|\eps\|_{\dot{\calH}_{0}^{1}}^{2}\aleq b^{3}|\log b|\cdot|t|
\end{align*}
and 
\begin{align*}
 & |(i\rd_{-}^{(\frkm)}z^{\flat},A_{\tht}[\psi_{1},\psi_{2}]\Re(\br{\psi_{3}}\psi_{4})\cdot\tfrac{1}{y}z^{\flat})_{r}|\\
 & \aleq\||z^{\flat}|_{-1}\|_{L^{\infty}}^{2}\|\psi_{1}\psi_{2}\|_{L^{1}}\|\psi_{3}\psi_{4}\|_{L^{1}}\\
 & \aleq\||z^{\flat}|_{-1}\|_{L^{\infty}}^{2}\|\eps\|_{L^{2}}^{2}\aleq b^{3}|\log b|\cdot|t|^{\delta},
\end{align*}
where we used \eqref{eq:z-flat-H2-bound}.

\uline{Case C}. Estimate of $(-i\Delta^{(\frkm)}z^{\flat},V_{R,1}\eps)_{r}$.

As in the previous case, we need to integrate by parts. We start by
writing
\[
(-i\Delta^{(\frkm)}z^{\flat},V_{R,1}\eps)_{r}=(i\rd_{-}^{(\frkm)}z^{\flat},(\rd_{y}V_{R,1})\eps)_{r}+(i\rd_{-}^{(\frkm)}z^{\flat},V_{R,1}(\rd_{-}^{(\frkm)}\eps))_{r}.
\]
We then notice that 
\[
y\rd_{y}V_{R,1}=V_{R,1}^{(I)}+V_{R,1}^{(II)},
\]
where $V_{R,1}^{(I)}$ and $V_{R,1}^{(II)}$ are defined similarly
to $V_{R,2}^{(I)}$ and $V_{R,2}^{(II)}$ with the only difference
being that their expressions contain at least one $\psi_{j}=\eps$. Thus
\begin{align*}
 & |(-i\Delta^{(\frkm)}z^{\flat},V_{R,1}\eps)_{r}|\\
 & \quad\aleq\||\rd_{-}^{(\frkm)}z^{\flat}|\cdot(|V_{R,1}^{(I)}|+|V_{R,1}|)\cdot|\eps|_{-1}\|_{L^{1}}+\|(\rd_{-}^{(\frkm)}z^{\flat})\tfrac{1}{y}V_{R,1}^{(II)}\eps\|_{L^{1}}.
\end{align*}
To estimate the first term, since the terms in $V_{R,1}^{(I)}$ have already
appeared in the expression of $V_{R,1}$, it can be easily estimated
by the duality estimate (Lemma~\ref{lem:duality-estimates-Holder}):
\begin{align*}
 & \||\rd_{-}^{(\frkm)}z^{\flat}|\cdot(|V_{R,1}^{(I)}|+|V_{R,1}|)\cdot|\eps|_{-1}\|_{L^{1}}\\
 & \aleq\||\rd_{-}^{(\frkm)}z^{\flat}|\cdot|\eps|_{-1}\|_{L^{2}}\|\eps\psi\|_{L^{2}}\\
 & \aleq\|\langle\log_{-}y\rangle\rd_{-}^{(\frkm)}z^{\flat}\|_{L^{\infty}}\|\tfrac{1}{\langle\log_{-}y\rangle}|\eps|_{-1}\|_{L^{2}}\|\eps\psi\|_{L^{2}}\\
 & \aleq\|\langle\log_{-}y\rangle\rd_{-}^{(\frkm)}z^{\flat}\|_{L^{\infty}}\|\eps\|_{\dot{\calH}_{0}^{1}}^{2}\\
 & \aleq b^{3}|\log b|\cdot|t|^{\delta},
\end{align*}
where in the third inequality, we used 
\begin{align*}
\|\eps\psi\|_{L^{2}} & \aleq\|\eps Q\|_{L^{2}}+\|\eps z^{\flat}\|_{L^{2}}+\|\eps^{2}\|_{L^{2}}\\
 & \aleq\|\eps\|_{\dot{\calH}_{0}^{1}}(1+\|\langle\log_{-}y\rangle z^{\flat}\|_{L^{2}}+\|\eps\|_{L^{2}})\aleq\|\eps\|_{\dot{\calH}_{0}^{1}}
\end{align*}
and in the last one, we used \eqref{eq:z-flat-H2-bound}. For the second
term $\|(\rd_{-}^{(\frkm)}z^{\flat})\tfrac{1}{y}V_{R,1}^{(II)}\eps\|_{L^{1}}$,
we have 
\begin{align*}
 & |(i\rd_{-}^{(\frkm)}z^{\flat},\rd_{y}\Re(\br{\psi}\eps)\cdot\eps)_{r}|\\
 & \aleq\|(\rd_{-}^{(\frkm)}z^{\flat})\eps\|_{L^{\infty}}\|\rd_{y}(\psi\eps)\|_{L^{1}}\\
 & \aleq\|\langle\log_{-}y\rangle|z^{\flat}|_{-1}\|_{L^{\infty}}\|\tfrac{1}{\langle\log_{-}y\rangle}\eps\|_{L^{\infty}}\|\rd_{y}(\psi\eps)\|_{L^{1}}\\
 & \aleq\|\langle\log_{-}y\rangle|z^{\flat}|_{-1}\|_{L^{\infty}}\|\eps\|_{\dot{\calH}_{0}^{1}}(\|\eps\|_{\dot{\calH}_{0}^{1}}+\lmb\|\eps\|_{L^{2}})\\
 & \aleq b^{3}|\log b|\cdot|t|^{\delta},
\end{align*}
where, in the second-to-last inequality, we used 
\begin{align*}
\|\rd_{y}(\psi\eps)\|_{L^{1}} & \aleq\|\rd_{y}(Q\eps)\|_{L^{1}}+\|\rd_{y}(z^{\flat}\eps)\|_{L^{1}}+\|\rd_{y}(\eps^{2})\|_{L^{1}}\\
 & \aleq\|\eps\|_{\dot{\calH}_{0}^{1}}(1+\|z^{\flat}\|_{L^{2}}+\|\eps\|_{L^{2}})+\|\rd_{y}z^{\flat}\|_{L^{2}}\|\eps\|_{L^{2}}\\
 & \aleq\|\eps\|_{\dot{\calH}_{0}^{1}}+\lmb\|\eps\|_{L^{2}},
\end{align*}
and in the last one, we used \eqref{eq:z-flat-H2-bound}, and we have
\begin{align*}
 & |(i\rd_{-}^{(\frkm)}z^{\flat},\tfrac{1}{y}A_{\tht}[\psi_{1},\psi_{2}]\Re(\br{\psi_{3}}\psi_{4})\cdot\eps)_{r}|\\
 & \aleq\|\langle\log_{-}y\rangle(\rd_{-}^{(\frkm)}z^{\flat})\eps\|_{L^{\infty}}\|\tfrac{1}{y\langle\log_{-}y\rangle}\eps\psi\|_{L^{1}}\\
 & \aleq\|\langle\log_{-}y\rangle^{2}|z^{\flat}|_{-1}\|_{L^{\infty}}\|\tfrac{1}{\langle\log_{-}y\rangle}\eps\|_{L^{\infty}}\|\tfrac{1}{y\langle\log_{-}y\rangle}\eps\|_{L^{2}}\\
 & \aleq\|\langle\log_{-}y\rangle^{2}|z^{\flat}|_{-1}\|_{L^{\infty}}\|\eps\|_{\dot{\calH}_{0}^{1}}^{2}\\
 & \aleq b^{3}|\log b|\cdot|t|^{\delta},
\end{align*}
where, in the last inequality, we used \eqref{eq:z-flat-H2-bound}.
This completes the proof of the claim \eqref{eq:energy-identity-step4-claim}.

\smallskip
\textbf{Step 6.} Estimates for the degenerate linear term $(\calL_{Q+z^{\flat}}-\calL_{z^{\flat}}^{(\frkm)})\epsilon$.

In this step, we show 
\begin{equation}
(i\nabla\td E[z^{\flat}],(\calL_{Q+z^{\flat}}-\calL_{z^{\flat}}^{(\frkm)})\epsilon)_{r}\aleq b^{3}|\log b|\cdot|t|^{\delta}.\label{eq:energy-identity-step5-claim}
\end{equation}
Here, we need to exploit the cancelation of the potentials of $\calL_{Q+z^{\flat}}-\calL_{z^{\flat}}^{(\frkm)}$.
Moreover, we need to use higher Sobolev bounds of $\rd_{-}^{(\frkm)}z^{\flat}$
(cf. Case A of Step 4), which utilize the degeneracy of $\rd_{-}^{(\frkm)}z$
near the origin, recorded in \eqref{eq:rd-z-ss-degen} and \eqref{eq:z-flat-degenerate-H4-bound}.

We start by expanding the term $(\calL_{Q+z^{\flat}}-\calL_{z^{\flat}}^{(\frkm)})\epsilon$
using the Coulomb form \eqref{eq:CSS-m-equiv}. First, we decompose
\begin{align*}
\calL_{Q+z^{\flat}}\eps & =-(\rd_{yy}+\tfrac{1}{y}\rd_{y})\eps+V_{Q,z^{\flat}}\eps+V_{\eps,Q+z^{\flat}}(Q+z^{\flat}),\\
\calL_{z^{\flat}}^{(\frkm)}\eps & =-(\rd_{yy}+\tfrac{1}{y}\rd_{y})\eps+V_{z^{\flat}}^{(\frkm)}\eps+V_{\eps,z^{\flat}}^{(\frkm)}z^{\flat},
\end{align*}
where 
\begin{align}
V_{Q,z^{\flat}} & =(\tfrac{A_{\tht}[Q+z^{\flat}]}{y})^{2}+A_{t}[Q+z^{\flat}]-|Q+z^{\flat}|^{2},\label{eq:def-VQ,z}\\
V_{\eps,Q+z^{\flat}} & =\tfrac{A_{\tht}[Q+z^{\flat}]\cdot2A_{\tht}[Q+z^{\flat},\eps]}{y^{2}}-\tint y{\infty}A_{\tht}[Q+z^{\flat}]\cdot2\Re((\br{Q+z^{\flat}})\eps)\tfrac{dy'}{y'}\label{eq:def-Veps,Qz}\\
 & \peq-\tint y{\infty}2A_{\tht}[Q+z^{\flat},\eps]|Q+z^{\flat}|^{2}\tfrac{dy'}{y'}-2\Re((\br{Q+z^{\flat}})\eps)\nonumber 
\end{align}
and 
\begin{align}
V_{z^{\flat}}^{(\frkm)} & =(\tfrac{\frkm+A_{\tht}[z^{\flat}]}{y})^{2}+A_{t}^{(\frkm)}[z^{\flat}]-|z^{\flat}|^{2},\label{eq:def-V-2z}\\
V_{\eps,z^{\flat}}^{(\frkm)} & =\tfrac{(\frkm+A_{\tht}[z^{\flat}])\cdot2A_{\tht}[z^{\flat},\eps]}{y^{2}}-\tint y{\infty}(\frkm+A_{\tht}[z^{\flat}])\cdot2\Re(\br{z^{\flat}}\eps)\tfrac{dy'}{y'}\label{eq:def-V-2epsz}\\
 & \peq-\tint y{\infty}2A_{\tht}[z^{\flat},\eps]|z^{\flat}|^{2}\tfrac{dy'}{y'}-2\Re(\br{z^{\flat}}\eps).\nonumber 
\end{align}
With the above notation, we can write 
\[
(\calL_{Q+z^{\flat}}-\calL_{z^{\flat}}^{(\frkm)})\eps=(V_{Q,z^{\flat}}-V_{z^{\flat}}^{(\frkm)})\eps+V_{\eps,Q+z^{\flat}}Q+(V_{\eps,Q+z^{\flat}}-V_{\eps,z^{\flat}}^{(\frkm)})z^{\flat}.
\]
Our aim is to show that each term after taking the inner product with
$i\nabla\td E[z^{\flat}]$ satisfies \eqref{eq:energy-identity-step5-claim}.
We also recall that 
\[
i\nabla\td E[z^{\flat}]=-i\Delta^{(\frkm)}z^{\flat}+i\calN^{(\frkm)}(z^{\flat})\quad\text{and}\quad|\calN^{(\frkm)}(z^{\flat})|\aleq\lmb^{2}|z^{\flat}|.
\]

\uline{Case A}. Estimate of $(i\nabla\td E[z^{\flat}],V_{\eps,Q+z^{\flat}}Q)_{r}$.

This case is very similar to Case A of Step 4. Indeed, we have by
the duality estimates (Lemma~\ref{lem:duality-estimates-Holder})
\begin{align*}
|(i\nabla\td E[z^{\flat}],V_{\eps,Q+z^{\flat}}Q)_{r}| & \aleq\|Q\nabla\td E[z^{\flat}]\|_{L^{2}}\|(Q+z^{\flat})\eps\|_{L^{2}}\\
 & \aleq\|\langle y\rangle^{-2}(|\rd_{-}^{(\frkm)}z^{\flat}|_{-1}+\lmb^{2}|z^{\flat}|)\|_{L^{2}}\|\eps\|_{\dot{\calH}_{0}^{1}}\\
 & \aleq b^{3}|\log b|^{\frac{1}{2}}\cdot|t|^{\delta},
\end{align*}
where we used \eqref{eq:z-flat-degenerate-H4-bound} and \eqref{eq:z-flat-H2-bound}.

\uline{Case B}. Estimate of $(i\nabla\td E[z^{\flat}],(V_{Q,z^{\flat}}-V_{z^{\flat}}^{(\frkm)})\eps)_{r}$.

First, we will decompose $V_{Q,z^{\flat}}-V_{z^{\flat}}^{(\frkm)}$
into two parts: 
\[
V_{Q,z^{\flat}}-V_{z^{\flat}}^{(\frkm)}=V_{B}^{(I)}+V_{B}^{(II)}.
\]
Using the formulas \eqref{eq:def-VQ,z} and \eqref{eq:def-V-2z} of
$V_{Q,z^{\flat}}$ and $V_{z^{\flat}}^{(\frkm)}$, and $|A_{\tht}[z^{\flat}]|_{1}\aleq1+y^{2}|z^{\flat}|^{2}$,
we can write $V_{Q,z^{\flat}}-V_{z^{\flat}}^{(\frkm)}$ as 
\begin{align*}
V_{Q,z^{\flat}}-V_{z^{\flat}}^{(\frkm)} & =\tfrac{2+A_{\tht}[Q]}{y^{2}}\cdot O_{|\cdot|_{1}}(1+y^{2}|z^{\flat}|^{2})+\tfrac{A_{\tht}[Q,z^{\flat}]}{y^{2}}\cdot O_{|\cdot|_{1}}(1+y^{2}|z^{\flat}|^{2})\\
 & \peq-\tint y{\infty}O_{L^{\infty}}(1)\cdot Q^{2}\tfrac{dy'}{y'}-\tint y{\infty}(2+A_{\tht}[Q])|z^{\flat}|^{2}\tfrac{dy'}{y'}\\
 & \peq-\tint y{\infty}O_{L^{\infty}}(1)\cdot2\Re(Qz^{\flat})\tfrac{dy'}{y'}-\tint y{\infty}2A_{\tht}[Q,z^{\flat}]|z^{\flat}|^{2}\tfrac{dy'}{y'}\\
 & \peq-Q^{2}-2\Re(Qz^{\flat}),
\end{align*}
where $f=O_{|\cdot|_{1}}(g)$ means some function $f$ satisfying
$|f|_{1}\aleq g$. Let $V_{B}^{(I)}$ collect the terms containing
the pointwise product $\Re(Qz^{\flat})$; let $V_{B}^{(II)}$ collect
the terms with spatial decay in the spirit of $Q+|2+A_{\tht}[Q]|\aleq\langle y\rangle^{-2}$.
More precisely, $V_{B}^{(I)}$ is the sum of the second, fifth, sixth,
and eights term, and $V_{B}^{(II)}$ is the sum of the first, third,
fourth, and seventh term.

Next, we claim that 
\begin{align}
\||V_{B}^{(I)}|_{1}\|_{L^{2}} & \aleq\|Qz^{\flat}\|_{L^{2}}+\lmb\|Qz^{\flat}\|_{L^{1}}\aleq b\cdot|t|^{\delta},\label{eq:claim-VBI}\\
|V_{B}^{(II)}|_{1} & \aleq\tfrac{1}{\langle y\rangle^{2}}(\tfrac{1}{y^{2}}+\lmb^{2}).\label{eq:claim-VBII}
\end{align}
Indeed, the proof of \eqref{eq:claim-VBI} follows from the pointwise
bound 
\[
|V_{B}^{(I)}|_{1}\aleq(\tfrac{1}{y^{2}}\tint 0y\cdot y'dy'+\tint y{\infty}\cdot\tfrac{dy'}{y'}+1)(|Qz^{\flat}|)+(\tint y{\infty}\cdot\tfrac{dy'}{y'}+1)(|A_{\tht}[Q,z^{\flat}]||z^{\flat}|^{2})
\]
and (using \eqref{eq:z-flat-H2-bound} and \eqref{eq:z-flat-H1-bound-2})
\begin{align*}
\||V_{B}^{(I)}|_{1}\|_{L^{2}} & \aleq\|Qz^{\flat}\|_{L^{2}}+\|A_{\tht}[Q,z^{\flat}]|z^{\flat}|^{2}\|_{L^{2}}\\
 & \aleq\|Qz^{\flat}\|_{L^{2}}+\|Qz^{\flat}\|_{L^{1}}\|z^{\flat}\|_{L^{2}}\|z^{\flat}\|_{L^{\infty}}\\
 & \aleq\|Qz^{\flat}\|_{L^{2}}+\lmb\|Qz^{\flat}\|_{L^{1}}\aleq b\cdot|t|^{\delta}.
\end{align*}
The proof of \eqref{eq:claim-VBII} follows from 
\[
|V_{B}^{(II)}|_{1}\aleq\tfrac{1}{\langle y\rangle^{2}}(\tfrac{1}{y^{2}}+\|z^{\flat}\|_{L^{\infty}}^{2})\aleq\tfrac{1}{\langle y\rangle^{2}}(\tfrac{1}{y^{2}}+\lmb^{2}).
\]

We turn to estimate $(i\nabla\td E[z^{\flat}],(V_{Q,z^{\flat}}-V_{z^{\flat}}^{(\frkm)})\eps)_{r}$.
As in Case B of Step 4, we need to integrate by parts the term $-\Delta^{(\frkm)}z^{\flat}$
of $\nabla\td E[z^{\flat}]$; we note the integration by parts formula
\begin{align}
(i\nabla\td E[z^{\flat}],V\eps)_{r} & =(i\rd_{-}^{(\frkm)}z^{\flat},V(\rd_{-}^{(\frkm)}\eps)+(\rd_{y}V)\eps)_{r}+(i\calN^{(\frkm)}(z^{\flat}),V\eps)_{r}\label{eq:int-by-pts-VB}
\end{align}
for a potential $V\in\{V_{B}^{(I)},V_{B}^{(II)}\}$. To estimate the
contribution of $V_{B}^{(I)}$, by \eqref{eq:int-by-pts-VB}, \eqref{eq:claim-VBI},
\eqref{eq:calN-2-ptwise-bound}, and \eqref{eq:z-flat-H2-bound},
we have 
\begin{align*}
 & |(i\nabla\td E[z^{\flat}],V_{B}^{(I)}\eps)_{r}|\\
 & \aleq\big(\|\langle\log_{-}y\rangle\rd_{-}^{(\frkm)}z^{\flat}\|_{L^{\infty}}\|\tfrac{1}{\langle\log_{-}y\rangle}|\eps|_{-1}\|_{L^{2}}\\
 & \qquad+\lmb^{2}\|\langle\log_{-}y\rangle z^{\flat}\|_{L^{2}}\|\tfrac{1}{\langle\log_{-}y\rangle}\eps\|_{L^{\infty}}\big)\cdot\||V_{B}^{(I)}|_{1}\|_{L^{2}}\\
 & \aleq\big(\|\langle\log_{-}y\rangle\rd_{-}^{(\frkm)}z^{\flat}\|_{L^{\infty}}+\lmb^{2}\big)\cdot\||V_{B}^{(I)}|_{1}\|_{L^{2}}\|\eps\|_{\dot{\calH}_{0}^{1}}\\
 & \aleq b|t|^{\delta}\cdot b|t|^{\delta}\cdot\|\eps\|_{\dot{\calH}_{0}^{1}}\aleq b^{3}|\log b|^{\frac{1}{2}}\cdot|t|^{\delta},
\end{align*}
To estimate the contribution of $V_{B}^{(II)}$, by \eqref{eq:int-by-pts-VB},
\eqref{eq:claim-VBII}, \eqref{eq:calN-2-ptwise-bound}, \eqref{eq:z-flat-H2-bound},
and \eqref{eq:z-flat-degenerate-H4-bound}, we have 
\begin{align*}
 & |(i\nabla\td E[z^{\flat}],V_{B}^{(II)}\eps)_{r}|\\
 & \aleq\|y\langle\log_{-}y\rangle|V_{B}^{(II)}|_{1}\cdot(\tfrac{1}{y}|\rd_{-}^{(\frkm)}z^{\flat}|+\lmb^{2}|z^{\flat}|)\|_{L^{2}}\|\tfrac{1}{\langle\log_{-}y\rangle}|\eps|_{-1}\|_{L^{2}}\\
 & \aleq(\|\tfrac{\langle\log_{-}y\rangle}{y\langle y\rangle^{2}}(\tfrac{1}{y}|\rd_{-}^{(\frkm)}z^{\flat}|+\lmb^{2}|z^{\flat}|)\|_{L^{2}}\\
 & \qquad+\lmb^{2}\|\tfrac{y\langle\log_{-}y\rangle}{\langle y\rangle^{2}}(\tfrac{1}{y}|\rd_{-}^{(\frkm)}z^{\flat}|+\lmb^{2}|z^{\flat}|)\|_{L^{2}})\cdot\|\eps\|_{\dot{\calH}_{0}^{1}}\\
 & \aleq b^{3}|\log b|^{\frac{1}{2}}\cdot|t|^{\delta}.
\end{align*}

\uline{Case C}. Estimate of $(i\nabla\td E[z^{\flat}],(V_{\eps,Q+z^{\flat}}-V_{\eps,z^{\flat}}^{(\frkm)})z^{\flat})_{r}$.

First, we will decompose $V_{\eps,Q+z^{\flat}}-V_{\eps,z^{\flat}}^{(\frkm)}$
into three parts: 
\[
V_{\eps,Q+z^{\flat}}-V_{\eps,z^{\flat}}^{(\frkm)}=V_{C}^{(I)}+V_{C}^{(II)}+V_{C}^{(III)}.
\]
Using the formulas \eqref{eq:def-Veps,Qz} and \eqref{eq:def-V-2epsz}
of $V_{\eps,Q+z^{\flat}}$ and $V_{\eps,z^{\flat}}^{(\frkm)}$, we
have 
\begin{align*}
 & V_{\eps,Q+z^{\flat}}-V_{\eps,z^{\flat}}^{(\frkm)}\\
 & =\tfrac{A_{\tht}[Q+z^{\flat}]\cdot2A_{\tht}[Q,\eps]}{y^{2}}+\tfrac{(2+A_{\tht}[Q])\cdot2A_{\tht}[z^{\flat},\eps]}{y^{2}}+\tfrac{2A_{\tht}[Q,z^{\flat}]\cdot2A_{\tht}[z^{\flat},\eps]}{y^{2}}\\
 & \peq-\tint y{\infty}A_{\tht}[Q+z^{\flat}]\cdot2\Re(Q\eps)\tfrac{dy'}{y'}-\tint y{\infty}(2+A_{\tht}[Q])\cdot2\Re(\br{z^{\flat}}\eps)\tfrac{dy'}{y'}-\tint y{\infty}2A_{\tht}[Q,z^{\flat}]\cdot2\Re(\br{z^{\flat}}\eps)\tfrac{dy'}{y'}\\
 & \peq-\tint y{\infty}2A_{\tht}[Q,\eps]|Q+z^{\flat}|^{2}\tfrac{dy'}{y'}-\tint y{\infty}2A_{\tht}[z^{\flat},\eps]Q^{2}\tfrac{dy'}{y'}-\tint y{\infty}2A_{\tht}[z^{\flat},\eps]\cdot2\Re(Qz^{\flat})\tfrac{dy'}{y'}\\
 & \peq-2\Re(Q\eps).
\end{align*}
Let $V_{C}^{(I)}$ collect the terms having the pointwise product
$\Re(Q\eps)$; let $V_{C}^{(II)}$ collect the terms having the pointwise
product $\Re(Qz^{\flat})$; let $V_{C}^{(III)}$ collect the terms
with spatial decay. More precisely, $V_{C}^{(I)}$ is the sum of the
first, fourth, seventh, and tenth terms, $V_{C}^{(II)}$ is the sum
of the third, sixth, and ninth terms, and $V_{C}^{(III)}$ is the
sum of the second, fifth, and eight terms.

Next, we claim that 
\begin{align}
\||V_{C}^{(I)}|_{1}\|_{L^{1}+y^{-2}L^{\infty}} & \aleq\|Q\eps\|_{L^{1}}+\|\eps\|_{\dot{\calH}_{0}^{1}}\aleq\|\eps\|_{\dot{\calH}_{0}^{1}}^{1-},\label{eq:claim-VCI}\\
\||V_{C}^{(II)}|_{1}\|_{L^{2}} & \aleq b^{\frac{1}{2}}|t|^{\delta}\cdot\|\eps\|_{\dot{\calH}_{0}^{1}}\label{eq:claim-VCII}\\
\|y|V_{C}^{(III)}|_{1}\|_{L^{2}} & \aleq\lmb\|\eps\|_{\dot{\calH}_{0}^{1}}.\label{eq:claim-VCIII}
\end{align}
Indeed, the proof of \eqref{eq:claim-VCI} (terms having $\Re(Q\eps)$)
follows from 
\[
|V_{C}^{(I)}|_{1}\aleq(\tfrac{1}{y^{2}}\tint 0y\cdot y'dy'+1)(|Q\eps|)+(\tint y{\infty}\cdot\tfrac{dy'}{y'}+1)(|A_{\tht}[Q,\eps]||Q+z^{\flat}|^{2})+Q|y\rd_{y}\eps|
\]
and the estimates\footnote{In the second display, we used the embedding $y^{-1}L^{2}\embed L^{1}+y^{-2}L^{\infty}$,
which follows from $L^{\infty}\cap y^{2}L^{1}\embed yL^{2}$ and taking
the dual.} 
\begin{align*}
||\tfrac{1}{y^{2}}\tint 0y|Q\eps|y'dy'\|_{y^{-2}L^{\infty}} & \aleq\|Q\eps\|_{L^{1}},\\
\|Q\cdot y\rd_{y}\eps\|_{L^{1}+y^{-2}L^{\infty}} & \aleq\|Q\cdot y\rd_{y}\eps\|_{y^{-1}L^{2}}\aleq\|\rd_{y}\eps\|_{L^{2}}\aleq\|\eps\|_{\dot{\calH}_{0}^{1}},\\
\|(\tint y{\infty}\cdot\tfrac{dy'}{y'}+1)(A_{\tht}[Q,\eps]|Q+z^{\flat}|^{2})\|_{L^{1}} & \aleq\|A_{\tht}[Q,\eps]|Q+z^{\flat}|^{2}\|_{L^{1}}\aleq\|Q\eps\|_{L^{1}}.
\end{align*}
The proof of \eqref{eq:claim-VCII} (terms having $\Re(Qz^{\flat})$)
follows from 
\begin{align*}
|V_{C}^{(II)}|_{1} & \aleq\|A_{\tht}[Q,z^{\flat}]\|_{L^{\infty}}\cdot(\tfrac{1}{y^{2}}\tint 0y\cdot y'dy'+\tint y{\infty}\cdot\tfrac{dy'}{y'}+1)(|z^{\flat}\eps|)\\
 & \quad+(\tint y{\infty}\cdot\tfrac{dy'}{y'}+1)(|A_{\tht}[z^{\flat},\eps]||Qz^{\flat}|)\\
 & \aleq\|Qz^{\flat}\|_{L^{1}}(\tfrac{1}{y^{2}}\tint 0y\cdot y'dy'+\tint y{\infty}\cdot\tfrac{dy'}{y'}+1)(|z^{\flat}\eps|)\\
 & \quad+\|z^{\flat}\|_{L^{\infty}}(\tint y{\infty}\cdot\tfrac{dy'}{y'}+1)(|QA_{\tht}[z^{\flat},\eps]|)
\end{align*}
and 
\begin{align*}
\||V_{C}^{(II)}|_{1}\|_{L^{2}} & \aleq\|Qz^{\flat}\|_{L^{1}}\|z^{\flat}\eps\|_{L^{2}}+\|z^{\flat}\|_{L^{\infty}}\|QA_{\tht}[z^{\flat},\eps]\|_{L^{2}}\\
 & \aleq(\|Qz^{\flat}\|_{L^{1}}+\lmb)\|z^{\flat}\eps\|_{L^{2}}\\
 & \aleq b^{\frac{1}{2}}|t|^{\delta}\cdot\|\langle\log_{-}y\rangle z^{\flat}\|_{L^{2}}\|\tfrac{1}{\langle\log_{-}y\rangle}\eps\|_{L^{\infty}}\\
 & \aleq b^{\frac{1}{2}}|t|^{\delta}\cdot\|\eps\|_{\dot{\calH}_{0}^{1}}.
\end{align*}
The proof of \eqref{eq:claim-VCIII} (terms with spatial decay) follows
from
\begin{align*}
 & y|V_{C}^{(III)}|_{1}\\
 & \aleq\tfrac{1}{\langle y\rangle^{2}y}\tint 0y|z^{\flat}\eps|y'dy'+y(\tint y{\infty}\cdot\tfrac{dy'}{y'}+1)(\tfrac{1}{\langle y'\rangle^{2}}|z^{\flat}\eps|)+y(\tint y{\infty}\cdot\tfrac{dy'}{y'}+1)(|Q^{2}A_{\tht}[z^{\flat},\eps]|)\\
 & \aleq(\tfrac{1}{y^{2}}\tint 0y\cdot y'dy'+\tint y{\infty}\cdot\tfrac{dy'}{y'}+1)(\tfrac{1}{\langle y\rangle}|z^{\flat}\eps|)+(\tint y{\infty}\cdot\tfrac{dy'}{y'}+1)(|\tfrac{1}{y^{2}}A_{\tht}[\tfrac{1}{\langle y'\rangle}z^{\flat},\eps]|)
\end{align*}
and 
\begin{align*}
\|y|V_{C}^{(III)}|_{1}\|_{L^{2}} & \aleq\|\tfrac{1}{\langle y\rangle}z^{\flat}\eps\|_{L^{2}}+\|\tfrac{1}{y^{2}}A_{\tht}[\tfrac{1}{\langle y'\rangle}z^{\flat},\eps]\|_{L^{2}}\\
 & \aleq\|\tfrac{1}{\langle y\rangle}z^{\flat}\eps\|_{L^{2}}\aleq\|z^{\flat}\|_{L^{\infty}}\|\tfrac{1}{y\langle\log_{-}y\rangle}\eps\|_{L^{2}}\aleq\lmb\|\eps\|_{\dot{\calH}_{0}^{1}}.
\end{align*}

We turn to estimate $(i\nabla\td E[z^{\flat}],(V_{\eps,Q+z^{\flat}}-V_{\eps,z^{\flat}}^{(\frkm)})z^{\flat})_{r}$.
As in Case C of Step 4, we integrate by parts: 
\begin{align}
(i\nabla\td E[z^{\flat}],V\eps)_{r} & =(i\rd_{-}^{(\frkm)}z^{\flat},V(\rd_{-}^{(\frkm)}z^{\flat})+(\rd_{y}V)z^{\flat})_{r}+(i\calN^{(\frkm)}(z^{\flat}),Vz^{\flat})_{r}\label{eq:int-by-pts-VC}
\end{align}
for a potential $V\in\{V_{C}^{(I)},V_{C}^{(II)},V_{C}^{(III)}\}$.
For each contribution, by \eqref{eq:int-by-pts-VC} and \eqref{eq:calN-2-ptwise-bound},
we have (using \eqref{eq:z-flat-H2-bound})
\begin{align*}
 & |(i\nabla\td E[z^{\flat}],V_{C}^{(I)}z^{\flat})_{r}|\\
 & \aleq(\||z^{\flat}|_{-1}\|_{L^{\infty}\cap yL^{2}}^{2}+\lmb^{2}\||z^{\flat}|^{2}\|_{L^{\infty}\cap y^{2}L^{1}})\||V_{C}^{(I)}|_{1}\|_{L^{1}+y^{-2}L^{\infty}}\\
 & \aleq(\||z^{\flat}|_{-1}\|_{L^{\infty}\cap yL^{2}}^{2}+\lmb^{2}\|z^{\flat}\|_{L^{\infty}\cap yL^{2}}^{2})\cdot\|\eps\|_{\dot{\calH}_{0}^{1}}^{1-}\\
 & \aleq(b|t|^{\delta})^{2}\cdot(b|\log b|^{\frac{1}{2}})^{1-}\aleq b^{3}|\log b|\cdot|t|^{\delta},
\end{align*}
and (using \eqref{eq:z-flat-H2-bound})
\begin{align*}
 & |(i\nabla\td E[z^{\flat}],V_{C}^{(II)}z^{\flat})_{r}|\\
 & \aleq(\||z^{\flat}|_{-1}^{2}\|_{L^{2}}+\lmb^{2}\||z^{\flat}|^{2}\|_{L^{2}})\||V_{C}^{(II)}|_{1}\|_{L^{2}}\\
 & \aleq(\||z^{\flat}|_{-1}\|_{L^{\infty}}\||z^{\flat}|_{-1}\|_{L^{2}}+\lmb^{2}\|z^{\flat}\|_{L^{\infty}}\|z^{\flat}\|_{L^{2}})\||V_{C}^{(II)}|_{1}\|_{L^{2}}\\
 & \aleq(b^{\frac{3}{2}}|t|^{\delta})\cdot(b^{\frac{1}{2}}|t|^{\delta}\|\eps\|_{\dot{\calH}_{0}^{1}})\aleq b^{3}|\log b|\cdot|t|^{\delta},
\end{align*}
and (using \eqref{eq:z-flat-H2-bound})
\begin{align*}
 & |(i\nabla\td E[z^{\flat}],V_{C}^{(III)}z^{\flat})_{r}|\\
 & \aleq\|\tfrac{1}{y}|z^{\flat}|_{-1}^{2}+\lmb^{2}|z^{\flat}\cdot\tfrac{1}{y}z^{\flat}|\|_{L^{2}}\|y|V_{C}^{(III)}|_{1}\|_{L^{2}}\\
 & \aleq(\||z^{\flat}|_{-1}\|_{L^{\infty}\cap yL^{2}}^{2}+\lmb^{2}\|z^{\flat}\|_{L^{\infty}\cap yL^{2}}^{2})\|y|V_{C}^{(III)}|_{1}\|_{L^{2}}\\
 & \aleq(b^{2}|t|^{2\delta})\cdot\lmb\|\eps\|_{\dot{\calH}_{0}^{1}}\aleq b^{3}|\log b|\cdot|t|^{\delta}.
\end{align*}
This completes the proof of the claim \eqref{eq:energy-identity-step5-claim}
and finishes the proof of \eqref{eq:energy-differential-equality}.

\smallskip
\textbf{Step 7.} Differential inequality.

In this step, we prove \eqref{eq:energy-differential-inequality}.
First, we show by adding a small correction to $\calE$ that the energy
identity \eqref{eq:energy-differential-equality} can be rewritten
as 
\begin{equation}
\rd_{t}\{\calE+o_{t\to0}(\tfrac{1}{\lmb^{2}}b^{2}|\log b|)\}-\Re(\rd_{t}\br{\bm{\zeta}}\cdot8\sqrt{8}\pi pq(4it)^{\frac{\nu-2}{2}})=o_{t\to0}(\tfrac{b}{\lmb^{4}}\cdot b^{2}|\log b|).\label{eq:energy-identity-step6-claim}
\end{equation}
Indeed, we have 
\begin{align*}
\Re(\rd_{t}\br{\bm{\lmb}}\cdot(4it)^{\frac{\nu-2}{2}}) & =\Re(\rd_{t}\br{\bm{\zeta}}\cdot(4it)^{\frac{\nu-2}{2}})\\
 & \quad-\rd_{t}\{\Re((\br{\bm{\zeta}-\bm{\lmb}})\cdot(4it)^{\frac{\nu-2}{2}})\}+\Re((\br{\bm{\zeta}-\bm{\lmb}})\cdot\rd_{t}(4it)^{\frac{\nu-2}{2}}).
\end{align*}
Using \eqref{eq:control-zeta-bound}, we see that 
\begin{align*}
|\Re((\br{\bm{\zeta}-\bm{\lmb}})\cdot(4it)^{\frac{\nu-2}{2}})| & \aleq|\bm{\zeta}-\bm{\lmb}|\cdot|t|^{\frac{\Re(\nu)-2}{2}}\aleq\lmb|t|^{\frac{\Re(\nu)-2}{2}}|\log b|^{-\frac{1}{2}},\\
|\Re(\br{\bm{\zeta}-\bm{\lmb}})\cdot\rd_{t}(4it)^{\frac{\nu-2}{2}})| & \aleq|\bm{\zeta}-\bm{\lmb}|\cdot|t|^{\frac{\Re(\nu)-2}{2}-1}\aleq\lmb|t|^{\frac{\Re(\nu)-2}{2}-1}|\log b|^{-\frac{1}{2}}.
\end{align*}
Since $\lmb|t|^{\frac{\Re(\nu)-2}{2}}\sim\frac{1}{\lmb^{2}}b^{2}|\log b|$,
$\frac{b}{\lmb^{2}}\sim|t|^{-1}$, and $|\log b|^{-\frac{1}{2}}=o_{t\to0}(1)$,
this completes the proof of \eqref{eq:energy-identity-step6-claim}.

Next, we compute 
\begin{align*}
 & -\rd_{t}(2\pi\log B_{0}\cdot|\bm{b}/\br{\bm{\zeta}}|^{2})\\
 & =-4\pi\log B_{0}\cdot\Re((\br{\bm{b}}/\bm{\zeta})\cdot\rd_{t}(\bm{b}/\br{\bm{\zeta}}))-2\pi\tfrac{(B_{0})_{t}}{B_{0}}|\bm{b}/\br{\bm{\zeta}}|^{2}\\
 & =-4\pi\log B_{0}\cdot\Re((\br{\bm{b}}/\bm{\zeta})\cdot\rd_{t}(\bm{b}/\br{\bm{\zeta}}))+O(\tfrac{b}{\lmb^{4}}\cdot b^{2}|\log b|^{\frac{1}{2}}).
\end{align*}
where we used $|\frac{(B_{0})_{t}}{B_{0}}|=|\frac{1}{2|t|}-\frac{\lmb_{t}}{\lmb}|\aleq\frac{b}{\lmb^{2}}|\log b|^{\frac{1}{2}}$,
which is a consequence of \eqref{eq:rough-control-lmb-gmm}. Next,
we use the modulation estimate \eqref{eq:mod-est-rd-b/zeta}, the
sign condition $b\geq0$, as well as the almost positivity \eqref{eq:P-positivity}
of $P$ to continue the previous display as 
\begin{align*}
 & =\Re\Big\{8\sqrt{8}\pi pq(4it)^{\frac{\nu-2}{2}}\cdot(\br{\bm{b}}/\bm{\zeta})+4\pi\log B_{0}\frac{\br{\bm{b}}P}{\lmb^{2}|\bm{\zeta}|^{2}}\Big\}-o_{t\to0}\Big(\frac{b}{\lmb^{4}}\cdot b^{2}|\log b|\Big)\\
 & \geq\Re\big\{8\sqrt{8}\pi pq(4it)^{\frac{\nu-2}{2}}\cdot(\br{\bm{b}}/\bm{\zeta})\big\}-o_{t\to0}(\tfrac{b}{\lmb^{4}}\cdot b^{2}|\log b|).
\end{align*}
Adding the above display to \eqref{eq:energy-identity-step6-claim}
and using \eqref{eq:mod-est-rd-zeta}, we have 
\begin{align*}
 & \rd_{t}\{\calE-2\pi\log B_{0}\cdot|\bm{b}/\br{\bm{\zeta}}|^{2}+o_{t\to0}(\tfrac{1}{\lmb^{2}}b^{2}|\log b|)\}\\
 & \quad\geq\Re\{(\rd_{t}\br{\bm{\zeta}}+(\br{\bm{b}}/\bm{\zeta}))\cdot8\sqrt{8}\pi pq(4it)^{\frac{\nu-2}{2}}\}-o_{t\to0}(\tfrac{b}{\lmb^{4}}\cdot b^{2}|\log b|)\\
 & \quad=-o_{t\to0}(\tfrac{b}{\lmb^{4}}\cdot b^{2}|\log b|).
\end{align*}
This completes the proof of \eqref{eq:energy-differential-inequality}.
\end{proof}

\subsection{$L^{2}$ energy estimates}

In order to close the $L^{2}$-bound of the bootstrap hypothesis,
we perform the energy method for $\eps$.
\begin{lem}[$L^{2}$ energy estimate]
We have 
\begin{equation}
\Big|\rd_{t}\|\eps^{\sharp}\|_{L^{2}}^{2}\Big|\aleq\tfrac{b}{\lmb^{2}}\cdot b|\log b|^{\frac{3}{2}}.\label{eq:L2-energy-est}
\end{equation}
\end{lem}

\begin{proof}
We start by rewriting the $\eps^{\sharp}$-equation as 
\[
i\rd_{t}\eps^{\sharp}+\Delta^{(0)}\eps^{\sharp}+\tht_{z}\eps^{\sharp}=-i\rd_{t}Q^{\sharp}+(\calL_{Q^{\sharp}+z}+\Delta^{(0)})\eps^{\sharp}+R_{Q^{\sharp}+z}(\eps^{\sharp})+R_{Q^{\sharp},z}-\Psi_{z}.
\]
Taking the inner product with $i\eps^{\sharp}$ and using the scaling
properties of the $\sharp/\flat$-operations, we obtain
\[
\tfrac{1}{2}\rd_{t}\|\eps^{\sharp}\|_{L^{2}}^{2}=-(\rd_{t}Q^{\sharp},\eps^{\sharp})_{r}+\tfrac{1}{\lmb^{2}}((\calL_{Q+z^{\flat}}+\Delta^{(0)})\eps+R_{Q+z^{\flat}}(\eps)+R_{Q,z^{\flat}}-\Psi_{z^{\flat}},i\eps)_{r}.
\]
For the modulation term, we use \eqref{eq:rough-control-lmb-gmm}
and $\|Q\eps\|_{L^{1}}\aleq b|\log b|$ (by \eqref{eq:weighted-L1})
to have 
\[
|-(\rd_{t}Q^{\sharp},\eps^{\sharp})_{r}|\aleq\tfrac{1}{\lmb^{2}}(|\tfrac{\lmb_{s}}{\lmb}|+|\gmm_{s}|)\|Q\eps\|_{L^{1}}\aleq\tfrac{1}{\lmb^{2}}\cdot b|\log b|^{\frac{1}{2}}\cdot b|\log b|\aleq\tfrac{b}{\lmb^{2}}\cdot b|\log b|^{\frac{3}{2}}.
\]
For the linear term $(\calL_{Q+z^{\flat}}+\Delta^{(0)})\eps$, we
notice that $((\calL_{Q+z^{\flat}}+\Delta^{(0)})\eps,i\eps)_{r}$
is a linear combination of $\calM_{\ast}(\psi_{1},\dots,\psi_{\ast})$
with $\psi_{j}\in\{Q+z^{\flat},\eps,i\eps\}$ and $\#\{j:\psi_{j}\in\{\eps,i\eps\}\}=2$.
Thus this contribution can be estimated by (cf. the proof of \eqref{eq:rough-control-tmp3})
\[
|\tfrac{1}{\lmb^{2}}((\calL_{Q+z^{\flat}}+\Delta^{(0)})\eps,i\eps)_{r}|\aleq\tfrac{1}{\lmb^{2}}\|\eps\|_{\dot{\calH}_{0}^{1}}^{2}\aleq\tfrac{b}{\lmb^{2}}\cdot b|\log b|.
\]
For the nonlinear term, we use \eqref{eq:rough-control-tmp3} to have
\[
|\tfrac{1}{\lmb^{2}}(R_{Q+z^{\flat}}(\eps),i\eps)_{r}|\aleq\tfrac{1}{\lmb^{2}}\|R_{Q+z^{\flat}}(\eps)\|_{L^{2}}\|\eps\|_{L^{2}}\aleq\tfrac{1}{\lmb^{2}}\|\eps\|_{\dot{\calH}_{0}^{1}}^{2}\|\eps\|_{L^{2}}\aleq\tfrac{b}{\lmb^{2}}\cdot b^{\frac{3}{2}-}.
\]
For the interaction and the $\Psi_{z^{\flat}}$ term, we use \eqref{eq:RQ,z-L2-est}
and \eqref{eq:def-Psi-zflat} to have 
\begin{align*}
|\tfrac{1}{\lmb^{2}}(R_{Q,z^{\flat}}-\Psi_{z^{\flat}},i\eps)_{r}| & \aleq\tfrac{1}{\lmb^{2}}(\|R_{Q,z^{\flat}}\|_{L^{2}}+\|\Psi_{z^{\flat}}\|_{L^{2}})\|\eps\|_{L^{2}}\\
 & \qquad\qquad\aleq\tfrac{1}{\lmb^{2}}\cdot b^{\frac{3}{2}}|t|^{\delta}\cdot b^{\frac{1}{2}}|\log b|\aleq\tfrac{b}{\lmb^{2}}\cdot b.
\end{align*}
This completes the proof.
\end{proof}

\subsection{\label{subsec:Closing-bootstrap}Closing bootstrap}

Finally, we finish the proof of Proposition~\ref{prop:bootstrap}.
\begin{proof}[Proof of Proposition~\ref{prop:bootstrap}]
\ 

\textbf{Step 1.} Smallness of $P$.

In this step, we claim that $P$ enjoys the following smallness: 
\begin{equation}
|P|\leq o_{t\to0}(b^{2}).\label{eq:closing-bootstrap-step1-claim}
\end{equation}
Note that the lower bound $P\geq-o_{t\to0}(b^{2})$ was already proved
in \eqref{eq:P-positivity}. The point here is that we can also obtain
the \emph{upper bound} $P\leq o_{t\to0}(b^{2})$. We will use \eqref{eq:energy-differential-inequality},
which is a consequence of the \emph{sign condition} $b\geq0$.

We integrate \eqref{eq:energy-differential-inequality} backwards
in time on $[t,\tau]$ to have 
\begin{align}
 & \Big\{\calE-2\pi\log B_{0}\cdot|\bm{b}/\br{\bm{\zeta}}|^{2}+o_{t\to0}(\tfrac{1}{\lmb^{2}}b^{2}|\log b|)\Big\}(t)\nonumber \\
 & \leq\Big\{\calE-2\pi\log B_{0}\cdot|\bm{b}/\br{\bm{\zeta}}|^{2}+o_{t\to0}(\tfrac{1}{\lmb^{2}}b^{2}|\log b|)\Big\}(\tau)\label{eq:closing-bootstrap-step1-1}\\
 & \quad+\int_{t}^{\tau}o_{t'\to0}\Big(\frac{b}{\lmb^{4}}\cdot b^{2}|\log b|\Big)dt'.\nonumber 
\end{align}
On one hand, we have at the initial time $t=\tau$ that 
\begin{equation}
|\calE-2\pi\log B_{0}\cdot|\bm{b}/\br{\bm{\zeta}}|^{2}|\leq o_{\tau\to0}(\tfrac{1}{\lmb^{2}(\tau)}b^{2}(\tau)|\log b(\tau)|),\label{eq:closing-bootstrap-step1-2}
\end{equation}
which follows from \eqref{eq:energy-coercivity}, \eqref{eq:initial-data-prop},
\eqref{eq:control-zeta-bound}, and \eqref{eq:proximity-b-bqnu}:
\begin{align*}
\calE(\tau) & =\tfrac{1}{2\lmb^{2}}\|L_{Q}\eps\|_{L^{2}}^{2}+o_{\tau\to0}(\tfrac{1}{\lmb^{2}}b^{2}|\log b|)\\
 & =\tfrac{1}{2\lmb^{2}}\cdot4\pi\log B_{0}\cdot|\bm{b}_{q,\nu}|^{2}+o_{\tau\to0}(\tfrac{1}{\lmb^{2}}b^{2}|\log b|)\\
 & =2\pi\log B_{0}\cdot|\bm{b}/\br{\bm{\zeta}}|^{2}+o_{\tau\to0}(\tfrac{1}{\lmb^{2}}b^{2}|\log b|).
\end{align*}
On the other hand, since $\frac{b}{\lmb^{4}}\cdot b^{2}|\log b|\sim|t|^{\mu}|\log|t||^{-1}$
with $\mu=\Re(\nu)-1>-1$ (by \eqref{eq:conv-rel-1}), we have 
\begin{equation}
\int_{t}^{\tau}\frac{b}{\lmb^{4}}\cdot b^{2}|\log b|dt'\aleq\frac{|t|^{\mu+1}}{|\log|t||}\sim\frac{1}{\lmb^{2}}\cdot b^{2}|\log b|.\label{eq:closing-bootstrap-step1-3}
\end{equation}
Substituting \eqref{eq:closing-bootstrap-step1-2} and \eqref{eq:closing-bootstrap-step1-3}
into \eqref{eq:closing-bootstrap-step1-1} yields the following control
at time $t$: 
\begin{equation}
\Big\{\calE-2\pi\log B_{0}\cdot|\bm{b}/\br{\bm{\zeta}}|^{2}+o_{t\to0}(\tfrac{1}{\lmb^{2}}b^{2}|\log b|)\Big\}(t)\leq o_{t\to0}\Big(\frac{1}{\lmb^{2}}\cdot b^{2}|\log b|\Big).\label{eq:closing-bootstrap-step1-4}
\end{equation}
Applying \eqref{eq:P-positivity} and \eqref{eq:energy-coercivity},
we obtain \eqref{eq:closing-bootstrap-step1-claim}.

\smallskip
\textbf{Step 2.} Closing the bootstrap bounds for the modulation parameters.

In this step, we prove the first and second rows of \eqref{eq:bootstrap-conclusion}.
The key is to use the smallness of $P$ \eqref{eq:closing-bootstrap-step1-claim},
which allows us to ignore the $P/\br{\bm{\zeta}}$-term in \eqref{eq:mod-est-rd-b/zeta}.
Using \eqref{eq:closing-bootstrap-step1-claim}, 
\begin{equation}
4\pi\log B_{0}=2\pi(\Re(\nu)+1)|\log|t||+O(\log|\log|t||),\label{eq:logB0-asymp}
\end{equation}
and \eqref{eq:def-p}, we can rewrite \eqref{eq:mod-est-rd-zeta}
and \eqref{eq:mod-est-rd-b/zeta} as 
\begin{align*}
|\rd_{t}\bm{\zeta}+(\bm{b}/\br{\bm{\zeta}})| & \aleq o_{t\to0}\Big(\frac{|t|^{\frac{\Re(\nu)}{2}}}{|\log|t||}\Big),\\
\Big|\rd_{t}(\bm{b}/\br{\bm{\zeta}})+\frac{4\sqrt{2}\Gmm(\tfrac{\nu}{2}+2)}{\Re(\nu)+1}\cdot q\frac{(4it)^{\frac{\nu}{2}-1}}{|\log|t||}\Big| & \aleq o_{t\to0}\Big(\frac{|t|^{\frac{\Re(\nu)}{2}-1}}{|\log|t||}\Big).
\end{align*}
Note that, using \eqref{eq:proximity-b-bqnu} and \eqref{eq:control-zeta-bound},
we have at the initial time $t=\tau$
\begin{align*}
\frac{\bm{b}(\tau)}{\br{\bm{\zeta}}(\tau)} & =\frac{\bm{b}_{q,\nu}(\tau)}{\br{\bm{\lmb}_{q,\nu}}(\tau)}+o_{\tau\to0}\Big(\frac{|\tau|^{\frac{\Re(\nu)}{2}}}{|\log|\tau||}\Big),\\
\bm{\zeta}(\tau) & =\bm{\lmb}_{q,\nu}(\tau)+o_{\tau\to0}\Big(\frac{|\tau|^{\frac{\Re(\nu)}{2}+1}}{|\log|\tau||}\Big).
\end{align*}
Integrating the above modulation equations backwards in time, we obtain
\begin{align}
(\bm{b}/\br{\bm{\zeta}})(t) & =\sqrt{2}i\frac{(\tfrac{\nu}{2}+1)\Gmm(\tfrac{\nu}{2})}{\Re(\nu)+1}\cdot q\frac{(4it)^{\frac{\nu}{2}}}{|\log|t||}+o_{t\to0}\Big(\frac{|t|^{\frac{\Re(\nu)}{2}}}{|\log|t||}\Big),\label{eq:closing-bootstrap-step2-1}\\
\bm{\zeta}(t) & =-\frac{\sqrt{2}}{4}\frac{\Gmm(\tfrac{\nu}{2})}{\Re(\nu)+1}\cdot q\frac{(4it)^{\frac{\nu}{2}+1}}{|\log|t||}+o_{t\to0}\Big(\frac{|t|^{\frac{\Re(\nu)}{2}+1}}{|\log|t||}\Big).\label{eq:closing-bootstrap-step2-2}
\end{align}
Applying \eqref{eq:control-zeta-bound} to \eqref{eq:closing-bootstrap-step2-2}
proves the first row of \eqref{eq:bootstrap-conclusion}. On the other
hand, multiplying \eqref{eq:closing-bootstrap-step2-1} and \eqref{eq:closing-bootstrap-step2-2}
gives 
\[
\bm{b}(t)=b_{q,\nu}(t)+o_{t\to0}\Big(\frac{|t|^{\Re(\nu)+1}}{|\log|t||}\Big),
\]
which combined with \eqref{eq:b-bound} yields the second row of \eqref{eq:bootstrap-conclusion}.

\smallskip
\textbf{Step 3.} Closing the bootstrap bounds for $\|L_{Q}\eps\|_{L^{2}}$.

In this step, we prove the third row of \eqref{eq:bootstrap-conclusion}.
This is an easy consequence of what we have proved in Step 1. Indeed,
by \eqref{eq:energy-coercivity}, \eqref{eq:logB0-asymp}, \eqref{eq:control-zeta-bound},
and \eqref{eq:proximity-b-bqnu}, we have 
\begin{align*}
 & \Big|\|L_{Q}\eps(t)\|_{L^{2}}^{2}-2\pi(\Re(\nu)+1)|\bm{b}_{q,\nu}(t)|^{2}\Big|\\
 & \quad\leq2\lmb^{2}\Big|\calE-2\pi\log B_{0}\cdot|\bm{b}/\br{\bm{\zeta}}|^{2}\Big|+o_{t\to0}(b^{2}|\log b|)).
\end{align*}
Thanks to \eqref{eq:closing-bootstrap-step1-4}, we obtain 
\[
\Big|\|L_{Q}\eps(t)\|_{L^{2}}^{2}-2\pi(\Re(\nu)+1)|\bm{b}_{q,\nu}(t)|^{2}\Big|=o_{t\to0}(b^{2}|\log b|).
\]
Taking the square root, this concludes the proof of the third row
of \eqref{eq:bootstrap-conclusion}.

\smallskip
\textbf{Step 4.} Closing the bootstrap bounds for $\|\eps\|_{L^{2}}$
and conclusion.

In this step, we prove the last row of \eqref{eq:bootstrap-conclusion}.
We integrate \eqref{eq:L2-energy-est} to have 
\[
\|\eps(t)\|_{L^{2}}^{2}\leq\|\eps(\tau)\|_{L^{2}}^{2}+O\Big(\int_{t}^{\tau}\frac{b}{\lmb^{2}}\cdot b|\log b|^{\frac{3}{2}}dt'\Big).
\]
Proceeding as in Step 1, while using \eqref{eq:initial-data-prop}
for $\|\eps(\tau)\|_{L^{2}}^{2}$ and $\frac{b}{\lmb^{2}}\cdot b|\log b|^{\frac{3}{2}}\sim|t|^{\mu}|\log|t||^{-\frac{1}{2}}$
with $\mu=\Re(\nu)>-1$, we obtain 
\[
\|\eps(t)\|_{L^{2}}^{2}\aleq b_{q,\nu}(\tau)+b(t)|\log b(t)|^{\frac{3}{2}}\aleq o_{t\to0}(b|\log b|^{2})\aleq o_{t\to0}(b_{q,\nu}|\log b_{q,\nu}|^{2}).
\]
This completes the proof of the last row of \eqref{eq:bootstrap-conclusion}
and hence the proof of Proposition~\ref{prop:bootstrap}.
\end{proof}
\bibliographystyle{abbrv}
\bibliography{References}

\end{document}